\documentclass{article}
\usepackage[margin=1in]{geometry}

\usepackage{amsmath}
\usepackage{amssymb}
\usepackage{amsthm}
\usepackage{bm}
\usepackage{mathrsfs}
\usepackage{enumerate}
\usepackage{color}
\usepackage{hyperref}
\hypersetup{
  colorlinks   = true,
  urlcolor     = blue,
  linkcolor    = blue,
  citecolor   = red
}
\usepackage{dsfont}
\usepackage[square,numbers]{natbib}

\def\R{\mathbb{R}}
\def\cR{\mathcal{R}}
\def\Z{\mathbb{Z}}
\def\bZ{\mathbf{Z}}
\def\P{\mathbb{P}}
\def\E{\mathbb{E}}
\def\cL{\mathcal{L}}
\def\N{\mathcal{N}}
\def\A{\mathbf{A}}
\def\D{\mathbf{D}}
\def\e{\mathbf{e}}

\def\x{\mathbf{x}}
\def\y{\mathbf{y}}
\def\a{\mathbf{a}}

\def\bUpsilon{\boldsymbol{\Upsilon}}
\def\b{\mathbf{b}}

\def\D{\mathbf{D}}
\def\M{\mathbf{M}}
\def\I{\mathbf{I}}
\def\T{\mathbf{T}}
\def\n{\mathbf{n}}
\def\X{\mathbf{X}}

\def\u{\mathbf{u}}
\def\v{\mathbf{v}}

\def\z{\mathbf{z}}
\def\s{\mathbf{s}}

\def\bR{\mathbf{R}}
\def\sr{\mathsf{r}}

\def\btheta{\boldsymbol{\theta}}
\def\bbeta{\boldsymbol{\eta}}
\def\bvarphi{\boldsymbol{\varphi}}
\def\boldeta{\boldsymbol{\eta}}

\def\bOmega{\boldsymbol{\Omega}}
\def\eps{\varepsilon}
\def\beps{\boldsymbol{\varepsilon}}
\def\d{\mathrm{d}}

\def\cE{\mathcal{E}}

\def\cT{\mathcal{T}}

\def\op{\mathrm{op}}
\def\tu{\Tilde{\u}}
\def\tx{\Tilde{\x}}
\def\sP{\mathsf{P}}

\def\deps{\partial_\eps|_{\eps = 0}}

\DeclareMathOperator{\Var}{Var}

\DeclareMathOperator{\diag}{diag}
\DeclareMathOperator{\Tr}{Tr}

\DeclareMathOperator{\indep}{indep}
\DeclareMathOperator{\RWF}{RWF}
\DeclareMathOperator{\WF}{WF}

\newcommand{\pnorm}[2]{\lVert #1\rVert_{#2}}
\newcommand{\bigpnorm}[2]{\big\lVert#1\big\rVert_{#2}}
\newcommand{\Bigpnorm}[2]{\Big\lVert#1\Big\rVert_{#2}}
\newcommand{\biggpnorm}[2]{\bigg\lVert#1\bigg\rVert_{#2}}

\newtheorem{theorem}{Theorem}[section]
\newtheorem{assumption}[theorem]{Assumption}
\newtheorem{lemma}[theorem]{Lemma}
\newtheorem{remark}[theorem]{Remark}
\newtheorem{proposition}[theorem]{Proposition}

\newtheorem{definition}[theorem]{Definition}
\newtheorem{example}[theorem]{Example}

\title{Learning single index model with gradient descent: spectral initialization and precise asymptotics}

\author{Yuchen Chen\thanks{Department of Statistics and Data Science, Carnegie Mellon University, \texttt{yuchenc3@andrew.cmu.edu}.}~ and Yandi Shen\thanks{Department of Statistics and Data Science, Carnegie Mellon University, \texttt{yandis@andrew.cmu.edu}.}}
\date{}

\begin{document}
\maketitle
\begin{abstract}
    Non-convex optimization plays a central role in many statistics and machine learning problems. Despite the landscape irregularities for general non-convex functions, some recent work showed that for many learning problems with random data and large enough sample size, there exists a region around the true signal with benign landscape. Motivated by this observation, a widely used strategy is a two-stage algorithm, where we first apply a spectral initialization to plunge into the region, and then run gradient descent for further refinement. While this two-stage algorithm has been extensively analyzed for many non-convex problems, the precise distributional property of both its transient and long-time behavior remains to be understood. In this work, we study this two-stage algorithm in the context of single index models under the proportional asymptotics regime. We derive a set of dynamical mean field equations, which describe the precise behavior of the trajectory of spectral initialized gradient descent in the large system limit. We further show that when the spectral initialization successfully lands in a region of benign landscape, the above equation system is asymptotically time translation invariant and exponential converging, and thus admits a set of long-time fixed points that represents the mean field characterization of the limiting point of the gradient descent dynamic. As a proof of concept, we demonstrate our general theory in the example of regularized Wirtinger flow for phase retrieval.
\end{abstract}
\setcounter{tocdepth}{2}
\tableofcontents

\section{Introduction}
Learning single index models (also known as generalized linear models) is of fundamental importance in many areas of modern statistics, machine learning, and signal processing. To set up the stage, suppose we have $n$ responses $\y = (y_1,\ldots,y_n)$ generated from the model
\begin{align}\label{eq:intro_model}
y_i = \varphi(\x_i^\top\btheta^\ast, z_i), \quad i=1,\ldots,n, 
\end{align}
where $\x_i\in\R^d$ is the sensing vector of the $i$th data point, $\btheta^\ast$ is the true signal of interest, and $z_i$ is some external  noise independent of $(\x_i,\btheta^\ast)$. Upon observing $\y$ and $\X = [\x_1^\top;\ldots;\x_n^\top] \in \R^{n\times d}$, our goal is to estimate the unknown $\btheta^\ast$. While earlier approaches have focused on convex formulations of the learning problem \cite{balan2006signal,conca2015algebraic,bandeira2014saving}, which is naturally amenable to the design of provable algorithms and theoretical analysis, it is widely observed in practice that popular algorithms like semidefinite
programming are prohibitively expensive when the data scale is large. To this end, a much more computationally efficient and widely used approach is to optimize over the empirical risk, given by
\begin{align}\label{eq:intro_objective}
\cL_n(\btheta) = \sum_{i=1}^n \bar L(\x_i^\top \btheta, y_i),
\end{align}
where $\bar L:\R^2 \rightarrow \R_+$ is some loss function. While (\ref{eq:intro_objective}) usually amounts to a non-convex objective with complicated landscape, it is widely observed in many data regimes that gradient descent (and its variants) can converge to the global minimum in relatively few steps, which piques great interest among researchers in understanding and explaining the success of such descent algorithms.  

Consider gradient descent, the simplest algorithm to minimize (\ref{eq:intro_objective}), which initializes from some $\btheta^0$ and then iterates as
\begin{align*}
\btheta^{t+1} = \btheta^t - \gamma \nabla_{\btheta} \cL_n(\btheta),
\end{align*}
where $\gamma > 0$ is the step size. Under fairly general assumptions on the sensing matrix $\X$ and appropriate sample size conditions, a recent line of work \cite{chen2017solving,ma2020implicit,chen2019gradient} rigorously justified the success of vanilla gradient descent by showing that in many noiseless problems (i.e., $\z = 0$), the gradient descent iterates will typically converge to the truth $\btheta^\ast$ after $O(\log n)$ steps. The key insight there is that for large enough sample size, $\cL_n(\btheta)$ admits a local region around $\btheta^\ast$ with benign landscape (i.e., well-conditioned Hessian matrix), and from both warm (e.g., spectral initialization) and cold starts (i.e., random initialization), the gradient descent iterates will plunge into and stay in this region after at most $O(\log n)$ steps, after which the convergence to $\btheta^\ast$ is exponential. These results, in the form of non-asymptotic high probability bounds, are however less useful in understanding the (precise) transient behavior of gradient descent and characterizing the properties of the convergence points (e.g., the global minimum), especially when there is noise. 

To these ends, a complementary line of work aims to derive more precise characterizations of both transient and long-time behavior of gradient descent, at the cost of stronger distributional assumptions on the data model (\ref{eq:intro_model}). One such example is the recent breakthrough paper \cite{celentano_high-dimensional_2021}, which studies the so-called mean field regime consisting of proportional asymptotics $n \asymp d$, sensing matrix $\X$ with independent entries, and some further distributional assumptions on the truth $\btheta^\ast$ and noise $\z = (z_1,\ldots,z_n)$. Within this regime, \cite{celentano_high-dimensional_2021} developed a set of dynamical mean field theory (DMFT), originating from the study of spin glass in statistical physics \cite{sompolinsky1981dynamic,sompolinsky1982relaxational}, to characterize the precise distribution of any fixed number of gradient descent iterates in the large system limit (i.e., $n,d\rightarrow\infty$).  The limiting fixed dimensional process comes in the form of a complicated integro-differential equation system, which contains rich information about the large system behavior of gradient descent, and in principle could be read off to guide practical considerations such as early stopping and choice of loss functions. Under an abstract convergence condition, \cite{celentano_high-dimensional_2021} also studied the long-time behavior of this limiting process and established some connection to the static mean field characterization of the objective $\cL_n(\btheta)$. 

Despite the remarkable progress of \cite{celentano_high-dimensional_2021}, a few important questions remain to be answered. First, it only considers the case where the gradient descent initialization $\btheta^0$ is independent of the sensing matrix $\X$. Without further knowledge of the truth $\btheta^\ast$, the most natural initialization is to draw $\btheta^0$ uniformly in random (i.e., cold start), and while gradient descent is still known to converge in some examples \cite{ma2020implicit}, both theory and simulation suggests that it typically takes $O(\log n)$ iterates to align with $\btheta^\ast$. This $\log n$ dependence, while mild in practice, still poses theoretical challenge as it cannot be captured by the finite step characterization of \cite{celentano_high-dimensional_2021} and its non-asymptotic extensions to larger but still $o(\log n)$ steps \cite{han2024entrywise}. Second, to analyze the complicated DMFT equation, \cite{celentano_high-dimensional_2021} proposed a general condition which dictates that the components of the DMFT equation are asymptotically time translation invariant (TTI) and exponentially converging, under which its long-time fixed points can be derived. To the best of our knowledge, however, there is currently very little guidance on how to verify such a condition (even for strongly convex loss functions), which limits the applicability of DMFT in understanding the long-time and large system behavior of gradient descent. 

 Motivated by its practical success, we consider in this work gradient descent with a spectral initialization given by
 \begin{align}\label{eq:intro_spec}
 \btheta^0 = \sqrt{d}\hat\btheta^s, 
 \end{align}
 where $\hat\btheta^s$ is the unit-norm leading eigenvector of certain matrix that we construct from the data $(\X,\y)$. As mentioned above, the advantage of such initialization  is that for large enough sample size, $\btheta^0$ is already inside a local region around $\btheta^\ast$ with benign landscape, after which the gradient descent iterates $\{\btheta^t\}_{t\geq 0}$ typically converge linearly and within $O(1)$ number of steps (as opposed to $O(\log n)$ steps for random initialization). Based on this algorithm, the second consideration above raises the following natural questions:
 \begin{enumerate}
 \item What is the precise pathwise behavior of such algorithm, in the sense of DMFT characterization?
 \item Under what (verifiable) conditions can we use the resulting DMFT equation to analyze the convergence point of such algorithm?
 \end{enumerate}
Answering these questions requires significant theoretical innovation beyond \cite{celentano_high-dimensional_2021}, and, as discussed above, also offers valuable insights to many practical questions of interest. 

Let us now summarize our main contributions towards answering these questions. 
\begin{itemize}
\item We derive the dynamical mean field equations for spectral initialized gradient descent, consisting of univariate stochastic processes $\{\theta^t\}_{t\geq 0},\{\eta^t\}_{t\geq 0}$, Gaussian processes $(\{u^t\}_{t\geq 0},u^\diamond)$, $(\{w^t\}_{t\geq 0},w^\ast)$, and covariance and response processes $C_\theta,C_\eta,R_\theta,R_\eta$ satisfying a self-consistent system of equations. The laws of these processes $\sP(\{\theta^t\}_{t\geq 0}),\sP(\{\eta^t\}_{t\geq 0})$ characterize the precise large system dynamical behavior of the above algorithm, via the following almost sure limits as $n,d\rightarrow\infty$ with $n/d\rightarrow \delta \in (0,\infty)$:
    \begin{equation}\label{eq:intro_result_1}
    \begin{aligned}
        \frac{1}{d}\sum_{j=1}^d \delta_{\left(\theta^{0}_j,...,\theta^{m}_j,\theta^\ast_j\right)} &\overset{W_2}{\to} \mathsf{P}\left(\theta^{0},...,\theta^{m},\theta^*\right),\\
        \frac{1}{n}\sum_{i=1}^n\delta_{\left(\eta^0_i,...,\eta^{m}_i,\eta^\ast_i,z_i\right)} &\overset{W_2}{\to} \sP\left(\eta^0,...,\eta^m,w^*,z\right),
    \end{aligned}
    \end{equation}
    where $\bbeta^t = \X\btheta^t, \bbeta^\ast = \X\btheta^\ast$, and $\overset{W_2}{\to}$ denotes Wasserstein-2 convergence. We refer to Theorem \ref{thm: discrete dmft asymp} for precise assumptions and statements.
    
    The high level proof strategy of (\ref{eq:intro_result_1}) follows that of \cite{celentano_high-dimensional_2021} by reducing to the state evolution analysis of an approximate message passing (AMP) algorithm, with the apparent difference being that the initialization $\btheta^0$ is no longer independent of $\X$. To this end, we consider a spectral initialized AMP algorithm with dependence on multiple past iterates and derive its state evoluton, which generalizes an early result in \cite{mondelli_approximate_2022} and might be of independent interest.

    \item 
    
    To verify the asymptotically TTI and exponentially converging condition proposed in \cite{celentano_high-dimensional_2021} (with some further modifications for spectral initialization), we formulate a high probability assumption directly on the high dimensional gradient descent dynamic $\{\btheta^t\}_{t\geq 0}$ (as opposed to on the limiting DMFT equations), which states that, after at most a constant number of iterates, the dynamic will enter and stay in a region with benign landscape (see Assumption \ref{ass:exp_tti_cond}). This condition is inspired by earlier non-asymptotic analysis \cite{chen2019gradient,ma2020implicit}, and has two interesting consequences:
    \begin{enumerate}
    \item We can verify that in such case, the DMFT system is indeed asymptotically TTI and exponentially converging, and similar to \cite{celentano_high-dimensional_2021}, its limiting variables (in the $L^2$ sense) can be described by a set of fixed point equations (see Proposition \ref{prop: dmft fixed point});
    \item On the event where the above assumption holds, the gradient descent dynamic $\{\btheta^t\}_{t\geq 0}$ necessarily admits a limiting point $\btheta^\infty$, and its precise distributional characterization is given by
    \begin{align}\label{intro:theta_infty}
    \frac{1}{d}\sum_{j=1}^d \delta_{\theta^\infty_j} \overset{W_2}{\to}\sP(\theta^\infty),
    \end{align}
    where $\theta^\infty$ is (part of) a solution to the prescribed long-time fixed point equations, and $\sP(\theta^\infty)$ denotes its law (see Theorem \ref{thm:dmft_exp_tti_condition}).
    \end{enumerate}
    
    For strongly convex problems, this assumption holds unconditionally over the entire $\R^d$, verifying a claim in \cite{celentano_high-dimensional_2021} and completing the results therein. For nonconvex problems, at least when the sample size is large enough and the global landscape is benign, the prescribed region can typically be chosen as a small $\ell_2$ ball (or some further refinement thereof) around the (unique) global minimum, and the Hessian upper and lower bounds can be made rigorous by leveraging existing analysis in \cite{chen2019gradient,ma2020implicit} and many other extensions. In such case, $\btheta^\infty$ above can be identified as the global minimum, and (\ref{intro:theta_infty}) allows us to characterize its precise distribution. As a proof of concept, we work out an example of noiseless phase retrieval under a regularized squared loss (see Proposition \ref{lm:pr_exp_tti_ass}).
\end{itemize}
\paragraph{Related Literature}

Dynamical mean field theory (DMFT) goes back to the study of disordered systems in statistical physics \cite{sompolinsky1981dynamic,sompolinsky1982relaxational}, where rigorous derivations were first obtained using pathwise large deviation techniques  \cite{arous1995large,arous1997symmetric,guionnet1997averaged}. Long-time behavior of DMFT has long been a question of interest in the physics literature, but mathematically rigorous results of convergence remain limited. Notable exceptions include \cite{dembo2007limiting,dembo2020dynamics}, which study the mean field behavior of Langevin dynamics for spherical spin glasses. For such dynamics, the DMFT system can be simplified and closed solely on deterministic correlation and response processes (known as Cugliandolo-Kurchan equations \cite{cugliandolo1993analytical}), and the authors were able to study equilibrium properties by directly analyzing such deterministic system. While such techniques can be extended to spiked matrix models \cite{liang2022high,bodin2021rank}, which can be viewed as ``planted" version of the pure spin glass model, it is unclear whether such analysis can be extended to the current DMFT system which involves random components. 

Following heuristic arguments from statistical physics, DMFT has been formalized in many statistical problems \cite{sarao2020complex,sarao2020optimization,sarao2020marvels,mignacco2021stochasticity,mignacco2022effective}. The recent breakthrough paper \cite{celentano_high-dimensional_2021}, which was further built on previous work on the analysis of general first order methods \cite{celentano2020estimation}, introduced a novel way to rigorously derive DMFT via a reduction to AMP algorithms and their state evolution analysis \cite{bayati_message_passing_2011,javanmard2013state}. Since then, their work has been applied and extended to obtain DMFT characterization in a variety of contexts, including stochastic gradient dynamics \cite{gerbelot2024rigorous}, nonasymptotic DMFT \cite{han2024entrywise}, empirical Bayes estimation and sampling \cite{fan2025dmft,fan2025dynamicalmeanfieldanalysisadaptive}, neural networks \cite{montanari2025dynamical,han2025precise}, among many others. Detailed analysis of the DMFT equation, on the other hand, is much more scarce. Aside from the conditional result of \cite{celentano_high-dimensional_2021} (based on the asymptotically TTI and exponential converging condition mentioned above), one exception is \cite{fan2025dynamicalmeanfieldanalysisadaptive}, which studies an empirical Bayes problem in the context of Bayesian linear model and derives the long-time behavior of an associate adaptive Langevin algorithm (with independent initialization). The current work, on the other hand, considers the more general context of single index models, and addresses significant additional challenges brought by the spectral initialization. 

Our leading non-convex example, phase retrieval, has a vast previous literature. 
Shifting from the worst-case study of injectivity thresholds \cite{balan2006signal,conca2015algebraic,bandeira2014saving}, we list below a few lines of study from an average-case perspective that are most related to the current work.
\begin{itemize}
\item As mentioned above, earlier approaches used convex relaxation which lifted the problem to a convex matrix optimization problem \cite{candes2015phase,candes2013phaselift,goldstein2018phasemax,walkspurger2015phase}. Despite the issue of computational cost, the convex formulation allows one to apply techniques like convex Gaussian min-max theorem (CGMT) \cite{thrampoulidis2014gaussian} to derive mean field asymptotics, e.g., precise sample complexity for perfect recovery \cite{abbasi2017performance,dhifallah2017phase}.
\item Spectral methods, which output the leading eigenvector of certain matrix constructed from the data, have been extensively analyzed both in terms of high probability bounds for general sensing matrix $\X$ \cite{ma2020implicit} and more precisely in the mean field regime \cite{netrapalli2013phase,lu2020phase,luo2019optimal,maillard2022construction}, and is known in the latter case to achieve the information-theoretic threshold of weak recovery \cite{mondelli2019fundamental,barbier2019optimal} (see also \cite{maillard2020phase} for threshold of strong recovery).
\item Gradient descent on the squared loss (known as Wirtinger flow) and many of its variants have also been extensively studied; see \cite{cai2016optimal,chen2017solving,zhang2017nonconvex,wang2017solving,luo2019solving,chen2019gradient,chi2019nonconvex,Soltanolkotabi2019structured,gao2020perturbed,li2020toward,hand2020optimal,ma2020implicit,cai2022solving} for an imcomplete list. A complementary line of study analyzes the geometric properties of the optimization landscape (e.g., the existence of spurious local minima), for both the vanilla Wirtinger flow objective  \cite{sun2018geometric} and its regularized version \cite{li2020toward}. We refer to \cite{sarao2019afraid,sarao2020complex,sarao2020marvels,antenucci2019glassy,cai2022solving,maillard2020landscape} for more detailed discussions on the connections between landscape and convergence of gradient methods. 
\item Another popular iterative algorithm is AMP, which was studied in the context of single index models in \cite{rangan2012iterative,javanmard2013state} and was first applied to phase retrieval in \cite{schniter2015compressive}. Several versions of AMP have been analyzed, including the Bayes-optimal AMP \cite{barbier2019optimal}, optimization-based AMP \cite{ma2019optimization}, and spectral initialized AMP \cite{mondelli_optimal_2022,mondelli_approximate_2022}.
\end{itemize}

Lastly, we mention in passing that both gradient descent (and its variants) \cite{bietti2025learning,troiani2024fundamental,collins2024hitting,bruna2025survey} and spectral methods \cite{kovavcevic2025spectral,defilippis2025optimal} have recently gained interest in the study of multi-index models, which generalizes (\ref{eq:intro_model}) to multiple signal directions. It is an interesting future direction to extend our theory to this more generalized setting.  

\paragraph{Paper organization} The rest of the paper is organized as follows. Section \ref{sec:main_results} introduces the detailed model setup and our main results. Main proofs are given in Sections \ref{sec:dmft_asymp_proof}-\ref{sec:conv_hd_res}, and the remaining proofs are given in Appendix \ref{sec: spec init amp}-\ref{sec:phase_conv}. 

\paragraph{Notation}
Bold quantities such as $\btheta, \M$ will denote vectors/matrices. When a scalar function $f:\R^d\rightarrow \R$ is applied to a matrix $\M \in \R^{n\times d}$, it is applied row-wise. A function is $2$-pseudo Lipschitz if $|f(x)-f(y)|\leq C(1+\pnorm{x}{}+\pnorm{y}{})\pnorm{x-y}{}$. On a Polish space $\mathcal{M}$, we use $\mathcal{P}_p(\mathcal{M})$ to denote probability distributions on $\mathcal{M}$ with finite $p$-th moment, and $W_p$ to denote the Wasserstein-$p$ metric on $\mathcal{P}_p(\mathcal{M})$. For a random variable $X$, we use $\sP(X)$ to denote its law. For a vector $\v$, we use $\pnorm{\v}{}$ to denote the $2$-norm, and $\diag(\v)$ to denote the diagonal matrix with $\v$ as its diagonal. For a matrix $\M$, we use $\pnorm{\M}{F},\pnorm{\M}{\op},\pnorm{\M}{\ast},\Tr(\M)$ to denote the usual Frobenius, operator, nuclear norm, and trace. For two matrices $\M,\M'$, we use $\M\circ\M'$ to denote the Hadamard (entrywise) product. Unless otherwise indicated, constants $C,C',c,c' >0$ are constants independent of $n,d$ which may change line to line.

\section{Main results}\label{sec:main_results}
\subsection{Model setup and preliminaries}
Let the deterministic vector $\btheta^* = (\theta^\ast_1,\ldots,\theta^\ast_d)\in\R^d$ be the signal of interest. As shown in (\ref{eq:intro_model}), we observe $n$ i.i.d.\@ sensing vectors $\x_i \in \R^d$ along with noisy observations given by
\begin{align}\label{eq:model}
    \y = \varphi(\X\btheta^*,\z),
\end{align}
where $\y = (y_1,\ldots,y_n) \in \R^n$, $\X = [\x_1^\top;\cdots;\x_n^\top]\in \R^{n\times d}$, $\z = (z_1,\ldots,z_n)$ are i.i.d. noise variables independent of $\X$, and $\varphi:\R\times\R \rightarrow\R$ is a link function applied entry-wise. 
Throughout this work, we make the following model assumptions regarding (\ref{eq:model}).
\begin{assumption}[Model Assumptions]\
\label{ass: model assumptions}
\begin{enumerate}
    \item $\lim_{n,d \rightarrow \infty} n/d = \delta \in (0,\infty)$.
    \item The signal $\btheta^\ast$ satisfies $\pnorm{\btheta^\ast}{} = \sqrt{d}$, and for some scalar variable $\theta^\ast$ with a finite moment generating function in a neighborhood around $0$, $\frac{1}{d}\sum_{j=1}^d \delta_{\theta^\ast_j} \overset{W_p}{\to} \sP(\theta^\ast)$ for any $p\geq 1$.
    \item The design $\X = (X_{ij})_{i\in[n],j\in[d]} \in \R^{n\times p}$ has independent entries satisfying $\E X_{ij} = 0$, $\Var X_{ij} = 1/d$, and $\pnorm{\sqrt{d} X_{ij}}{\psi_2} \leq C$, where $\pnorm{\cdot}{\psi}$ is the sub-Gaussian norm.
    \item $z_1,\ldots,z_n$ are i.i.d. distributed as $\sP(z)$ for some scalar variable $z$ with a finite moment generating function in a neighborhood around $0$.
    \item $u\mapsto \varphi(u,v)$ is continuous, weakly differentiable in $u$ for any $v\in\R$, and $\sup_{u,v} |\partial_u \varphi(u,v)| \leq C$. 
\end{enumerate}
\end{assumption}

Given a pre-processing function $\mathcal{T}_s:\R \rightarrow \R$, define
\begin{align}\label{def:M_n}
    \M_n := \sum_{i=1}^n \mathcal{T}_s(y_i)\x_i\x_i^\top.
\end{align}
The spectral estimator $\hat{\btheta}^s$ is defined as the principal eigenvector of $\M_n$, with normalization $\pnorm{\hat{\btheta}^s}{}=1$. We make the following assumption on the pre-processing function $\cT_s(\cdot)$.
\begin{assumption}[Spectral Initalization]\label{ass: spectral initialization}
The pre-processing function $\cT_s(\cdot)$ is nonnegative, bounded, and Lipschitz, satisfies $\P(Z_s = 0) < 1$, and with $\tau := \inf\{z:\P(Z_s\leq z)=1\}$ assumed to be finite, it holds that 
\begin{align*}
\lim_{\lambda \rightarrow\tau^+}\E\Big[\frac{Z_s}{(\lambda-Z_s)^2}\Big]= \lim_{\lambda \rightarrow\tau^+} \E\Big[\frac{Z_sG^2}{\lambda-Z_s}\Big] = \infty,
\end{align*}
where the expectation is over the following randomness: $G\sim N(0,1)$ is independent of $z \sim \sP(z)$, $Y = \varphi(G,z)$, and $Z_s=\mathcal{T}_s(Y)$.
\end{assumption}
These are standard assumptions in the analysis of spectral estimators for single index models, and we refer to \cite{mondelli_optimal_2022} and references therein for further discussions and relaxations. We now summarize some previous results on the performance of the spectral estimator \cite{lu2020phase,mondelli2019fundamental,mondelli_optimal_2022} under Assumption \ref{ass: spectral initialization}. For any $\delta\in(0,\infty)$, define on the domain $\lambda \in (\tau,\infty)$ the functions
\begin{align*}
    \phi(\lambda) := \lambda \E\Big[\frac{Z_sG^2}{\lambda-Z_s}\Big], \quad \psi_\delta(\lambda) := \lambda\Big(\frac{1}{\delta} +\E\Big[\frac{Z_s}{\lambda-Z_s}\Big]\Big).
\end{align*}
It can be readily checked that the function $\psi_\delta$ is convex and has a minimizer which we denote by $\overline{\lambda}_\delta$. Define
\begin{align*}
    \zeta_\delta(\lambda) = \psi_\delta\big(\max(\lambda,\overline{\lambda}_\delta)\big).
\end{align*}
Under Assumption \ref{ass: spectral initialization}, there exists a unique solution $\lambda_\delta^*$ \cite[Theorem 1]{lu2020phase} to
\begin{align}\label{def:lambda_delta}
    \zeta_\delta(\lambda) = \phi(\lambda).
\end{align}
For future reference, define
\begin{align}\label{def:T_func}
\cT(u):=\frac{\cT_s(u)}{\lambda_\delta^* - \cT_s(u)}.
\end{align}

The following result characterizes the high-dimensional limit of the overlap between the spectral estimator $\hat\btheta^s$ and the signal $\btheta^\ast$, along with the limiting leading eigenvalues of $\M_n$ defined in (\ref{def:M_n}); see \cite[Theorem 1]{lu2020phase} and \cite[Lemma 2]{mondelli2019fundamental}.
\begin{proposition}\label{prop: spec est overlap}
Under Assumptions \ref{ass: model assumptions} and \ref{ass: spectral initialization}, the limiting overlap satisfies
\begin{align*}
    \frac{\Big|\langle\hat{\btheta}^s,\btheta^*\rangle\Big|^2}{\pnorm{\hat{\btheta}^s}{}^2\pnorm{\btheta^\ast}{}^2} \xrightarrow{a.s.} a^2 := \begin{cases}
        0 & \psi'_\delta(\lambda_\delta^*) \leq 0\\
        \frac{\psi'_\delta(\lambda_\delta^*)}{\psi'_\delta(\lambda_\delta^*)-\phi'(\lambda_\delta^*)} & \psi_\delta'(\lambda_\delta^*) > 0.
    \end{cases}
\end{align*}

Moreover, with $\lambda_1(\M_n)$ and $\lambda_2(\M_n)$ denoting the first and second leading eigenvalues of $\M_n$, it holds almost surely as $n,d\rightarrow\infty$ that
\begin{align*}
\lim_{n,d\rightarrow\infty} \lambda_1(\M_n) = \delta \psi_\delta(\lambda^\ast_\delta), \quad \lim_{n,d\rightarrow\infty} \lambda_2(\M_n) = \delta \psi_\delta(\bar{\lambda}_\delta).
\end{align*}
\end{proposition}
In other words, when $\psi_\delta'(\lambda^\ast_\delta) > 0$, the spectral estimator has a non-zero limiting overlap with the truth $\btheta^\ast$, and the matrix $\M_n$ admits a strictly positive eigengap in the large system limit. Our main results, stated in Theorems \ref{thm: discrete dmft asymp} and \ref{thm:dmft_exp_tti_condition} below, will operate in this regime.

To estimate $\btheta^\ast$ in (\ref{eq:model}), we consider a ridge regularized empirical risk of the form
\begin{equation}
    \mathcal{L}_n(\btheta) := \sum_{i=1}^n \bar{L}(\x_i^\top\btheta,y_i) + \frac{\lambda}{2}\pnorm{\btheta}{}^2,
\end{equation}
where $\bar{L}: \R\times \R \rightarrow\R_+$ is the loss function. Using the representation (\ref{eq:model}), this amounts to
\begin{align}\label{eq:GD_objective}
\cL_n(\btheta) = \sum_{i=1}^n L(\x_i^\top \btheta, \x_i^\top\btheta^\ast,z_i) + \frac{\lambda}{2}\pnorm{\btheta}{}^2, 
\end{align}
where $L(a,b,c) = \bar{L}(a, \varphi(b,c))$. We study the optimization of $\cL_n$ through gradient descent
\begin{align}\label{eq: gradient descent}
    \btheta^{t+1} = \btheta^t - \gamma\nabla_{\btheta}\cL_n(\btheta^t) = \btheta^t - \gamma\X^\top\ell(\X\btheta^t,\X\btheta^\ast,\z) - \gamma\lambda\btheta^t,
\end{align}
where $\ell(a,b,c) = \partial_a L(a,b,c)$ is the (weak) derivative of $L$ with respect to the first argument and is applied entry-wise, $\gamma>0$ is a fixed step size, and the initialization is the (normalized) spectral estimator in (\ref{def:M_n}):
\begin{align}\label{eq:gradient_des_init}
\btheta^0 = \sqrt{d}\hat{\btheta}^s.
\end{align}
In cases with global sign symmetry (e.g., phase retrieval), we pick without loss of generality the sign of $\btheta^\ast$ such that $\langle \hat{\btheta}^s,\btheta^\ast\rangle \geq 0$ so that $a \geq 0$ in Proposition \ref{prop: spec est overlap}.

We make the following assumptions for the loss function of gradient descent.
\begin{assumption}[Loss function]
    \label{ass: ell lipschitz} 
    The function $(x_1,x_2,x_3)\mapsto\ell(x_1,x_2,x_3)$ is continuously twice differentiable and there exists some $C > 0$ such that
    \begin{align*}
    \sup_{(x_1,x_2,x_3)\in\R^3} \max_{I \in \{(1),(2),(3),(11),(12),(22)\}}\Big|\partial_{I}\ell(x_1,x_2,x_3) \Big| \leq C,
    \end{align*}
    where $\partial_I$ denotes the derivative with respect to the arguments in the index set $I$.
\end{assumption}

Before introducing the main results, we give two examples of models satisfying our model assumptions.

\begin{example}[Linear model with robust loss]\label{ex:model_examples}
For the linear model, the link function is $\varphi(b,c) = b +c$, with observation
\begin{equation*}
    y_i = \x_i^\top\btheta^\ast + z_i, \quad i=1,\ldots,n.
\end{equation*}
Under the robust loss $L(a,b,c) = \rho(a-b-c)$ in (\ref{eq:GD_objective}) for some (weakly) differentiable $\rho:\R\rightarrow\R_+$, the gradient descent dynamic in (\ref{eq: gradient descent}) takes the form
\begin{equation*}
    \btheta^{t+1} = \btheta^t - \gamma\X^\top \rho'(\X\btheta^t - \X\btheta^\ast - \z) - \gamma\lambda\btheta^t,
\end{equation*}
where $\rho'$ is applied entrywise. In this case, $\ell(a,b,c) = \rho'(a-b-c)$, and it can be readily checked that such $\ell$ satisfies Assumption \ref{ass: ell lipschitz} as long as $\pnorm{\rho''}{\infty}\vee \pnorm{\rho'''}{\infty} \leq C$.
\end{example}

\begin{example}[Phase retrieval with regularized squared loss]\label{ex:phase_retrieval}
A prototypical example of single index models that often leads to non-convex objective functions is phase retrieval, where we only observe (a noisy version of) the magnitude of linear measurements. More precisely, we observe
\begin{equation*}
    y_i = |\x_i^\top\btheta^\ast| + z_i, \quad i=1,\ldots,n,
\end{equation*}
with link function $\varphi(b,c) = |b| + c$ in (\ref{eq:model}). We now consider a regularized version of the squared loss defined as follows. Let $h:[0,\infty) \rightarrow [0,1]$ be a smooth function satisfying
\begin{align}\label{eq:truncation}
        h(u) = 1 \text{ for } u \in [0,L], \quad
        h(u) \in [0,1] \text{ for } u \in (L,U), \quad
        h(u) = 0 \text{ for } u \in [U,\infty),
\end{align}
and moreover, $h'(u),h''(u)$ exist with $\sup_{u\in\R_+} \{|h'(u)|,|h''(u)|\} < C$, where $C,L,U >0$ are some constants. Examples of this type of function are given in \cite{li2020toward}; see also Proposition \ref{lm:pr_exp_tti_ass} and related discussions. We consider the regularized squared loss, also known as regularized Wirtinger flow, defined as 
\begin{align}\label{eq:smoothed_WF_loss}
    L^{\RWF}(a,b,c) = \frac{1}{2}\big(a^2-\varphi(b,c)^2\big)^2h\big(a^2\big)h\big(\varphi(b,c)^2\big).
\end{align}
The first partial derivative takes the form 
\begin{align*}
    \ell^{\RWF}(a,b,c) &= \partial_a  L^{\RWF}(a,b,c)\\
    &=2\Big(a^2-\varphi(b,c)^2\Big)ah(a^2)h\big(\varphi(b,c)^2\big) + \Big(a^2-\varphi(b,c)^2\Big)^2h'(a^2)ah\big(\varphi(b,c)^2\big),
\end{align*}
and using the bounded support of $h,h',h''$, one can readily check that it satisfies Assumption \ref{ass: ell lipschitz}. We refer to Proposition \ref{lm:pr_exp_tti_ass} for a detailed discussion of our general theory applied to this example.
\end{example}

\subsection{Dynamical mean field approximation under spectral initialization}
Consider the following dynamical mean field system. Initialize from
\begin{align}\label{def:dmft_theta0}
\theta^0 := a\theta^\ast + u^\diamond,
\end{align}
where $\theta^\ast \sim \sP(\theta^\ast)$ is independent of $u^\diamond \sim \N(0,1-a^2)$ with $\sP(\theta^\ast)$ given in Assumption \ref{ass: model assumptions}, and $a$ is the overlap given in Proposition \ref{prop: spec est overlap}. Further define the correlation functions
\begin{align}\label{def:theta_cov} 
C_\theta(0,0) := \E[\theta^0\theta^0] = 1, \quad C_\theta(0,\ast) := \E[\theta^0\theta^\ast] = a, \quad C_\theta(\ast,\ast) = 1, 
\end{align}
and response functions
\begin{align*}
R_\theta(0,\diamond) := \E[r_\theta(0,\diamond)], \quad r_\theta(0,\diamond) := 1.
\end{align*}

At time $t\in\Z_+$, given $\{R_\theta(s,r)\}_{0\leq r < s\leq t}$, $\{R_\theta(s,\diamond)\}_{0\leq s\leq t}$, $\{C_\theta(s,r)\}_{0\leq r \leq s \leq t}$, $\{C_\theta(s,\ast)\}_{0\leq s\leq t}$, define (in the probability space of $z$) 
\begin{align}\label{def:dmft_eta}
\eta^t &= -\sum_{s=0}^{t-1}\ell(\eta^{s},w^*,z)R_\theta(t,s) + \cT(y)w^0R_\theta(t,\diamond) + w^t.
\end{align}
Moreover, given the process $\{\eta^s\}_{0\leq s\leq t}$, define 
\begin{equation}\label{eq:r_eta_response}
\begin{aligned}
r_\eta(t,s) &= 
\partial_1\ell(\eta^t,w^*,z)\Big(-\sum_{r=s+1}^{t-1}r_\eta(r,s)R_\theta(t,r)-\partial_1\ell(\eta^s,w^*,z)R_\theta(t,s)\Big), \qquad 0 \leq s < t,\\
 r_\eta(t,*) &= \partial_1\ell(\eta^t,w^*,z)\Big(-\sum_{r=0}^{t-1}r_\eta(r,*)R_\theta(t,r)\Big) + \partial_2\ell(\eta^t,w^*,z),\\
r_\eta(t,\diamond) &= \partial_1\ell(\eta^t,w^*,z)\Big(-\sum_{s=0}^{t-1}r_\eta(s,\diamond)R_\theta(t,s) + \cT(y)R_\theta(t,\diamond)\Big),\\
r_\eta(t,\diamond\diamond)
&=\partial_1\ell(\eta^t,w^*,z) \Big(-\sum_{s=0}^{t-1}r_\eta(s,\diamond\diamond)R_\theta(t,s) +\cT'(y)\varphi'(w^*,z)w^0 R_\theta(t,\diamond)\Big).
\end{aligned}
\end{equation}
Here $(w^\ast,\{w^s\}_{0\leq s\leq t})$ is a centered Gaussian process independent of $z$ with covariance structure
\begin{align}\label{eq:w_cov}
\E[w^r w^s] = C_\theta(r,s), \quad \E[w^s w^\ast] = C_\theta(s,\ast), \quad \E[w^\ast w^\ast] = C_\theta(\ast,\ast),
\end{align}
function $\cT(\cdot)$ is defined in (\ref{def:T_func}), and
\begin{align*}
y := \varphi(w^\ast, z).
\end{align*}
With these processes constructed, we can now define the correlation functions  $\{C_\eta(s,r)\}_{0\leq r \leq s \leq t}$, $\{C_\eta(s,\diamond)\}_{0\leq s\leq t}$, $C_\eta(\diamond, \diamond)$ as
\begin{align}\label{def:eta_cov}    
    C_\eta(s,r) = \delta\E\big[\ell(\eta^{s},w^*,z)\ell(\eta^{r},w^*,z)\big], \quad 
    C_\eta(s,\diamond) = 
        -\frac{\delta}{\lambda_\delta^*}\E\big[\ell(\eta^{s},w^*,z) \cT_s(y)\eta^0\big], \quad C_\eta(\diamond,\diamond) = 1-a^2,
\end{align}
and response functions $\{R_\eta(s,r)\}_{0\leq r < s\leq t}$, $\{R_\eta(s,\ast)\}_{0\leq s\leq t}$, $\{R_\eta(s,\diamond)\}_{0\leq s\leq t}$, $\{R_\eta(s,\diamond\diamond)\}_{0\leq s\leq t}$ as
\begin{equation}\label{def:eta_res}
\begin{gathered}
R_\eta(s,r) = \delta\E[r_\eta(s,r)], \quad
    R_\eta(s,*) = \delta\E[r_\eta(s,\ast)],\\
    R_\eta(s,\diamond) = \delta\E[r_\eta(s,\diamond)], \quad R_\eta(s,\diamond\diamond) = \delta\E[r_\eta(s,\diamond\diamond)].
\end{gathered}
\end{equation}

Conversely, suppose that up to time $t\in\Z_+$, we are given $\{\eta^s\}_{0\leq s\leq t}$, response functions $\{R_\eta(r,s)\}_{0\leq r < s \leq t}$, $\{R_\eta(s,\ast),R_\eta(s,\diamond),R_\eta(s,\diamond\diamond)\}_{0\leq s\leq t}$ and covariance functions $\{C_\eta(s,r)\}_{0\leq r\leq s \leq t}$, $\{C_\eta(s,\diamond)\}_{0\leq s \leq t}$, $C_\eta(\diamond,\diamond)$. Define (in the probability space of $\theta^\ast$)
\begin{align}\label{def:dmft_theta}
    \theta^{t+1} = \theta^t +\gamma\Big(- \lambda\theta^t -\delta \E\big[\partial_1\ell(\eta^t,w^*,z)\big]\theta^{t}-\sum_{s=0}^{t-1}R_\eta(t,s)\theta^{s} - R_\eta(t,\diamond)\theta^0 - \big(R_\eta(t,\ast) + R_\eta(t,\diamond\diamond)\big)\theta^* + u^{t}\Big).
\end{align}
Here $(u^\diamond,\{u^s\}_{0\leq s \leq t})$ is a centered Gaussian process independent of $\theta^\ast$ with covariance structure
\begin{equation}\label{eq:u_cov}
    \E[u^r u^s] = C_\eta(r,s), \quad \E[u^ru^\diamond] = C_\eta(r,\diamond), \quad\E[u^\diamond u^\diamond] = C_\eta(\diamond,\diamond).
\end{equation}
Further define 
\begin{equation}\label{eq:r_theta_response}
\begin{aligned}
    r_\theta(t+1,s) &= 
    \begin{cases}
    \big(1-\gamma\lambda-\gamma\delta\E[\partial_1\ell(\eta^t,w^*,z)]\big) r_\theta(t,s) - \gamma\sum_{r=s+1}^{t-1}R_\eta(t,r)r_\theta(r,s) & s \leq t-1,\\
    \gamma & s = t,
    \end{cases}\\
    r_\theta(t+1,\diamond) &= \begin{cases}
    \left(1-\gamma\lambda-\gamma\delta\E[\partial_1\ell(\eta^{t},w^*,z)]\right) r_\theta(t,\diamond) - \gamma\sum_{r=0}^{t-1}R_\eta(t,r)r_\theta(r,\diamond) - \gamma R_\eta(t,\diamond) &t \geq 0\\
    1 & t = -1.
    \end{cases}
\end{aligned}
\end{equation}
Then we can construct the correlation functions $\{C_\theta(s,r)\}_{0\leq r\leq s\leq t+1}$, $\{C_\theta(s,\ast)\}_{0\leq s\leq t+1}$, and $C_\theta(\ast,\ast)$ as
\begin{align}\label{def:theta_cov}
C_\theta(s,r) = \E[\theta^{s}\theta^{r}], \quad C_\theta(s,*) = \E[\theta^{s}\theta^{*}], \quad C_\theta(\ast,\ast) = \E[(\theta^\ast)^2] = 1,   
\end{align}
and response functions $\{R_\theta(s,r)\}_{0\leq r<s\leq t+1}$, $\{R_\theta(s,\diamond)\}_{0\leq s\leq t+1}$ as
\begin{align} \label{def:theta_res}
R_\theta(s,r) = \E[r_\theta(s,r)], \quad R_\theta(s,\diamond) = \E[r_\theta(s,\diamond)].
\end{align}
In summary, the above DMFT system is recursively defined in the following order: with $\theta^0$ in (\ref{def:dmft_theta0}),
\begin{align*}
&\theta^0 \longrightarrow \{C_\theta(s,r)\}_{0\leq r\leq s\leq 0}, \{C_\theta(s,\ast)\}_{0\leq s\leq 0}, \{R_\theta(s,r)\}_{0\leq r < s\leq 0}, \{R_\theta(s,\diamond)\}_{0\leq s\leq 0}\longrightarrow \eta^0\\
 &\longrightarrow \{C_\eta(s,r)\}_{0\leq r\leq s\leq 0}, \{C_\eta(s,\diamond)\}_{0\leq s\leq 0}, C_\eta(\diamond,\diamond), \{R_\eta(s,r)\}_{0\leq r < s\leq 0}, \{R_\eta(s,\ast),R_\eta(s,\diamond),R_\eta(s,\diamond\diamond)\}_{0\leq s\leq 0}\\
&\longrightarrow \theta^1 \longrightarrow \{C_\theta(s,r)\}_{0\leq r\leq s\leq 1}, \{C_\theta(s,\ast)\}_{0\leq s\leq 1}, \{R_\theta(s,r)\}_{0\leq r < s\leq 1}, \{R_\theta(s,\diamond)\}_{0\leq s\leq 1}\longrightarrow \eta^1\\
&\longrightarrow \{C_\eta(s,r)\}_{0\leq r\leq s\leq 1}, \{C_\eta(s,\diamond)\}_{0\leq s\leq 1}, C_\eta(\diamond,\diamond), \{R_\eta(s,r)\}_{0\leq r < s\leq 1}, \{R_\eta(s,\ast),R_\eta(s,\diamond),R_\eta(s,\diamond\diamond)\}_{0\leq s\leq 1}\\
&\longrightarrow\theta^2 \rightarrow \ldots.
\end{align*}

Our first main result is the following distributional characterization of the gradient descent iterates (\ref{eq: gradient descent}) with spectral initialization. Given (\ref{eq: gradient descent}), further define
\begin{align}\label{def:eta_highd}
\bbeta^\ast := \X\btheta^\ast, \quad \bbeta^t := \X\btheta^t, \quad t\geq 0.
\end{align}

\begin{theorem}\label{thm: discrete dmft asymp}
    Suppose Assumptions \ref{ass: model assumptions}, \ref{ass: spectral initialization}, \ref{ass: ell lipschitz} hold, and $\psi_\delta'(\lambda_\delta^*) > 0$ in Proposition \ref{prop: spec est overlap}. Then for any fixed integer $m\geq 0$, almost surely as $n,d \rightarrow\infty$, we have that
    \begin{align*}
        \frac{1}{d}\sum_{j=1}^d \delta_{\left(\theta^{0}_j,...,\theta^{m}_j,\theta^\ast_j\right)} &\overset{W_2}{\to} \mathsf{P}\left(\theta^{0},...,\theta^{m},\theta^*\right),\\
        \frac{1}{n}\sum_{i=1}^n\delta_{\left(\eta^0_i,...,\eta^{m}_i,\eta^\ast_i,z_i\right)} &\overset{W_2}{\to} \sP\left(\eta^0,...,\eta^m,w^*,z\right).
    \end{align*}
\end{theorem}
\begin{remark}[Independent initialization]
    If the gradient descent (\ref{eq: gradient descent}) is initialized from some $\btheta^{0,\indep}$ independent of $\X$ that also satisfies $\frac{1}{d}\sum_{j=1}^d \delta_{\theta^{0,\indep}_j}\overset{W_2}{\to} \mathsf{P}(\theta^{0,\indep})$ for some scalar variable $\theta^{0,\indep}$ with $\E[(\theta^{0,\indep})^2] < \infty$, then the empirical distribution of $\{\btheta^t\}_{t\geq 0}$ and $\{\bbeta^t\}_{t\geq 0}$ is still described by Theorem \ref{thm: discrete dmft asymp}, with the DMFT law on the right given by (\ref{def:dmft_theta0})-(\ref{def:theta_res}) except for (i) $\theta^0$ in (\ref{def:dmft_theta0}) is replaced by $\theta^{0,\indep}$ independent of everything else and (ii)  all terms involving $\diamond$ or $\diamond\diamond$ (i.e., $r_{\theta}(t,\diamond)$, $r_\eta(t,\diamond)$, $r_\eta(t,\diamond\diamond)$) are set to $0$. In this case, Theorem \ref{thm: discrete dmft asymp} reduces to the result of \cite{celentano_high-dimensional_2021} up to some minor rescaling.
\end{remark}

The proof of Theorem \ref{thm: discrete dmft asymp} is given in Section \ref{sec:dmft_asymp_proof}. The proof idea closely follows that of \cite{celentano_high-dimensional_2021}, with the apparent difference being that the initialization $\btheta^0$ is no longer independent of $\X$. To this end, we consider a spectral initialized AMP algorithm with dependence on multiple past iterates and derive its precise distributional behavior (see Theorem \ref{thm: amp spec} ahead), which generalizes an early result in \cite{mondelli_approximate_2022} and might be of independent interest. 

\subsection{Convergence of DMFT system}
Our second main result involves studying long-time fixed points of the DMFT system with spectral initialization, starting with the following condition under which the convergence of DMFT system can be made precise. We use $\pnorm{X}{L^2} = (\E[X^2])^{1/2}$ to denote the $L^2$ norm of a random variable, and $X_n \overset{L^2}{\to} X$ to denote $L^2$ convergence for $(\{X_n\}_{n\geq 0},X)$ defined in the same space.

\begin{definition}[Approximately TTI and exponential convergence]
\label{def:exp_conv_spec_init}
We say the spectral initialized DMFT system (\ref{def:dmft_theta0})-(\ref{def:theta_res}) is approximately time translation invariant and exponentially converging if for some deterministic constants $C,c > 0$, the following conditions are satisfied:
\begin{enumerate}
    \item There exist univariate random variables $\theta^\infty$ (in the space of $(\{u^t\}_{t\geq 0},u^\diamond,\theta^\ast)$) and $\eta^\infty$ (in the space of $(\{w^t\}_{t\geq 0},w^\ast,z)$) such that
    \begin{align*}
        \pnorm{\theta^t-\theta^\infty}{L^2} \leq Ce^{-ct}, \quad 
        \pnorm{\eta^t-\eta^\infty}{L^2} \leq Ce^{-ct}.
    \end{align*}
    \item There exist Gaussian variables $u^\infty$ (in the space of $(\{u^t\}_{t\geq 0}, u^\diamond,\theta^\ast)$) and $w^\infty$ (in the space of $(\{w^t\}_{t\geq 0}, w^\ast,z)$) such that 
    \begin{align*}
        \pnorm{u^t-u^\infty}{L^2} \leq Ce^{-ct}, \quad 
        \pnorm{w^t-w^\infty}{L^2} \leq Ce^{-ct}.
    \end{align*}
    \item (Approximate time translation invariance of response) There exist a univariate random process $\{r_\eta(t)\}_{t > 0}$ and random variable $r_\eta^\ast$ (in the space of $(\{w^t\}_{t\geq 0}, w^\ast,z)$), and deterministic process $\{R_\theta(t)\}_{t > 0}$ such that as $t\rightarrow\infty$, 
    \begin{align*}
    R_\theta(t+s,t) \to R_\theta(s), \quad
    r_\eta(t+s,t)\overset{L^2}{\to} r_\eta(s), \quad r_\eta(t,\ast)\overset{L^2}{\to} r_\eta^*,
    \end{align*}
    for any fixed integer $s > 0$.

    \item (Decay of response from spectral initialization) The response functions $R_\theta(t,\diamond),R_\eta(t,\diamond),R_\eta(t,\diamond\diamond)$  satisfy as $t\rightarrow\infty$ that
    \begin{align*}
        R_\theta(t,\diamond) \rightarrow 0, \quad R_\eta(t,\diamond) \rightarrow 0, \quad
        R_\eta(t,\diamond\diamond) \rightarrow 0.
    \end{align*}
    \item (Decay of response from random initialization) For all $t\geq 0$ and $s\geq 1$, the response functions $R_\theta(t,s), r_\eta(t,s)$ satisfy
    \begin{align}\label{eq:response_exp_decay}
        |R_\theta(t+s,t)| \leq Ce^{-cs}, \quad \pnorm{r_\eta(t+s,t)}{L^2} \leq Ce^{-cs}.
    \end{align}
\end{enumerate}
\end{definition}
\begin{remark}
Definition \ref{def:exp_conv_spec_init} is close in spirit to \cite[Definition 4.2]{celentano_high-dimensional_2021}, with two differences: (i) For the approximate TTI property in part (3), we only require it to hold pointwise in $s$; (ii) we require in part (4) that the response functions resulting from spectral initialization decay to zero. We also note that in order for the subsequent Proposition \ref{prop: dmft fixed point} to hold, the exponential decay factor $Ce^{-ct}$ in part (1)(2)(5) can be relaxed to any positive sequence $\eps(t)$ such that $t\eps(t)\rightarrow 0$ as $t\rightarrow\infty$.
\end{remark}

Under Definition \ref{def:exp_conv_spec_init}, the following result, whose proof is given in Section \ref{sec:fixed_point}, shows that the convergence point $(\theta^\infty, \eta^\infty,u^\infty,(w^\infty,w^\ast), \{R_\theta(s)\}_{s\geq 1}, \{r_\eta(s)\}_{s\geq 1}, r_\eta^\ast)$ of the DMFT system must satisfy a system of self-consistent equations. The proof strategy is similar to \cite{celentano_high-dimensional_2021}, except that that we treat the discrete-time case explicitly and deal with the additional response functions that arise from the spectral initialization.

\begin{proposition}
\label{prop: dmft fixed point}
Suppose the spectral initialized DMFT system satisfies Definition \ref{def:exp_conv_spec_init}. Then the limiting variables $(\theta^\infty,\eta^\infty,u^\infty, (w^\infty,w^\ast))$ therein satisfy the fixed point equations 
    \begin{gather}
            0 = -(\lambda+\delta\Gamma^\infty + R_\eta^\infty)\theta^\infty - R_\eta^*\theta^* + u^\infty,\label{eq:fix_1}\\
            \eta^\infty = -R_\theta^\infty \ell(\eta^\infty,w^\ast,z) + w^\infty.\label{eq:fix2}
    \end{gather}
    Here $u^\infty \sim \N(0,C^\infty_\eta)$ (resp. $(w^\infty,w^*)\sim \N(0,C^\infty_\theta)$) is independent of $\theta^\ast$ (resp. $z$), with
    \begin{align}\label{eq:fix3}
        C^\infty_\eta = \delta\E\big[\ell(\eta^\infty,w^*,z)^2\big], \quad C^\infty_\theta = \E\big[(\theta^\infty,\theta^\ast)(\theta^\infty,\theta^\ast)^\top\big],
    \end{align}
    and the response values $(R_\theta^\infty,R_\eta^\infty,R_\eta^*)$ are deterministic constants satisfying, with $\Gamma^\infty := \E[\partial_1 \ell(\eta^\infty,w^\ast,z)]$, 
            \begin{gather}(R_\theta^\infty)^{-1} = \lambda + \delta\Gamma^\infty + R_\eta^\infty,\label{eq:fix5}\\
            \delta\Gamma^\infty + R_\eta^\infty = \delta(R_\theta^\infty)^{-1}\E\Big[\Big(1-\Big(1+\partial_1\ell(\eta^{\infty},w^*,z)R_\theta^\infty\Big)^{-1}\Big)\Big],\label{eq:fix6}\\
            R_\eta^* = \delta\E\Big[\big(1+\partial_1 \ell(\eta^\infty,w^*,z)R_\theta^\infty\big)^{-1}\partial_2 \ell(\eta^\infty,w^*,z)\Big],\label{eq:fix7}
            \end{gather}
    where we assume that the inverses exist.
\end{proposition}
\begin{remark}
Under Definition \ref{def:exp_conv_spec_init}, the above system has at least one solution, but in general multiple solutions could exist. We refer to \cite{celentano_high-dimensional_2021} for some discussion on the connection to static mean field characterization of $\cL_n(\btheta)$ via the convex Gaussian min-max theorem \cite{thrampoulidis2014gaussian} (see also Remark \ref{rmk:nonunique_solution} below).
\end{remark}

With the above characterization, a natural question is when will the conditions in Definition \ref{def:exp_conv_spec_init} be satisfied. Given the DMFT approximation result in Theorem \ref{thm: discrete dmft asymp}, a sensible sufficient condition is that the spectral initialized gradient descent will converge over an $(n,d)$-independent horizon. We formulate such a condition over the landscape along the gradient descent trajectory.

\begin{assumption}\label{ass:exp_tti_cond} 
There exists a convex region $\cR_n \subset\R^d$ and constants $t_0 \in \Z_+$, $c_0>0, c>1$ independent of $n,d$ such that with probability $1-n^{-c}$, the following holds.
    \begin{enumerate}
        \item For $t\geq t_0$, $\btheta^t \in \cR_n$.
        \item $\lambda_{\min}(\nabla_{\btheta}^2 \cL_n(\btheta)) \geq c_0$ for all $\btheta \in \cR_n$. 
    \end{enumerate} 
    We note that under Assumption \ref{ass: ell lipschitz}, we also have $\lambda_{\max}(\nabla_{\btheta}^2 \cL_n(\btheta)) \leq L$ for all $\btheta\in \cR_n$ where $L=\lambda + 4(1+\sqrt{\delta})^2\pnorm{\partial_1\ell}{\infty}$ is a constant. Here $\lambda_{\min}$ and $\lambda_{\max}$ denote the smallest and largest eigenvalues respectively.
\end{assumption}

The above assumption, inspired by previous convergence analysis of gradient descent in non-convex problems \cite{ma2020implicit}, says that after at most a constant number of steps, spectral initialized gradient descent will enter and stay within a region of benign landscape (i.e., with a well-conditioned Hessian matrix). By standard analysis, gradient descent will converge linearly to certain limit point $\btheta^\infty$ in such a region, hence the overall algorithm will converge over an $(n,d)$-independent horizon. 

The following result establishes the validity of Definition \ref{def:exp_conv_spec_init} under the above condition, and as a corollary, provides the precise distributional characterization of the limit point $\btheta^\infty$. Its proof is given in Section \ref{sec:proof_long_time}. 

\begin{theorem}\label{thm:dmft_exp_tti_condition}
Suppose Assumptions \ref{ass: model assumptions}, \ref{ass: spectral initialization}, \ref{ass: ell lipschitz}, \ref{ass:exp_tti_cond} hold, $\psi_\delta'(\lambda_\delta^*) > 0$ in Proposition \ref{prop: spec est overlap}, and the step size satisfies $\gamma < 2/L$. Then the DMFT is approximately TTI and exponentially convergent in the sense of Definition \ref{def:exp_conv_spec_init}, and almost surely for large enough $n,d$,  there exists a vector $\btheta^\infty \in \R^d$ such that $\pnorm{\btheta^t - \btheta^\infty}{} \rightarrow 0$ as $t\rightarrow\infty$, and
\begin{align*}
\frac{1}{d}\sum_{j=1}^d \delta_{\theta^\infty_j} \overset{W_2}{\to} \sP(\theta^\infty),
\end{align*}
where $\theta^\infty$ is (part of) a solution to the fixed-point equation in Proposition \ref{prop: dmft fixed point}.
\end{theorem}
\begin{remark}
The convergence $\pnorm{\btheta^t-\btheta^\infty}{}\rightarrow 0$ as $t\rightarrow\infty$ is only guaranteed on the high probability event of Assumption \ref{ass:exp_tti_cond}, hence $\btheta^\infty$ is only well-defined almost surely for large enough $n,d$, rather than with probability $1$ for any finite $n,d$. In concrete applications, however, $\btheta^\infty$ can typically be identified with the (unique) global minimum of $\cL_n(\btheta)$ (which is indeed well-defined with probability $1$ for any finite $n,d$), at least when the sample size is large enough and the global landscape is benign.
\end{remark}
\begin{remark}\label{rmk:nonunique_solution}
We emphasize that in general the system (\ref{eq:fix_1})-(\ref{eq:fix7}) may have multiple solutions, which potentially correspond to the static mean field characterizations of the local minima of $\cL_n(\btheta)$. In such case, $\theta^\infty$ (and other components of the fixed point) in the above theorem refers to the specific one prescribed by Definition \ref{def:exp_conv_spec_init}, which could depend on the specific choice of the pre-processing function $\cT_s(\cdot)$ in the initialization. While we will not pursue this in the current work, we conjecture that for a general class of loss functions and large enough aspect ratio $\delta$, with high probability $\cL_n(\btheta)$ admits a unique global minimum (up to global symmetry) and no spurious local minima (see \cite{li2020toward} for some rigorous results in the context of Example \ref{ex:model_examples}), and (\ref{eq:fix_1})-(\ref{eq:fix7}) will admit a unique solution (up to global symmetry) that corresponds to the mean field characterization of this global minimum; see \cite{celentano_high-dimensional_2021,asgari2025local} for some results along this line for (strongly) convex problems. 
\end{remark}

To the best of our knowledge, Theorem \ref{thm:dmft_exp_tti_condition} is new even for (strongly) convex problems. More precisely, for single index models with strongly convex loss (or convex but with positive regularization $\lambda >0$), Assumption \ref{ass:exp_tti_cond} is automatically satisfied with $\cR_n = \R^d$ and $c_0$ the strong convexity parameter, and Theorem \ref{thm:dmft_exp_tti_condition} provides a precise distributional characterization of $\btheta^\infty$, which is the unique global minimum $\hat\btheta$ of $\cL_n(\btheta)$ in this case. Restricting to the case of independent initialization, Theorem \ref{thm:dmft_exp_tti_condition} partially addresses the question posed in Remark 4.3 of \cite{celentano_high-dimensional_2021} regarding the validity of Definition \ref{def:exp_conv_spec_init}. 

Moving on to nonconvex problems, the key step is to identity the region $\cR_n$ and establish the Hessian lower bound in Assumption \ref{ass:exp_tti_cond}. Assuming the sample size is large enough (i.e., large $\delta$) and the global landscape is benign (i.e., unique global minimum and no spurious local minima), $\cR_n$ can typically be chosen to be a small $\ell_2$ ball (or some further refinements thereof) around the unique global minimum $\hat\btheta$ of $\cL_n(\btheta)$, and the high probability Hessian lower bound in part (2) could be established by the following general strategy: (i) show that $\hat\btheta$ is close to $\btheta^\ast$; (ii) show that the ``population version" Hessian $\nabla_{\btheta}^2 \E[\cL_n(\btheta)]$ is well-conditioned at $\btheta^\ast$ and thus also in some small $\ell_2$ neighborhood around $\btheta^\ast$ (or equivalently, $\hat\btheta$) by continuity; (iii) show the same bound for $\nabla_{\btheta}^2 \cL_n(\btheta)$ via concentration arguments.

As a proof of concept, we will follow \cite{li2020toward} and work out the details of the phase retrieval Example \ref{ex:phase_retrieval} in the noiseless case (i.e., $\z=0$ in (\ref{eq:model})). Its proof is given in Section \ref{sec:phase_conv}.

\begin{proposition}\label{lm:pr_exp_tti_ass}
    Consider Example \ref{ex:phase_retrieval} in the noiseless case with $\z=0$ in (\ref{eq:model}), and assume additionally that $\X$ has i.i.d. $\N(0,1/d)$ entries. For any $M > 0$, let the spectral pre-processing function be
    \begin{align*}
    \cT_s(y) = \big(\min(y,M)\big)^2, \quad y\in\R_+,
    \end{align*}
    which satisfies Assumption \ref{ass: spectral initialization} and $\psi'_\delta(\lambda^\ast_\delta) > 0$ in Proposition \ref{prop: spec est overlap}. For sufficiently large $U = 2L$ in the definition of the truncation function $h(\cdot)$ in (\ref{eq:truncation}), there exists some $\delta_0 = \delta_0(U)$ such that for $\delta > \delta_0$, the following holds with probability at least $1-e^{-cn}$ for some universal $c > 0$:
    \begin{enumerate}
    \item The spectral estimator $\hat\btheta^s$ satisfies $\pnorm{\hat{\btheta}^s - \btheta^\ast/\sqrt{d}}{} \leq 1/5$;
    \item With $\cR_n := \{\btheta \in \R^d:\pnorm{\btheta-\btheta^\ast}{}/\sqrt{d}\leq 1/5\}$, it holds that $\lambda_{\min}(\nabla^2_{\btheta} \cL_n(\btheta)) > 1/50$ for any $\btheta\in \cR_n$.
    \end{enumerate}
    Moreover, Theorem \ref{thm:dmft_exp_tti_condition} holds with $\btheta^\infty = \btheta^\ast$ and $\theta^\infty = \theta^\ast$, and the remaining components of the solution to (\ref{eq:fix_1})-(\ref{eq:fix7}) are given (sequentially) by: $u^\infty = 0, \eta^\infty = w^\infty = w^\ast$, $R_\theta^\infty$ is the unique solution to
    \begin{align*}
    \frac{1}{\delta} = \E\Big[\frac{\partial_1\ell(w^\ast,w^\ast,0)R_\theta^\infty}{1+\partial_1\ell(w^\ast,w^\ast,0)R_\theta^\infty}\Big],
    \end{align*}
    and
    \begin{align*}
    R_\eta^\infty = \frac{1}{R_\theta^\infty} - \delta\E\Big[\frac{\partial_1\ell(w^\ast,w^\ast,0)}{1+\partial_1\ell(w^\ast,w^\ast,0)}\Big],\quad R_\eta^\ast = -\delta\E[\partial_1\ell(w^\ast,w^\ast,0)] - R_\eta^\infty.
    \end{align*}
\end{proposition}
\begin{remark}
In the noiseless case, \cite{li2020toward} showed that for large enough $\delta$, the objective $\cL_n(\btheta)$ admits no spurious local minima and a unique global minimum $\hat\btheta = \btheta^\ast$ (up to global sign symmetry). We believe that such a benign global landscape also holds for the noisy case and large enough $\delta$, in which case Theorem \ref{thm:dmft_exp_tti_condition} can be applied with $\btheta^\infty = \hat\btheta$ to yield a precise distributional characterization of the global minimum.
\end{remark}
\begin{remark}
The loss function in (\ref{eq:smoothed_WF_loss}) is a regularized version of the squared loss of Wirtinger flow, which is defined as
\begin{align}\label{eq:WF_loss}
    L^{\WF}(a,b,c) = \frac{1}{2}\big(a^2-\varphi(b,c)^2\big)^2,
\end{align}
and has been extensively studied (see e.g., \cite{ma2020implicit,chen2017solving,chen2019gradient}). Unfortunately, this example currently still falls out of the scope of Theorems \ref{thm: discrete dmft asymp} and \ref{thm:dmft_exp_tti_condition}, mostly because both existing upper and lower bounds of the Hessian in Assumption \ref{ass:exp_tti_cond} require as many as $n = Cd\log d$ samples for some large $C$, and as a consequence, the gradient descent step size $\gamma$ has to be chosen as $\gamma = c/\log n$. It is conjectured in \cite{sarao2020complex} that for the loss (\ref{eq:WF_loss}), the idealized gradient flow dynamic (i.e., gradient descent with infinitismal step size $\gamma$) with random initialization will converge to $\btheta^\ast$ under the proportional regime $n = Cd$, but a rigorous proof is still lacking, even for gradient flow with spectral initialization.
\end{remark}

\section{Proof of DMFT characterization (Theorem \ref{thm: discrete dmft asymp})}
\label{sec:dmft_asymp_proof}

\subsection{AMP algorithm with spectral initialization}
We start by introducing an AMP algorithm with spectral initialization and analyzing its state evolution properties. Let $\y = \varphi(\X\btheta^\ast, \z)$ as in (\ref{eq:model}) and $\bZ_s = \diag(\cT_s(\y)) \in \R^{n\times n}$. Suppose Assumption \ref{ass: model assumptions} holds so that
almost surely as $n,d\rightarrow\infty$, 
\begin{align}\label{eq:amp_init_cond}
    \frac{1}{n}\sum_{i=1}^n \delta_{z_i} \overset{W_2}{\to} \sP(z), \quad \frac{1}{d}\sum_{j=1}^d \delta_{ \theta^\ast_j} \overset{W_2}{\to} \sP(\theta^\ast). 
\end{align}
Let $\btheta^0$ be defined by (\ref{eq:gradient_des_init}) and $\lambda^\ast_\delta$ be given by (\ref{def:lambda_delta}). For nonlinearities $f_i = (f_{i,1}, f_{i,2}):\R^{2(i+1)+1} \rightarrow \R^2$ and $g_i = (g_{i,1}, g_{i,2}): \R^{2i+2} \rightarrow\R^2$ applied row-wise, consider the following AMP algorithm:
\begin{align}
\begin{split}
\label{eq:spec_init_AMP}
    \b^i &= \X g_i(\a^1,...,\a^i;\btheta^0,\btheta^*) + \frac{1}{\delta} \sum_{j=0}^{i-1} f_j(\b^0,...,\b^j;\z)\zeta_{i,j}-\Big(\frac{1}{\lambda_\delta^*}\bZ_s \X\btheta^0, \mathbf{0}\Big)\zeta_{i,-1} \in\R^{n\times 2},\\
    \a^{i+1} &= -\frac{1}{\delta}\X^\top f_i(\b^0,...,\b^i;\z) + \sum_{j=0}^{i} g_j(\a^1,...,\a^j;\btheta^0,\btheta^*)\xi_{i,j} \in \R^{d\times 2},
\end{split}
\end{align}
initialized with $g_0(u,v) = (u,v)$ so that $g_0(\btheta^0,\btheta^*) = (\btheta^0,\btheta^*)$ and $\b^0 = (\X\btheta^0 - \frac{1}{\lambda_\delta^*}\bZ_s \X\btheta^0,\X\btheta^\ast)$. Let $\{u^j\}_{j\geq 0} = \{(u^j_1,u^j_2)\}_{j\geq 0}$ and $\{w^j\}_{j\geq 0} = \{(w^j_1,w^j_2)\}_{j\geq 0}$ be two centered Gaussian processes in $\R^2$ with covariance recursively defined by
\begin{align}\label{eq:amp_gp_cov}
\begin{split}
\E[u^i(u^{j})^\top] &= \frac{1}{\delta}\E\Big[f_{i-1}(w^0,...,w^{i-1};z)f_{j-1}(w^0,...,w^{j-1};z)^\top\Big],   \quad 1\leq j \leq i<\infty,\\
\E[u^i(u^0)^\top] &= \begin{pmatrix}
    -\E[f_{i-1,1}(w^0,...,w^{i-1};z)\cT(\varphi(w_2^0,z))w_1^0] & 0\\
    0 & 0
\end{pmatrix},  \quad i\geq 1,\\
\E[u^0(u^0)^\top] &= \begin{pmatrix}
    1-a^2 & 0\\
    0 & 0
\end{pmatrix},  \\
\E[w^i(w^{j})^\top] &=\E\Big[g_{i}(u^1,...,u^{i};a\theta^*+u^0_1,\theta^*)g_{j}(u^1,...,u^{j};a\theta^*+u^0_1,\theta^*)^\top\Big], \quad 0\leq j \leq i<\infty.
\end{split}
\end{align}
where $a$ is the limiting overlap of the spectral estimator as given in Proposition \ref{prop: spec est overlap}, $\cT(u) = \frac{\cT_s(u)}{\lambda^\ast_\delta - \cT_s(u)}$ is defined by (\ref{def:T_func}), and $\theta^* \sim \sP(\theta^*)$, $z \sim \sP(z)$, $\{u^j\}_{j\geq 0}$, $\{w^j\}_{j\geq 0}$ are mutually independent. From this covariance structure, one may recursively define Gaussian vectors $u^0 \rightarrow w^0 \rightarrow u^1 \rightarrow w^1 \rightarrow \ldots$, and the Onsager correction terms in (\ref{eq:spec_init_AMP}) are recursively defined by
\begin{equation}\label{def:xi_zeta}
\begin{aligned}
    \xi_{ij} &= \left(\E\left[\frac{\partial}{\partial w^j}f_i(w^0,...,w^i;z)\right]\right)^\top \in \R^{2\times 2},\quad 0\leq j \leq i\\
    \zeta_{ij} &= \left(\E\left[\frac{\partial}{\partial u^{j+1}}g_i(u^1,...,u^i;a\theta^*+u^0_1,\theta^*)\right]\right)^\top \in\R^{2\times 2}, \quad -1\leq j \leq i-1,
\end{aligned}
\end{equation}
in the order of $\{\zeta_{0j}\}_{-1\leq j \leq -1} \rightarrow \{\xi_{0j}\}_{0\leq j \leq 0} \rightarrow \{\zeta_{1j}\}_{-1\leq j \leq 0} \rightarrow \{\xi_{1j}\}_{0\leq j \leq 1} \rightarrow \ldots$.

The following result, whose proof is given in Appendix \ref{sec: spec init amp}, characterizes the distribution of the AMP iterates in (\ref{eq:spec_init_AMP}). It slightly generalizes the result of \cite{mondelli_approximate_2022} to be applicable to AMP algorithms that depend on multiple past iterates.
\begin{theorem}\label{thm: amp spec} 
 Assume that the nonlinearities $\{f_i\},\{g_i\}$ are continuous and Lipschitz in their first $(i+1)$ and $i$ arguments respectively, $\lim_{n,d\rightarrow\infty} n/d = \delta$, the distribution of $\X\in\R^{n\times d}$ satisfies Assumption \ref{ass: model assumptions}-(3), and that $\psi_\delta'(\lambda_\delta^*) > 0$. Suppose $\btheta^\ast \in \R^d$ and $\z\in\R^n$ are independent of $\X$ and satisfy the distributional assumption of (\ref{eq:amp_init_cond}), where $\E[z^2], \E[(\theta^\ast)^2] < \infty$. For any fixed integer $m\geq 0$ and 2-pseudo Lipschitz functions $\psi: \R^{2m} \times \R^2 \rightarrow \R$ and $\phi:\R^{2(m+1)} \times \R\rightarrow\R$, it holds almost surely as $n,d\rightarrow\infty$ that
    \begin{align*}
        \frac{1}{d}\sum_{j=1}^d \psi\Big((\a^1)_j,...,(\a^m)_j;(\btheta^0)_j,(\btheta^*)_j\Big) &\rightarrow \E\left[\psi(u^1,...,u^m;a\theta^*+u^0_1,\theta^*)\right]\\
        \frac{1}{n}\sum_{i=1}^n \phi\Big((\b^0)_i,...,(\b^m)_i;(\z)_i\Big) &\rightarrow \E\left[\phi(w^0,...,w^m;z)\right].
    \end{align*}
\end{theorem}

\subsection{Reduction to AMP algorithm}
To connect the AMP algorithm (\ref{eq:spec_init_AMP}) back to the gradient descent dynamic (\ref{eq: gradient descent}), we show the existence of nonlinearities $\{f_j\}_{j\geq 0},\{g_j\}_{j\geq 0}$ such that the gradient descent iterates can be represented as (separable) functions of the AMP iterates, as detailed in the following lemma.

\begin{lemma}
    \label{lm: AMP nonlinearity choice}
    There exist continuous nonlinearities $g_j=(g_{j,1},g_{j,2}):\R^{2j+2} \rightarrow\R^2$ and $f_j=(f_{j,1},f_{j,2}):\R^{2(j+1)+1} \rightarrow\R^2$ which are also Lipschitz in their first $j$ and $j+1$ arguments such that for any $j\geq 0$,
    \begin{align}\label{eq:f_g_choice}
        g_j(\a^1,...,\a^j;\btheta^0,\btheta^*) = (\btheta^{j},\btheta^*),  \quad f_j(\b^0,...,\b^j;\z) = (\ell(\X\btheta^j,\X\btheta^*,\z),\mathbf{0}).
    \end{align}
\end{lemma}
\begin{proof}
    We construct $g_j$ and $f_j$ inductively.
    
    \noindent\textbf{Base Case}
    By definition we have $g_0(\theta^0,\theta^\ast) = (\theta^0, \theta^\ast)$ so that $g_0(\btheta^0,\btheta^*) = (\btheta^0,\btheta^*)$, proving the first claim. This leads to $\zeta_{0,-1} = [1,0;0,0]$, and
    \begin{align}
        \b^0 = \Big(\X\btheta^0-\frac{1}{\lambda_\delta^*}\bZ_s\X\btheta^0,\X\btheta^*\Big) = \Big(\Big(\I_n-\frac{1}{\lambda_\delta^*}\bZ_s\Big)\X\btheta^0,\X\btheta^*\Big).
    \end{align}
    By direct calculation, we have
    \begin{equation*}
        \Big(\I_n-\frac{1}{\lambda_\delta^*}\bZ_s\Big)^{-1} = \diag\Big(1+\cT(\varphi(\X\btheta^*,\z))\Big) = \diag\Big(1+\cT(\y)\Big),
    \end{equation*}
    where $\cT(y) = \frac{\cT_s(y)}{\lambda_\delta^*-\cT_s(y)}$ is given by (\ref{def:T_func}). By writing $\b^0 = (\b^0_1,\b^0_2)$, this implies 
    \begin{align}\label{eq:represent_Xtheta0}
    \X\btheta^0 = \big(\bm{1}_n+\cT(\varphi(\b^0_2,\z))\big)\circ\b^0_1,
    \end{align}
    hence we can define $f_0 = (f_{0,1},f_{0,2})$ satisfying the second claim by
    \begin{align*}
        f_{0,1}(b^0;z) = \ell\big((1+\cT(\varphi(b^0_2,z)))b^0_1,b^0_2,z\big), \quad f_{0,2}(b^0;z) = 0.
    \end{align*}
    
    \noindent\textbf{Inductive Step} Assume we have defined $\{g_\ell\}_{\ell=0}^j$ and $\{f_\ell\}_{\ell=0}^j$ up to some $j\geq 0$ satisfying (\ref{eq:f_g_choice}). We now define $g_{j+1}$ and $f_{j+1}$. For $g_{j+1}$, note that
    \begin{align*}
        \a^{j+1} = -\frac{1}{\delta}\X^\top f_j(\b^0,...,\b^j;\z) +\sum_{i=0}^j g_i(\a^1,...,\a^i;\btheta^0,\btheta^*)\xi_{ji},
    \end{align*}
    which implies
    \begin{align}
    \label{eq: amp postprocess ind a}
        -\frac{1}{\delta}\X^\top f_{j,1}(\b^0,...,\b^j;\z) = \a^{j+1}_1 - \sum_{i=0}^j g_{i,1}(\a^1,...,\a^i;\btheta^0,\btheta^*)(\xi_{ji})_{11} - \sum_{i=0}^j \btheta^* (\xi_{ji})_{21}.
    \end{align}
    On the other hand, we have by (\ref{eq: gradient descent}) that
    \begin{align*}
        \btheta^{j+1} &= \btheta^{j} - \gamma\X^\top \ell(\X\btheta^{j},\X\btheta^*,\z) - \gamma\lambda \btheta^{j}\\
        &=\btheta^{j} - \gamma\X^\top f_{j,1}(\b^0,...,\b^j;\z) - \gamma\lambda \btheta^{j}\\
        &= (1-\gamma\lambda)g_{j,1}(\a^1,...,\a^j;\btheta^0,\btheta^*) - \gamma\X^\top f_{j,1}(\b^0,...,\b^j;\z)\\
        &= (1-\gamma\lambda)g_{j,1}(\a^1,...,\a^j;\btheta^0,\btheta^*) + \gamma\delta\Big[\a^{j+1}_1 - \sum_{i=0}^j g_{i,1}(\a^1,...,\a^i;\btheta^0,\btheta^*)(\xi_{ji})_{11} - \sum_{i=0}^j \btheta^* (\xi_{ji})_{21}\Big],
    \end{align*}
    where the second and third steps follow from the inductive assumption, and the last equality follows from (\ref{eq: amp postprocess ind a}).
    To satisfy $g_{j+1}(\a^1,...,\a^{j+1};\btheta^0,\btheta^*) = (\btheta^{j+1},\btheta^*)$, we define
    \begin{align}
    \label{eq:def g}
        \notag g_{j+1,1}(a^1,...,a^{j+1};\theta^0,\theta^*) &= (1-\gamma\lambda)g_{j,1}(a^1,...,a^j;\theta^0,\theta^*)\\
        \notag&\quad + \gamma\delta\Big[a^{j+1}_1 - \sum_{i=0}^j g_{i,1}(a^1,...,a^i;\theta^0,\theta^*)(\xi_{ji})_{11} - \sum_{i=0}^j \theta^* (\xi_{ji})_{21}\Big],\\
        g_{j+1,2}(a^1,...,a^{j+1};\theta^0,\theta^*) &= \theta^*.
    \end{align}
    Next we construct $f_{j+1}$. With $g_{j+1}$ defined above, we have
    \begin{align}
        \notag\b_1^{j+1} &= \X g_{j+1,1}(\a^1,...,\a^{j+1};\btheta^0,\btheta^*) + \frac{1}{\delta} \sum_{i=0}^{j} f_{i,1}(\b^0,...,\b^i;\z)(\zeta_{j+1,i})_{11}-\frac{1}{\lambda_\delta^*}\bZ_s \X\btheta^0(\zeta_{j+1,-1})_{11}\\
        \notag&=\X\btheta^{j+1} + \frac{1}{\delta} \sum_{i=0}^{j} f_{i,1}(\b^0,...,\b^i;\z)(\zeta_{j+1,i})_{11}-\frac{1}{\lambda_\delta^*}\bZ_s\diag\Big(1+\cT(\varphi(\b_2^0,\z))\Big)\b_1^0(\zeta_{j+1,-1})_{11}\\
        &=\X\btheta^{j+1} + \frac{1}{\delta} \sum_{i=0}^{j} f_{i,1}(\b^0,...,\b^i;\z)(\zeta_{j+1,i})_{11}-\diag\Big(\cT(\varphi(\b_2^{0},\z))\Big)\b_1^0(\zeta_{j+1,-1})_{11},
    \end{align}
    where we apply (\ref{eq:represent_Xtheta0}). This implies
    \begin{equation}\label{eq:represent_Xtheta}
        \X\btheta^{j+1} = \b_1^{j+1} -\frac{1}{\delta} \sum_{i=0}^{j} f_{i,1}(\b^0,...,\b^i;\z)(\zeta_{j+1,i})_{11} + \diag(\cT(\varphi(\b_2^{0},\z)))\b_1^0(\zeta_{j+1,-1})_{11}.
    \end{equation}
    To satisfy $f_{j+1}(\b^0,...,\b^{j+1};\z) = (\ell(\X\btheta^{j+1},\X\btheta^*,\z),\mathbf{0})$, define
    \begin{align}\label{eq: def f}
    \begin{split}
        f_{j+1,1}(b^0,...,b^{j+1};z) &= \ell\Big(b_1^{j+1} -\frac{1}{\delta} \sum_{i=0}^{j} f_{i,1}(b^0,...,b^i;z)(\zeta_{j+1,i})_{11} + \cT(\varphi(b_2^{0},z))b_1^0(\zeta_{j+1,-1})_{11},b_2^{0},z\Big),\\
        f_{j+1,2}(b^0,...,b^{j+1};z) &= 0.
    \end{split}
    \end{align}
    By induction, it is easy to see that the defined $f_i,g_i$ satisfy the  Lipschitz conditions of Theorem \ref{thm: amp spec} if $\pnorm{\partial_1 \ell}{\infty}, \pnorm{\partial_2 \ell}{\infty}, \pnorm{\cT'}{\infty}, \pnorm{\partial_1 \varphi}{\infty}$ are bounded, which are satisfied by Assumptions \ref{ass: model assumptions}, \ref{ass: spectral initialization} and \ref{ass: ell lipschitz}.
\end{proof}

We are now ready to prove Theorem \ref{thm: discrete dmft asymp}.
\begin{proof}[Proof of Theorem \ref{thm: discrete dmft asymp}]
Let $f_i,g_i$ be defined as in Lemma \ref{lm: AMP nonlinearity choice}. Define for $i\geq 0$ 
\begin{align*}
    \theta^i = g_{i,1}(u^1,...,u^{i};\theta^0,\theta^*),
\end{align*}
initialized from $\theta^0=a\theta^*+u_1^0$. Recalling from (\ref{eq:represent_Xtheta}) that
    \begin{align*}
        \X\btheta^{i} = \b_1^{i} -\frac{1}{\delta} \sum_{j=0}^{i-1} f_{j,1}(\b^0,...,\b^j;\z)(\zeta_{i,j})_{11} + \diag(\cT(\varphi(\b_2^{0},\z)))\b_1^0(\zeta_{i,-1})_{11},
    \end{align*}
further define for $i\geq 0$
\begin{equation}\label{eq:def_eta}
    \eta^i = w_1^{i} -\frac{1}{\delta} \sum_{j=0}^{i-1} f_{j,1}(w^0,...,w^j;z)(\zeta_{i,j})_{11} + \cT(\varphi(w^{0}_2,z))w_1^{0}(\zeta_{i,-1})_{11}.
\end{equation}
By Lemma \ref{lm: AMP nonlinearity choice}, it holds that
\begin{align*}
    f_{i,1}(w^0,\ldots,w^i;z) = \ell(\eta^i, w^0_2, z).
\end{align*}
Applying the state evolution of Theorem \ref{thm: amp spec}, we see that for any fixed integer $m\in\Z_+$, it holds almost surely as $n,d\rightarrow\infty$ that
\begin{align*}
    \frac{1}{d}\sum_{j=1}^d \delta_{\left(\theta_{j}^{0},...,\theta_{j}^{m},\theta^*_j\right)} \overset{W_2}{\rightarrow} \sP(\theta^0,...,\theta^m,\theta^*),\quad\frac{1}{n}\sum_{i=1}^n\delta_{\left(\eta_{i}^{0},...,\eta_{i}^{m},\eta^*_i,z_i\right)} \overset{W_2}{\rightarrow} \sP(\eta^0,...,\eta^m,w^0_2,z).
\end{align*}
We define the following quantities (on the left) using the state evolution variables (on the right):
\begin{gather}
    \theta_\gamma^{i} := \theta^i, \qquad \eta_\gamma^{i} := \eta^i  \label{eq:discrete_equiva_1}\\ 
    (u_\gamma^\diamond, \{u_\gamma^{i}\}_{i\geq 0}) := (u_1^0,\delta \{u_1^{i+1}\}_{i\geq 0}),  \qquad(w^*_\gamma, \{w_\gamma^{i}\}_{i\geq 0}) := (w^0_2, \{w^i_1\}_{i\geq 0}), \label{eq:discrete_equiva_2}\\
    \Big\{\frac{\partial \theta_\gamma^{i}}{\partial u_\gamma^{j}}\Big\}_{i > j \geq 0} := \Big\{\frac{1}{\delta}\frac{\partial g_{i,1}}{\partial u^{j+1}_1}\Big\}_{i>j\geq 0}, \quad \Big\{\frac{\partial \theta_\gamma^i}{\partial u^\diamond_\gamma}\Big\}_{i\geq 0} := \Big\{\frac{\partial g_{i,1}}{\partial u^0_{1}}\Big\}_{i\geq 0},\label{eq:discrete_equiva_3}\\
    \Big\{\frac{\partial \ell_\gamma^{i}}{\partial w_\gamma^{j}}\Big\}_{i>j\geq 0} := \Big\{\frac{\partial f_{i,1}}{\partial w_1^{j}}\Big\}_{i>j\geq 0}, \quad \Big\{\frac{\partial \ell_\gamma^{i}}{\partial w_\gamma^{*}}\Big\} := \Big\{\frac{\partial f_{i,1}}{\partial w_2^{0}}\Big\}_{i > 0} \label{eq:discrete_equiva_5}.
\end{gather}
Based on these definitions, further define
\begin{equation}\label{def:R_gamma}
\begin{aligned}
R_\theta^\gamma(i,j) &:= \E\Big[\frac{\partial\theta^i_\gamma}{\partial u^j_\gamma}\Big] = \frac{1}{\delta}\E\Big[\frac{\partial g_{i,1}}{\partial u^{j+1}_1}\Big] = \frac{1}{\delta}(\zeta_{ij})_{11}, \quad i > j \geq 0,\\
R_\theta^\gamma(i,\diamond) &:= \E\Big[\frac{\partial\theta^i_\gamma}{\partial u^\diamond_\gamma}\Big] = \E\Big[\frac{\partial g_{i,1}}{\partial u^0_1}\Big] = (\zeta_{i,-1})_{11}, \quad i\geq 0,\\
R_\eta^\gamma(i,j) &:= \delta\E\Big[\frac{\partial \ell^i_\gamma}{\partial w^j_\gamma}\Big] = \delta\E\Big[\frac{\partial f_{i,1}}{\partial w^j_1}\Big] = \delta(\xi_{ij})_{11}, \quad i > j \geq 0,\\
R_\eta^\gamma(i, \ast) &:= \delta\E\Big[\frac{\partial \ell^i_\gamma}{\partial w^\ast_\gamma}\Big] = \delta \E\Big[\frac{\partial f_{i,1}}{\partial w^0_2}\Big] = \delta (\xi_{i0})_{21}, \quad i \geq 0. 
\end{aligned}
\end{equation}
We check that these variables (on the left) satisfy the DMFT system. 

We first verify the covariance kernels of $(u^\diamond_\gamma, \{u^i_\gamma\}_{i\geq 0})$ and $(w^\ast_\gamma, \{w^i_\gamma\}_{i\geq 0}\})$ in (\ref{eq:discrete_equiva_2}) are identical to those in (\ref{def:eta_cov})-(\ref{def:theta_cov}). By the covariance structure of $\{u^i\}_{i\geq 0}$ and $\{w^i\}_{i\geq 0}$ in (\ref{eq:amp_gp_cov}), we have: with $y^\ast_\gamma = \varphi(w^\ast_\gamma, z)$,
\begin{equation}\label{eq:check_covariance}
\begin{aligned}
\E[u^i_\gamma u^j_\gamma] &= \delta^2 \E[u^{i+1}_1u^{j+1}_1] = \delta\E\Big[f_{i,1}(w^0,\ldots,w^i;z)f_{j,1}(w^0,\ldots,w^j;z)\Big] = \delta\E[\ell(\eta^i_\gamma,w^\ast_\gamma,z)\ell(\eta^j_\gamma, w^\ast_\gamma, z)],\\
\E[u^i_\gamma u^\diamond_\gamma] &= \delta\E[u^{i+1}_1u^0_1] = -\delta \E\Big[f_{i,1}(w^0,\ldots,w^i;z)\cT(\varphi(w^0_2,z))w^0_1\Big] = -\delta \E[\ell(\eta_\gamma^i,w^\ast_\gamma,z)\cT(y^\ast_\gamma)w^0_1]\\
&\stackrel{(*)}{=} -\frac{\delta}{\lambda^\ast_\delta}\E\Big[\ell(\eta_\gamma^i, w^\ast_\gamma,z)\cT_s(y^\ast_\gamma)\eta^0_\gamma\Big],\\
\E[(u^\diamond_\gamma)^2] &= \E[(u^0_1)^2] = 1-a^2,\\
\E[w^i_\gamma w^j_\gamma] &= \E[w^i_1w^j_1] = \E\Big[g_{i,1}(u^1,\ldots,u^i;a\theta^\ast + u^0_1,\theta^\ast)g_{j,1}(u^1,\ldots,u^j;a\theta^\ast + u^0_1,\theta^\ast)\Big] = \E[\theta^i_\gamma\theta_\gamma^j],\\
\E[w^i_\gamma w^\ast_\gamma] &= \E[w^0_2w^i_1] = \E[\theta^\ast g_{i,1}(u^1,\ldots,u^i;a\theta^\ast + u^0_1,\theta^\ast)] = \E[\theta^\ast \theta^i_\gamma],\\
\E[(w^\ast_\gamma)^2] &= \E[(w^0_2)^2] = \E[(\theta^\ast)^2] = 1,
\end{aligned}
\end{equation}
where $(*)$ follows from $w^0_1 = \frac{\eta^0}{1+\cT(y^\ast_\gamma)}$ via (\ref{eq:def_eta}). In view of (\ref{eq:discrete_equiva_1}), this shows the equivalence between covariance of  $(u^\diamond_\gamma, \{u^i_\gamma\}_{i\geq 0})$ and $(w^\ast_\gamma, \{w^i_\gamma\}_{i\geq 0}\})$ to that of $(u^\diamond, \{u^s\}_{s\geq 0})$ in (\ref{def:eta_cov}), (\ref{eq:u_cov}) and $(w^\ast, \{w^s\}_{s\geq 0})$ in (\ref{eq:w_cov}), (\ref{def:theta_cov}).

Next we derive the recursions for $\{\frac{\partial \theta_\gamma^{i}}{\partial u_\gamma^{j}}\}_{i>j\geq 0}$, $\{\frac{\partial \theta_\gamma^{i}}{\partial u_\gamma^{0}}\}_{i>0}$, $\{\frac{\partial \ell_\gamma^{i}}{\partial w_\gamma^{j}}\}_{i>j\geq 0}$, $\{\frac{\partial \ell_\gamma^{i}}{\partial w_\gamma^{*}}\}_{i>0}$. By definition, it suffices to derive the recursions for the $g$-partial derivatives $\{\frac{\partial g_{i,1}}{\partial u_1^{j+1}}(u^1,\ldots,u^i;a\theta^\ast+u^0_1,\theta^\ast)\}_{i>j\geq 0}$, $\{\frac{\partial g_{i,1}}{\partial u_1^{0}}(u^1,\ldots,u^i;a\theta^\ast+u^0_1,\theta^\ast)\}_{i>0}$, and $f$-partial derivatives $\{\frac{\partial f_{i,1}}{\partial w_1^{j}}(w^0,...,w^i;z)\}_{i>j\geq 0}$, and $\{\frac{\partial f_{i,1}}{\partial w^0_2}(w^0,...,w^i;z)\}_{i>0}$. We first compute the $f$-partial derivatives. By the construction of $\{f_i\}_{i\geq 0}$ in (\ref{eq: def f}), for any $j > 0$, we have
\begin{align}\label{eq:xi_ii_1}
 \frac{\partial f_{i,1}}{\partial w_1^j}(w^0,\ldots,w^i;z) = \partial_1\ell(\eta^i,w^0_2,z), \quad i = j,
\end{align}
and for $i > j$, we have
\begin{align}\label{eq:partial_f_w}
    \notag\frac{\partial f_{i,1}}{\partial w_1^j}(w^0,\ldots,w^i;z) &= 
        \partial_1\ell(\eta^i,w^0_2,z)\Big[-\frac{1}{\delta} \sum_{s=j+1}^{i-1}\frac{\partial f_{s,1}}{\partial w_1^j}(\zeta_{i,s})_{11}-\frac{1}{\delta}\partial_1\ell(\eta^j,w^0_2,z)(\zeta_{i,j})_{11}\Big]\\
        &= -\partial_1\ell(\eta^i,w^0_2,z)\Big[\sum_{s=j+1}^{i-1}\frac{\partial f_{s,1}}{\partial w_1^j}R_\theta^\gamma(i,s) + \partial_1\ell(\eta^j,w^0_2,z)R_\theta^\gamma(i,j)\Big].
\end{align}
For $j=0$, using $(\zeta_{0,-1})_{11} = 1$, we have
\begin{align}\label{eq:xi_ii_2}
\frac{\partial f_{i,1}}{\partial w_1^0}(w^0,\ldots,w^i;z) = 
\partial_1\ell(\eta^0,w^0_2,z)\big(1+\cT(\varphi(w^0_2,z))\big), \quad i = 0,
\end{align}
and for $i > 0$, we have
\begin{align*}
    &\frac{\partial f_{i,1}}{\partial w_1^0}(w^0,\ldots,w^i;z)\\
    &= 
        \partial_1\ell(\eta^i,w^0_2,z)\Big[-\frac{1}{\delta} \sum_{s=1}^{i-1}\frac{\partial f_{s,1}}{\partial w_1^0}(\zeta_{i,s})_{11}-\frac{1}{\delta}\partial_1\ell(\eta^0,w^0_2,z)\big(1+\mathcal{T}(\varphi(w^0_2,z))\big)(\zeta_{i,0})_{11}+\mathcal{T}(\varphi(w^0_2,z))(\zeta_{i,-1})_{11}\Big]\\
        &= \partial_1\ell(\eta^i,w^0_2,z)\Big[- \sum_{s=1}^{i-1}\frac{\partial f_{s,1}}{\partial w_1^0}R_\theta^\gamma(i,s)-\partial_1\ell(\eta^0,w^0_2,z)\big(1+\mathcal{T}(y^\ast_\gamma)\big)R^\gamma_\theta(i,0) + \mathcal{T}(y^\ast_\gamma)R_\theta^\gamma(i,\diamond)\Big],
\end{align*}
Lastly, we have 
\begin{align*}
    \frac{\partial f_{i,1}}{\partial w_2^0}(w^0,\ldots,w^i;z) &= \partial_1\ell(\eta^i,w^0_2,z)\Big[-\frac{1}{\delta}\sum_{s=0}^{i-1} \frac{\partial f_{s,1}}{\partial w_2^0}(\zeta_{i,s})_{11}+\cT'(\varphi(\omega^*,z))\varphi'(w^0_2,z)w_1^0(\zeta_{i,-1})_{11}\Big]+\partial_2\ell(\eta^i,w^0_2,z)\\
    &= \partial_1\ell(\eta^i,w^0_2,z)\Big[-\sum_{s=0}^{i-1} \frac{\partial f_{s,1}}{\partial w_2^0}R_\theta^\gamma(i,s)+\cT'(y^\ast_\gamma)\varphi'(w^0_2,z)w_1^0R_\theta^\gamma(i,\diamond)\Big]+\partial_2\ell(\eta^i,w^0_2,z).
\end{align*}

Next we compute the $g$-partial derivatives. By the construction of $\{g_j\}_{j\geq 0}$ in (\ref{eq:def g}), we have for any $i > j \geq 0$ that $\frac{\partial g_{i,1}}{\partial u_1^{j+1}} = \gamma\delta$ if $i = j+1$, and for $i \geq j+2,$
\begin{align}\label{eq:partial_g_u}
    \notag\frac{\partial g_{i,1}}{\partial u_1^{j+1}} &= 
        \Big(1-\gamma\lambda-\gamma\delta(\xi_{i-1,i-1})_{11}\Big)\frac{\partial g_{i-1,1}}{\partial u_1^{j+1}}-\gamma\delta\sum_{s=j+1}^{i-2}\frac{\partial g_{s,1}}{\partial u_1^{j+1}}(\xi_{i-1,s})_{11}\\
        &= \Big(1-\gamma\lambda-\gamma\delta \E[ \partial_1\ell(\eta^{i-1},w^0_2,z)]\Big)\frac{\partial g_{i-1,1}}{\partial u_1^{j+1}}-\gamma\sum_{s=j+1}^{i-2}\frac{\partial g_{s,1}}{\partial u_1^{j+1}}R_\eta^\gamma(i-1,s),
\end{align}
where we apply the fact
\begin{align}\label{eq:xi_ii_id}
    (\xi_{ii})_{11} &= \E\Big[\frac{\partial f_{i,1}(w^0,...,w^i;z)}{\partial w_1^i}\Big] = \E\Big[ \partial_1\ell(\eta^i,w^0_2,z)\big(1+\bm{1}\{i=0\}\mathcal{T}(\varphi(w^0_2,z))\big)\Big], \quad i \geq 0,
\end{align}
obtained by combining (\ref{eq:xi_ii_1}) and (\ref{eq:xi_ii_2}). 
Similarly, for $i > j = -1$, we have $\frac{\partial g_{i,1}}{\partial u_1^{0}} = 1$ if $i = 0$, and if $i > 0$,
\begin{align*}
\frac{\partial g_{i,1}}{\partial u_1^{0}} &=
\big(1-\gamma\lambda-\gamma\delta(\xi_{i-1,i-1})_{11}\big)\frac{\partial g_{i-1,1}}{\partial u_1^0}-\gamma\delta\sum_{s=0}^{i-2}\frac{\partial g_{s,1}}{\partial u_1^0}(\xi_{i-1,s})_{11}\\
&= \Big(1-\gamma\lambda-\gamma\delta\E\big[\partial_1\ell(\eta^{i-1},w^0_2,z)\big(1+\bm{1}\{i=1\}\mathcal{T}(y^\ast_\gamma)\big)\big]\Big)\frac{\partial g_{i-1,1}}{\partial u_1^0}-\gamma\sum_{s=0}^{i-2}\frac{\partial g_{s,1}}{\partial u_1^0}R_\eta^\gamma(i-1,s),
\end{align*}
where we apply (\ref{eq:xi_ii_id}). 

The following lemma connects the recursions derived above to the DMFT responses $r_\eta(t,s), r_\eta(t,\ast), r_\eta(t,\diamond), r_\eta(t,\diamond\diamond)$ in (\ref{eq:r_eta_response}) and $r_\theta(t,s),r_\theta(t,\diamond)$ in (\ref{eq:r_theta_response}), whose proof is given at the end of the section.
\begin{lemma}\label{lem:response_correspondence}
    Recall the processes $\{\frac{\partial \theta^t_\gamma}{\partial u^s_\gamma}\}_{t>s\geq 0}$, $\{\frac{\partial\theta^t_\gamma}{\partial u^\diamond_\gamma}\}_{t\geq 0}$, $\{\frac{\partial \ell^t_\gamma}{\partial w^s_\gamma}\}_{t>s\geq 0}$, $\{\frac{\partial \ell^t_\gamma}{\partial w^\ast_\gamma}\}_{t\geq 0}$ defined in (\ref{eq:discrete_equiva_3})-(\ref{eq:discrete_equiva_5}). The following identities hold:
    \begin{align}
    \frac{\partial\theta_\gamma^t}{\partial u_\gamma^s} &= r_\theta(t,s), \quad t>s\geq 0,\label{eq:response_id_3}\\
    \frac{\partial\theta_\gamma^t}{\partial u_\gamma^\diamond} &= r_\theta(t,\diamond), \quad t\geq 0,\label{eq:response_id_5}\\
    \frac{\partial\ell_\gamma^t}{\partial w^s_\gamma} &= r_\eta(t,s), \quad t > s \geq 1,\label{eq:response_id_4}\\
       \frac{\partial\ell_\gamma^t}{\partial w^0_\gamma}   &= r_\eta(t,0) + r_\eta(t,\diamond), \quad t\geq 1, \label{eq:response_id_1}\\
       \frac{\partial\ell_\gamma^t}{\partial w^*_\gamma} &= r_\eta(t,*) + r_\eta(t,\diamond\diamond) \quad t\geq 0. \label{eq:response_id_2}
    \end{align}
Consequently, by taking expectation on both sides and (\ref{def:R_gamma}), 
\begin{equation}\label{eq:response_R_correspond}
\begin{aligned}
R_\theta^\gamma(t,s) &= \frac{1}{\delta}(\zeta_{t,s})_{11} = R_\theta(t,s), \quad t > s \geq 0\\ 
R_\theta^\gamma(t,\diamond) &= (\zeta_{t,-1})_{11} = R_\theta(t,\diamond), \quad t \geq 0,\\ 
R_\eta^\gamma(t,s) & = \delta(\xi_{t,s})_{11} = R_\eta(t,s), \quad t > s \geq 1\\
R_\eta^\gamma(t,0) &= \delta(\xi_{t,0})_{11} = R_\eta(t,0) + R_\eta(t,\diamond), \quad t \geq 1\\
R_\eta^\gamma(t,\ast) &= \delta (\xi_{t,0})_{21} = R_\eta(t,\ast) + R_\eta(t,\diamond\diamond), \quad t\geq 0.
\end{aligned}
\end{equation}
\end{lemma}

Lastly, let us verify the recursions of $\{\theta^i_\gamma\}_{i\geq 0}$ and $\{\eta^i_\gamma\}_{i\geq 0}$ as defined by (\ref{eq:discrete_equiva_1}). By the construction of $\{g_i\}$ and $\{f_i\}$ in (\ref{eq:def g}) and (\ref{eq: def f}), we can write
\begin{align*}
    \theta_\gamma^{i+1} &= (1-\gamma\lambda)\theta_\gamma^i + \gamma\delta\Big[u^{i+1}_1 - \sum_{j=0}^i \theta_\gamma^j(\xi_{ij})_{11} - \sum_{j=0}^i \theta^* (\xi_{ij})_{21}\Big]\\
    &= \big(1-\gamma\lambda-\gamma\delta(\xi_{ii})_{11}\big)\theta_\gamma^i - \gamma\delta\sum_{j=0}^{i-1}\theta_\gamma^j(\xi_{ij})_{11} - \gamma\delta\theta^\ast\sum_{j=0}^{i} (\xi_{ij})_{21} + \gamma\delta u_1^{i+1}\\
    &= \big(1-\gamma\lambda-\gamma\delta(\xi_{ii})_{11}\big)\theta_\gamma^i - \gamma\sum_{j=0}^{i-1}R_\eta^\gamma(i,j)\theta_\gamma^j - \gamma (R_\eta(i,\ast) + R_\eta(i,\diamond\diamond))\theta^\ast + \gamma u^i_\gamma,
\end{align*}
where we apply the fact that $\sum_{j=0}^{i} (\xi_{ij})_{21} = (\xi_{i0})_{21} = \big(R_\eta(i,\ast) + R_\eta(i,\diamond\diamond)\big)/\delta$ by (\ref{eq:response_R_correspond}), since by construction (\ref{eq: def f}), $f_{i,1}$ does not depend on $w_2^1,...,w_2^i$ but only on $w_2^0$. If $i = 0$, we have by (\ref{eq:xi_ii_id}) that $(\xi_{00})_{11} = \E[\partial_1 \ell(\eta^0,w^0_2,z)(1+\cT(y^\ast_\gamma))] = \E[\partial_1 \ell(\eta^0,w^0_2,z)] + R_\eta(0,\diamond)/\delta$ and $(\xi_{00})_{21} = (R_\eta(0,\ast) + R_\eta(0,\diamond\diamond))/\delta$, so
\begin{align*}
\theta^1_\gamma = \theta^0_\gamma + \gamma\Big[-\lambda\theta^0_\gamma - \delta\E[\partial_1 \ell(\eta^0,w^0_2,z)] \theta^0_\gamma - R_\eta(0,\diamond)\theta^0_\gamma - (R_\eta(0,\ast) + R_\eta(0,\diamond\diamond))\theta^\ast + u^0_\gamma\Big],
\end{align*}
which matches $\theta^1$ in (\ref{def:dmft_theta}). For $i\geq 1$, we have $(\xi_{ii})_{11} = \E[\partial_1 \ell(\eta^i,w^0_2,z)]$, so using $R_\eta^\gamma(i,j) = R_\eta(i,j)$ for $i > j > 0$ and $R_\eta^\gamma(i,0) = R_\eta(i,0) + R_\eta(i,\diamond)$ via (\ref{eq:response_R_correspond}), we have 
\begin{align}
    \theta^{i+1}_\gamma &= \theta_\gamma^i +\gamma\Big(- \lambda\theta_\gamma^i -\delta \E\left[\partial_1\ell(\eta^i,w_2^0,z)\right]\theta_\gamma^{i}-\sum_{j=0}^{i-1}R_\eta(i,j)\theta_\gamma^j - R_\eta(i,\diamond)\theta_\gamma^0 - \big(R_\eta(i,\ast) + R_\eta(i,\diamond\diamond)\big)\theta^\ast +  u_\gamma^{i}\Big),
\end{align}
which matches (\ref{def:dmft_theta}). We also have by (\ref{eq:def_eta}) that 
\begin{align*}
    \notag\eta^i_\gamma &= -\frac{1}{\delta} \sum_{j=0}^{i-1} \ell(\eta^j_\gamma,w^0_2,z)(\zeta_{i,j})_{11} + \cT(\varphi(w_2^{0},z))w_1^0(\zeta_{i,-1})_{11}+w_1^{i}\\
    &= - \sum_{j=0}^{i-1} R_\theta(i,j)\ell(\eta^j_\gamma,w^0_2,z) + \cT(\varphi(w_2^{0},z))w_1^0R_\theta(i,\diamond)+w_1^i\\
    &=-\sum_{j=0}^{i-1}R_\theta(i,j)\ell(\eta^j_\gamma,w^0_2,z) + \mathcal{T}(y^\ast_\gamma)w^0_1 R_\theta(i,\diamond) + w^i_1,
\end{align*}
which matches (\ref{def:dmft_eta}). The proof is complete.
\end{proof}

\begin{proof}[Proof of Lemma \ref{lem:response_correspondence}]
(\ref{eq:response_id_3}) follows by comparing the recursion (\ref{eq:partial_g_u}) with that of $r_\theta(t,s)$ in (\ref{eq:r_theta_response}), and (\ref{eq:response_id_4}) follows by comparing (\ref{eq:partial_f_w}) with that of $r_\eta(t,s)$ in (\ref{eq:r_eta_response}). To see (\ref{eq:response_id_5}), we have $\frac{\partial \theta^0_\gamma}{\partial u^\diamond_\gamma} = r_\theta(0,\diamond) = 1$, and using $R_\eta(0,\diamond) = \delta\E[\partial_1\ell(\eta^0,w^0_2,z)\cT(y)]$
\begin{align*}
\frac{\partial \theta^1_\gamma}{\partial u^\diamond_\gamma} &= \Big(1-\gamma\lambda - \gamma\delta \E[\partial_1\ell(\eta^0, w^0_2,z)]\Big) - \gamma\delta \E[\partial_1\ell(\eta^0,w^0_2,z)\cT(y^\ast_\gamma)] \\
&= \Big(1-\gamma\lambda - \gamma\delta \E[\partial_1\ell(\eta^0, w^0_2,z)]\Big) - \gamma R_\eta(0,\diamond) = r_\theta(1,\diamond). 
\end{align*}
For $t \geq 2$, we have by (\ref{eq:response_R_correspond}) that
\begin{align*}
\frac{\partial \theta^t_\gamma}{\partial u^\diamond_\gamma} &= \Big(1-\gamma\lambda-\gamma\delta\E\big[\partial_1\ell(\eta^{t-1},w^0_2,z)\big]\Big)\frac{\partial \theta^{t-1}_\gamma}{\partial u^\diamond_\gamma}-\gamma\sum_{s=0}^{t-2} \frac{\partial\theta^s_\gamma}{\partial u^\diamond_\gamma} R_\eta^\gamma(t-1,s)\\
&= \Big(1-\gamma\lambda-\gamma\delta\E\big[\partial_1\ell(\eta^{t-1},w^0_2,z)\big]\Big)\frac{\partial \theta^{t-1}_\gamma}{\partial u^\diamond_\gamma}-\gamma\sum_{s=0}^{t-2} \frac{\partial\theta^s_\gamma}{\partial u^\diamond_\gamma} R_\eta(t-1,s) - \gamma R_\eta(t-1,\diamond) = r_\theta(t,\diamond).
\end{align*}
Next we prove (\ref{eq:response_id_1}) and (\ref{eq:response_id_2}) by induction.

\textbf{Proof of (\ref{eq:response_id_1})} For the baseline case $t =1$, we have 
\begin{align*}
\frac{\partial \ell^1_\gamma}{\partial w^0_\gamma} = \partial_1\ell(\eta^1, w^0_2, z)\cdot\big[-\partial_1\ell(\eta^0, w^0_2, z)\big(1+\cT(y^\ast_\gamma)\big)R_\theta(1,0) + \cT(y^\ast_\gamma)R_\theta(1,\diamond)\big],
\end{align*}
while $r_\eta(1,0) = -\partial_1\ell(\eta^1, w^\ast, z)\partial_1\ell(\eta^0, w^\ast, z)R_\theta(1,0)$, and $r_\eta(1,\diamond) = \partial_1\ell(\eta^1, w^\ast, z)\cdot\big[-\partial_1\ell(\eta^0, w^\ast, z)\cT(y)R_\theta(1,0) + \cT(y)R_\theta(1,\diamond)\big]$, proving the baseline case. Suppose now the claim holds for some $t\geq 1$. For $t+1$, we have
\begin{align*}
\frac{\partial \ell^{t+1}_\gamma}{\partial w^0_\gamma} &= \partial_1\ell(\eta^{t+1},w^0_2,z)\Big[-\sum_{r=1}^t \frac{\partial \ell^r_\gamma}{\partial w^0_\gamma}R_\theta^\gamma(t+1,r) - \partial_1 \ell(\eta^0,w^0_2,z)\big(1+\cT(y^\ast_\gamma)\big)R_\theta^\gamma(t+1,0) + \cT(y^\ast_\gamma)R_\theta(t+1,\diamond)\Big]\\
&= \partial_1\ell(\eta^{t+1},w^0_2,z)\Big[-\sum_{r=1}^t \big(r_\eta(r,0) + r_\eta(r,\diamond)\big)R_\theta(t+1,r)\\
&\quad - \partial_1 \ell(\eta^0,w^0_2,z)\big(1+\cT(y^\ast_\gamma)\big)R_\theta(t+1,0) + \cT(y^\ast_\gamma)R_\theta(t+1,\diamond)\Big],
\end{align*}
applying the induction hypothesis. On the other hand, we have
\begin{align*}
r_\eta(t+1,0) &= \partial_1\ell(\eta^{t+1},w^*,z)\Big[ -\sum_{r=1}^{t}r_\eta(r,0)R_\theta(t+1,r)-\partial_1\ell(\eta^{0},w^*,z)R_\theta(t+1,0)\Big],\\
r_\eta(t+1,\diamond) &= \partial_1\ell(\eta^{t+1},w^*,z)\bigg(-\sum_{r=0}^{t}r_\eta(r,\diamond)R_\theta(t+1,r) + \cT(y)R_\theta(t+1,\diamond)\bigg)\\
&= \partial_1\ell(\eta^{t+1},w^*,z)\bigg(-\sum_{r=1}^{t}r_\eta(r,\diamond)R_\theta(t+1,r) - \partial_1\ell(\eta^0,w^\ast,z)\cT(y)R_\theta(t+1,0) + \cT(y)R_\theta(t+1,\diamond)\bigg),
\end{align*}
concluding the induction for (\ref{eq:response_id_1}).

\noindent \textbf{Proof of (\ref{eq:response_id_2})} The baseline case $t = 0$ holds as
\begin{align*}
\frac{\partial \ell^0_\gamma}{\partial w^\ast_\gamma} &= \partial_1 \ell(\eta^0, w^0_2, z)\cT'(y^\ast_\gamma)\varphi'(w^0_2, z)w^0_1 R_\theta(0,\diamond) + \partial_2 \ell(\eta^0, w^0_2, z),\\
r_\eta(0,\ast) &= \partial_2 \ell(\eta^0, w^\ast, z), \quad r_\eta(0,\diamond\diamond) = \partial_1 \ell(\eta^0, w^\ast, z)\cT'(y)\varphi'(w^\ast, z)w^0 R_\theta(0,\diamond).
\end{align*}
Suppose the claim holds up to some $t\geq 0$. For $t+1$, we have
\begin{align*}
\frac{\partial \ell^{t+1}_\gamma}{\partial w^\ast_\gamma} &=  \partial_1 \ell(\eta^{t+1}, w^0_2,z)\Big[- \sum_{s=0}^t \frac{\partial \ell^s_\gamma}{\partial w^\ast_\gamma} R_\theta^\gamma(t+1,s) + \cT'(y^\ast_\gamma)\varphi'(w^0_2, z)w^0_1 R_\theta^\gamma(t+1,\diamond)\Big] + \partial_2 \ell(\eta^{t+1},w^0_2,z)\\
&= \partial_1 \ell(\eta^{t+1}, w^0_2,z)\Big[- \sum_{s=0}^t \big(r_\eta(s,\ast) + r_\eta(s,\diamond\diamond)\big) R_\theta(t+1,s) + \cT'(y^\ast_\gamma)\varphi'(w^0_2, z)w^0_1 R_\theta(t+1,\diamond)\Big] + \partial_2 \ell(\eta^{t+1},w^0_2,z),
\end{align*}
by the induction hypothesis.
The claim now holds by the definition of $r_\eta(t+1,\ast)$ and $r_\eta(t+1,\diamond\diamond)$.
\end{proof}

\section{Proof of fixed point system (Proposition \ref{prop: dmft fixed point})}
\label{sec:fixed_point}


\begin{proof}[Proof of Proposition \ref{prop: dmft fixed point}]
        We first introduce some notation. For any $t\geq0$, set
        \begin{equation*}
            \Gamma^t := \E[\partial_1\ell(\eta^t,w^\ast,z)].
        \end{equation*}

        Note that under regularity conditions on $\ell$ and $\eta^t\rightarrow_{L^2}\eta^\infty$, we have
        \begin{equation*}
            \Gamma^t \rightarrow\Gamma^\infty := \E[\partial_1\ell(\eta^\infty,w^\ast,z)].
        \end{equation*}
        Further define the values 
        \begin{equation}\label{def:R_eta_theta_infty}
        R_\eta^\infty = \sum_{s=1}^\infty R_\eta(s), \qquad R_\theta^\infty = \sum_{s=1}^\infty R_\theta(s),
    \end{equation}
    which exist by Definition \ref{def:exp_conv_spec_init}-(5). Lastly, let $R_\eta^\ast := \delta\E[r_\eta^\ast]$ and $R_\eta(t) := \delta\E[r_\eta(t)]$, and by Definition \ref{def:exp_conv_spec_init}-(2), we have
    \begin{align*}
    R_\eta(t+s,t) \rightarrow R_\eta(s), \quad R_\eta(t,\ast) \rightarrow R_\eta^\ast
    \end{align*}
    as $t\rightarrow\infty$ for fixed $s$.

    We first show (\ref{eq:fix_1}).
    We know from the recursion (\ref{def:dmft_theta})  that
    \begin{equation*}
            \frac{\theta^{t+1}-\theta^t}{\gamma} =- \lambda\theta^t -\delta \Gamma^t\theta^{t}-\sum_{s=0}^{t-1}R_\eta(t,s)\theta^{s} - R_\eta(t,\diamond)\theta^0 - \big(R_\eta(t,\ast) + R_\eta(t,\diamond\diamond)\big)\theta^* + u^{t}.
    \end{equation*}
    By Definition \ref{def:exp_conv_spec_init}-(1), as $t\rightarrow \infty$, the left side satisfies $(\theta^{t+1}-\theta^t)/\gamma \overset{L^2}{\rightarrow} 0$. Under the conditions in Definition \ref{def:exp_conv_spec_init}, we further have the convergence
    \begin{gather*}
        \lambda\theta^t \overset{L^2}{\to}\lambda\theta^\infty,\quad \delta\Gamma^t\theta^t \overset{L^2}{\to}\delta\Gamma^\infty\theta^\infty, \quad u^t \overset{L^2}{\to} u^\infty\\
        R_\eta(t,\diamond)\theta^0 \overset{L^2}{\to} 0, \quad \big(R_\eta(t,\ast) + R_\eta(t,\diamond\diamond)\big)\theta^* \overset{L^2}{\to} R_\eta^\ast \theta^\ast,
    \end{gather*}
    the last one following from $R_\eta(t,\ast) = \delta\E[r_\eta(t,\ast)] \rightarrow \delta \E[r_\eta^\ast] = R_\eta^\ast$. It remains to show $\sum_{s=0}^{t-1}R_\eta(t,s)\theta^{s} \overset{L^2}{\to} R_\eta^\infty \theta^\infty$. To this end, we claim that
    \begin{align}\label{eq:R_eta_sum}
    \lim_{t\rightarrow\infty} \sum_{s=0}^{t-1} R_\theta(t,t-s-1) = R_\theta^\infty, \quad \lim_{t\rightarrow\infty} \sum_{s=0}^{t-1} R_\eta(t,t-s-1) = R_\eta^\infty,
    \end{align}
    with the two values on the right given by (\ref{def:R_eta_theta_infty}).
    We prove the second claim, and the first one follows from similar arguments. Note that
    \begin{align*}
    \Big|\sum_{s=0}^{t-1} R_\eta(t,t-s-1) - R_\eta^\infty\Big| &= \Big|\sum_{s=0}^{t-1} R_\eta(t,t-s-1) - \sum_{s=0}^{\infty} R_\eta(s+1)\Big|\\
    &\leq \sum_{s=0}^{t-1} |R_\eta(t,t-s-1) - R_\eta(s+1)| + \sum_{s=t}^\infty |R_\eta(s+1)|.
    \end{align*}
    Note that $|R_\eta(s)| \leq Ce^{-cs}$, which follows from 
    \begin{align*}
    |R_\eta(s)| \leq |R_\eta(t+s,t) - R_\eta(s)| + |R_\eta(t+s,t)| \leq |R_\eta(t+s,t) - R_\eta(s)| + Ce^{-cs},
    \end{align*}
    followed by taking $t\rightarrow\infty$. This implies $\sum_{s=t}^\infty |R_\eta(s+1)| \leq Ce^{-ct}$. On the other hand, for any integer $\Delta$ and large enough $t$, again by $|R_\eta(s)| \leq Ce^{-cs}$ we have
    \begin{align*}
    \sum_{s=0}^{t-1} |R_\eta(t,t-s-1) - R_\eta(s+1)| &\leq \Delta \max_{0\leq s\leq \Delta} |R_\eta(t,t-s) - R_\eta(s)| + \sum_{s=\Delta+1}^{t-1} (|R_\eta(t,t-s-1)| + |R_\eta(s+1)|)\\
    &\leq \Delta \max_{0\leq s\leq \Delta} |R_\eta(t,t-s) - R_\eta(s)| + Ce^{-c\Delta}. 
    \end{align*}
     Then by taking $t\rightarrow\infty$ followed by $\Delta \rightarrow\infty$, we have by Definition \ref{def:exp_conv_spec_init}-(3) that $\sum_{s=0}^{t-1} |R_\eta(t,t-s-1) - R_\eta(s+1)| \rightarrow 0$ as $t\rightarrow\infty$, proving the claim (\ref{eq:R_eta_sum}).
    We can then write
    \begin{align*}
    |\sum_{s=0}^{t-1} R_\eta(t,s)\theta^s - R_\eta^\infty\theta^\infty| &= |\sum_{s=0}^{t-1} R_\eta(t,t-s-1)\theta^{t-s-1} - R_\eta^\infty\theta^\infty|\\
    &\leq \underbrace{\sum_{s=0}^{t-1} |R_\eta(t,t-s-1)||\theta^{t-s-1} - \theta^\infty|}_{(I_t)} + \underbrace{\Big|\sum_{s=0}^{t-1} R_\eta(t,t-s-1) - R_\eta^\infty\Big| |\theta^\infty|}_{(II_t)}.
    \end{align*}
    Using claim (\ref{eq:R_eta_sum}), we have $(II_t) \overset{L^2}{\to} 0$. On the other hand, by Definition \ref{def:exp_conv_spec_init}-(1)(5), we have 
    \begin{align*}
        \pnorm{(I_t)}{L^2} \leq \sum_{s=0}^{t-1}|R_\eta(t,t-s-1)|\pnorm{\theta^{t-s-1}-\theta^\infty}{L^2} \leq C\sum_{s=0}^{t-1} e^{-c(s+1)}e^{-c(t-s-1)} = Cte^{-ct} \rightarrow 0.
    \end{align*}
    This concludes the claim $\sum_{s=0}^{t-1}R_\eta(t,s)\theta^{s} \overset{L^2}{\to} R_\eta^\infty \theta^\infty$, thereby showing (\ref{eq:fix_1}).
    
    To see (\ref{eq:fix2}), we have by (\ref{def:dmft_eta}) the recursion
    \begin{equation*}
        \eta^t = -\sum_{s=0}^{t-1}\ell(\eta^{s},w^*,z)R_\theta(t,s) + \cT(y^*)w^0R_\theta(t,\diamond) + w^t.
    \end{equation*}
    The left side satisfies $\eta^t \overset{L^2}{\to} \eta^\infty$ by Definition \ref{def:exp_conv_spec_init}-(1). By definition \ref{def:exp_conv_spec_init}-(2)(5), as $t\rightarrow\infty$, it holds that
    \begin{equation*}
        \cT(y^*)w^0R_\theta(t,\diamond) \overset{L^2}{\to} 0, \quad w^t \overset{L^2}{\to} w^\infty.
    \end{equation*}
    For the other term, we have
    \begin{align*}
        &\Big|\sum_{s=0}^{t-1}\ell(\eta^{s},w^*,z)R_\theta(t,s) - R_\theta^\infty \ell(\eta^\infty, w^\ast, z)\Big| =\Big|\sum_{s=0}^{t-1}\ell(\eta^{t-s-1},w^*,z)R_\theta(t,t-s-1) - R_\theta^\infty \ell(\eta^\infty, w^\ast, z)\Big|\\
        &\leq \underbrace{\sum_{s=0}^{t-1}\big|\ell(\eta^{t-s-1},w^*,z)-\ell(\eta^\infty,w^*,z)\big||R_\theta(t,t-s-1)|}_{(I_t)} + \underbrace{\Big|\sum_{s=0}^{t-1}R_\theta(t,t-s-1) - R_\theta^\infty\Big||\ell(\eta^\infty,w^*,z)|}_{(II_t)}.
    \end{align*}
    Term $(II_t) \overset{L^2}{\to} 0$ by claim (\ref{eq:R_eta_sum}). By Assumption \ref{ass: ell lipschitz}, term $(I_t)$ satisfies 
    \begin{align*}
    \pnorm{(I_t)}{L^2} &\leq \sum_{s=0}^{t-1} |R_\theta(t,t-s-1)|\pnorm{\ell(\eta^{t-s-1,w^\ast,z}) - \ell(\eta^\infty,w^\ast,z)}{L^2}\\
    &\leq C\sum_{s=0}^{t-1}e^{-c(t-s-1)}\pnorm{\eta^{t-s-1} - \eta^\infty}{L^2} \leq C\sum_{s=0}^{t-1}e^{-c(t-s-1)}e^{-cs} \rightarrow 0.
    \end{align*}
    This concludes that $\sum_{s=0}^{t-1}\ell(\eta^{s},w^*,z)R_\theta(t,s) \overset{L^2}{\to} R_\theta^\infty \ell(\eta^\infty, w^\ast, z)$, which implies (\ref{eq:fix2}) after some rearranging. 

    To see (\ref{eq:fix3}), since $u^t\overset{L^2}{\to} u^\infty$, we have
    \begin{align*}
        C^\infty_\eta = \lim_{t\rightarrow\infty} \E[(u^t)^2] =  \lim_{t\rightarrow\infty} C_\eta(t,t) = \E\Big[\ell^2(\eta^\infty,w^\ast,z)\Big].
    \end{align*}
    Similarly, $(w^t,w^\ast) \overset{L^2}{\to} (w^\infty,w^\ast)$, so 
    \begin{align*}
        C_\theta^\infty = \lim_{t\rightarrow\infty}\E[(w^t,w^\ast)(w^t,w^\ast)^\top] = \lim_{t\rightarrow\infty} 
        \begin{bmatrix}
        C_\theta(t,t) & C_\theta(t,\ast)\\
        C_\theta(t,\ast) & C_\theta(\ast,\ast)
        \end{bmatrix}
        = \E\Big[(\theta^\infty,\theta^\ast)(\theta^\infty,\theta^\ast)^\top\Big].
    \end{align*}
    
    Next we show (\ref{eq:fix5}). Using (\ref{eq:r_theta_response}) and that $r_\theta(t,s) = R_\theta(t,s)$, we have for any fixed $s\geq 1$ that
    \begin{equation*}
        R_\theta(t+1+s,t) = (1-\gamma\lambda-\gamma\delta\Gamma^{t+s})R_\theta(t+s,t)-\gamma\sum_{r=t+1}^{t+s-1}R_\eta(t+s,r)R_\theta(r,t).
    \end{equation*}
    For such fixed $s$, taking $t\rightarrow\infty$ yields
    \begin{equation*}
        R_\theta(s+1) = (1-\gamma\lambda-\gamma\delta\Gamma^{\infty})R_\theta(s) - \gamma\sum_{r=1}^{s-1}R_\eta(r)R_\theta(s-r),
    \end{equation*}
    where we apply the fact that
    \begin{align*}
        \sum_{r=t+1}^{t+s-1}R_\eta(t+s,r)R_\theta(r,t) = \sum_{r=1}^{s-1}R_\eta(t+s,t+r)R_\theta(t+r,t) \rightarrow \sum_{r=1}^{s-1}R_\eta(s-r)R_\theta(r) = \sum_{r=1}^{s-1}R_\eta(r)R_\theta(s-r).
    \end{align*}
    The above display can be rearranged as
    \begin{equation*}
        \frac{R_\theta(s)-R_\theta(s+1)}{\gamma} = \lambda R_\theta(s) + \delta \Gamma^\infty R_\theta(s) + \sum_{r=1}^{s-1}R_\eta(r)R_\theta(s-r).
    \end{equation*}
    Summing both sides from $s=1$ to $\infty$, we get
    \begin{equation*}
        \sum_{s=1}^\infty  \frac{R_\theta(s)-R_\theta(s+1)}{\gamma} = \lambda R_\theta^\infty + \delta \Gamma^\infty R_\theta^\infty + \sum_{s=1}^\infty\sum_{r=1}^{s-1}R_\eta(r)R_\theta(s-r).
    \end{equation*}
    Using $R_\theta(t+1,t) = \gamma$ for all $t\geq 0$, we have $R_\theta(1) = \gamma$, so we can compute the left side by
    \begin{align*}
         \sum_{s=1}^\infty  \frac{R_\theta(s)-R_\theta(s+1)}{\gamma} &= \frac{R_\theta(1)}{\gamma} = 1.
    \end{align*}
    The double sum on the right side can be computed as
    \begin{align*}
        \sum_{s=1}^\infty\sum_{r=1}^{s-1}R_\eta(r)R_\theta(s-r) &= \sum_{r=1}^\infty R_\eta(r)\sum_{s=r+1}^\infty R_\theta (s-r) = R_\theta^\infty \sum_{r=1}^\infty R_\eta(r) = R_\theta^\infty R_\eta^\infty.
    \end{align*}
    In summary, we have $1 = (\lambda+\delta\Gamma^\infty+R_\eta^\infty)R_\theta^\infty$, which yields (\ref{eq:fix5}) after some rearranging. 

    Next we show (\ref{eq:fix6}). Using (\ref{eq:r_eta_response}), we have for any fixed $s\geq 1$ that
    \begin{align*}
        r_\eta(t+s,t) = 
\partial_1\ell(\eta^{t+s},w^*,z)\Big(-\sum_{r=t+1}^{t+s-1}r_\eta(r,t)R_\theta(t+s,r)-\partial_1\ell(\eta^t,w^*,z)R_\theta(t+s,t)\Big).
    \end{align*}
    For such fixed $s$, taking the $L^2$ limit of both sides as $t\rightarrow \infty$, we get
    \begin{equation*}
             r_\eta(s) = 
\partial_1\ell(\eta^{\infty},w^*,z)\Big(-\sum_{r=1}^{s-1} R_\theta(s-r)r_\eta(r)-\partial_1\ell(\eta^\infty,w^*,z)R_\theta(s)\Big),
    \end{equation*}
    where we apply
    \begin{align*}
    \sum_{r=t+1}^{t+s-1} r_\eta(r,t) R_\theta(t+s,r) = \sum_{r=1}^{s-1} r_\eta(t+r,t) R_\theta(t+s,t+r) \overset{L^2}{\to} \sum_{r=1}^{s-1} R_\theta(s-r)r_\eta(r).
    \end{align*}
    Summing both sides from $s=1$ to $\infty$ (such $L^2$ limit exists due to (\ref{eq:response_exp_decay})), we get
    \begin{align*}
        \sum_{s=1}^\infty r_\eta(s) &= 
\partial_1\ell(\eta^{\infty},w^*,z)\Big(-\sum_{s=1}^\infty\sum_{r=1}^{s-1}r_\eta(r)R_\theta(s-r)-\partial_1\ell(\eta^\infty,w^*,z)R_\theta^\infty\Big)\\
&= \partial_1\ell(\eta^{\infty},w^*,z)\Big(-R_\theta^\infty\sum_{s=1}^\infty r_\eta(s)-\partial_1\ell(\eta^\infty,w^*,z)R_\theta^\infty\Big),
    \end{align*}
   which further implies that
    \begin{align*}
        \sum_{s=1}^\infty r_\eta(s) &= -\Big(1+\partial_1\ell(\eta^{\infty},w^*,z)R_\theta^\infty\Big)^{-1}\partial_1\ell(\eta^{\infty},w^*,z)^2 R_\theta^\infty\\
        &=-\partial_1\ell(\eta^{\infty},w^*,z) + \Big(1+\partial_1\ell(\eta^{\infty},w^*,z)R_\theta^\infty\Big)^{-1}\partial_1\ell(\eta^{\infty},w^*,z)\\
        &= -\partial_1\ell(\eta^{\infty},w^*,z) + \Big[1-\Big(1+\partial_1\ell(\eta^{\infty},w^*,z)R_\theta^\infty\Big)^{-1}\Big](R_\theta^\infty)^{-1}.
    \end{align*}
    Taking expectations on both sides and rearranging gives (\ref{eq:fix6}).

    Lastly we show (\ref{eq:fix7}). From (\ref{eq:r_eta_response}), we have 
    \begin{align}\label{eq:r_eta_ast}
        \notag r_\eta(t,*) &= \partial_1\ell(\eta^t,w^*,z)\Big(-\sum_{r=0}^{t-1}r_\eta(r,*)R_\theta(t,r)\Big) + \partial_2\ell(\eta^t,w^*,z)\\
        &=\partial_1\ell(\eta^t,w^*,z)\Big(-\sum_{r=1}^{t}r_\eta(t-r,*)R_\theta(t,t-r)\Big) + \partial_2\ell(\eta^t,w^*,z).
    \end{align}
    We claim that 
    \begin{align}\label{eq:convolution_convergence}
    A(t) := \sum_{r=1}^t r_\eta(t-r,\ast)R_\theta(t,t-r) \overset{L^2}{\to} R_\theta^\infty r_\eta^\ast.
    \end{align}
    To see this, let $A_0(t) := \sum_{r=1}^t r_\eta(t-r,\ast)R_\theta(r)$. Then
    \begin{align*}
    \pnorm{A(t) - A_0(t)}{L^2} &\leq \sum_{r=1}^t \pnorm{r_\eta(t-r,\ast)}{L^2} \big|R_\theta(t,t-r) - R_\theta(r)\big| \leq \sup_{t\geq 0}\pnorm{r_\eta(t,\ast)}{L^2}\sum_{r=1}^t \big|R_\theta(t,t-r) - R_\theta(r)\big|\\
    &\leq \sup_{t\geq 0}\pnorm{r_\eta(t,\ast)}{L^2} \Big(\Delta\max_{1\leq r\leq \Delta} \big|R_\theta(t,t-r) - R_\theta(r)\big| + \sum_{r=\Delta+1}^t \big(|R_\theta(t,t-r)| + |R_\theta(r)|\big)\Big)\\
    &\leq \sup_{t\geq 0}\pnorm{r_\eta(t,\ast)}{L^2} \Big(\Delta\max_{1\leq r\leq \Delta} \big|R_\theta(t,t-r) - R_\theta(r)\big| + Ce^{-c\Delta}\Big),
    \end{align*}
    the last step following from Definition \ref{def:exp_conv_spec_init}-(5). Since $r_\eta(t,\ast) \overset{L^2}{\to} r_\eta^\ast$, we have $\sup_{t\geq 0}\pnorm{r_\eta(t,\ast)}{L^2} < \infty$, so taking $t\rightarrow\infty$ followed by $\Delta \rightarrow\infty$, we have $\pnorm{A(t) - A_0(t)}{L^2} \rightarrow 0$ as $t\rightarrow\infty$. On the other hand, 
    \begin{align*}
    \pnorm{A_0(t) - r_\eta^\ast R_\theta^\infty}{L^2} &=  \Bigpnorm{\sum_{r=1}^t r_\eta(t-r,\ast)R_\theta(r) - r_\eta^\ast \sum_{r=1}^\infty R_\theta(r)}{L^2}\\
    &\leq \underbrace{\sum_{r=1}^t |R_\theta(r)|\pnorm{r_\eta(t-r,\ast) - r_\eta^\ast}{L^2}}_{(I_t)} + \underbrace{\pnorm{r_\eta^\ast}{L^2}\sum_{r=t+1}^\infty |R_\theta(r)|}_{(II_t)}. 
    \end{align*}
    It is clear that $(II_t) \leq Ce^{-ct}$, while it holds for any positive integer $\Delta$ that
    \begin{align*}
    (I_t) &\leq \max_{1\leq r\leq \Delta} \pnorm{r_\eta(t-r,\ast)-r_\eta^\ast}{L^2}\sum_{r=1}^\Delta |R_\theta(r)| + \Big(\sup_{t\geq 0}\pnorm{r_\eta^\ast(t,\ast)}{L^2} + \pnorm{r_\eta^\ast}{L^2}\Big)\sum_{r=\Delta+1}^t |R_\theta(r)|\\
    &\leq \max_{1\leq r\leq \Delta} \pnorm{r_\eta(t-r,\ast)-r_\eta^\ast}{L^2}\sum_{r=1}^\Delta |R_\theta(r)| + C\Big(\sup_{t\geq 0}\pnorm{r_\eta^\ast(t,\ast)}{L^2} + \pnorm{r_\eta^\ast}{L^2}\Big)e^{-c\Delta}.
    \end{align*}
    So by taking $t\rightarrow\infty$ followed by $\Delta \rightarrow\infty$, we have $(I_t) \rightarrow 0$ and hence $A_0(t) \overset{L^2}{\to} r_\eta^\ast R_\eta^\ast$, thereby showing the claim (\ref{eq:convolution_convergence}).
    
    Now taking the $L^2$ limit as $t\rightarrow \infty$ of (\ref{eq:r_eta_ast}), and using further that $\partial_1 \ell(\eta^t,w^\ast,z) \overset{L^2}{\to} \partial_1 \ell(\eta^\infty,w^\ast,z)$ and $\partial_2 \ell(\eta^t,w^\ast,z) \overset{L^2}{\to} \partial_2 \ell(\eta^\infty,w^\ast,z)$, we get
    \begin{align*}
        r_\eta^\ast = -\partial_1\ell(\eta^\infty,w^*,z)r_\eta^\ast R_\theta^\infty + \partial_2\ell(\eta^\infty,w^*,z),
    \end{align*}
    which further implies
    \begin{equation*}
        r_\eta^\ast = \Big(1+\partial_1\ell(\eta^\infty,w^*,z)R_\theta^\infty\Big)^{-1}\partial_2\ell(\eta^\infty,w^*,z).
    \end{equation*}
    Taking expectation and recalling $R_\eta^\ast = \delta\E[r_\eta^\ast]$ then gives (\ref{eq:fix7}).
\end{proof}

\section{Long-time analysis of DMFT system (Theorem \ref{thm:dmft_exp_tti_condition})}
\label{sec:conv_hd_res}
In this section, our goal is to give a high-dimensional interpretation of the response functions $r_\eta,r_\theta,R_\eta,$ and $R_\theta$ that will aid in understanding the time translation invariant and exponential decay properties of these response function for the long-time analysis in Theorem \ref{thm:dmft_exp_tti_condition}. To do this, we introduce as an intermediary, an artificial dynamic which consists of $T$ stages of power iteration and then proceeds with gradient descent.

Here is a roadmap of this section:
\begin{itemize}
    \item In Section \ref{subsec:art_dynamic}, we derive the DMFT of an artificial dynamic which starts from a random initialization, performs $T$ rounds of power iteration in the first stage then does gradient descent in the second stage. This artificial dynamic serves as an approximation of gradient descent with spectral initialization, where the artifical dynamic approaches gradient descent with spectral initalization when $T\rightarrow\infty$.

    \item In Section \ref{subsec:hd_res_art}, we derive high-dimensional responses for the DMFT of the artificial dynamic. The technical component involves understanding the response functions as perturbations of the high dimensional artificial dynamic which is given in Lemma \ref{lm: low dim response}.

    \item In Section \ref{subsec:art_res_conv}, we prove the key result Proposition \ref{prop:art_spec_dmft_coup} which ties the limit of the artificial responses, $r_{T,\eta},r_{T,\theta},R_{T,\eta},$ and $R_{T,\theta}$, as $T\rightarrow\infty$ to the spectral initialized responses $r_\eta,r_\theta,R_\eta,$ and $R_\theta$.

    \item In Section \ref{sec:proof_long_time}, we give the proof of the main result Theorem \ref{thm:dmft_exp_tti_condition} by leveraging the high-dimensional responses derived in the previous sections.
\end{itemize}

\subsection{Artificial power iteration dynamic and DMFT characterization}
\label{subsec:art_dynamic}
Let $\M_n = \X^\top \bZ_s\X \in \R^{d\times d}$ be the pre-processing matrix defined in (\ref{def:M_n}) with $\bZ_s = \diag(\cT_s(\y))$. 
Consider the artificial dynamic $\btheta_T^{-T} \rightarrow \cdots \rightarrow \btheta_T^{0} \rightarrow \btheta_T^{1}\rightarrow\cdots$ with initialization $\btheta_T^{-T} = \btheta^\ast$ (we note that this is not a practical algorithm and only developed for theoretical analysis). This dynamic is divided into two stages.

\noindent\textbf{First Stage}
    In the first stage
    \begin{equation}\label{def:art_dyn_first}
        \btheta_T^{t+1} = \frac{\M_n \btheta_T^{t}}{\hat{\beta}_T^{t+1}} \qquad -T \leq t \leq -1,
    \end{equation}
where
\begin{align}\label{def:beta_def}
    \hat{\beta}_T^{t+1} := \lim_{n,d\rightarrow\infty} \frac{\pnorm{\M_n \btheta_T^{t}}{}}{\sqrt{d}},
\end{align}
is the high dimensional limit of the normalization of the iterate $\btheta_T^{t+1}$. As a consequence of Proposition \ref{prop:discreteDmftArt} below, such limit exists almost surely.

\noindent\textbf{Second Stage} The second stage follows the original gradient descent dynamics (\ref{eq: gradient descent}): 
\begin{equation}\label{def:art_dyn_second}
    \btheta_{T}^{t+1} = \btheta_{T}^{t} -\gamma\X^\top\ell(\X\btheta_{T}^{t},\X\btheta^*,\z) - \gamma\lambda\btheta_{T}^{t} \qquad t \geq 0,
\end{equation}
where $\gamma>0$ is the stepsize for the gradient descent stage. Further define
\begin{align}\label{def:eta_T_highd}
\bbeta^t_T = \X\btheta^t_T, \quad \bbeta^\ast = \X\btheta^\ast. 
\end{align}
The artificial dynamic iterates carry the $T$ in the subscript to distinguish from the original dynamic (\ref{eq: gradient descent}).

The asymptotic behavior of the artificial high-dimensional dynamic (\ref{def:art_dyn_first})-(\ref{def:art_dyn_second}) is given by the following low-dimensional artificial DMFT system, in similar spirit to (\ref{def:dmft_theta0})-(\ref{def:theta_res}). We initialize with
\begin{equation*}
    (\theta^{-T}_T, \theta^\ast) \sim \mathsf{P}(\theta^\ast,\theta^*),
\end{equation*}
and let
\begin{align*}
C_{T,\theta}(-T,-T) = C_{T,\theta}(-T,\ast) = \E[(\theta^\ast)^2] = 1.
\end{align*}

At integer time $t \in [-T,\infty)$, given $\{R_{T,\theta}(r,s)\}_{-T\leq s<r\leq t}$, $\{C_{T,\theta}(r,s)\}_{-T\leq s\leq r\leq t}$, $\{C_{T,\theta}(r,\ast)\}_{-T \leq r\leq t}$, define
\begin{align}\label{def:artificialEta}
    \eta_{T}^{t} &= \begin{cases}
        -\cT_s(y_T)\sum_{s=-T}^{t-1}\eta_{T}^s R_{T,\theta}(t,s) + w_{T}^{t} & -T \leq t \leq 0\\
        -\sum_{s=0}^{t-1}\ell(\eta_{T}^{s},w_{T}^{*},z)R_{T,\theta}(t,s) -\cT_s(y_T)\sum_{s=-T}^{-1}\eta_{T}^s R_{T,\theta}(t,s)+ w_{T}^{t}  &t>0.
    \end{cases}
\end{align}

Moreover, given the process $\{\eta_{T}^{s}\}_{-T\leq s\leq t}$, define for $-T\leq s < t$: 
\begin{equation}\label{eq:r_T_eta_recursion}
\begin{aligned}
    r_{T,\eta}(t,s) &= \begin{cases}
         \cT_s(y_T)\Big(-\sum_{r=s+1}^{t-1}R_{T,\theta}(t,r)r_{T,\eta}(r,s) - \cT_s(y_T)R_{T,\theta}(t,s)\Big) &t < 0\\
         \partial_1\ell(\eta_{T}^{t},w^*_{T},z)\Big( -\sum_{r=s+1}^{t-1}R_{T,\theta}(t,r)r_{T,\eta}(r,s)-\partial_1\ell(\eta_{T}^{s},w^*_{T},z)R_{T,\theta}(t,s)\Big) &t\geq0, s \geq 0\\
          \partial_1\ell(\eta_{T}^{t},w^*_{T},z)\Big( -\sum_{r=s+1}^{t-1}R_{T,\theta}(t,r)r_{T,\eta}(r,s)-\cT_s(y_T) R_{T,\theta}(t,s)\Big) &t\geq0, s \leq -1
     \end{cases},\\
      r_{T,\eta}(t,\sharp) &= \begin{cases}
          \cT_s(y_T)\Big(-\sum_{r=-T}^{t-1}R_{T,\theta}(t,r)r_{T,\eta}(r,\sharp)\Big)+ \cT'_s(y_{T})\varphi'(w_T^\sharp,z)\eta_{T}^t  &t <0\\
          \partial_1\ell(\eta_{T}^{t},w^*_{T},z)\Big(-\sum_{r=-T}^{t-1}r_{T,\eta}(r,\sharp)R_{T,\theta}(t,r)\Big)   &t\geq 0
      \end{cases},\\
      r_{T,\eta}(t,*) &= \partial_1\ell(\eta_{T}^{t},w^*_{T},z)\Big[-\sum_{r=0}^{t-1}r_{T,\eta}(r,*)R_{T,\theta}(t,r)\Big] + \partial_2\ell(\eta_{T}^{t},w^*_{T},z),\quad t\geq 0.
\end{aligned}
\end{equation}
Here, $(w^*_T,w^\sharp_T,\{w_T^s\}_{-T\leq s \leq t})$ is a mean-zero Gaussian vector with covariance 
\begin{equation}\label{eq:w_art}
    \E[w_T^rw_T^s] = C_{T,\theta}(r,s) \qquad \E[w_T^rw_T^\ast] = C_{T,\theta}(r,\ast) \qquad w_T^\sharp = w_T^\ast,
\end{equation}
and is independent of $z \sim \sP(z)$, and $y_T = \varphi(w^\sharp_T, z)$.

Then we can define correlation $\{C_{T,\eta}(r,s)\}_{-T\leq s \leq r\leq t}$ up to time $t$ by
\begin{align}\label{def:artificial_C_eta}
    C_{T,\eta}(r,s) &= \begin{cases}
        \delta\E\big[\ell(\eta_{T}^{r},w_{T}^*,z)\ell(\eta_{T}^{s},w_{T}^*,z)\big]  &r\geq0,s\geq 0\\\delta\E\big[\cT_s(y_T)^2\eta_{T}^r\eta_{T}^s\big] &r<0,s<0\\\delta\E\big[\ell(\eta_T^{r},w_T^*,z)\cT_s(y_T)\eta_{T}^s\big]  &r\geq 0, s<0
    \end{cases},
\end{align}
and response $\{R_{T,\eta}(r,s)\}_{-T\leq s < r\leq t}$, $\{R_{T,\eta}(r,\ast)\}_{0\leq r\leq t}$, $\{R_{T,\eta}(r,\sharp)\}_{-T\leq r\leq t}$ up to time $t$ by
\begin{equation}
    R_{T,\eta}(r,s) = \delta\E\big[r_{T,\eta}(r,s)\big], \quad
    R_{T,\eta}(r,*) = \delta\E\big[r_{T,\eta}(r,*)\big], \quad
    R_{T,\eta}(r,\sharp) = \delta\E\big[r_{T,\eta}(r,\sharp)\big].
\end{equation}

Conversely, suppose that up to integer time $t\in[-T,\infty)$, we are given $\{R_{T,\eta}(r,s)\}_{-T\leq s<r\leq t}$, and $\{C_{T,\eta}(r,s)\}_{-T\leq s\leq r\leq t}$. Define the processes
\begin{align}
\label{def:r_T_theta}
    \theta_{T}^{t+1} &= \begin{cases}
        \frac{1}{\beta_T^{t+1}}\bigg(\delta\E[\cT_s(y_T)]\theta_{T}^t+\sum_{s=-T}^{t-1}R_{T,\eta}(t,s)\theta_{T}^s  + R_{T,\eta}(t,\sharp)\theta^*-u_{T}^t\bigg) &t<0\\
        \theta_{T}^t +\gamma\bigg(- \lambda\theta_{T}^t -\delta \E\left[\partial_1\ell(\eta_{T}^{t},w_{T}^{*},z)\right]\theta_{T}^{t}\\
        -\sum_{s=-T}^{t-1}R_{T,\eta}(t,s)\theta_{T}^{s} -(R_{T,\eta}(t,\ast) + R_{T,\eta}(t,\sharp))\theta^*  +  u_{T}^{t}\bigg) &t\geq0
    \end{cases} \\
    \label{eq:r_T_theta_recursion}
     r_{T,\theta}(t+1,s) &= \begin{cases}
         \frac{1}{\beta_T^{t+1}}\bigg(\delta\E[\cT_s(y_{T})]r_{T,\theta}(t,s)+\sum_{r=s+1}^{t-1}R_{T,\eta}(t,r)r_{T,\theta}(r,s)\bigg) &t <0\\
         \left(1-\gamma\lambda-\gamma\delta\E[\partial_1\ell(\eta_{T}^{t},w_{T}^*,z)]\right) r_{T,\theta}(t,s) - \gamma\sum_{r=s+1}^{t-1}R_{T,\eta}(t,r)r_{T,\theta}(r,s) &t \geq 0
     \end{cases}
\end{align} with initialization 
\begin{equation*}
r_{T,\theta}(s+1,s) = 
    \begin{cases}
        -\frac{1}{\beta_T^{s+1}} &\qquad s <0\\
        \gamma &\qquad s\geq0.
    \end{cases}
\end{equation*}
Here $\{u_T^s\}_{-T\leq s\leq t}$ is a mean-zero Gaussian vector independent of $z$ and with covariance
\begin{equation}\label{eq:u_art}
    \E[u_T^ru_T^s] = C_{T,\eta}(r,s),
\end{equation}
and
\begin{align}\label{def:beta_seq}
\beta_T^{t+1} = \sqrt{\E\Big[\sum_{i=-T}^t R_{T,\eta}(t,i)\theta^i_T + R_{T,\eta}(t,\sharp)\theta^\ast - u^t_T\Big]^2}, \quad t < 0.
\end{align}
Then we can define correlation $\{C_{T,\theta}(r,s)\}_{-T\leq s \leq r\leq t+1}$, $\{C_{T,\theta}(r,\ast)\}_{-T\leq s\leq t+1}$, $C_{T,\theta}(\ast,\ast)$ up to time $t+1$ by
\begin{equation}\label{def:artificial_dmft_CR_theta}
\begin{gathered}
    C_{T,\theta}(s,r) = \E[\theta^{s}_{T}\theta^{r}_{T}], \quad C_{T,\theta}(s,*) = \E[\theta^{s}_{T}\theta^{*}], \quad C_{T,\theta}(\ast,\ast) = \E[(\theta^\ast)^2],
\end{gathered} 
\end{equation}
and response $\{R_{T,\theta}(r,s)\}_{-T\leq s < r\leq t+1}$ up to time $t+1$ by
\begin{align}
R_{T,\theta}(r,s) = \E[r_{T,\theta}(r,s)].
\end{align}
In summary, the above DMFT system is recursively defined in the following order:
\begin{align*}
&\theta_T^{-T} = \theta^\ast \longrightarrow \{C_{T,\theta}(s,r)\}_{-T\leq r\leq s\leq -T}, \{C_{T,\theta}(s,\ast)\}_{-T\leq s\leq -T}, \{R_{T,\theta}(s,r)\}_{-T\leq r < s\leq -T}\\ 
 &\longrightarrow \eta^{-T}_T \longrightarrow \{C_{T,\eta}(s,r)\}_{-T\leq r\leq s\leq -T}, \{R_{T,\eta}(s,r)\}_{-T\leq r < s\leq -T}, \{R_{T,\eta}(s,\sharp)\}_{-T\leq s\leq -T}\\
&\longrightarrow \theta^{-T+1}_T \longrightarrow \{C_{T,\theta}(s,r)\}_{-T\leq r\leq s\leq -T+1}, \{C_{T,\theta}(s,\ast)\}_{-T\leq s\leq -T+1}, \{R_{T,\theta}(s,r)\}_{0\leq r < s\leq -T+1}\\
&\longrightarrow \eta_T^{-T+1} \longrightarrow \{C_{T,\eta}(s,r)\}_{0\leq r\leq s\leq -T+1}, \{R_{T,\eta}(s,r)\}_{-T\leq r < s\leq -T+1}, \{R_{T,\eta}(s,\sharp)\}_{-T\leq s\leq -T+1}\\
&\longrightarrow \theta^{-T+2}_T \longrightarrow \ldots \\
&\longrightarrow \theta^0_T \longrightarrow \{C_{T,\theta}(s,r)\}_{-T\leq r\leq s\leq 0}, \{C_{T,\theta}(s,\ast)\}_{-T\leq s\leq 0}, \{R_{T,\theta}(s,r)\}_{0\leq r < s\leq 0}\\
&\longrightarrow \eta_T^{0} \longrightarrow \{C_{T,\eta}(s,r)\}_{0\leq r\leq s\leq 0}, \{R_{T,\eta}(s,r)\}_{-T\leq r < s\leq 0}, \{R_{T,\eta}(s,\sharp)\}_{-T\leq s\leq 0}, \{R_{T,\eta}(s,\ast)\}_{0\leq s\leq 0}\\
&\longrightarrow \theta^1_T \longrightarrow \{C_{T,\theta}(s,r)\}_{-T\leq r\leq s\leq 1}, \{C_{T,\theta}(s,\ast)\}_{-T\leq s\leq 1}, \{R_{T,\theta}(s,r)\}_{0\leq r < s\leq 1}\\
&\longrightarrow \eta_T^{1} \longrightarrow \{C_{T,\eta}(s,r)\}_{0\leq r\leq s\leq 1}, \{R_{T,\eta}(s,r)\}_{-T\leq r < s\leq 1}, \{R_{T,\eta}(s,\sharp)\}_{-T\leq s\leq 1}, \{R_{T,\eta}(s,\ast)\}_{0\leq s\leq 1}\\
&\longrightarrow \theta^2_T \longrightarrow \ldots.
\end{align*}
The following result formalizes the DMFT approximation of the artificial dynamic (\ref{def:art_dyn_first})-(\ref{def:art_dyn_second}). Since its proof is similar to that of Theorem \ref{thm: discrete dmft asymp}, we defer it to Section \ref{appendix:proof_power_dmft}.
\begin{proposition}\label{prop:discreteDmftArt}
    Suppose Assumptions \ref{ass: model assumptions}, \ref{ass: spectral initialization}, and \ref{ass: ell lipschitz} hold. Then $\hat\beta^t_T$ in (\ref{def:beta_def}) is well defined and equals $\beta^t_T$ in (\ref{def:beta_seq}) for all integer $-T+1\leq t\leq 0$. Moreover, for any fixed integer $m\geq -T$, almost surely as $n,d \rightarrow\infty$, it holds that
    \begin{align*}
        \frac{1}{d}\sum_{j=1}^d \delta_{\big(\theta^{-T+1}_{T,j},...,\theta^{m}_{T,j},\theta^*_j\big)} &\overset{W_2}{\to} \sP\big(\theta_T^{-T+1},...,\theta_T^{m},\theta^*\big),\\
        \frac{1}{n}\sum_{i=1}^n\delta_{\big(\eta^{-T}_{T,i},...,\eta_{T,i}^{m},\eta^*_i,z_i\big)} &\overset{W_2}{\to} \sP\left(\eta_T^{-T},...,\eta_T^{m},w_T^*,z\right).
    \end{align*}
\end{proposition}

\subsection{High dimensional response of artificial dynamic}
\label{subsec:hd_res_art}
Consider the following perturbed dynamics. For any $j\in[d]$ and the basis vector $\e_j\in\R^d$ define
\begin{align}\label{def:perturbed_dynamic_theta}
    \btheta_T^{t+1,(s,j),\eps} &= \begin{cases}
      \frac{\X^\top \diag(\cT_s(\y)) \X \btheta_T^{t,(s,j),\eps} - \eps\e_j\mathbf{1}_{t=s}}{\beta_T^{t+1}} & -T \leq t \leq -1\\
    \btheta_{T}^{t,(s,j),\eps} -\gamma\big[\X^\top\ell(\X\btheta_{T}^{t,(s,j),\eps},\X\btheta^*,\z) + \lambda\btheta_{T}^{t,(s,j),\eps}-\eps\e_j\mathbf{1}_{t=s}\big] & t \geq 0,  
    \end{cases}
\end{align}
with initialization $\btheta_T^{-T,(s,j),\eps} = \btheta^\ast$, and $\{\beta^t_T\}$ is given by (\ref{def:beta_def}).
For any $i\in[n]$ and the basis vector $\e_i\in\R^n$ define the perturbed dynamics
\begin{equation}\label{def:perturbed_dynamic_eta}
\begin{aligned}
    \btheta_T^{t+1,[s,i],\eps}&=\begin{cases}
        \frac{\X^\top \diag(\cT_s(\y)) \X \btheta_T^{t,[s,i],\eps} -\eps\X^\top\diag(\cT_s(\y))\e_i\mathbf{1}_{t=s}}{\beta_T^{t+1}} & -T \leq t \leq -1\\
        \btheta_{T}^{t,[s,i],\eps} -\gamma\big[\X^\top\ell(\X\btheta_{T}^{t,[s,i],\eps},\X\btheta^*,\z) + \lambda\btheta_{T}^{t,[s,i],\eps}-\eps\X^\top\diag\Big(\partial_1\ell(\X\btheta_T^s,\X\btheta^*,\z)\Big)\e_i\mathbf{1}_{t=s}\big] & t \geq 0
    \end{cases}\\
        \btheta_T^{t+1,[\sharp,i],\eps} &= 
        \begin{cases}
            \frac{\X^\top \diag(\cT_s(\y)) \X \btheta_T^{t,[\sharp,i],\eps} + \eps\X^\top\diag(\cT_s'(\y)\circ\varphi'(\bbeta^*,\z)\circ\bbeta_T^t)\e_i}{\beta_T^{t+1}} & -T \leq t \leq -1\\
            \btheta_{T}^{t,[\sharp,i],\eps} -\gamma\big[\X^\top\ell(\X\btheta_{T}^{t,[\sharp,i],\eps},\X\btheta^*,\z) + \lambda\btheta_{T}^{t,[\sharp,i],\eps}\big] & t \geq 0
        \end{cases}\\
            \btheta_T^{t+1,[\ast,i],\eps} &=
            \begin{cases}
                \frac{\X^\top \diag(\cT_s(\y)) \X \btheta_T^{t,[\ast,i],\eps}}{\beta_T^{t+1}} & -T \leq t \leq -1\\
                \btheta_{T}^{t,[\ast,i],\eps} -\gamma\big[\X^\top\ell(\X\btheta_{T}^{t,[\ast,i],\eps},\X\btheta^*,\z) + \lambda\btheta_{T}^{t,[\ast,i],\eps}+\eps\X^\top\diag(\partial_2\ell(\bbeta^t_T,\bbeta^*,\z))\e_i\big] & t \geq 0,
            \end{cases}
\end{aligned}
\end{equation}
where all perturbed dynamics are initialized at $\btheta^\ast$. We further define
\begin{align}\label{def:perturbed_dynamic_eta2}
\bbeta_T^{t,[s,i],\eps} := \X\btheta_T^{t,[s,i],\eps}, \quad \bbeta_T^{t,[\sharp,i],\eps} := \X\btheta_T^{t,[\sharp,i],\eps}, \quad \bbeta_T^{t,[\ast,i],\eps} := \X\btheta_T^{t,[\ast,i],\eps}.
\end{align}

In addition, we introduce the following notation. For a fixed process $\bar{\eta} := \{(\eta^t,\eta^\ast,z)\} \in \R^3$ indexed by integers $t \in [-T,\infty)$, let $y = \varphi(\eta^\ast,z)$ and define the auxiliary processes
\begin{align*}
\label{eq:r_T_eta_recursion_aux}
    r_{T,\eta}^{(\bar{\eta})}(t,s) &= \begin{cases}
         \cT_s(y)\Big(-\sum_{r=s+1}^{t-1}R_{T,\theta}(t,r)r_{T,\eta}^{(\bar{\eta})}(r,s) - \cT_s(y)R_{T,\theta}(t,s)\Big) &t < 0\\
         \partial_1\ell(\eta^{t},\eta^*,z)\Big( -\sum_{r=s+1}^{t-1}r_{T,\eta}^{(\bar{\eta})}(r,s)R_{T,\theta}(t,r)-\partial_1\ell(\eta^{s},\eta^\ast,z)R_{T,\theta}(t,s)\Big) &t\geq0, s \geq 0\\
          \partial_1\ell(\eta^{t},\eta^\ast,z)\Big( -\sum_{r=s+1}^{t-1}r_{T,\eta}^{(\bar{\eta})}(r,s)R_{T,\theta}(t,r)-\cT_s(y) R_{T,\theta}(t,s)\Big) &t\geq0, s \leq -1
     \end{cases}\\
    r_{T,\eta}^{(\bar{\eta})}(t,t) &= \cT_s(y) \quad t< 0,\quad\quad r_{T,\eta}^{(\bar{\eta})}(t,t) = \partial_1 \ell(\eta^t,\eta^\ast,z), \quad t \geq 0,\\
    r_{T,\eta}^{(\bar{\eta})}(t,\sharp) &= \begin{cases}
          \cT_s(y)\Big(-\sum_{r=-T}^{t-1}R_{T,\theta}(t,r)r_{T,\eta}^{(\bar{\eta})}(r,\sharp)\Big)+ \cT'_s(y)\varphi'(\eta^\ast,z)\eta^t  &t <0\\
          \partial_1\ell(\eta^{t},\eta^\ast,z)\Big(-\sum_{r=-T}^{t-1}r_{T,\eta}^{(\bar{\eta})}(r,\sharp)R_{T,\theta}(t,r)\Big)   &t\geq 0
      \end{cases}\\
      r_{T,\eta}^{(\bar{\eta})}(t,\ast) &= \partial_1\ell(\eta^{t},\eta^\ast,z)\Big[-\sum_{r=0}^{t-1}r_{T,\eta}^{(\bar{\eta})}(r,\ast)R_{T,\theta}(t,r)\Big] + \partial_2\ell(\eta^{t},\eta^\ast,z),\quad t\geq 0.
\end{align*}
When $\bar\eta_T = \{(\eta^t_T, w^\ast_T, z)\}$ is the DMFT process (\ref{def:artificialEta}) for the artificial dynamic (\ref{def:art_dyn_first})-(\ref{def:art_dyn_second}), by definition it holds that
\begin{align*}
    r_{T,\eta}^{(\bar{\eta}_T)}(t,s) = r_{T,\eta}(t,s),\quad r_{T,\eta}^{(\bar{\eta}_T)}(t,\sharp) = r_{T,\eta}(t,\sharp),\quad r_{T,\eta}^{(\bar{\eta}_T)}(t,\ast) = r_{T,\eta}(t,\ast),
\end{align*}
where the responses on the right side are given by (\ref{eq:r_T_eta_recursion}), and by taking expectation,
\begin{align*}
\delta\E\big[r_{T,\eta}^{(\bar{\eta}_T)}(t,s)\big] = R_{T,\eta}(t,s), \quad \delta\E\big[r_{T,\eta}^{(\bar{\eta}_T)}(t,\sharp)\big] = R_{T,\eta}(t,\sharp), \quad \delta\E\big[r_{T,\eta}^{(\bar{\eta}_T)}(t,\ast)\big] = R_{T,\eta}(t,\ast).
\end{align*}

The following lemma is the key to relating the perturbed dynamics to the artificial response functions.

\begin{lemma}\label{lm: low dim response} 
Suppose Assumptions \ref{ass: model assumptions}, \ref{ass: spectral initialization}, and \ref{ass: ell lipschitz} hold. For each $i\in[n]$, let $\bar{\eta}_{T,i}$ be the $i$th coordinate process of $\{\bbeta^t_T, \bbeta^\ast,\z\}$ given by (\ref{def:eta_T_highd}). For each $t>s$, $j \in [d], i \in [n]$, we have 
    \begin{align*}
        \deps \theta^{t,(s,j),\eps}_{T,j} &= r_{T,\theta}(t,s) + E_{T,\theta}^{t,(s,j)},\\
        \cT_s(y_i)\deps \eta^{t,[s,i],\eps}_{T,i} &= -r_{T,\eta}^{(\bar{\eta}_{T,i})}(t,s) + E_{T,\eta}^{t,[s,i]}, \qquad t<0,\\
    \partial_1\ell(\eta_{T,i}^t,\eta_i^*,z_i)\deps \eta^{t,[s,i],\eps}_{T,i} &= -r_{T,\eta}^{(\bar{\eta}_{T,i})}(t,s) + E_{T,\eta}^{t,[s,i]}, \qquad t \geq 0,\\
       \cT_s(y_i)\deps \eta^{t,[\sharp,i],\eps}_{T,i} &= r_{T,\eta}^{(\bar{\eta}_{T,i})}(t,\sharp)- \cT'_s(y_i)\varphi'(\eta^\ast_{T,i},z_i)\eta^t_{T,i}+ E_{T,\eta}^{t,[\sharp,i]}, \qquad t<0, \\
    \partial_1\ell(\eta_{T,i}^t,\eta_i^*,\eps_i)\deps \eta^{t,[\sharp,i],\eps}_{T,i} &= r_{T,\eta}^{(\bar{\eta}_{T,i})}(t,\sharp) + E_{T,\eta}^{t,[\sharp,i]}, \qquad t \geq 0,\\
       \partial_1 \ell(\eta^t_{T,i}, \eta^\ast_i, z_i)\deps \eta^{t,[\ast,i],\eps}_{T,i} &= r_{T,\eta}^{(\bar\eta_{T,i})}(t,\ast) - \partial_2 \ell(\eta^t_{T,i}, \eta^\ast_i,z_i) +  E_{T,\eta}^{t,[\ast,i]}, 
    \end{align*}
    where, almost surely as $n,d\rightarrow\infty$,
    \begin{align*}
        \frac{1}{d}\sum_{j=1}^d|E_{T,\theta}^{t,(s,j)}|^p \vee \frac{1}{n}\sum_{i=1}^n |E_{T,\eta}^{t,[s,i]}|^p \vee \frac{1}{n}\sum_{i=1}^n |E_{T,\eta}^{t,[\sharp,i]}|^p \vee \frac{1}{n}\sum_{i=1}^n |E_{T,\eta}^{t,[\ast,i]}|^p\rightarrow 0,
    \end{align*}
    for any $p\geq 1$ and fixed $T > 0$.
\end{lemma}

Before giving the proof, we state a corollary of Lemma \ref{lm: low dim response} that can be applied to understand the response functions. Define the high-dimensional response matrices $\bR_{T,\theta}(t,s)\in \R^{d\times d}$, $\bR_{T,\eta}(t,s), \bR_{T,\eta}(t,\sharp), \bR_{T,\eta}(t,\ast) \in \R^{n\times n}$ as
\begin{align*}
    (\bR_{T,\theta}(t,s))_{ij} &:= \deps \theta^{t,(s,j),\eps}_{T,i},\\
    (\bR_{T,\eta}(t,s))_{ij} &:= 
    \begin{cases}
    -\cT_s(y_i)\deps \eta^{t,[s,j],\eps}_{T,i} & t< 0\\
     -\deps\ell\big(\eta^{t,[s,j],\eps}_{T,i},\eta^*_i,z_i\big) & t\geq 0
    \end{cases},\\
    (\bR_{T,\eta}(t,\sharp))_{ij} &:= 
    \begin{cases}
    \cT_s(y_i)\deps \eta_{T,i}^{t,[\sharp,j],\eps} + \cT_s'(y_i)\varphi'(\eta^\ast_i,z_i)\eta_{T,i}^t & t< 0\\
    \deps \ell\big(\eta^{t,[\sharp,j],\eps}_{T,i}, \eta^\ast_i,z_i\big) & t\geq 0
    \end{cases},\\
    (\bR_{T,\eta}(t,\ast))_{ij} &:= \deps \ell\big(\eta_{T,i}^{t,[\ast,j],\eps}, \eta^\ast_i,z_i\big) + \partial_2\ell(\eta_{T,i}^t,\eta^\ast_i,z_i).
\end{align*}

\begin{proposition}\label{prop: high dim response}
Suppose Assumptions \ref{ass: model assumptions}, \ref{ass: spectral initialization}, and \ref{ass: ell lipschitz} hold. Let $F:\R\rightarrow\R$ be a $2$-pseudo Lipschitz function. Then almost surely, it holds that
\begin{align*}
&\lim_{n,d\rightarrow\infty} \frac{1}{d}\sum_{j=1}^d F\Big(\deps \theta^{t,(s,j),\eps}_{T,j}\Big) = F( r_{T,\theta}(t,s)),\\
&\lim_{n,d\rightarrow\infty}\bigg\{\frac{1}{n}\sum_{i=1}^n F\Big(-\cT_s(y_i)\deps \eta^{t,[s,i],\eps}_{T,i}\Big),   \frac{1}{n}\sum_{i=1}^n F\big(r_{T,\eta}^{(\bar{\eta}_{T,i})}(t,s)\big)\bigg\} = \E[F\big(r_{T,\eta}(t,s)\big)], \qquad t<0,\\
    &\lim_{n,d\rightarrow\infty}\bigg\{\frac{1}{n}\sum_{i=1}^n F\Big(- \partial_1\ell(\eta_{T,i}^t,\eta_i^*,z_i)\partial_\eps|_{\eps=0} \eta^{t,[s,i],\eps}_{T,i}\Big),  \frac{1}{n}\sum_{i=1}^n F\big(r_{T,\eta}^{(\bar\eta_{T,i})}(t,s)\big)\bigg\} = \E[F(r_{T,\eta}(t,s))], \quad t\geq 0,\\
       &\lim_{n,d\rightarrow\infty}\bigg\{ \frac{1}{n}\sum_{i=1}^n F\Big(\cT_s(y_i)\deps \eta^{t,[\sharp,i],\eps}_{T,i} +  \cT'_s(y_i)\varphi'(\eta^\ast_{T,i},z_i)\eta^t_{T,i}\Big),  \frac{1}{n}\sum_{i=1}^n F\big(r_{T,\eta}^{(\bar{\eta}_{T,i})}(t,\sharp)\big)\bigg\} = \E[F(r_{T,\eta}(t,\sharp))], \qquad t<0,\\
    &\lim_{n,d\rightarrow\infty} \bigg\{\frac{1}{n}\sum_{i=1}^n F\Big(\partial_1\ell(\eta_{T,i}^t,\eta_i^*,z_i)\deps \eta^{t,[\sharp,i],\eps}_{T,i}\Big), \frac{1}{n}\sum_{i=1}^n F\big(r_{T,\eta}^{(\bar{\eta}_{T,i})}(t,\sharp)\big)\bigg\} = \E[F(r_{T,\eta}(t,\sharp))], \qquad t \geq 0,\\
       &\lim_{n,d\rightarrow\infty} \bigg\{\frac{1}{n}\sum_{i=1}^n F\Big(\partial_1 \ell(\eta^t_{T,i}, \eta^\ast_i, z_i)\deps \eta^{t,[\ast,i],\eps}_{T,i} + \partial_2 \ell(\eta^t_{T,i}, \eta^\ast_i,z_i)\Big), \frac{1}{n}\sum_{i=1}^n F\big(r_{T,\eta}^{(\bar\eta_{T,i})}(t,\ast)\big)\bigg\} = \E[F(r_{T,\eta}(t,\ast))].   
\end{align*}
In particular, as $n,d\rightarrow\infty$, it holds almost surely  that
\begin{equation*}
\begin{split}
        \frac{1}{d}\Tr(\bR_{T,\theta}(t,s)) \rightarrow R_{T,\theta}(t,s), \quad 
        \frac{1}{d}\Tr(\bR_{T,\eta}(t,s)) \rightarrow R_{T,\eta}(t,s),\\
        \frac{1}{d}\Tr(\bR_{T,\eta}(t,\sharp)) \rightarrow R_{T,\eta}(t,\sharp), \quad 
        \frac{1}{d}\Tr(\bR_{T,\eta}(t,\ast)) \rightarrow R_{T,\eta}(t,\ast).
    \end{split}
    \end{equation*}
\end{proposition}
\begin{proof}
    Under the model assumptions on $\ell$, we can recursively check that the processes $r_{T,\eta}^{(\bar{\eta}_{T,i})}(t,s),r_{T,\eta}^{(\bar{\eta}_{T,i})}(t,\sharp),$ and $r_{T,\eta}^{(\bar{\eta}_{T,i})}(t,\ast)$ are Lipschitz functions with respect to $\bar{\eta}_{T,i}$. Then the composition with $F$ is $2-$pseudo Lipschitz. Then applying Proposition \ref{prop:discreteDmftArt}, we get
    \begin{align*}
        \lim_{n,d\rightarrow\infty}\frac{1}{n}\sum_{i=1}^nF(r_{T,\eta}^{(\bar{\eta}_{T,i})}(t,s)) &= \E[F(r_{T,\eta}(t,s))]\\
        \lim_{n,d\rightarrow\infty}\frac{1}{n}\sum_{i=1}^nF(r_{T,\eta}^{(\bar{\eta}_{T,i})}(t,\sharp)) &= \E[F(r_{T,\eta}(t,\sharp))]\\
        \lim_{n,d\rightarrow\infty}\frac{1}{n}\sum_{i=1}^nF(r_{T,\eta}^{(\bar{\eta}_{T,i})}(t,\ast)) &= \E[F(r_{T,\eta}(t,\ast))].
    \end{align*}
    By definition it is true that
    \begin{equation*}
        \lim_{n,d\rightarrow\infty}\frac{1}{n}\sum_{i=1}^nF(r_{T,\theta}(t,s)) =  F(r_{T,\theta}(t,s)).
    \end{equation*}
    The remaining claims follow from applying Lemma \ref{lm: low dim response}.
\end{proof}

Proposition \ref{prop: high dim response} relates the response functions to derivatives of high-dimensional perturbed dynamics in (\ref{def:perturbed_dynamic_theta})-(\ref{def:perturbed_dynamic_eta}). This is helpful since by computing these derivatives explicitly, we can further relate the response functions to certain matrix recursions, as detailed in the following lemma. Recall the definitions of $\bbeta^{t,[s,i],\eps}_T, \bbeta^{t,[\sharp,i], \eps}_T, \bbeta^{t,[\ast,i],\eps}_T$ in (\ref{def:perturbed_dynamic_eta2}). 

\begin{lemma}\label{lm: response matrix recursion}
    For any integer $t \geq -T$, define the matrices
    \begin{equation}\label{def:D_Omega}
        \D^t_T := \diag\big(\partial_1\ell(\boldeta^t_T,\boldeta^*,\z)\big) \in \R^{n\times n} \qquad \bOmega^t_T := (1-\gamma\lambda)\I_d -\gamma \X^\top\D^t_T\X \in \R^{d\times d}. 
    \end{equation}
    For integer $t \in [-T,-1]$, define the matrix
    \begin{align}\label{def:Upsilon}
        \bUpsilon^t_T := \frac{1}{\beta_T^{t+1}}\X^\top\diag(\cT_s(\y))\X.
    \end{align}
    Then for any $j \in [d], i \in [n]$, and $t\geq s+1$,
    \begin{align}
        \deps \btheta_T^{t,(s,j),\eps} &= \begin{cases}
            -\frac{1}{\beta_T^{s+1}} \bUpsilon_T^{t-1}\cdots\bUpsilon_T^{s+1}\e_j &s\leq-1,t\leq -1\\
            -\frac{1}{\beta_T^{s+1}} \bOmega^{t-1}_T\cdots\bOmega^0_T\bUpsilon_T^{-1}\cdots\bUpsilon_T^{s+1}\e_j&s\leq-1,t> -1\\
            \gamma\bOmega^{t-1}_T\cdots\bOmega^{s+1}_T \e_j & s> -1
        \end{cases},\label{eq:matrix_recursion_1}\\
        \deps \boldeta_T^{t,[s,i],\eps} &= \begin{cases}
        -\frac{1}{\beta_T^{s+1}} \X\bUpsilon_T^{t-1}\cdots\bUpsilon_T^{s+1}\X^\top\diag(\cT_s(\y))\e_i &s\leq-1,t\leq -1\\
            -\frac{1}{\beta_T^{s+1}} \X\bOmega^{t-1}_T\cdots\bOmega^0_T\bUpsilon_T^{-1}\cdots\bUpsilon_T^{s+1}\X^\top\diag(\cT_s(\y))\e_i&s\leq-1,t> -1\\
            \gamma\X\bOmega^{t-1}_T\cdots\bOmega^{s+1}_T\X^\top\D_T^s \e_i & s> -1
        \end{cases}\label{eq:matrix_recursion_2}.
        \end{align}
        Here for $t = s+1$, $\bOmega^{t-1}_T\ldots \bOmega^{s+1}_T$ is understood as $\I_d$ in the case $s > -1$ and $\bOmega^{t-1}_T\ldots, \bOmega^0_T\bUpsilon^{-1}_{T}\ldots \bUpsilon^{s+1}_T$ is understood as $\I_d$ in the case $s = -1,t=0$.
        Furthermore, for any $i\in[n]$ it holds that 
        \begin{align}
        \deps \bbeta_T^{t,[\sharp,i],\eps} &= \begin{cases}
            \sum_{s=-T}^{t-1}\frac{1}{\beta_T^{s+1}}\X\bUpsilon_T^{t-1}\cdots\bUpsilon_T^{s+1}\X^\top\diag(\cT_s'(\y)\circ\varphi'(\bbeta^*,\z)\circ\bbeta_T^s)\e_i &t<0\\
            \X\bOmega_T^{t-1}\cdots\bOmega_T^{0}\sum_{s=-T}^{-1}\frac{1}{\beta_T^{s+1}}\bUpsilon_T^{-1}\cdots\bUpsilon_T^{s+1}\X^\top\diag(\cT_s'(\y)\circ\varphi'(\bbeta^*,\z)\circ\bbeta_T^s)\e_i & t\geq 0
        \end{cases}\label{eq:matrix_recursion_3}\\
        \deps \bbeta_T^{t,[\ast,i],\eps}&= -\gamma\sum_{s=0}^{t-1}\X\bOmega_T^{t-1}\cdots\bOmega_T^{s+1}\X^\top\diag(\partial_2\ell(\bbeta_T^s,\bbeta^\ast,\z))\e_i, \quad t\geq 0\label{eq:matrix_recursion_4}.
    \end{align}
Consequently, for any integer $H \geq 0$, there exists some $C > 0$ (depending on $H,T,\gamma$) such that almost surely for large enough $n,d\rightarrow\infty$,
\begin{align}\label{eq:derivative_bound}
\max_{-T\leq s < t \leq H}\max_{i\in[n],j\in[d]}\pnorm{\deps \btheta_T^{t,(s,j),\eps}}{} \vee \pnorm{\deps \bbeta_T^{t,[s,i],\eps}}{} \vee \pnorm{\deps \bbeta_T^{t,[\sharp,i],\eps}}{} \vee \pnorm{\deps \bbeta_T^{t,[\ast,i],\eps}}{} \leq C.  
\end{align}
\end{lemma}
\begin{proof}
The claims (\ref{eq:matrix_recursion_1}) and (\ref{eq:matrix_recursion_2}) hold since for $t \geq s+1$, we have $\deps \btheta_T^{t,(s,j),\eps} = \bUpsilon^{t-1}_T \deps \btheta_T^{t-1,(s,j),\eps}$ for $t\leq 0$ and $\deps \btheta_T^{t,(s,j),\eps} = \bOmega^{t-1}_T \deps \btheta_T^{t-1,(s,j),\eps}$ if $t \geq 1$, and similarly for $\deps \btheta_T^{t,[s,i],\eps}$, while
\begin{align*}
\deps \btheta_T^{s+1,(s,j),\eps} = 
\begin{cases}
-\frac{\e_j}{\beta^{s+1}_T} & s \leq -1, \\
\gamma \e_j & s > -1,
\end{cases}\quad 
\deps \btheta_T^{s+1,[s,i],\eps} =
\begin{cases}
-\frac{\X^\top \diag(\cT_s(\y))\e_i}{\beta^{s+1}_T} & s \leq -1,\\
\gamma \X^\top \D^s_T \e_i & s > -1.
\end{cases}
\end{align*}
Next (\ref{eq:matrix_recursion_3}) holds since for $t \geq 0$
\begin{align*}
\deps \btheta_T^{t,[\sharp,i],\eps} = \bOmega^{t-1}_T\deps \btheta_T^{t-1,[\sharp,i],\eps} = \ldots = \bOmega^{t-1}_T\ldots \bOmega^{0}_T \deps \btheta_T^{0,[\sharp,i],\eps},
\end{align*}
where
\begin{align*}
\deps\btheta_T^{0,[\sharp,i],\eps} &= \bUpsilon^{-1}_T \deps\btheta_T^{-1,[\sharp,i],\eps} + \frac{\X^\top \diag(\cT'(\y)\circ\varphi'(\bbeta^\ast,\z)\circ\bbeta^{-1}_T)\e_i}{\beta^0_T}\\
&= \ldots = \sum_{s=-T}^{-1}  \bUpsilon^{-1}_T\ldots \bUpsilon^{s+1}_T\frac{\X^\top \diag(\cT'(\y)\circ\varphi'(\bbeta^\ast,\z)\circ\bbeta^{s}_T)\e_i}{\beta^{s+1}_T}.
\end{align*}
The case of $t < 0$ follows similarly. For (\ref{eq:matrix_recursion_4}), we clearly have $\deps \bbeta_T^{0,[\ast,i],\eps} = 0$, and for $t\geq 1$, 
\begin{align*}
\deps \btheta_T^{t,[\ast,i],\eps} &= \bOmega_T^{t-1}\deps \btheta_T^{t-1,[\ast,i],\eps} - \gamma \X^\top \diag(\partial_2\ell(\bbeta^{t-1}_T,\bbeta^\ast,\z))\e_i\\
&= \ldots = -\gamma\sum_{s=0}^{t-1} \bOmega_T^{t-1}\ldots \bOmega^{s+1}_T\X^\top \diag(\partial_2\ell(\bbeta^s_T,\bbeta^\ast,\z))\e_i,
\end{align*}
as desired.
\end{proof}
To motivate the need for the above analysis, consider for instance the response function $R_{T,\theta}(t,s)$ where $t>s>0$. A combination of Proposition \ref{prop: high dim response} and Lemma \ref{lm: response matrix recursion} shows that almost surely,
\begin{equation}\label{eq:art_response_example}
    R_{T,\theta}(t,s) = \lim_{n,d\rightarrow\infty} \frac{1}{d}\Tr\Big(\gamma\bOmega_T^{t-1}\cdots\bOmega_T^{s+1}\Big).
\end{equation}
That is, we can understand the response $R_{T,\theta}$ as a normalized trace which is in principle much easier to analyze compared to the recursion (\ref{eq:r_T_theta_recursion}). Understanding these response functions as limits of matrix values is a key step in proving the responses are approximately tti and exponentially decay. As such, this type of reasoning will be used extensively in Section \ref{sec:proof_long_time}.

The remaining parts of this section are used to prove Lemma \ref{lm: low dim response}. We first derive some dynamical cavity estimates in Section \ref{subsecsec:loo_estimate}, and then complete the proof in Section \ref{subsecsec:proof_key_lemma}.

\subsubsection{Dynamical Cavity Estimates}\label{subsecsec:loo_estimate}

In this section, we develop some dynamical cavity estimates  that will be useful in the proof of Lemma \ref{lm: low dim response}. We introduce the following cavity dynamics. Let
\begin{equation}\label{def:Xy_leave_col}
    \X^{(j)} = (X_{ik}\mathbf{1}_{k\neq j})_{i\in[n],k\in[d]} \quad \text{ and }\quad \y^{(j)} = \varphi(\X^{(j)}\btheta^*,\z)
\end{equation}
be the cavity versions of $\X$ and $\y$ by leaving out the $j$th column of $\X$. For each $j\in[d]$, define the column cavity version of (\ref{def:art_dyn_first})-(\ref{def:art_dyn_second}) as
\begin{equation}\label{def:cavity_col_dynamic}
    \btheta_T^{t+1,(j)} =\begin{cases}
        \frac{(\X^{(j)})^\top \diag(\cT_s(\y^{(j)})) \X^{(j)} \btheta_T^{t,(j)}}{\beta_T^{t+1}} & -T \leq t<0\\
        \btheta_{T}^{t,(j)} -\gamma(\X^{(j)})^\top\ell(\X^{(j)}\btheta_{T}^{t,(j)},\X^{(j)}\btheta^*,\z) - \gamma\lambda\btheta_{T}^{t,(j)} & t \geq 0,
    \end{cases}
\end{equation}
which is initialized at $\btheta_T^{-T,(j)} = \btheta^\ast$. Similarly, let
\begin{equation}\label{def:Xy_leave_row}
    \X^{[i]} = (X_{kj}\mathbf{1}_{k\neq i})_{k\in[n],j\in[d]} \quad \text{ and } \quad\y^{[i]} = \varphi(\X^{[i]}\btheta^*,\z)
\end{equation}
be the cavity versions of $\X$ and $\y$ by leaving out the $i$th row of $\X$. For each $i\in[n]$, define the row cavity version of (\ref{def:art_dyn_first})-(\ref{def:art_dyn_second}) as
\begin{equation}\label{def:cavity_row_dynamic}
    \btheta_T^{t+1,[i]} =\begin{cases}
        \frac{(\X^{[i]})^\top \diag(\cT_s(\y^{[i]})) \X^{[i]} \btheta_T^{t,[i]}}{\beta_T^{t+1}} & -T \leq t \leq -1\\
        \btheta_{T}^{t,[i]} -\gamma(\X^{[i]})^\top\ell(\X^{[i]}\btheta_{T}^{t,[i]},\X^{[i]}\btheta^*,\z) - \gamma\lambda\btheta_{T}^{t,[i]} & t \geq 0,
    \end{cases}
\end{equation}
which is initialized at $\btheta_T^{-T,[i]} = \btheta^\ast$.  Further define
\begin{align*}
\bbeta^{t,(j)}_T = \X^{(j)}\btheta^{t,(j)}_T, \quad \bbeta^{t,[i]}_T = \X^{[i]}\btheta^{t,[i]}_T, \quad \bbeta^{\ast,(j)} = \X^{(j)}\btheta^\ast, \quad \bbeta^{\ast,[i]} = \X^{[i]}\btheta^\ast.
\end{align*}
Note that for both cavity dynamics, we keep the same sequence $\{\beta_T^t\}$ as in the original artificial dynamic (\ref{def:art_dyn_first})-(\ref{def:art_dyn_second}). In the following, for any $j\in[d],i\in[n]$, we will use the notation
\begin{equation}
    \btheta_T^t = (\theta_{T,j}^t,\btheta_{T,-j}^t) \quad \text{ and } \quad \bbeta_T^t = (\eta_{T,i}^t,\bbeta_{T,-i}^t)
\end{equation}
to isolate single coordinates of $\btheta^t_T$ and $\bbeta^t_T$.

The following lemma controls the $\ell_2$ difference between the cavity dynamics (\ref{def:cavity_col_dynamic}), (\ref{def:cavity_row_dynamic}) and the artificial dynamic (\ref{def:art_dyn_first})-(\ref{def:art_dyn_second}).
\begin{lemma}\label{lem:cavity_l2_estimate}
    Suppose Assumptions \ref{ass: model assumptions}, \ref{ass: spectral initialization} and \ref{ass: ell lipschitz} hold. Fix any $T,H>0$. There exists a constant $C>0$ (which can depend on $T,H$ but $n,d$ independent) such that for any $\gamma>0$, integer $t \in [-T,H]$ and $j\in[d]$, $i\in [n]$
    \begin{align}
        \frac{\pnorm{\btheta_T^t}{}}{\sqrt{d}} \vee \frac{\pnorm{\btheta_T^{t,(j)}}{}}{\sqrt{d}} \vee \frac{\pnorm{\btheta_T^{t,[i]}}{}}{\sqrt{d}} \leq C, \label{eq:cavity_1}\\
        \pnorm{\btheta_T^{t,(j)}-\btheta_T^t}{} \leq C\sqrt{\log d}(1 + |\theta_j^\ast|),\label{eq:cavity_2}\\
        \pnorm{\btheta_T^{t,[i]} - \btheta_T^t}{} \leq C\log d(1+|z_i| )\label{eq:cavity_3},
    \end{align}
    almost surely for large $n,d$.
\end{lemma}
\begin{proof}
    For some large fixed constant $C_0>0,$ define the event
    \begin{equation}
        \label{def:loo_event}
        \mathcal{E}_n = \Big\{\pnorm{\X}{\rm op}\leq C_0,\pnorm{\btheta^*}{} \vee\pnorm{\z}{}\leq C_0\sqrt{d}\Big\}.
    \end{equation}
    By Assumption \ref{ass: model assumptions}, $\mathcal{E}_n$ holds almost surely for large enough $n,d$. We note that on $\mathcal{E}_n$, the cavity versions $\pnorm{\X^{(j)}}{\rm op}\leq C_0$ and $\pnorm{\X^{[i]}}{\rm op}\leq C_0$ as well for any $j \in [d],i\in[n]$.
    
    To see (\ref{eq:cavity_1}), first consider the case $t>0$. By iterating $\btheta^t_T = \btheta^{t-1}_T - \gamma(\X^\top \ell(\X\btheta^{t-1}_T, \X\btheta^\ast,\z) + \lambda \btheta_T^{t-1})$, we have 
    \begin{equation*}
        \btheta_T^t = \btheta_T^0 - \gamma\sum_{s=0}^{t-1} \Big[\X^\top\ell(\X\btheta_T^{s},\X\btheta^\ast,\z) + \lambda\btheta_T^s\Big].
    \end{equation*}
    Then using Assumption \ref{ass: ell lipschitz}, on the event $\cE_n$ we have
    \begin{align*}
        \frac{\pnorm{\btheta_T^t}{}}{\sqrt{d}} &\leq \frac{\pnorm{\btheta_T^0}{}}{\sqrt{d}} + \gamma\sum_{s=0}^{t-1} \Big[\pnorm{\X}{\rm op} \frac{\pnorm{\ell(\X\btheta_T^s,\X\btheta^\ast,\z)}{}}{\sqrt{d}}+ \lambda \frac{\pnorm{\btheta_T^s}{}}{\sqrt{d}}\big]\\
        &\leq \frac{\pnorm{\btheta_T^0}{}}{\sqrt{d}} + C\gamma\sum_{s=0}^{t-1} \Big[\pnorm{\X}{\rm op}\Big(1+\frac{\pnorm{\X}{\rm op}\pnorm{\btheta_T^s}{}}{\sqrt{d}}+ \frac{\pnorm{\X}{\rm op}\pnorm{\btheta^\ast}{}}{\sqrt{d}}+\frac{\pnorm{\z}{}}{\sqrt{d}}\Big)+\lambda \frac{\pnorm{\btheta_T^s}{}}{\sqrt{d}}\big]\\
        &\leq \frac{\pnorm{\btheta_T^0}{}}{\sqrt{d}} + C\gamma\sum_{s=0}^{t-1}\Big[1+\frac{\pnorm{\btheta_T^s}{}}{\sqrt{d}}\Big].
    \end{align*}
    This leads to the recursive formula
    \begin{equation*}
        1+\frac{\pnorm{\btheta_T^t}{}}{\sqrt{d}} \leq 1+\frac{\pnorm{\btheta_T^0}{}}{\sqrt{d}}+ C\gamma\sum_{s=0}^{t-1}\Big[1+\frac{\pnorm{\btheta_T^s}{}}{\sqrt{d}}\Big],
    \end{equation*}
    which after iterating yields
    \begin{equation*}
        \frac{\pnorm{\btheta_T^t}{}}{\sqrt{d}} \leq (1+C\gamma)^t\Big[1+\frac{\pnorm{\btheta_T^0}{}}{\sqrt{d}}\Big].
    \end{equation*}
    It remains to control $\limsup_{n,d\rightarrow\infty}\pnorm{\btheta_T^0}{}/\sqrt{d}$. By definition (\ref{def:art_dyn_first}), we have the expansion
    \begin{equation}
    \btheta_T^0=\frac{(\X^\top\diag(\cT_s(\y))\X)^T\btheta^\ast}{\beta_T},
    \end{equation}
    where 
    \begin{equation}
        \beta_T = \lim_{n,d\rightarrow\infty}\frac{\pnorm{(\X^\top\diag(\cT_s(\y))\X)^T\btheta^\ast}{}}{\sqrt{d}}.
    \end{equation}
    On the event $\mathcal{E}_n$ we have $\pnorm{\X^\top \diag\big(\cT_s(\y)\big)\X}{\op} \leq C$, hence $\pnorm{\btheta_T^0}{}/\sqrt{d} \leq C' \pnorm{\btheta^\ast}{}/\sqrt{d}$, where $C'>0$ is a constant depending on $T$. On the event $\mathcal{E}_n$, we have $\pnorm{\btheta^\ast}{}/\sqrt{d} \leq C_0$ which implies $\pnorm{\btheta^0_T}{}/\sqrt{d}\leq C$, completing the proof. The case $t\leq 0$ and bounds for $\pnorm{\btheta_T^{t,(j)}}{}/\sqrt{d}, \pnorm{\btheta_T^{t,[i]}}{}/\sqrt{d}$ follow from similar arguments.

    Next we show (\ref{eq:cavity_2}). Again consider the case $t>0$. Direct computation shows 
    \begin{equation*}
        \btheta_T^t-\btheta_T^{t,(j)} = (1-\gamma\lambda)(\btheta_T^{t-1}-\btheta_T^{t-1,(j)}) - \gamma\Big(\X^\top\ell(\X\btheta_T^{t-1},\X\btheta^\ast,\z)-(\X^{(j)})^\top\ell(\X^{(j)}\btheta_T^{t-1,(j)},\X^{(j)}\btheta^\ast,\z)\Big),
    \end{equation*}
    which implies
    \begin{align}\label{eq:loo_main}
        \pnorm{\btheta_T^t-\btheta_T^{t,(j)}}{} &\leq (1-\gamma\lambda)\pnorm{\btheta_T^{t-1}-\btheta_T^{t-1,(j)}}{} +\gamma \Bigpnorm{\X^\top\ell(\X\btheta_T^{t-1},\X\btheta^\ast,\z)-(\X^{(j)})^\top\ell(\X^{(j)}\btheta_T^{t-1,(j)},\X^{(j)}\btheta^\ast,\z)}{}.
    \end{align}
    The second term satisfies
    \begin{align*}
       &\Bigpnorm{\X^\top\ell(\X\btheta_T^{t-1},\X\btheta^\ast,\z)-(\X^{(j)})^\top\ell(\X^{(j)}\btheta_T^{t-1,(j)},\X^{(j)}\btheta^\ast,\z)}{}\\
       &\qquad \leq \Bigpnorm{(\X-\X^{(j)})^\top \ell(\X^{(j)}\btheta_T^{t-1,(j)},\X^{(j)}\btheta^\ast,\z)}{} + \Bigpnorm{\X^\top\Big(\ell(\X\btheta_T^{t-1},\X\btheta^\ast,\z)-\ell(\X^{(j)}\btheta_T^{t-1,(j)},\X^{(j)}\btheta^\ast,\z)\Big)}{}.
    \end{align*}
    Using Assumption \ref{ass: ell lipschitz}, the second term here can be further bounded by
    \begin{align*}
    &\Bigpnorm{\X^\top\Big(\ell(\X\btheta_T^{t-1},\X\btheta^\ast,\z)-\ell(\X^{(j)}\btheta_T^{t-1,(j)},\X^{(j)}\btheta^\ast,\z)\Big)}{}\\
    &\leq C\pnorm{\ell(\X\btheta_T^{t-1},\X\btheta^\ast,\z)-\ell(\X^{(j)}\btheta_T^{t-1,(j)},\X^{(j)}\btheta^\ast,\z)}{}\\
    &\leq C\pnorm{\X\btheta_T^{t-1}-\X^{(j)}\btheta_T^{t-1,(j)}}{} + C\pnorm{(\X-\X^{(j)})\btheta^\ast}{}\\
    &\leq C\pnorm{\btheta_T^{t-1}-\btheta_T^{t-1,(j)}}{} + C\pnorm{(\X-\X^{(j)})\btheta_T^{t-1,(j)}}{} + C\pnorm{(\X-\X^{(j)})\btheta^\ast}{}.
    \end{align*}
    Plugging this into (\ref{eq:loo_main}), we have
    \begin{align*}
        \pnorm{\btheta_T^t-\btheta_T^{t,(j)}}{} &\leq 
        (1 + C\gamma) \pnorm{\btheta_T^{t-1}-\btheta_T^{t-1,(j)}}{}\\
        &\quad + \underbrace{C\gamma\Big(\pnorm{(\X-\X^{(j)})^\top \ell(\X^{(j)}\btheta_T^{t-1,(j)},\X^{(j)}\btheta^\ast,\z)}{} + \pnorm{(\X-\X^{(j)})\btheta_T^{t-1,(j)}}{} + \pnorm{(\X-\X^{(j)})\btheta^\ast}{2}\Big)}_{\Delta_{t-1,j}}.
    \end{align*}
    We can iterate this bound to get
    \begin{align*}
        \pnorm{\btheta_T^t-\btheta_T^{t,(j)}}{} &\leq (1+C\gamma)^t\pnorm{\btheta_T^{0}-\btheta_T^{0,(j)}}{} +\sum_{s=0}^{t-1}(1+C\gamma)^{t-1-s}\Delta_{s,j} \leq C\pnorm{\btheta_T^{0}-\btheta_T^{0,(j)}}{} + C\max_{0\leq s \leq t-1}\Delta_{s,j}.
    \end{align*}
   We now bound $\pnorm{\btheta_T^{0}-\btheta_T^{0,(j)}}{}$.
    By construction, we have
    \begin{align*}
        \pnorm{\btheta_T^{0}-\btheta_T^{0,(j)}}{} &= \frac{1}{\beta_T^0}\Bigpnorm{\X^\top\diag(\cT_s(\y))\X\btheta_T^{-1}-(\X^{(j)})^\top\diag(\cT_s(\y^{(j)}))\X^{(j)}\btheta_T^{-1,(j)}}{}\\
        &\leq\frac{1}{\beta_T^0}\Big[\Bigpnorm{\X^\top\diag(\cT_s(\y))\X(\btheta_T^{-1}-\btheta_T^{-1,(j)})}{}\\
        &\qquad +\underbrace{\Bigpnorm{\Big(\X^\top\diag(\cT_s(\y))\X-(\X^{(j)})^\top\diag(\cT_s(\y^{(j)}))\X^{(j)}\Big)\btheta_T^{-1,(j)}}{}\Big]}_{\Omega_{-1,j}}\\
        &\leq \frac{C}{\beta^0_T}\Big(\pnorm{\btheta_T^{-1}-\btheta_T^{-1,(j)}}{} + \Omega_{-1,j}\Big).
    \end{align*}
    Iterating this bound and using that $\btheta_T^{-T}=\btheta_T^{-T,(j)} = \btheta^\ast$, we arrive at
    $\pnorm{\btheta_T^{0}-\btheta_T^{0,(j)}}{} \leq C\max_{-T\leq s \leq -1}\Omega_{s,j}$, which then implies
    \begin{align}\label{eq:col_cavity_main}
    \pnorm{\btheta_T^t-\btheta_T^{t,(j)}}{} \leq C\Big(\max_{0\leq s \leq t-1}\Delta_{s,j} + \max_{-T\leq s \leq -1}\Omega_{s,j}\Big).
    \end{align}

    Let us now bound $\Delta_{s,j}$ and $\Omega_{s,j}$. For $\Delta_{s,j}$, we will bound the term $\pnorm{(\X-\X^{(j)})^\top \ell(\X^{(j)}\btheta_T^{s,(j)},\X^{(j)}\btheta^\ast,\z)}{}$, and the other terms $\pnorm{(\X-\X^{(j)})\btheta_T^{s,(j)}}{}, \pnorm{(\X-\X^{(j)})\btheta^\ast}{}$ therein will follow from similar arguments. Note that $\X^{(j)} = \X-\x_j\e_j^\top$, where $\x_j$ denotes the $j$th column of $\X$. Then we can write
    \begin{equation*}
       \Bigpnorm{(\X-\X^{(j)})^\top \ell(\X^{(j)}\btheta_T^{s,(j)},\X^{(j)}\btheta^\ast,\z)}{} =  \Big|\x_j^\top\ell(\X^{(j)}\btheta_T^{s,(j)},\X^{(j)}\btheta^\ast,\z)\Big|. 
    \end{equation*}
    By independence between $\x_j$ and $\ell(\X^{(j)}\btheta_T^{s,(j)},\X^{(j)}\btheta^\ast,\z)$, conditioning on $\X_{-j},\btheta^\ast,\z$, $\x_j^\top\ell(\X^{(j)}\btheta_T^{s,(j)},\X^{(j)}\btheta^\ast,\beps)$ is sub-Gaussian with variance proxy bounded by a constant multiple of
    $\frac{1}{d}\pnorm{\ell(\X^{(j)}\btheta_T^{s,(j)},\X^{(j)}\btheta^\ast,\z)}{}^2$, which implies the tail bound
    \begin{equation*}
        \P\Bigg(|\x_j^\top\ell(\X^{(j)}\btheta_T^{s,(j)},\X^{(j)}\btheta^\ast,\z)|\geq C\sqrt{\frac{\pnorm{\ell(\X^{(j)}\btheta_T^{s,(j)},\X^{(j)}\btheta^\ast,\z)}{}^2\log d}{d}}\Bigg) \leq 2d^{-10},
    \end{equation*}
    for large enough $C,c \geq 0$.
    Note that the above computation holds for any $s,t$. Taking  union bound over $j\in[d], 0\leq s\leq t-1$ and applying Borel-Cantelli, we have
    \begin{equation*}
         \Big|\x_j^\top\ell(\X^{(j)}\btheta_T^{s,(j)},\X^{(j)}\btheta^\ast,\z)\Big| \leq C\sqrt{\frac{\pnorm{\ell(\X^{(j)}\btheta_T^{s,(j)},\X^{(j)}\btheta^\ast,\z)}{}^2\log d}{d}}
    \end{equation*}
    for all $0\leq s\leq t-1$ and $j\in[d]$ almost surely for large enough $n,d$. By Assumption \ref{ass: ell lipschitz} and (\ref{eq:cavity_1}) established above, we have
    \begin{equation*}
        \pnorm{\ell(\X^{(j)}\btheta_T^{s,(j)},\X^{(j)}\btheta^\ast,\z)}{} \leq C\big(\sqrt{n} + \pnorm{\X^{(j)}\btheta_T^{s,(j)}}{} + \pnorm{\X^{(j)}\btheta^\ast}{} + \pnorm{\z}{}\big) \leq C\sqrt{d}
    \end{equation*}
    almost surely for large enough $n,d$. Combining the above bounds yields
    \begin{equation*}
        \max_{0\leq s\leq t-1}\max_{j\in[d]} \Big|\x_j^\top\ell(\X^{(j)}\btheta_T^{s,(j)},\X^{(j)}\btheta^\ast,\z)\Big| \leq C\sqrt{\log d},
    \end{equation*}
    which, together with similar bounds for $\pnorm{(\X-\X^{(j)})\btheta_T^{s,(j)}}{}, \pnorm{(\X-\X^{(j)})\btheta^\ast}{}$, yields
    \begin{align}\label{eq:Delta_bound}
        \max_{0\leq s\leq t-1}\max_{j\in[d]} \Delta_{s,j} \leq C\sqrt{\log d}
    \end{align}
    almost surely for large enough $n,d$.
    
    Next we bound $\Omega_{s,j}$. We start with the bound
    \begin{align*}
        \Omega_{s,j} &= \Bigpnorm{\Big(\X^\top\diag(\cT_s(\y))\X-(\X^{(j)})^\top\diag(\cT_s(\y^{(j)}))\X^{(j)}\big)\btheta_T^{s,(j)}}{}\\
        &\leq \underbrace{\Bigpnorm{\Big(\X^\top\diag(\cT_s(\y^{(j)}))\X-(\X^{(j)})^\top\diag(\cT_s(\y^{(j)}))\X^{(j)}\Big)\btheta_T^{s,(j)}}{}}_{(I)}\\
        &\quad + \underbrace{\Bigpnorm{\Big(\X^\top\diag(\cT_s(\y))\X -\X^\top\diag(\cT_s(\y^{(j)}))\X\Big)\btheta_T^{s,(j)}}{}}_{(II)}.
    \end{align*}
    To bound $(I)$, using $\X^{(j)} = \X-\x_j\e_j^\top$, we have
    \begin{align*}
        &\X^\top\diag(\cT_s(\y^{(j)}))\X-(\X^{(j)})^\top\diag(\cT_s(\y^{(j)}))\X^{(j)}\\
        &= \X^\top\diag(\cT_s(\y))\x_j\e_j^\top + \e_j\x_j^\top\diag(\cT_s(\y^{(j)}))\X - \e_j\x_j^\top\diag(\cT_s(\y^{(j)}))\x_j\e_j^\top\\
        &= \X^\top\diag(\cT_s(\y^{(j)}))\x_j\e_j^\top + \e_j\x_j^\top\diag(\cT_s(\y^{(j)}))\X^{(j)}.
    \end{align*}
    Since $\cT_s(\cdot)$ is bounded, this implies that, almost surely for large enough $n,d$,
    \begin{align*}
        (I) &\leq C\Big(|\theta_{T,j}^{s,(j)}| + |\x_j^\top\diag(\cT_s(\y^{(j)}))\X^{(j)}\btheta_T^{s,(j)}|\Big).
    \end{align*}
    Note that $\theta_{T,j}^{s,(j)}=0$ for $-T < s \leq -1$ and $\theta_{T,j}^{-T,(j)}=\theta^\ast_j$. Using the independence between $\x_j$ and $\diag(\cT_s(\y^{(j)}))\X^{(j)}\btheta_T^{s,(j)}$ and  similar argument as before, we have
    \begin{equation*}
        \Big|\x_j^\top\diag(\cT_s(\y^{(j)}))\X^{(j)}\btheta_T^{s,(j)}\Big| \leq C\sqrt{\frac{\pnorm{\X^{(j)}\btheta_T^{s,(j)}}{}^2\log d}{d}} \leq C\sqrt{\log d}
    \end{equation*}
    uniformly over $-T\leq s\leq -1$ and $j\in[d]$ for large enough $n,d$, where the last step follows from (\ref{eq:cavity_1}) established above.
    Thus we conclude that 
    \begin{equation*}
        (I) \leq C(|\theta_j^\ast| + \sqrt{\log d}).
    \end{equation*}
    To bound $(II)$, we have
    \begin{align*}
        (II) &= \Bigpnorm{\X^\top\diag\big(\cT_s(\y)-\cT_s(\y^{(j)})\big)\X\btheta_T^{s,(j)}}{} \leq \Bigpnorm{\diag\big(\cT_s(\y)-\cT_s(\y^{(j)})\big)\X\btheta_T^{s,(j)}}{},
    \end{align*}
    and, with $\x_i \in \R^d$ denoting the $i$th row of $\X$ and $\x_i^{(j)} = (X_{ik}\bm{1}_{k\neq j})_{1\leq k\leq d}$ denoting its cavity version, 
    \begin{align*}
    \Big|\Big[\diag\big(\cT_s(\y)-\cT_s(\y^{(j)})\big)\X\btheta_T^{s,(j)}\Big]_i\Big| &= \Big|\big(\cT_s(\varphi(\x_i^\top \btheta^\ast)) - \cT_s(\varphi(\x_i^{(j)\top}\btheta^\ast))\big)\x_i^\top \btheta^{s,(j)}_T\Big|\\
    &\leq C|\theta_j^\ast||X_{ij}|\big(|X_{ij}||\theta^{s,(j)}_{T,j}| + |\x_i^{(j)\top} \btheta^{s,(j)}_T|\big).
    \end{align*}
    This implies
    \begin{align*}
    (II)^2 \leq C|\theta_j^\ast|^2\Big(\sum_{i=1}^n |X_{ij}|^4 |\theta_{T,j}^{s,(j)}|^2 + \sum_{i=1}^n X_{ij}^2(\x_i^{(j)}\btheta^{s,(j)}_T)^2\Big).
    \end{align*}
    The first term satisfies $\sum_{i=1}^n |X_{ij}|^4 |\theta_{T,j}^{s,(j)}|^2 \leq C\pnorm{\btheta^{s,(j)}_T}{}^2/d \leq C$ by (\ref{eq:cavity_1}) established above. For the second term, let $\A^{(j)} := \diag\big(\X^{(j)}\btheta^{s,(j)}_T\btheta^{s,(j)\top}_T\X^{(j)\top}\big)$ be the diagonal matrix in $\R^{n\times n}$ with diagonals of $\X^{(j)}\btheta^{s,(j)}_T\btheta^{s,(j)\top}_T\X^{(j)\top}$, then $\sum_{i=1}^n X_{ij}^2(\x_i^{(j)}\btheta^{s,(j)}_T)^2 = \x_j^\top \A^{(j)} \x_j$, and using independence between $\x_j$ and $\A^{(j)}$ and Hanson-Wright inequality, 
    \begin{align*}
    \P\Big(\Big|\x_j^\top \A^{(j)} \x_j - \Tr \A^{(j)}/d\Big| \geq C \big(\pnorm{\A^{(j)}}{F}\frac{\sqrt{\log d}}{d} + \pnorm{\A^{(j)}}{\op}\frac{\log d}{d}\big)\Big) \leq 2d^{-10}.
    \end{align*}
    We may evaluate $\Tr \A^{(j)}, \pnorm{\A^{(j)}}{F}, \pnorm{\A^{(j)}}{\op}$ as follows:
    \begin{align*}
    \Tr \A^{(j)} &= \Tr \X^{(j)}\btheta^{s,(j)}_T\btheta^{s,(j)\top}_T\X^{(j)\top} = \pnorm{\X^{(j)}\btheta^{s,(j)}_T}{}^2 \leq C\pnorm{\btheta^{s,(j)}_T}{}^2 \leq Cd,\\
    \pnorm{\A^{(j)}}{F} &\leq \pnorm{\X^{(j)}\btheta^{s,(j)}_T\btheta^{s,(j)\top}_T\X^{(j)\top}}{F} = \pnorm{\X^{(j)}\btheta^{s,(j)}_T}{}^2 \leq Cd,\\
     \pnorm{\A^{(j)}}{\op} &= \max_{1\leq i\leq n} (\x_i^{(j)\top} \btheta^{s,(j)}_T)^2 \leq \max_{1\leq i\leq n}\pnorm{\x_i}{}^2 \cdot \pnorm{\btheta^{s,(j)}_T}{}^2 \leq Cd.
    \end{align*}
    By Borel-Cantelli, almost surely for large enough $n,d$ it holds that $\x_j^\top \A^{(j)} \x_j \leq C\log d$, which implies $(II) \leq C|\theta^\ast_j|\sqrt{\log d}$. Combining the bounds of $(I)$ and $(II)$ yields
    \begin{align*}
    \max_{-T\leq s\leq -1} \Omega_{s,j} \leq C\sqrt{\log d}|\theta^\ast_j|,
    \end{align*}
    which, along with (\ref{eq:Delta_bound}), implies the desired bound for $\pnorm{\btheta^t_T - \btheta^{t,(j)}_T}{}$ in view of (\ref{eq:col_cavity_main}).

    Finally we prove (\ref{eq:cavity_3}). We consider the $t>0$ case and note that the $t\leq0$ case will follow by similar arguments. By similar arguments as in (\ref{eq:cavity_2}), we have the bound
    \begin{equation*}
        \pnorm{\btheta_T^t-\btheta_T^{t,[i]}}{} \leq C \Big(\max_{0\leq s\leq t-1}\Delta_{s,i} + \max_{-T\leq s\leq -1}\Omega_{s,i}\Big),
    \end{equation*}
    where
    \begin{align*}
        \Delta_{s,i} &= C\gamma\Big(\pnorm{(\X-\X^{[i]})^\top \ell(\X^{[i]}\btheta_T^{s,[i]},\X^{[i]}\btheta^\ast,\z)}{} + \pnorm{(\X-\X^{[i]})\btheta_T^{s,[i]}}{} + \pnorm{(\X-\X^{[i]})\btheta^\ast}{}\Big),\\
        \Omega_{s,i} &= \Bigpnorm{\Big(\X^\top\diag(\cT_s(\y))\X-(\X^{[i]})^\top\diag(\cT_s(\y^{[i]}))\X^{[i]}\Big)\btheta_T^{s,[i]}}{}.
    \end{align*}
    We first bound the term $\Delta_{s,i}$. To bound the first term $\pnorm{(\X-\X^{[i]})^\top \ell(\X^{[i]}\btheta_T^{s,[i]},\X^{[i]}\btheta^\ast,\z)}{}$, using that $\X-\X^{[i]} = \e_i\x_i^\top$ (where $\e_i$ is the canonical basis in $\R^n$ and $\x_i\in\R^d$ is the $i$th row of $\X$), we have
    \begin{align*}
        &\pnorm{(\X-\X^{[i]})^\top \ell(\X^{[i]}\btheta_T^{s,[i]},\X^{[i]}\btheta^\ast,\z)}{} = \pnorm{\x_i\e_i^\top \ell(\X^{[i]}\btheta_T^{s,[i]},\X^{[i]}\btheta^\ast,\z)}{}\\
        &\leq C\Big|\e_i^\top \ell(\X^{[i]}\btheta_T^{s,[i]},\X^{[i]}\btheta^\ast,\z)\Big| = C|\ell(0,0,z_i)| \leq C(1+|z_i|).
    \end{align*}
    The second term $\pnorm{(\X-\X^{[i]})\btheta_T^{s,[i]}}{}$ satisfies $\pnorm{(\X-\X^{[i]})\btheta_T^{s,[i]}}{} = |\x_i^\top \btheta_T^{s,[i]}|$, so using independence between $\x_i$ and $\btheta^{s,[i]}_T$, we can use similar arguments as above to obtain
    \begin{equation*}
        \max_{0\leq s\leq t-1}\max_{i\in [n]} |\x_i^\top \btheta_T^{s,[i]}| \leq C\sqrt{\log d}
    \end{equation*}
    almost surely for large $n,d$. A similar argument also yields $\max_{0\leq s\leq t-1}\max_{i\in [n]} \pnorm{(\X-\X^{[i]})\btheta^\ast}{} \leq C\sqrt{\log d}$, so we conclude that
    \begin{equation*}
         \max_{0\leq s\leq t-1} \Delta_{s,i} \leq C(1+|z_i|\sqrt{\log d}).
    \end{equation*}
    almost surely for large $n,d$. Next, we move on to bounding $\Omega_{s,i}$. We can bound $\Omega_{s,i}$ by two terms
    \begin{align*}
        \Omega_{s,i} &\leq \underbrace{\Bigpnorm{\Big(\X^\top\diag(\cT_s(\y^{[i]}))\X-(\X^{[i]})^\top\diag(\cT_s(\y^{[i]}))\X^{[i]}\Big)\btheta_T^{s,[i]}}{}}_{(I)}\\
        &\quad + \underbrace{\Bigpnorm{\Big(\X^\top\diag(\cT_s(\y))\X -\X^\top\diag(\cT_s(\y^{[i]}))\X\Big)\btheta_T^{s,[i]}}{}}_{(II)}.
    \end{align*}
    For $(I)$, note that
    \begin{align*}
        &\X^\top\diag(\cT_s(\y^{[i]}))\X-(\X^{[i]})^\top\diag(\cT_s(\y^{[i]}))\X^{[i]}\\
        &= \X^\top\diag(\cT_s(\y^{[i]}))\Big(\X-\X^{[i]}\Big) + \Big(\X-\X^{[i]}\Big)^\top\diag(\cT_s(\y^{[i]}))\X^{[i]}\\
        &=\X^\top\diag(\cT_s(\y^{[i]}))\e_i\x_i^\top + \x_i\e_i^\top\diag(\cT_s(\y^{[i]}))\X,
    \end{align*}
    which implies that
    \begin{align*}
        (I) &\leq \pnorm{\X^\top\diag(\cT_s(\y^{[i]}))\e_i\x_i^\top\btheta_T^{s,[i]}}{} + \pnorm{\x_i\e_i^\top\diag(\cT_s(\y^{[i]}))\X\btheta_T^{s,[i]}}{} \leq C |\x_i^\top\btheta_T^{s,[i]}|.
    \end{align*}
    Using the independence between $\x_i$ and  $\btheta_T^{s,[i]}$, we again have $|\x_i^\top\btheta_T^{s,[i]}| \leq C\sqrt{\log d}$
    almost surely for large $n,d$. For term $(II)$, we have
    \begin{align*}
        (II) &\leq C\Bigpnorm{\diag\big(\cT_s(\y)-\cT_s(\y^{[i]})\big)\X\btheta_T^{s,[i]}}{} = C\Big|\cT_s(\varphi(\x_i^\top\btheta^\ast,z_i)) - \cT_s(\varphi(0,z_i))\Big||\x_i^\top \btheta^{s,[i]}_T|\\
        &\leq  C|\x_i^\top\btheta^\ast||\x_i^\top\btheta_T^{s,[i]}|\leq C\log(d),
    \end{align*}
    where the last line uses similar independence arguments as before and holds almost surely for large $n,d$. Combining the bounds of $(I)$ and $(II)$ yields that $\max_{-T\leq s\leq -1}\Omega_{s,i} \leq C\log d$, so we conclude that $\pnorm{\btheta_T^t-\btheta_T^{t,[i]}}{} \leq C\log d(1+|z_i|)$, as desired.
\end{proof}

\subsubsection{Proof of Lemma \ref{lm: low dim response}}\label{subsecsec:proof_key_lemma}
With these lemmas, we are now ready to prove Lemma \ref{lm: low dim response}.

\begin{proof}[Proof of Lemma \ref{lm: low dim response}]    
    \noindent\textbf{Analysis for $\deps \theta^{t,(s,j),\eps}_{T,j}$ and $\deps\eta_{T,i}^{t+1,[s,i],\eps}$} We perform a joint induction starting from $s = -T$ and $t = -T+1$, in which case
    \begin{align*}
        \deps \theta^{-T+1,(-T,j),\eps}_{T,j} &= -\frac{1}{\beta_T^{-T+1}} = r_{T,\theta}(-T+1,-T),\\
        \cT_s(y_i)\deps \eta^{-T+1,[-T,i],\eps}_{T,i} &= -\frac{1}{\beta_T^{-T+1}}\cT_s(y_i)\cT_s(y_i)\pnorm{\X^\top \e_i}{2}^2 = -r_{T,\eta}^{(\bar\eta_{T,i})}(-T+1,-T) + E_\eta^{-T+1,[-T,i]},
    \end{align*}
    where standard concentration shows that $\frac{1}{n}\sum_{i=1}^n |E_\eta^{-T+1,[-T,i]}|^p \rightarrow 0$ almost surely as $n,d\rightarrow\infty$ for any $p\geq 1$. This establishes the base case. Now suppose the claim holds up to some $t$ and all $s<t$. By Lemma \ref{lm: response matrix recursion}, we have
    \begin{equation}\label{eq:induction_conseq_1}
    \begin{aligned}
    \frac{1}{d}\sum_{j=1}^d \deps \theta^{t,(s,j),\eps}_{T,j} &= 
    \begin{cases}
    -\frac{1}{\beta_T^{s+1}} \frac{1}{d}\Tr \big(\bUpsilon_T^{t-1}\cdots\bUpsilon_T^{s+1}\big) &s\leq-1,t\leq -1\\
            -\frac{1}{\beta_T^{s+1}} \frac{1}{d}\Tr\big(\bOmega^{t-1}_T\cdots\bOmega^0_T\bUpsilon_T^{-1}\cdots\bUpsilon_T^{s+1}\big) & s\leq-1,t> -1\\
            \gamma\frac{1}{d}\Tr\big(\bOmega^{t-1}_T\cdots\bOmega^{s+1}_T\big) & s> -1
    \end{cases}\\
    &\rightarrow_{a.s.} r_{T,\theta}(t,s) = R_{T,\theta}(t,s).
    \end{aligned}
    \end{equation}
    Similarly, if $t \geq 0$ we have
    \begin{equation}\label{eq:induction_conseq_2}
    \begin{aligned}
    \frac{1}{n}\sum_{i=1}^n \partial_1\ell(\eta^{t}_{T,i},\eta^\ast_i,z_i)\deps \eta_{T,i}^{t,[s,i],\eps} &= 
    \begin{cases}
        -\frac{1}{n}\Tr\Big(\frac{1}{\beta_T^{s+1}}\D^t_T \X\bOmega^{t-1}_T\cdots\bOmega^0_T\bUpsilon_T^{-1}\cdots\bUpsilon_T^{s+1}\X^\top\diag(\cT_s(\y))\Big) & s\leq-1\\
        \gamma\frac{1}{n}\Tr\Big(\D^t_T\X\bOmega^{t-1}_T\cdots\bOmega^{s+1}_T\X^\top\D_T^s\Big) & s> -1.
    \end{cases}\\
    &\rightarrow_{a.s.} -\E[r_{T,\eta}^{(\bar\eta_T)}(t,s)] = -\frac{1}{\delta}R_{T,\eta}(t,s),
    \end{aligned}
    \end{equation}
    and if $t \leq -1$, it holds that
    \begin{equation}\label{eq:induction_conseq_3}
    \begin{aligned}
    \frac{1}{n}\sum_{i=1}^n \cT_s(y_i)\deps \eta_{T,i}^{t,[s,i],\eps} &= -\frac{1}{\beta_T^{s+1}}\frac{1}{n}\Tr\Big(\diag(\cT_s(\y)) \X\bUpsilon_T^{t-1}\cdots\bUpsilon_T^{s+1}\X^\top\diag(\cT_s(\y))\Big)\\
    &\rightarrow_{a.s.} -\E[r_{T,\eta}^{(\bar\eta_T)}(t,s)] = -\frac{1}{\delta}R_{T,\eta}(t,s).
    \end{aligned}
    \end{equation}
    
    We will now show the claim for $t+1$ and all $s<t+1$. We assume that $t>0$ and $s<0$, and the other cases will follow from similar arguments.
    Let $\X =[\x_j, \X_{-j}]$ with $\x_j \in \R^n$ denoting the $j$th column of $\X$. Since $t >s $, by writing $\btheta_T^{t,(s,d),\eps} = (\theta^{t,(s,j),\eps}_j, \btheta_{T,-j}^{t,(s,j),\eps})$ with $\theta_{T,j}^{t,(s,j),\eps}$ denoting its $j$th component, we have
    \begin{align*}
        \theta^{t+1,(s,j),\eps}_{T,j} = \theta^{t,(s,j),\eps}_{T,j} - \gamma\left[\x_j^\top\ell(\X\btheta_T^{t,(s,j),\eps},\X\btheta^*,\z)+\lambda\theta_{T,j}^{t,(s,j),\eps}\right].
    \end{align*}
    Taking the $\eps$ derivative yields, with $\{\D^t_T\}$ defined in (\ref{def:D_Omega}), we have
    \begin{align}\label{eq:mu_derivative_1}
    \notag\deps\theta^{t+1,(s,j),\eps}_{T,j} &= \deps\theta^{t,(s,j),\eps}_{T,j} - \gamma\x_j^\top\Big[\partial_1\ell(\bbeta^t_T,\bbeta^*,\z) \circ \Big(\X_{-j}\deps\btheta_{T,-j}^{t,(s,j),\eps} + \x_j\deps\theta_{T,j}^{t,(s,j),\eps}\Big)\Big] - \gamma \lambda \deps\theta_{T,j}^{t,(s,j),\eps}\\
    &= (1-\gamma\lambda-\gamma\x_j^\top \D
    _T^t\x_j)\deps\theta_{T,j}^{t,(s,j),\eps}-\gamma\langle\X_{-j}^\top \D_T^t\x_j,\deps\btheta_{T,-j}^{t,(s,j),\eps}\rangle.
    \end{align}
    Next we compute $\deps\btheta_{T,-j}^{t,(s,j),\eps}$. Note that
    \begin{align*}
    \btheta_{T,-j}^{t,(s,j),\eps} = \btheta_{T,-j}^{t-1,(s,j),\eps} - \gamma \X_{-j}^\top\ell(\X\btheta_T^{t-1,(s,j),\eps},\X\btheta^*,\z)- \gamma \lambda \btheta_{T,-j}^{t-1,(s,j),\eps}.
    \end{align*}
    Taking $\eps$ derivative yields
    \begin{align*}
        \deps\btheta_{T,-j}^{t,(s,j),\eps} = \Big((1-\gamma\lambda)\I_{d-1}-\gamma\X_{-j}^\top \D^{t-1}_T\X_{-j}\Big)\deps\btheta_{T,-j}^{t-1,(s,j),\eps} - \gamma \X_{-j}^\top \D^{t-1}_T\x_j\deps \theta_{T,j}^{t-1,(s,j),\eps}.
    \end{align*}
    Define 
    \begin{align}\label{def:bOmega_reduced}
    \bOmega^{t}_{T,-j}:=(1-\gamma\lambda)\I_{d-1}-\gamma\X_{-j}^\top \D_T^t\X_{-j} \in \R^{(d-1)\times (d-1)}.
    \end{align}
    Iterating the above identity, we get
    \begin{align}\label{eq:theta_derivative_1}
    \deps\btheta_{T,-j}^{t,(s,j),\eps} = \bOmega^{t-1}_{T,-j}\cdots\bOmega^{0}_{T,-j}\deps\btheta_{T,-j}^{0,(s,j),\eps} - \gamma\sum_{\ell=0}^{t-1}\bOmega^{t-1}_{T,-j}\cdots\bOmega^{\ell+1}_{T,-j}\X_{-j}^\top \D^\ell_T\x_j\deps \theta^{\ell,(s,j),\eps}_{T,j},
    \end{align}
    where we understand $\bOmega^{t-1}_{T,-j}\cdots\bOmega^{\ell+1}_{T,-j}$ as $\I_{d-1}$ for $\ell = t-1$. To further compute $\deps \btheta_{T,-j}^{0,(s,j),\eps}$, note that
    \begin{equation*}
        \btheta_{T,-j}^{0,(s,j),\eps} = \frac{\X_{-j}^\top\diag(\cT_s(\y))\big(\X_{-j}\btheta_{T,-j}^{-1,(s,d),\eps} + \x_j\theta_{T,j}^{-1,(s,j),\eps}\big)}{\beta_T^0}.
    \end{equation*}
    Define
    \begin{align}\label{def:Upsilon_reduced}
        \bUpsilon_{T,-j}^{t} := \frac{1}{\beta_T^{t+1}}\X_{-j}^\top\diag(\cT_s(\y))\X_{-j} \in \R^{(d-1)\times (d-1)}.
    \end{align}
    Then taking $\eps$ derivative, we get
    \begin{equation}
        \deps \btheta_{T,-j}^{0,(s,j),\eps} = \bUpsilon_{T,-j}^{-1} \deps\btheta_{T,-j}^{-1,(s,j),\eps} + \frac{1}{\beta_T^0}\X_{-j}^\top\diag(\cT_s(\y)) \x_j\deps\theta_{T,j}^{-1,(s,j),\eps}.
    \end{equation}
    Iterating in this manner, we arrive at
    \begin{align*}
        \deps \btheta_{T,-j}^{0,(s,j),\eps} &=\bUpsilon_{T,-j}^{-1} \cdots \bUpsilon_{T,-j}^{s+1} \deps\btheta_{T,-j}^{s+1,(s,j),\eps} + \sum_{\ell=s+1}^{-1}\frac{1}{\beta_T^{\ell+1}}\bUpsilon_{T,-j}^{-1} \cdots \bUpsilon_{T,-j}^{\ell+1}\X_{-j}^\top\diag(\cT_s(\y)) \x_j\deps\theta_{T,j}^{\ell,(s,j),\eps}\\ 
        &=\sum_{\ell=s+1}^{-1}\frac{1}{\beta_T^{\ell+1}}\bUpsilon_{T,-j}^{-1} \cdots \bUpsilon_{T,-j}^{\ell+1} \X_{-j}^\top\diag(\cT_s(\y))\x_j\deps\theta_{T,j}^{\ell,(s,j),\eps},
    \end{align*}
    where we understand $\bUpsilon_{T,-j}^{-1} \cdots \bUpsilon_{T,-j}^{\ell+1}$ as $\I_{d-1}$ for $\ell = -1$, and the last line follows from $\deps\btheta_{T,-j}^{s+1,(s,j),\eps} = 0$. Plugging the above expression into (\ref{eq:theta_derivative_1}) and then further into (\ref{eq:mu_derivative_1}), we get
    \begin{align*}
        \deps\theta^{t+1,(s,j),\eps}_{T,j} &= \underbrace{(1-\gamma\lambda-\gamma\x_j^\top\D^t_T\x_j)\deps\theta^{t,(s,j),\eps}_{T,j}}_{(I_j)}\\
        &+\underbrace{\gamma^2\sum_{\ell=0}^{t-1}\x_j^\top \D_T^t\X_{-j}\bOmega^{t-1}_{T,-j}\cdots\bOmega^{\ell+1}_{T,-j}\X_{-j}^\top \D_T^\ell\x_j\deps\theta^{\ell,(s,j),\eps}_{T,j}}_{(II_j)}\\
        &+\underbrace{\gamma\sum_{\ell=s+1}^{-1}\frac{1}{\beta_T^{\ell+1}}\x_j^\top\D_T^t\X_{-j}\bOmega^{t-1}_{T,-j}\cdots\bOmega^{0}_{T,-j}\bUpsilon_{T,-j}^{-1} \cdots \bUpsilon_{T,-j}^{\ell+1}\diag(\cT_s(\y)) \x_j\deps\theta_{T,j}^{\ell,(s,j),\eps}}_{(III_j)}.
    \end{align*} 
    To analyze $(I_j)$, define the following cavity version of $\D^{t}_T$ in (\ref{def:D_Omega}) by leaving out the $j$th column:
    \begin{align*}
    \D^{t,(j)}_T := \diag\Big(\partial_1\ell(\bbeta_T^{t,(j)}, \bbeta^{\ast,(j)},\z)\Big) \in\R^{n\times n},
    \end{align*}
    then we have 
    \begin{align*}
    (I_j) = (1-\gamma\lambda-\gamma\x_j^\top\D^{t,(j)}_T\x_j)r_{T,\theta}(t,s) + \sr^{(j)}_{1,1} + \sr^{(j)}_{1,2}, 
    \end{align*}
    where
    \begin{align*}
    \sr^{(j)}_{1,1} = \gamma\x_j^\top (\D_T^{t,(j)} - \D_T^t)\x_j \deps \theta^{t,(s,j),\eps}_{T,j}, \quad \sr^{(j)}_{1,2} = \Big(1-\gamma\lambda-\gamma\x_j^\top\D^{t,(j)}_T\x_j\Big) E^{t,(s,j)}_{T,\theta}.
    \end{align*}
    For the first term, we have
    \begin{align*}
    \x_j^\top (\D^{t,(j)} - \D^t)\x_j \leq \sqrt{\sum_{k=1}^n (\x_j)_k^4} \cdot \pnorm{\partial_1\ell(\bbeta_T^{t,(j)}, \bbeta^{\ast,(j)},\z) - \partial_1\ell(\bbeta_T^{t}, \bbeta^{\ast},\z)}{} \leq C\frac{(\log d)^{3/2}}{\sqrt{d}}\Big(1 + |\theta^\ast_j| \Big)
    \end{align*}
    almost surely for large enough $n,d$, where we apply Assumption \ref{ass: ell lipschitz} and Lemma \ref{lem:cavity_l2_estimate}. By Assumption \ref{ass: model assumptions} and $\pnorm{\deps \btheta^{t,(s,j),\eps}}{} \leq C$ as established in (\ref{eq:derivative_bound}), we have $\frac{1}{d}\sum_{j=1}^d |\sr^{(j)}_{1,1}|^p \rightarrow_{a.s.} 0$ for any $p\geq 1$. Furthermore, using the induction hypothesis on $E^{t,(s,j)}_{T,\theta}$ and the fact that $\x_j^\top \D^{t,(j)}_T \x_j \leq \pnorm{\x_j}{}^2\pnorm{\D^{t,(j)}_T}{\op} \leq C$ almost surely for large enough $n,d$, we have $\frac{1}{d}\sum_{j=1}^d |\sr^{(j)}_{1,2}|^p \rightarrow 0$ almost surely as $n,d\rightarrow\infty$. On the other hand, by the independence between $\x_j$ and $\D^{t,(j)}_T$ and Hanson-Wright inequality, we have
    \begin{align*}
    \P\bigg(|\x_j^\top \D^{t,(j)}_T \x_j - \frac{1}{d}\Tr\D^{t,(j)}_T| \geq \max\Big(\frac{C\log d}{d}\pnorm{\D^{t,(j)}_T}{\op}, \frac{C\sqrt{\log d}}{d} \pnorm{\D^{t,(j)}_T}{F}\Big)\bigg) \leq e^{-cd}.
    \end{align*}
    Using $\pnorm{\D^{t,(j)}_T}{F} \leq C\sqrt{d}$, a union bound over $j\in[d]$, and the Borel-Cantelli lemma, we conclude that $(1-\gamma\lambda-\gamma\x_j^\top\D^{t,(j)}_T\x_j)r_{T,\theta}(t,s) = (1-\gamma\lambda-\gamma\Tr \D^{t,(j)}/d)r_{T,\theta}(t,s) + \sr^{(j)}_{1,3}$, with $\frac{1}{d}\sum_{j=1}^d |\sr^{(j)}_{1,3}|^p \rightarrow 0$ almost surely as $n,d\rightarrow\infty$ for any $p\geq 1$. Lastly, it is easy to see that $\frac{1}{d}\sum_{j=1}^d|\Tr \D^{t,(j)}_T/d - \Tr \D^{t}_T/d|^p \rightarrow 0$ a.s. for $n,d\rightarrow\infty$, and $\frac{1}{d}\Tr \D^{t}_T \rightarrow_{a.s.} \delta\E[\partial_1 \ell(\eta^t_T, w^\ast_T,z)]$ by Proposition \ref{prop:discreteDmftArt}, therefore we conclude that 
    \begin{align*}
    (I_j) = (1-\gamma\lambda-\gamma\delta\E[\partial_1\ell(\eta^t_T, w^\ast_T,z)])r_{T,\theta}(t,s) + \sr^{(j)}_1,
    \end{align*}
    satisfying $\frac{1}{d}\sum_{j=1}^d |\sr^{(j)}_1|^p \rightarrow 0$ almost surely as $n,d\rightarrow\infty$ for any $p\geq 1$.

    To analyze $(II_j)$, we have $(II_j) = \gamma^2\sum_{\ell=0}^{t-1} (II_{j,\ell})$, where
    \begin{align*}
    (II_{j,\ell}) = \x_j^\top \D_T^t\X_{-j}\bOmega^{t-1}_{T,-j}\cdots\bOmega^{\ell+1}_{T,-j}\X_{-j}^\top \D_T^\ell\x_j\deps\theta^{\ell,(s,j),\eps}_{T,j}.
    \end{align*}
    We now apply a similar leave one out argument as above to analyze each $(II_{j,\ell})$. Let
    \begin{align*}
    \bOmega^{t,(j)}_{T,-j} = (1-\gamma\lambda)\I_{d-1}-\gamma\X_{-j}^\top \D_T^{t,(j)}\X_{-j} \in \R^{(d-1)\times (d-1)}
    \end{align*}
    be the cavity version of (\ref{def:bOmega_reduced}) by leaving out the $j$th column. Then
    \begin{align}\label{eq:matrix_prod_main}
    \notag&\Big|\x_j^\top \D_T^t\X_{-j}\bOmega^{t-1}_{T,-j}\cdots\bOmega^{\ell+1}_{T,-j}\X_{-j}^\top \D_T^\ell\x_j - \x_j^\top \D_T^{t,(j)}\X_{-j}\bOmega^{t-1,(j)}_{T,-j}\cdots\bOmega^{\ell+1,(j)}_{T,-j}\X_{-j}^\top \D_T^{\ell,(j)}\x_j\Big|\\
    \notag&\leq \Big|\x_j^\top \D_T^t\X_{-j}\bOmega^{t-1}_{T,-j}\cdots\bOmega^{\ell+1}_{T,-j}\X_{-j}^\top \D_T^\ell\x_j - \x_j^\top \D_T^{t,(j)}\X_{-j}\bOmega^{t-1}_{T,-j}\cdots\bOmega^{\ell+1}_{T,-j}\X_{-j}^\top \D_T^{\ell,(j)}\x_j\Big|\\
    &\quad + \sum_{k=t-1}^{\ell+1} \Big|\x_j^\top \D^{t,(j)}_T \X_{-j}\bOmega^{t-1,(j)}_{T,-j}\ldots \bOmega^{k+1,(j)}_{T,-j}(\bOmega^{k,(j)}_{T,-j} - \bOmega^k_{T,-j})\bOmega^{k-1}_{T,-j}\ldots \bOmega^{\ell+1}_{T,-j}\X_{-j}^\top \D_{T}^{\ell,(j)}\x_j\Big|.
    \end{align}
    For each $k$ in the above summand, using the fact that $|\a^\top \D \b| \leq \pnorm{\a}{\infty}\pnorm{\b}{}\pnorm{\D}{F}$ for diagonal matrix $\D$ and vectors $\a,\b$, we have
    \begin{align*}
    &\Big|\x_j^\top \D^{t,(j)}_T \X_{-j}\bOmega^{t-1,(j)}_{T,-j}\ldots \bOmega^{k+1,(j)}_{T,-j}(\bOmega^{k,(j)}_{T,-j} - \bOmega^k_{T,-j})\bOmega^{k-1}_{T,-j}\ldots \bOmega^{\ell+1}_{T,-j}\X_{-j}^\top \D_{T}^{\ell,(j)}\x_j\Big|\\
    &=\gamma\Big|\x_j^\top \D^{t,(j)}_T \X_{-j}\bOmega^{t-1,(j)}_{T,-j}\ldots \bOmega^{k+1,(j)}_{T,-j}\X_{-j}^\top(\D^{k,(j)}_{T,-j} - \D^k_{T,-j})\X_{-j}\bOmega^{k-1}_{T,-j}\ldots \bOmega^{\ell+1}_{T,-j}\X_{-j}^\top \D_{T}^{\ell,(j)}\x_j\Big|\\
    &\leq \gamma \pnorm{\big(\x_j^\top \D^{t,(j)}_T \X_{-j}\bOmega^{t-1,(j)}_{T,-j}\ldots \bOmega^{k+1,(j)}_{T,-j}\X_{-j}^\top\big)^\top}{\infty}\pnorm{\D^{k,(j)}_{T,-j} - \D^k_{T,-j}}{F}\pnorm{\X_{-j}\bOmega^{k-1}_{T,-j}\ldots \bOmega^{\ell+1}_{T,-j}\X_{-j}^\top \D_{T}^{\ell,(j)}\x_j}{}\\
    &\leq C\frac{\log d}{\sqrt{d}}(1 + |\theta^\ast_j|),
    \end{align*}
    where we use the independence between $\x_j$ and $\D^{t,(j)}_T \X_{-j}\bOmega^{t-1,(j)}_{T,-j}\ldots \bOmega^{k+1,(j)}_{T,-j}\X_{-j}^\top$ and Lemma \ref{lem:cavity_l2_estimate}. A similar estimate holds for the first term of (\ref{eq:matrix_prod_main}). On the other hand, similar estimates as above yield
    \begin{align*}
    &\frac{1}{d}\sum_{j=1}^d\Big|\x_j^\top \D_T^{t,(j)}\X_{-j}\bOmega^{t-1,(j)}_{T,-j}\cdots\bOmega^{\ell+1,(j)}_{T,-j}\X_{-j}^\top \D_T^{\ell,(j)}\x_j - \frac{1}{d}\Tr\Big(\D_T^{t,(j)}\X_{-j}\bOmega^{t-1,(j)}_{T,-j}\cdots\bOmega^{\ell+1,(j)}_{T,-j}\X_{-j}^\top \D_T^{\ell,(j)}\Big)\Big|^p \rightarrow_{a.s.} 0,\\
    &\frac{1}{d}\sum_{j=1}^d\Big|\frac{1}{d}\Tr\Big(\D_T^{t,(j)}\X_{-j}\bOmega^{t-1,(j)}_{T,-j}\cdots\bOmega^{\ell+1,(j)}_{T,-j}\X_{-j}^\top \D_T^{\ell,(j)}\Big) - \frac{1}{d}\Tr\Big(\D_T^{t}\X_{-j}\bOmega^{t-1}_{T,-j}\cdots\bOmega^{\ell+1}_{T,-j}\X_{-j}^\top \D_T^{\ell}\Big)\Big|^p \rightarrow_{a.s.} 0,\\
    &\frac{1}{d}\sum_{j=1}^d\Big|\frac{1}{d}\Tr\Big(\D_T^{t}\X_{-j}\bOmega^{t-1}_{T,-j}\cdots\bOmega^{\ell+1}_{T,-j}\X_{-j}^\top \D_T^{\ell}\Big) - \frac{1}{d}\Tr\Big(\D_T^{t}\X\bOmega^{t-1}_{T}\cdots\bOmega^{\ell+1}_{T}\X^\top \D_T^{\ell}\Big)\Big|^p \rightarrow_{a.s.} 0.
    \end{align*}
    Lastly, by case $s > -1$ of (\ref{eq:induction_conseq_2}), we have
    \begin{align*}
    \frac{1}{d}\Tr\Big(\D_T^{t}\X\bOmega^{t-1}_{T}\cdots\bOmega^{\ell+1}_{T}\X^\top \D_T^{\ell}\Big) \rightarrow_{a.s.} \frac{\delta}{\gamma} \cdot \Big(-\frac{1}{\delta}\Big)R_{T,\eta}(t,\ell) = -\frac{1}{\gamma}R_{T,\eta}(t,\ell).
    \end{align*}
    We hence conclude that each $(II_{j,\ell}) = -R_{T,\eta}(t,\ell)r_{T,\theta}(\ell,s)/\gamma + \sr^{(j)}_{2,\ell}$, with $\frac{1}{d} \sum_{j=1}^d |\sr^{(j)}_{2,\ell}|^p \rightarrow_{a.s.} 0$ almost surely as $n,d\rightarrow\infty$ for each $p\geq 1$, which further implies that
    \begin{align*}
    (II_j) = -\gamma\sum_{\ell=0}^{t-1} R_{T,\eta}(t,\ell)r_{T,\theta}(\ell,s) + \sr^{(j)}_2, 
    \end{align*}
    with $\frac{1}{d}\sum_{j=1}^d |\sr^{(j)}_2|^p \rightarrow_{a.s.} 0$ as $n,d\rightarrow\infty$.

    The analysis of $(III_j)$ is similar. Let
    \begin{align}
        \bUpsilon_{T,-j}^{t,(j)} := \frac{1}{\beta_T^{t+1}}\X_{-j}^\top\diag(\cT_s(\y^{(j)}))\X_{-j} \in \R^{(d-1)\times (d-1)}.
    \end{align}
    be the cavity version of (\ref{def:Upsilon_reduced}) by leaving out the $j$th column of $\X$, with $\y^{(j)}$ given by (\ref{def:Xy_leave_col}). Using similar estimates as above, we have for each $s+1\leq \ell \leq -1$
    \begin{align*}
    &\frac{1}{d}\sum_{j=1}^d\Big|\x_j^\top\D_T^t\X_{-j}\bOmega^{t-1}_{T,-j}\cdots\bOmega^{0}_{T,-j}\bUpsilon_{T,-j}^{-1} \cdots \bUpsilon_{T,-j}^{\ell+1}\diag(\cT_s(\y)) \x_j\\
    &- \x_j^\top\D_T^{t,(j)}\X_{-j}\bOmega^{t-1,(j)}_{T,-j}\cdots\bOmega^{0,(j)}_{T,-j}\bUpsilon_{T,-j}^{-1,(j)} \cdots \bUpsilon_{T,-j}^{\ell+1,(j)}\diag(\cT_s(\y)) \x_j\Big|^p \rightarrow_{a.s.} 0,
    \end{align*}
    and moreover,
    \begin{align*}
    &\frac{1}{d}\sum_{j=1}^d\Big|\x_j^\top\D_T^{t,(j)}\X_{-j}\bOmega^{t-1,(j)}_{T,-j}\cdots\bOmega^{0,(j)}_{T,-j}\bUpsilon_{T,-j}^{-1,(j)} \cdots \bUpsilon_{T,-j}^{\ell+1,(j)}\diag(\cT_s(\y)) \x_j\\
    &- \frac{1}{d}\Tr\Big(\D_T^{t,(j)}\X_{-j}\bOmega^{t-1,(j)}_{T,-j}\cdots\bOmega^{0,(j)}_{T,-j}\bUpsilon_{T,-j}^{-1,(j)} \cdots \bUpsilon_{T,-j}^{\ell+1,(j)}\diag(\cT_s(\y))\Big)\Big|^p \rightarrow_{a.s.}0,\\
    &\frac{1}{d}\sum_{j=1}^d\Big|\frac{1}{d}\Tr\Big(\D_T^{t,(j)}\X_{-j}\bOmega^{t-1,(j)}_{T,-j}\cdots\bOmega^{0,(j)}_{T,-j}\bUpsilon_{T,-j}^{-1,(j)} \cdots \bUpsilon_{T,-j}^{\ell+1,(j)}\diag(\cT_s(\y))\Big)\\
    &-\frac{1}{d}\Tr\Big(\D_T^{t}\X\bOmega^{t-1}_{T}\cdots\bOmega^{0}_{T}\bUpsilon_{T}^{-1} \cdots \bUpsilon_{T}^{\ell+1}\diag(\cT_s(\y))\Big)\Big|^p \rightarrow_{a.s.} 0.
    \end{align*}
    Lastly, we have by case $s \leq -1$ of (\ref{eq:induction_conseq_2}) that
    \begin{align*}
    \frac{1}{d}\Tr\Big(\frac{1}{\beta^{t+1}_T}\D_T^{t}\X\bOmega^{t-1}_{T}\cdots\bOmega^{0}_{T}\bUpsilon_{T}^{-1} \cdots \bUpsilon_{T}^{\ell+1}\diag(\cT_s(\y))\Big) \rightarrow -\delta R_{T,\eta}(t,\ell).
    \end{align*}
    In summary, we have
    \begin{align*}
    (III_j) = -\gamma \sum_{\ell=s+1}^{-1}R_{T,\eta}(t,\ell)r_{T,\theta}(\ell,s) + \sr^{(j)}_3,
    \end{align*}
    with $\frac{1}{d}\sum_{j=1}^d |\sr^{(j)}_3|^p \rightarrow_{a.s.} 0$ as $n,d\rightarrow\infty$.

    Putting together the above estimates for $(I_j),(II_j),(III_j)$, we have
    \begin{align*}
    \deps\theta^{t+1,(s,j),\eps}_{T,j} &= (1-\gamma\lambda-\gamma\delta\E[\partial_1\ell(\eta^t_T, w^\ast_T,z)])r_{T,\theta}(t,s) - \gamma\sum_{\ell=s+1}^{t-1} R_{T,\eta}(t,\ell)r_{T,\theta}(\ell,s) + E^{t+1,(s,j)}_{T,\theta}\\
    &= r_{T,\theta}(t+1,s) + E^{t+1,(s,j)}_{T,\theta}, 
    \end{align*}
    concluding the induction for $\deps\theta^{t,(s,j),\eps}_{T,j}$.
    
    Next we show the induction for $\partial_\eps|_{\eps=0} \eta^{t,[s,i],\eps}_{T,i}$. As above, we will only show the induction for the case when $t>0$ and $s<0$ as the computations in the other cases are simpler. Let $\X = [\x_i^\top;\X_{-i}]^\top$, where $\x_i \in \R^d$ denotes the $i$th row of $\X$ and $\X_{-i} \in \R^{(n-1)\times d}$ denotes the remaining $n-1$ rows of $\X$. Write 
\begin{align*}
\bbeta^{t,[s,i],\eps}_T = (\bbeta^{t,[s,i],\eps}_{T,-i},\eta_{T,i}^{t,[s,i],\eps}), \quad \bbeta^t_T = (\bbeta_{T,-i}^t,\eta_{T,i}^{t}),\quad \bbeta^\ast = (\bbeta^\ast_{-i}, \eta^\ast_i), \quad \z = (\z_{-i}, z_i).
\end{align*}
Additionally, for any $t\in \Z_+$ we use the notation
\begin{equation}\label{def:DandOmega}
\begin{aligned}
\D_{T,-i}^t &:= \diag\Big(\partial_1 \ell(\bbeta_{T,-i}^t,\boldeta_{-i}^*,\z_{-i})\Big) \in \R^{(n-1)\times (n-1)},\\ \bOmega_{T,-i}^t &:= (1-\gamma\lambda) \I_d - \gamma\X_{-i}^\top\D_{T,-i}^t\X_{-i} \in \R^{d \times d},\\
\bUpsilon^t_{T,-i} &:= \frac{1}{\beta_T^{t+1}}\X_{-i}^\top\diag(\cT_s(\y_{-i}))\X_{-i} \in \R^{d\times d}.
\end{aligned}
\end{equation}
We may expand $\eta^{t,[s,i],\eps}_{T,i}$ as
\begin{align*}
    \eta_{T,i}^{t+1,[s,i],\eps} &= \eta_{T,i}^{t,[s,i],\eps} -\gamma\x_i^\top\X^\top\ell(\bbeta^{t,[s,i],\eps}_T,\bbeta^*,\z) - \gamma\lambda\eta_{T,i}^{t,[s,i],\eps}\\
    &= (1-\gamma\lambda)\eta_{T,i}^{t,[s,i],\eps} - \gamma\x_i^\top \Big(\X_{-i}^\top\ell(\bbeta_{T,-i}^{t,[s,i],\eps},\bbeta_{-i}^*,\z_{-i}) + \x_i \ell(\eta_{T,i}^{t,[s,i],\eps},\eta^*_i,z_i)\Big).
\end{align*}
Taking derivative on both sides yields
\begin{align}\label{eq:kappa_expansion}
    \notag\deps\eta_{T,i}^{t+1,[s,i],\eps} &= \Big(1-\gamma\lambda - \gamma\pnorm{\x_i}{}^2\partial_1\ell(\eta^t_{T,i},\eta^*_i,z_i)\Big)\deps\eta^{t,[s,i],\eps}_{T,i} -\gamma \x_i^\top \X_{-i}^\top \D^t_{T,-i} \X_{-i} \deps \btheta^{t,[s,i],\eps}_T\\
    \notag&= \Big(1-\gamma\lambda - \gamma\pnorm{\x_i}{}^2\partial_1\ell(\eta_{T,i}^t,\eta^*_i,z_i)\Big)\deps\eta_{T,i}^{t,[s,i],\eps}\\
    & - (1-\gamma\lambda)\x_i^\top\deps\btheta_T^{t,[s,i],\eps} +  \x_i^\top\bOmega_{T,-i}^t\deps\btheta_T^{t,[s,i],\eps},
\end{align}
applying $\gamma \X_{-i}^\top \D^t_{-i}\X_{-i} = (1-\gamma \lambda)\I_d - \bOmega^t_{T,-i}$ in the second step. To derive $\deps\btheta_T^{t,[s,i],\eps}$, taking derivative on both sides of the first definition of (\ref{def:perturbed_dynamic_eta}), we get
\begin{align*}
    \deps\btheta_T^{t,[s,i],\eps} &= \Big((1-\gamma\lambda)\I_d - \X_{-i}^\top\D^{t-1}_{T,-i}\X_{-i}\Big) \deps\btheta_T^{t-1,[s,i],\eps} - \gamma\x_i\partial_1\ell(\eta_{T,i}^{t-1},\eta^*_i,z_i)\deps\eta_{T,i}^{t-1,[s,i],\eps}\\
    &=\bOmega_{T,-i}^{t-1} \deps\btheta_T^{t-1,[s,i],\eps} - \gamma\x_i\partial_1\ell(\eta_{T,i}^{t-1},\eta^*_i,z_i)\deps\eta_{T,i}^{t-1,[s,i],\eps}.
\end{align*}
Iterating the above identity yields
\begin{align}\label{eq:theta_row_derivative}
    \deps\btheta_T^{t,[s,i],\eps} &= \bOmega_{T,-i}^{t-1}\cdots \bOmega_{T,-i}^{0}\deps\btheta_T^{0,[s,i],\eps} -\gamma\sum_{j=0}^{t-1} \bOmega_{T,-i}^{t-1}\cdots \bOmega_{T,-i}^{j+1}\x_i\partial_1\ell(\eta_{T,i}^j,\eta^*_i,z_i)\deps \eta_{T,i}^{j,[s,i],\eps}.
\end{align}
where we understand $\bOmega^{t-1}_{T,-i}\ldots\bOmega^{j+1}_{T,-i}$ as $\I_d$ for $j = t-1$. 
To further compute $\deps\btheta_T^{0,[s,i],\eps}$, note that
\begin{equation*}
    \deps\btheta_T^{0,[s,i],\eps} = \bUpsilon_{T,-i}^{-1}\deps\btheta_T^{-1,[s,i],\eps} +\frac{1}{\beta_T^0}\x_i\cT_s(y_i)\deps\eta_{T,i}^{-1,[s,i],\eps}.
\end{equation*}
Iterating the above identity yields 
\begin{align*}
    \deps\btheta_T^{0,[s,i],\eps} &= \bUpsilon_{T,-i}^{-1}\cdots \bUpsilon_{T,-i}^{s+1}\deps\btheta_T^{s+1,[s,i],\eps} + \sum_{j=s+1}^{-1}\frac{1}{\beta_T^{j+1}}\bUpsilon_{T,-i}^{-1}\cdots \bUpsilon_{T,-i}^{j+1}\x_i\cT_s(y_i)\deps\eta_{T,i}^{j,[s,n],\eps}\\
    &=-\frac{1}{\beta_T^{s+1}}\bUpsilon_{T,-i}^{-1}\cdots \bUpsilon_{T,-i}^{s+1}\x_i\cT_s(y_i) + \sum_{j=s+1}^{-1}\frac{1}{\beta_T^{j+1}}\bUpsilon_{T,-i}^{-1}\cdots \bUpsilon_{T,-i}^{j+1}\x_i\cT_s(y_i)\deps\eta_{T,i}^{j,[s,i],\eps}.
\end{align*}
Plugging the above expression of $\deps \btheta^{0,[s,i],\eps}_T$ into (\ref{eq:theta_row_derivative}), we have
\begin{align*}
\deps\btheta_T^{t,[s,i],\eps} &= -\frac{1}{\beta^{s+1}_T} \bOmega^{t-1}_{T,-i}\ldots \bOmega^{0}_{T,-i}\bUpsilon_{T,-i}^{-1}\cdots \bUpsilon_{T,-i}^{s+1}\x_i\cT_s(y_i) + \sum_{j=s+1}^{-1} \frac{1}{\beta^{j+1}_T} \bOmega^{t-1}_{T,-i}\ldots \bOmega^{0}_{T,-i}\bUpsilon_{T,-i}^{-1}\cdots \bUpsilon_{T,-i}^{s+1}\x_i\cT_s(y_i)\\
&-\gamma\sum_{j=0}^{t-1} \bOmega_{T,-i}^{t-1}\cdots \bOmega_{T,-i}^{j+1}\x_i\partial_1\ell(\eta_{T,i}^j,\eta^*_i,z_i)\deps \eta_{T,i}^{j,[s,i],\eps}.
\end{align*}
Plugging the above identity into (\ref{eq:kappa_expansion}), we get the decomposition
\begin{align}
    \begin{split}
        \deps\eta_{T,i}^{t+1,[s,i],\eps} &=\underbrace{-\gamma\pnorm{\x_i}{}^2\partial_1\ell(\eta_{T,i}^t,\eta^*_i,\eps_i)\deps\eta_{T,i}^{t,[s,i],\eps}}_{(I_i)} +\underbrace{(1-\gamma\lambda)\deps\eta_{T,i}^{t,[s,i],\eps}}_{(II_i)}\\
        &\underbrace{+\frac{1}{\beta_T^{s+1}}(1-\gamma\lambda)\x_i^\top \bOmega_{T,-i}^{t-1}\cdots\bOmega_{T,-i}^{s+1}\bUpsilon_{T,-i}^{-1}\cdots \bUpsilon_{T,-i}^{s+1}\x_i\cT_s(y_i)}_{(III_i)}\\
        &\underbrace{-(1-\gamma\lambda)\sum_{j=s+1}^{-1}\frac{1}{\beta_T^{j+1}}\x_i^\top\bOmega_{T,-i}^{t-1}\cdots \bOmega_{T,-i}^{0}\bUpsilon_{T,-i}^{-1}\cdots \bUpsilon_{T,-i}^{j+1}\x_i\cT_s(y_i)\deps\eta_{T,i}^{j,[s,i],\eps}}_{(IV_i)}\\
        &\underbrace{+\gamma(1-\gamma\lambda)\sum_{j=0}^{t-1}\x_i^\top \bOmega_{T,-i}^{t-1}\cdots \bOmega_{T,-i}^{j+1}\x_i\partial_1\ell(\eta_{T,i}^j,\eta^*_i,\eps_i)\deps\eta_{T,i}^{j,[s,i],\eps}}_{(V_i)}\\
        &\underbrace{-\frac{1}{\beta_T^{s+1}}\x_i^\top \bOmega_{T,-i}^{t}\cdots\bOmega_{T,-i}^{s+1}\bUpsilon_{T,-i}^{-1}\cdots \bUpsilon_{T,-i}^{s+1}\x_i\cT_s(y_i)}_{(VI_i)}\\
        &\underbrace{-\gamma\sum_{j=0}^{t-1}\x_i^\top \bOmega_{T,-i}^{t}\cdots \bOmega_{T,-i}^{j+1}\x_i\partial_1\ell(\eta_{T,i}^j,\eta^*_i,\eps_i)\deps\eta_{T,i}^{j,[s,i],\eps}}_{(VII_i)}\\
        &+\underbrace{\sum_{j=s+1}^{-1}\frac{1}{\beta_T^{j+1}}\x_i^\top\bOmega_{T,-i}^{t}\cdots \bOmega_{T,-i}^{0}\bUpsilon_{T,-i}^{-1}\cdots \bUpsilon_{T,-i}^{j+1}\x_i\cT_s(y_i)\deps\eta_{T,i}^{j,[s,i],\eps}}_{(VIII_i)}.
    \end{split}
\end{align}
By the induction hypothesis in (\ref{eq:induction_conseq_1}) and proper leave one out arguments above, we have 
\begin{align*}
    (I_i) &=  R_{T,\theta}(t+1,t)r_{T,\eta}^{(\bar\eta_{T,i})}(t,s) + \sr^{[i]}_1,\\
    (II_i) &= -(1-\gamma\lambda)(\partial_1\ell(\eta_{T,i}^t,\eta^*_i,z_i))^{-1}r^{(\bar\eta_{T,i})}_{T,\eta}(t,s) + \sr^{[i]}_2,\\
    (III_i) &= -(1-\gamma\lambda)R_{T,\theta}(t,s)\cT_s(y_i)+ \sr^{[i]}_3,\\
    (IV_i) &= -(1-\gamma\lambda)\sum_{j=s+1}^{-1} R_{T,\theta}(t,j)r_{T,\eta}^{(\bar\eta_{T,i})}(j,s)+ \sr^{[i]}_4,\\
    (V_i) &= -(1-\gamma\lambda)\sum_{j=0}^{t-1} R_{T,\theta}(t,j)r^{(\bar\eta_{T,i})}_{T,\eta}(j,s)+ \sr^{[i]}_5,\\
    (VI_i) &= R_{T,\theta}(t+1,s)\cT_s(y_i)+ \sr^{[i]}_6,\\
    (VII_i) &= \sum_{j=0}^{t-1} R_{T,\theta}(t+1,j)r^{(\bar\eta_{T,i})}_{T,\eta}(j,s)+ \sr^{[i]}_7,\\
    (VIII_i) &= \sum_{j=s+1}^{-1} R_{T,\theta}(t+1,j)r^{(\bar\eta_{T,i})}_{T,\eta}(j,s)+ \sr^{[i]}_8.
\end{align*}
Note that the main terms of $(II_i)-(IV_i)$ add up to 0 by definition of $r_{T,\eta}^{(\bar\eta_{T,i})}$:
\begin{align*}
-(\partial_1\ell(\eta_{T,i}^t,\eta^*_i,z_i))^{-1}r^{(\bar\eta_{T,i})}_{T,\eta}(t,s)-R_{T,\theta}(t,s)\cT_s(y_i) - \sum_{j=s+1}^{t-1} R_{T,\theta}(t,j)r_{T,\eta}^{(\bar\eta_{T,i})}(j,s) = 0.
\end{align*}
Adding up the rest of the terms yields
\begin{align*}
   \deps\eta_{T,i}^{t+1,[s,i],\eps} = R_{T,\theta}(t+1,s)\cT_s(y_i) + \sum_{j=s+1}^t R_{T,\theta}(t+1,j)r_{T,\eta}^{(\bar\eta_{T,i})}(j,s) + \bar E^{t+1,[s,i]}_{T,\eta},   
\end{align*}
where $\frac{1}{n}\sum_{i=1}^n |\bar E^{t+1,[s,i]}_{T,\eta}|^p \rightarrow_{a.s.} 0$ as $n,d\rightarrow\infty$. This further implies that
\begin{align*}
\partial_1\ell(\eta^t_{T,i}, \eta^\ast_i,z_i)\deps\eta_{T,i}^{t+1,[s,i],\eps} =& \partial_1\ell(\eta^t_{T,i}, \eta^\ast_i,z_i)\Big(R_{T,\theta}(t+1,s)\cT_s(y_i) + \sum_{j=s+1}^t R_{T,\theta}(t+1,j)r_{T,\eta}^{(\bar\eta_{T,i})}(j,s)\Big) + E^{t+1,[s,i]}_{T,\eta}\\
&= -r_{T,\eta}^{(\bar\eta_{T,i})}(t+1,s) + E^{t+1,[s,i]}_{T,\eta},
\end{align*}
where $\frac{1}{n}\sum_{i=1}^n |\bar E^{t+1,[s,i]}_{T,\eta}|^p \rightarrow_{a.s.} 0$ as $n,d\rightarrow\infty$. This concludes the induction for $\deps\eta_{T,i}^{t+1,[s,i],\eps}$.

\paragraph{Analysis of $\btheta_T^{t+1,[s,\sharp],\eps}$}
We first consider the case $t<0$ and prove by induction. The base case $t=-T$ clearly holds since $\deps \eta^{-T,[\sharp,i],\eps}_{T,i} = 0$, while $r_{T,\eta}^{(\bar\eta_{T,i})}(-T,\sharp) = \cT'_s(y_i)\varphi'(\eta^\ast_i,z_i)\eta^{-T}_{T,i}$, so the identity holds exactly. Assume now the claim holds up to some $t < 0$. Then by definition
\begin{align*}
\deps \btheta^{t+1,[\sharp,i],\eps}_T &= \frac{\X^\top \bZ_s\X}{\beta^{t+1}_T}\deps \btheta_T^{t,[\sharp,i],\eps} + \frac{1}{\beta^{t+1}_T}\X^\top \big(\cT_s'(\y)\circ\bvarphi'(\bbeta^\ast,\z)\circ\bbeta^{t}_T\big)\e_i\\
&= \Big(\frac{\X_{-i}^\top \bZ_{s,-i}\X_{-i}}{\beta^{t+1}_T} + \frac{1}{\beta^{t+1}_T}\cT_s(y_i)\x_i\x_i^\top\Big)\deps \btheta^{t,[\sharp,i],\eps}_T + \frac{1}{\beta^{t+1}_T}\x_i \cT_s'(y_i)\varphi'(\eta^\ast_i,z_i)\eta^t_{T,i}\\
&= \bUpsilon^{t}_{T,-i}\deps \btheta^{t-1,[\sharp,i],\eps}_T + \frac{\cT_s(y_i)\x_i}{\beta^{t+1}_T}\deps \eta^{t,[\sharp,i],\eps}_T + \frac{1}{\beta^{t+1}_T}\x_i \cT_s'(y_i)\varphi'(\eta^\ast_i,z_i)\eta^t_{T,i},
\end{align*}
where $\bZ_{s,-i} = \diag(\y_{-i})$ and $\bUpsilon^t_{T,-i}$ is given in (\ref{def:DandOmega}).
Iterating this identity yields
\begin{align}\label{eq:theta_deriv}
\deps\btheta^{t+1,[\sharp,i],\eps}_{T} = \sum_{s=-T}^{t}\bUpsilon^{t}_{T,-i}\ldots\bUpsilon^{s+1}_{T,-i} \Big[\frac{\cT_s(y_i)\x_i}{\beta^{s+1}_T}\deps \eta^s_{T,i} + \frac{1}{\beta^{s+1}_T}\x_i \cT_s'(y_i)\varphi'(\eta^\ast_i,z_i)\eta^s_{T,i}\Big].
\end{align}
This implies
\begin{align*}
\deps \eta^{t+1,[\sharp,i],\eps}_{T,i} &= \x_i^\top \deps \btheta_T^{t+1,[\sharp,i],\eps}\\
&= \sum_{s=-T}^{t} \x_i^\top \frac{1}{\beta^{s+1}_{T}}\bUpsilon^{t}_{T,-i}\ldots\bUpsilon^{s+1}_{T,-i}\x_i \cdot \Big(\cT_s(y_i)\deps \eta^{s,[\sharp,i],\eps}_{T,i} + \cT_s'(y_i)\varphi'(\eta^\ast_i,z_i)\eta^s_{T,i}\Big).
\end{align*}
By the induction hypothesis, for each $s\in [-T,t]$, it holds that $\cT_s(y_i)\deps \eta^{s,[\sharp,i],\eps}_{T,i} + \cT_s'(y_i)\varphi'(\eta^\ast_i,z_i)\eta^s_{T,i} = r_{T,\eta}^{(\bar\eta_{T,i})}(s,\sharp) + E^{s,[\sharp,i]}_{T,\eta}$, with $\frac{1}{n}\sum_{i=1}^n |E^{s,[\sharp,i]}_{T,\eta}|^p \rightarrow_{a.s.} 0$ as $n,d\rightarrow\infty$. On the other hand, a similar leave one out argument as above yields that
\begin{align*}
\frac{1}{n}\sum_{i=1}^n\Big|\x_i^\top \frac{1}{\beta^{s+1}_{T}}\bUpsilon^{t}_{T,-i}\ldots\bUpsilon^{s+1}_{T,-i}\x_i - \frac{1}{d}\Tr\Big(\frac{1}{\beta^{s+1}_T}\bUpsilon^{t}_{T,-i}\ldots\bUpsilon^{s+1}_{T,-i}\Big)\Big|^p \rightarrow_{a.s.} 0,
\end{align*}
and by (\ref{eq:induction_conseq_1}), $\frac{1}{d}\Tr\Big(\frac{1}{\beta^{s+1}_T}\bUpsilon^{t}_{T,-i}\ldots\bUpsilon^{s+1}_{T,-i}\Big) \rightarrow_{a.s.} -R_{T,\theta}(t+1,s)$. This entails that
\begin{align*}
\deps \eta^{t+1,[\sharp,i],\eps}_{T,i} =  -\sum_{s=-T}^{t} R_{T,\theta}(t+1,s)r_{\eta,T}^{(\bar\eta_{T,i})}(s,\sharp) + \bar E^{t+1,[\sharp,i]}_{T,\eta},
\end{align*}
where $\frac{1}{n}\sum_{i=1}^n |\bar E^{t+1,[\sharp,i]}_{T,\eta}|^p \rightarrow_{a.s.} 0$ as $n,d\rightarrow\infty$. Therefore we have as desired that
\begin{align*}
&\cT_s(y_i)\deps \eta^{t+1,[\sharp,i],\eps}_{T,i} + \cT_s'(y_i)\varphi'(\eta^\ast_i, z_i)\eta^{t+1}_{T,i}\\
&= \cT_s(y_i)\Big(-\sum_{s=-T}^{t} R_{T,\theta}(t+1,s)r^{(\bar\eta_{T,i})}_{T,\eta}(s,\sharp)\Big) + \cT_s'(y_i)\varphi'(\eta^\ast_i, z_i)\eta^{t+1}_{T,i} + E^{t+1,[\sharp,i]}_{T,\eta} = r_{T,\eta}^{(\bar\eta_{T,i})}(t+1,\sharp) + E^{t+1,[\sharp,i]}_{T,\eta},
\end{align*}
where $\frac{1}{n}\sum_{i=1}^n |E^{t+1,[\sharp,i]}_{T,\eta}|^p \rightarrow_{a.s.} 0$ as $n,d\rightarrow\infty$. This establishes the claim up to $r_{T,\eta}^{(\bar\eta_{T,i})}(0,\sharp)$.

Next suppose the claim holds up to some $t\geq 0$. Then we have
\begin{align*}
\deps \btheta^{t+1,[\sharp,i],\eps}_{T} &= \deps \Big[(1-\gamma\lambda)\btheta^{t,[\sharp,i],\eps}_T - \gamma\X^\top \ell(\X\btheta^{t,[\sharp,i],\eps}_T,\bbeta^\ast,\z)\Big]\\
&= (1-\gamma\lambda)\deps \btheta^{t,[\sharp,i],\eps}_T - \gamma \X^\top \diag\big(\partial_1\ell(\X\btheta^t_T, \bbeta^\ast,\z)\big)\X\deps \btheta^{t,[\sharp,i],\eps}_T\\
&= \underbrace{\Big(1-\gamma\lambda - \X_{-i}^\top \partial_1\ell(\bbeta^t_{T,-i},\bbeta^\ast_{-i},\z_{-i})\X_{-i}\Big)}_{\bOmega^t_{T,-i}}\deps \btheta^{t,[\sharp,i],\eps}_T - \gamma\partial_1 \ell(\eta^t_{T,i},\eta^\ast_i,z_i)\x_i \deps \eta^{t,[\sharp,i],\eps}_{T,i},
\end{align*}
with $\bOmega^t_{T,-i}$ given in (\ref{def:DandOmega}). Iterating this identity yields
\begin{align*}
\deps \btheta^{t+1,[\sharp,i],\eps}_T &= \bOmega^{t}_{T,-i}\ldots \bOmega^{0}_{T,-i}\deps \btheta_T^{0,[\sharp,i],\eps} - \gamma\sum_{s=0}^t \bOmega^t_{T,-i}\ldots\bOmega^{s+1}_{T,-i}\partial_1\ell(\eta^s_{T,i},\eta^\ast_i,z_i) \x_i \deps \eta^{s,[\sharp,i],\eps}_{T,i}\\
&=\bOmega^{t}_{T,-i}\ldots \bOmega^{0}_{T,-i} \Big(\sum_{s=-T}^{-1}\bUpsilon^{-1}_{T,-i}\ldots\bUpsilon^{s+1}_{T,-i} \Big[\frac{\cT_s(y_i)\x_i}{\beta^{s+1}_T}\deps \eta^{s,[\sharp,i],\eps}_{T,i} + \frac{1}{\beta^{s+1}_T}\x_i \cT_s'(y_i)\varphi'(\eta^\ast_i,z_i)\eta^s_{T,i}\Big]\Big)\\
&\quad - \gamma\sum_{s=0}^t \bOmega^t_{T,-i}\ldots\bOmega^{s+1}_{T,-i}\partial_1\ell(\eta^s_{T,i},\eta^\ast_i,z_i) \x_i \deps \eta^{s,[\sharp,i],\eps}_{T,i},
\end{align*}
where we apply (\ref{eq:theta_deriv}) with $t = -1$. This implies that
\begin{align*}
\deps \eta^{t+1,[\sharp,i],\eps}_{T,i} &= \sum_{s=-T}^{-1}\x_i^\top \frac{1}{\beta^{s+1}_T} \bOmega^{t}_{T,-i}\ldots \bOmega^{0}_{T,-i}\bUpsilon^{-1}_{T,-i}\ldots \bUpsilon^{s+1}_{T,-i}\x_i \cdot \underbrace{\Big(\cT_s(y_i)\deps \eta^{s,[\sharp,i],\eps}_{T,i} + \cT_s'(y_i)\varphi'(\eta^\ast_i,z_i)\eta^s_{T,i}\Big)}_{r_{T,\eta}^{(\bar\eta_{T,i})}(s,\sharp) + E^{s,[\sharp,i]}_{T,\eta}}\\
&\quad - \gamma\sum_{s=0}^t \x_i^\top \bOmega^{t}_{T,-i}\ldots \bOmega^{s+1}_{T,-i} \x_i \cdot \underbrace{\partial_1\ell(\eta^s_{T,i},\eta^\ast_i,z_i)\deps \eta^{s,[\sharp,i],\eps}_{T,i}}_{r_{T,\eta}^{(\bar\eta_{T,i})}(s,\sharp) + E^{s,[\sharp,i]}_{T,\eta}}.
\end{align*}
Using a similar leave one out argument as above and (\ref{eq:induction_conseq_1}), we have
\begin{align*}
&\frac{1}{n}\sum_{i=1}^n\Big|\x_i^\top \frac{1}{\beta^{s+1}_T} \bOmega^{t}_{T,-i}\ldots \bOmega^{0}_{T,-i}\bUpsilon^{-1}_{T,-i}\ldots \bUpsilon^{s+1}_{T,-i}\x_i - \frac{1}{d}\Tr\Big[\frac{1}{\beta^{s+1}_T} \bOmega^{t}_{T,-i}\ldots \bOmega^{0}_{T,-i}\bUpsilon^{-1}_{T,-i}\ldots \bUpsilon^{s+1}_{T,-i}\Big]\Big|^p\rightarrow_{a.s.} 0,\\
&\frac{1}{d}\Tr\Big[\frac{1}{\beta^{s+1}_T} \bOmega^{t}_{T,-i}\ldots \bOmega^{0}_{T,-i}\bUpsilon^{-1}_{T,-i}\ldots \bUpsilon^{s+1}_{T,-i}\Big] \rightarrow_{a.s.} -R_{T,\theta}(t+1,s),
\end{align*}
and
\begin{align*}
&\frac{1}{n}\sum_{i=1}^n \Big|\x_i^\top \bOmega^{t}_{T,-i}\ldots \bOmega^{s+1}_{T,-i} \x_i - \frac{1}{d}\Tr\Big[\bOmega^{t}_{T,-i}\ldots \bOmega^{s+1}_{T,-i}\Big]\Big|^p \rightarrow_{a.s.} 0,\\
&\frac{1}{d}\Tr\Big[\bOmega^{t}_{T,-i}\ldots \bOmega^{s+1}_{T,-i}\Big] \rightarrow_{a.s.} R_{T,\theta}(t+1,s)/\gamma.
\end{align*}
This implies
\begin{align*}
\partial_1\ell(\eta^{t+1}_{T,i},\eta^\ast_i,z_i)\deps \eta^{t+1,[\sharp,i],\eps}_{T,i} &= \partial_1\ell(\eta^{t+1}_{T,i},\eta^\ast_i,z_i)\Big(-\sum_{s=-T}^t R_{T,\theta}(t+1,s)r_{T,\eta}^{(\bar\eta_{T,i})}(s,\sharp)\Big) + E^{t+1,[\sharp,i]}_{T,\eta}\\
&= r_{T,\eta}^{(\bar\eta_{T,i})}(t+1,\sharp) + E^{t+1,[\sharp,i]}_{T,\eta}, 
\end{align*}
where $\frac{1}{n}\sum_{i=1}^n |E^{t+1,[\sharp,i]}_{T,\eta}|^p \rightarrow_{a.s.} 0$ as $n,d\rightarrow\infty$. This concludes the induction.

\paragraph{Analysis of $\btheta_T^{t+1,[s,\ast],\eps}$} 
We prove by induction. The baseline case $t = 0$ clearly holds, since $\deps \eta^{0,[\ast,i],\eps}_{T,i} = 0$ and $r_{T,\eta}^{(\bar\eta_{T,i})}(0,\ast) = \partial_2 \ell(\eta^0_{T,i}, \eta^\ast_i,z_i)$. Suppose the claim holds up to some $t\geq 0$. With $\bOmega_{T,-i}^t$ defined in (\ref{def:DandOmega}), we have
\begin{align*}
\deps\btheta^{t+1,[\ast,i],\eps}_T &= \Big((1-\gamma\lambda)\I_d - \gamma \X^\top \diag\big(\partial_1 \ell(\bbeta^{t}_T,\bbeta^\ast,\z)\big)\X\Big)\deps\btheta^{t,[\ast,i],\eps}_T - \gamma\X^\top \diag(\partial_2\ell(\bbeta_T^{t},\bbeta^\ast,\z))\e_i\\
&= \bOmega^{t}_{T,-i}\deps\btheta^{t,[\ast,i],\eps}_T - \gamma\x_i\partial_1 \ell(\eta^{t}_{T,i},\eta^\ast_i,z_i)\deps\eta^{t,[\ast,i],\eps}_{T,i} - \gamma \x_i\partial_2\ell(\eta^t_{T,i},\eta^\ast_i,z_i).
\end{align*}
Iterating this identity yields
\begin{align*}
\deps\btheta^{t+1,[\ast,i],\eps}_T = -\gamma\sum_{s=0}^{t}\bOmega^{t}_{T,-i}\ldots\bOmega^{s+1}_{T,-i}\x_i\Big(\partial_1 \ell(\eta^{s}_{T,i},\eta^\ast_i,z_i) \deps \eta^{s,[\ast,i],\eps}_{T,i} + \partial_2 \ell(\eta^s_{T,i},\eta^\ast_i,z_i)\Big),
\end{align*}
which further implies that
\begin{align*}
\deps\eta^{t+1,[\ast,i],\eps}_{T,i} = -\gamma\sum_{s=0}^{t}\x_i^\top\bOmega^{t}_{T,-i}\ldots\bOmega^{s+1}_{T,-i}\x_i\underbrace{\Big(\partial_1 \ell(\eta^{s}_{T,i},\eta^\ast_i,z_i) \deps \eta^{s,[\ast,i],\eps}_{T,i} + \partial_2 \ell(\eta^s_{T,i},\eta^\ast_i,z_i)\Big)}_{r_{T,\eta}^{(\bar\eta_{T,i})}(s,\ast) + E^{s,[\ast,i]}_{T,\eta}}.
\end{align*}
By a similar leave one out argument, we have
\begin{align*}
&\frac{1}{n}\sum_{i=1}^n \Big|\x_i^\top \bOmega^{t}_{T,-i}\ldots \bOmega^{s+1}_{T,-i}\x_i - \frac{1}{d}\Tr\Big(\bOmega^{t}_{T,-i}\ldots \bOmega^{s+1}_{T,-i}\Big)\Big|^p \rightarrow_{a.s.}0,\\
&\frac{1}{d}\Tr\Big(\bOmega^{t}_{T,-i}\ldots \bOmega^{s+1}_{T,-i}\Big) \rightarrow_{a.s.} R_{T,\theta}(t,s)/\gamma.
\end{align*}
Hence
\begin{align*}
&\partial_1\ell(\eta^{t+1}_{T,i}, \eta^\ast_i,z_i)\deps\eta^{t+1,[\ast,i],\eps}_{T,i} + \partial_2\ell(\eta^{t+1}_{T,i}, \eta^\ast_i,z_i)\\
&= \partial_1\ell(\eta^{t+1}_{T,i}, \eta^\ast_i,z_i)\Big(-\sum_{s=0}^t R_{T,\theta}(t,s)r_{T,\eta}^{(\bar\eta_{T,i})}(s,\ast)\Big)+ \partial_2\ell(\eta^{t+1}_{T,i}, \eta^\ast_i,z_i) + E^{t+1,[\ast,i]}_{T,\eta}\\
&= r_{T,\eta}^{(\bar\eta_{T,i})}(t+1,\ast) + E^{t+1,[\ast,i]}_{T,\eta},
\end{align*}
where $\frac{1}{n}\sum_{i=1}^n |E^{t+1,[\ast,i]}_{T,\eta}|^p \rightarrow_{a.s.} 0$ as $n,d\rightarrow\infty$. The proof is complete.
\end{proof}

\subsection{Convergence of artificial DMFT system}
\label{subsec:art_res_conv}
The goal of this section is to show the convergence of the artificial DMFT (\ref{def:artificialEta})-(\ref{def:artificial_dmft_CR_theta}) to the spectral initialization DMFT (\ref{def:dmft_eta})-(\ref{def:theta_res}) as the number of iterates in the first stage $T \rightarrow \infty$, in the sense of the following proposition. Recall that for a sequence of variables $\{X_i\}_{i\geq 0}$ and $X$ in the same space, $\lim_{n\rightarrow \infty} X_n \overset{L^2}{=} X$ is notation for $\lim_{n\rightarrow\infty} \pnorm{X_n - X}{L^2} = \lim_{n\rightarrow\infty} \sqrt{\E (X_n -X)^2} = 0$.

\begin{proposition}
\label{prop:art_spec_dmft_coup}
Suppose that Assumptions \ref{ass: model assumptions}, \ref{ass: spectral initialization} and \ref{ass: ell lipschitz} hold and moreover $\psi_\delta'(\lambda_\delta^*) > 0$. For any fixed integer $h > 0$, there exists a coupling of the randomness in (\ref{def:artificialEta})-(\ref{def:artificial_dmft_CR_theta}) and (\ref{def:dmft_eta})-(\ref{def:theta_res}) such that the following holds.
\begin{itemize}
\item[(1)] As $T\rightarrow\infty$,
\begin{align*}
&\sum_{t=0}^h \E (\theta^t - \theta^t_T)^2 + \sum_{t=0}^h \E (\eta^t - \eta^t_T)^2 \rightarrow 0,\\
&\E(u^{-1}_T + \sqrt{\delta}\lambda^\ast_\delta u^\diamond)^2 + \sum_{t=0}^h \E(u^t_T - u^t)^2 + \E(w^\ast_T - w^\ast)^2 + \sum_{t=0}^h \E(w^t_T - w^t)^2 \rightarrow 0.
\end{align*}
\item[(2)] For any $t > s \geq 0$, we have that
    \begin{align*}
        \lim_{T\rightarrow\infty}r_{T,\theta}(t,s) = r_\theta(t,s), \quad \lim_{T\rightarrow\infty}r_{T,\eta}(t,s) \overset{L^2}{=} r_\eta(t,s), \quad 
        \lim_{T\rightarrow\infty}r_{T,\eta}(s,\ast) \overset{L^2}{=} r_\eta(s,\ast).
    \end{align*}
    As a consequence, for any $t > s \geq 0$, it holds that
    \begin{align*}
        \lim_{T\rightarrow\infty}R_{T,\theta}(t,s) = R_\theta(t,s), \quad \lim_{T\rightarrow\infty}R_{T,\eta}(t,s) = R_\eta(t,s),\quad \lim_{T\rightarrow\infty}R_{T,\eta}(s,\ast) = R_\eta(s,\ast).
    \end{align*}
\item[(3)] For any $t>s$ with $s<0$, we have that
 \begin{align*}
        -\lambda_\delta^\ast \lim_{T\rightarrow\infty} \sum_{s=-T}^{-1} R_{T,\theta}(t,s) = R_\theta(t,\diamond), \quad
        \lim_{T\rightarrow\infty}\sum_{s=-T}^{-1}R_{T,\eta}(t,s)  = R_\eta(t,\diamond), \quad
        \lim_{T\rightarrow\infty}R_{T,\eta}(t,\sharp) = R_{\eta}(t,\diamond\diamond).
    \end{align*}
\end{itemize}
\end{proposition}

Proposition \ref{prop:art_spec_dmft_coup} is the main result of this subsection. In the remainder of this subsection, we prove this result in three steps, each corresponding to one part of the Proposition. 

\paragraph{Constructing the Coupling}
Recall that the artificial DMFT is the low-dimensional characterization of a high-dimensional artificial power iteration dynamic $\btheta_T^{-T},...,\btheta_T^0,\btheta_T^1,...$ given in Section \ref{subsec:art_dynamic} while the spectral initialized DMFT is the low-dimensional characterization of a spectral initialized gradient descent $\btheta^0,\btheta^1,...$ given by (\ref{eq: gradient descent}). We first show that the second stage of the artificial high-dimensional dynamic $\btheta_T^0,\btheta_T^1,...$ converges to the high-dimensional dynamic of interest $\btheta^0,\btheta^1,...$ as $T\rightarrow\infty$.

For notation, we use $\lambda_{1,\infty} = \delta \psi_\delta(\lambda_\delta^\ast),\lambda_{2,\infty} := \delta\psi_\delta(\bar\lambda_\delta)$ to denote the almost sure limit of the first and second eigenvalues of $\M_n = \X^\top\diag(\cT_s(\y))\X$ as shown by Proposition \ref{prop: spec est overlap}. We note that $\lambda_{1,\infty} > \lambda_{2,\infty}$ with a strictly positive gap under our assumption that $\psi_\delta'(\lambda_\delta^*) > 0$.

\begin{lemma}\label{lem:hdFirstStage}
Suppose Assumptions \ref{ass: model assumptions}  and \ref{ass: spectral initialization} hold and that $\psi_\delta'(\lambda_\delta^*) > 0$. Then almost surely it holds that
\begin{align*}
    \limsup_{n,d\rightarrow\infty}\frac{1}{\sqrt{d}}\pnorm{\sqrt{d}\hat{\btheta}^s - \btheta_T^0}{} \leq C \Big(\frac{\lambda_{2,\infty}}{\lambda_{1,\infty}}\Big)^T
\end{align*}
for some universal $C > 0$. Moreover, suppose Assumption \ref{ass: ell lipschitz}, then there exists $C > 0$ depending on $\gamma,\lambda,\pnorm{\partial_1 \ell}{\infty}$ such that 
\begin{align*}
\limsup_{n,d\rightarrow\infty}\frac{1}{\sqrt{d}}\pnorm{\btheta^t - \btheta_T^t}{} \leq C^t \Big(\frac{\lambda_{2,\infty}}{\lambda_{1,\infty}}\Big)^T.
\end{align*}
\end{lemma}
\begin{proof}
    We consider the following dynamic
    \begin{equation}
        \label{def:true_power_iter}
        \btheta_{PI,T}^{t+1} = \sqrt{d}\frac{\X^\top \diag(\cT_s(\y)) \X \btheta_{PI,T}^{t}}{\pnorm{\X^\top \diag(\cT_s(\y)) \X \btheta_{PI,T}^{t}}{2}} \qquad -T \leq t \leq -1,
    \end{equation}
    which is initialized at $\btheta_{PI,T}^{-T} = \btheta^\ast$. Then we have
    \begin{align*}
        \frac{1}{\sqrt{d}}\pnorm{\sqrt{d}\hat{\btheta}^s - \btheta_T^0}{2} &\leq \underbrace{\frac{1}{\sqrt{d}}\pnorm{\sqrt{d}\hat{\btheta}^s - \btheta_{PI,T}^0}{2}}_{(I)} + \underbrace{\frac{1}{\sqrt{d}}\pnorm{\btheta_{PI,T}^0 - \btheta_T^0}{2}}_{(II)}.
    \end{align*}
    For the first term, let $\lambda_{1,n}$ and $\lambda_{2,n}$ denote the first and second eigenvalues (ordered by magnitude) of $\M_n = \X^\top\diag(\cT_s(\y))\X$, and $a_n = \langle\hat\btheta^{s},\btheta^\ast\rangle/\sqrt{d}$. By standard analysis of power iteration, we have the bound
    \begin{equation*}
        (I) \leq C\frac{(\lambda_{2,n}/\lambda_{1,n})^T\sqrt{1-a_{n}^2}}{\sqrt{a_{n}^2 + (\lambda_{2,n}/\lambda_{1,n})^{2T}(1-a_{n}^2)}}.
    \end{equation*}
    Thus, in view of the almost sure limit $a_n \rightarrow a$ given by Proposition \ref{prop: spec est overlap}, taking the high-dimensional limit yields
    \begin{equation*}
        \lim_{n,d\rightarrow\infty} (I) \leq C \Big(\frac{\lambda_{2,\infty}}{\lambda_{1,\infty}}\Big)^T.
    \end{equation*}
    For the second term, we have $\btheta_{PI,T}^0 = \sqrt{d}\frac{\M_n^T\btheta^\ast}{\pnorm{\M_n^T\btheta^\ast}{}}$, and $\btheta_{T}^0 = \frac{\M_n^T\btheta^\ast}{\lim_{n,d\rightarrow\infty}(\pnorm{\M_n^T\btheta^\ast}{}/\sqrt{d})}$, hence
    \begin{align*}
        (II) \leq \pnorm{\M_n^T}{\rm op}\frac{\pnorm{\btheta^\ast}{}}{\sqrt{d}}\Big(\frac{\sqrt{d}}{\pnorm{\M_n^T\btheta^\ast}{}} - \lim_{n,d\rightarrow\infty}\frac{\sqrt{d}}{\pnorm{\M_n^T\btheta^\ast}{}}\Big).
    \end{align*}
    Almost surely the term $\pnorm{\M_n^T}{\op} = \pnorm{\M_n}{\op}^T$ converges to $\lambda_{1,\infty}^T$ as $n,d \rightarrow \infty$, $\frac{\pnorm{\btheta^\ast}{}}{\sqrt{d}} = 1$, and 
    \begin{align*}
    \Big[\frac{\sqrt{d}}{\pnorm{\M_n^T\btheta^\ast}{}} - \lim_{n,d\rightarrow\infty}\frac{\sqrt{d}}{\pnorm{\M_n^T\btheta^\ast}{}}\Big] \xrightarrow{n,d\rightarrow\infty} 0,
    \end{align*}
    so $\lim_{n,d\rightarrow\infty} (II) = 0$, concluding the first claim. Next note that for any $t\geq 0$
    \begin{align*}
    \pnorm{\btheta^{t+1} - \btheta^{t+1}_T}{} &\leq |1-\gamma\lambda|\pnorm{\btheta^{t} - \btheta^{t}_T}{} + \gamma \pnorm{\X^\top \big(\ell(\X\btheta^t, \X\btheta^\ast,\z) - \ell(\X\btheta^t_T, \X\btheta^\ast,\z)\big)}{}\\
    &\leq \big(|1-\gamma\lambda| + \gamma\pnorm{\X}{\op}^2\pnorm{\partial_1\ell}{\infty}\big)\pnorm{\btheta^t - \btheta^t_T}{}.
    \end{align*}
    Almost surely for large enough $n,d$ we have $\pnorm{\X}{\op} \leq C(\delta)$, yielding the second claim. 
\end{proof}

Using the above high-dimensional convergence and the corresponding state evolutions of the spectral initialized DMFT (Theorem \ref{thm: discrete dmft asymp}) and artificial DMFT (Proposition \ref{prop:discreteDmftArt}), we can establish the following convergence of the law of the artificial DMFT to the law of the spectral initialized DMFT as $T\rightarrow\infty$.

\begin{lemma}\label{lm:dmft_T_conv}
Suppose Assumptions \ref{ass: model assumptions}, \ref{ass: spectral initialization}, \ref{ass: ell lipschitz} hold and $\psi_\delta'(\lambda_\delta^*) > 0$. Then for any fixed $h\geq 0$, as $T\rightarrow \infty$ it holds that
\begin{equation}
\begin{gathered}
     \sP(\theta^\ast,\{\theta_T^t\}_{t=0}^h) \overset{W_2}{\to} \sP(\theta^\ast,\{\theta^t\}_{t=0}^h), \quad \sP(z,w^\ast_T,\{\eta_T^t\}_{t=0}^h) \overset{W_2}{\to} \sP(z,w^\ast,\{\eta^t\}_{t=0}^h),\\
        \sP(u_T^{-1},\{u_T^t\}_{t=0}^h) \overset{W_2}{\to} \sP(-\lambda^\ast_\delta u^{\diamond},\{u^t\}_{t=0}^h), \quad \sP(w_T^\ast,\{w_T^t\}_{t=0}^h) \overset{W_2}{\to} \sP(w^\ast,\{w^t\}_{t=0}^h).
\end{gathered}
\end{equation}
\end{lemma}
\begin{proof}
    To see the convergence of $\sP(\theta^\ast, \{\theta^t_T\}_{t=0}^h)$, let $\phi$ be any $2$-pseudo Lipschitz test function and it suffices to show
    \begin{equation*}
        \lim_{T\rightarrow\infty} \E[\phi(\theta^\ast,\theta_T^0,...,\theta_T^h)] = \E[\phi(\theta^\ast,\theta^0,...,\theta^h)].
    \end{equation*}
    Almost surely for any finite $n,d$, we have
    \begin{align*}
        &\Big|\E[\phi(\theta^\ast,\theta_T^0,...,\theta_T^h)] - \E[\phi(\theta^\ast,\theta^0,...,\theta^h)]\Big| \leq \underbrace{\Big|\E[\phi(\theta^\ast,\theta_T^0,...,\theta_T^h)] - \frac{1}{d}\sum_{j=1}^d\phi(\theta^\ast_j,\theta_{T,j}^0,...,\theta_{T,j}^h)\Big|}_{(I_{n,d})}\\
        &+\underbrace{\Big|\frac{1}{d}\sum_{j=1}^d\phi(\theta^\ast_j,\theta_{T,j}^0,...,\theta_{T,j}^h)-\frac{1}{d}\sum_{j=1}^d\phi(\theta^\ast_j,\theta_{j}^0,...,\theta_{j}^h)\Big|}_{(II_{n,d})} + \underbrace{\Big|\frac{1}{d}\sum_{j=1}^d\phi(\theta^\ast_j,\theta_{j}^0,...,\theta_j^h)-\E[\phi(\theta^\ast,\theta^0,...,\theta^h)]\Big|}_{(III_{n,d})}.
    \end{align*}
    For any $T$, $\lim_{n,d\rightarrow\infty}(I_{n,d}) = \lim_{n,d\rightarrow\infty}(III_{n,d}) = 0$ by the DMFT characterization in Theorem \ref{thm: discrete dmft asymp} and Proposition \ref{prop:discreteDmftArt}. For $(II_{n,d})$, using the definition of 2-pseudo Lipschitz, we have
    \begin{align*}
        (II_{n,d}) \leq C\Big(1+\sum_{t=0}^h\frac{\pnorm{\btheta_T^t}{}}{\sqrt{d}}+\sum_{t=0}^h\frac{\pnorm{\btheta^t}{}}{\sqrt{d}}\Big)\sum_{t=0}^h\frac{\pnorm{\btheta_T^t-\btheta^t}{}}{\sqrt{d}}.
    \end{align*}
    By Lemma \ref{lem:hdFirstStage}, we have $\lim_{T\rightarrow\infty}\limsup_{n,d\rightarrow\infty}\frac{1}{\sqrt{d}}\sum_{t=0}^h\pnorm{\btheta_T^t-\btheta^t}{} = 0$ almost surely. On the other hand, by Theorem \ref{thm: discrete dmft asymp} and again Lemma \ref{lem:hdFirstStage}, it is easy to see that   $\limsup_{T\rightarrow\infty}\limsup_{n,d\rightarrow\infty}\frac{1}{\sqrt{d}}(1+\sum_{t=0}^h\pnorm{\btheta_T^t}{}+\sum_{t=0}^h\pnorm{\btheta^t}{}) < \infty$, therefore we conclude that $\lim_{T\rightarrow\infty}\limsup_{n,d\rightarrow\infty} (II_{n,d}) = 0$. The claim follows by taking $n,d\rightarrow\infty$ and then $T\rightarrow\infty$. A similar argument yields the convergence of $\sP(z,w^\ast_T,\{\eta^t_T\}_{t=0}^h)$.

    To establish the convergence of $\sP(u^{-1}_T, \{u^t_T\}_{t=0}^h)$, it suffices to show the convergence of the covariance kernel. For $r,s\geq 0$, we have $\E[u^r_Tu^s_T] = C_{T,\eta}(r,s) = \delta\E[\ell(\eta^r_T,w^\ast_T,z)\ell(\eta^s_T,w^\ast_T,z)]$ by (\ref{def:artificial_C_eta}) and $\E[u^ru^s] = C_\eta(r,s) = \delta\E[\ell(\eta^r,w^\ast,z)\ell(\eta^s,w^\ast,z)]$ by (\ref{def:eta_cov}). By Assumption \ref{ass: ell lipschitz}, the map $(\eta,w^\ast,z) \mapsto \ell(\eta,w^\ast,z)$ has linear growth, then $(\eta,\eta',w^\ast,z) \mapsto \ell(\eta,w^\ast,z)\ell(\eta',w^\ast,z)$ is continuous and has quadratic growth, so by dominated convergence we have $\E[u^r_Tu^s_T] \rightarrow \E[u^ru^s]$. On the other hand, for $r\geq 0$, we have $\E[u^{-1}_Tu^r_T] = \delta\E[\ell(\eta^r_T,w^\ast_T,z)\cT_s(\varphi(w^\ast_T,z))\eta^{-1}_T]$ by (\ref{def:artificial_C_eta}) and $\E[u^\diamond u^r] = -\frac{\delta}{\lambda^\ast_\delta}\E[\ell(\eta^s,w^\ast,z)\cT_s(\varphi(w^\ast,z))\eta^0]$ by (\ref{def:eta_cov}). Note that
    \begin{align*}
    &\Big|\E\big[\ell(\eta^r_T, w^\ast_T, z)\cT_s(\varphi(w^\ast_T,z))\eta^{-1}_T\big] - \E\big[\ell(\eta^r_T, w^\ast_T, z)\cT_s(\varphi(w^\ast,z))\eta^0\big]\Big|\\
    &\leq \E\Big[\big|\ell(\eta^r_T, w^\ast_T, z)\big|\big|\cT_s(\varphi(w^\ast_T,z))\big|\big|\eta^{-1}_T - \eta^0_T\big|\Big] + \Big|\E\big[\ell(\eta^r_T, w^\ast_T, z)\cT_s(\varphi(w^\ast_T,z))\eta^{0}_T\big] - \E\big[\ell(\eta^r, w^\ast, z)\cT_s(\varphi(w^\ast,z))\eta^0\big]\Big|.
    \end{align*}
    By Lemma \ref{lm:first_stage_reg}-(1), we have $\E(\eta^{-1}_T - \eta^0_T)^2 \leq 2\big(\E[(\eta^{-1}_T - \eta^0)^2] + \E[(\eta^{0}_T - \eta^0)^2]\big) \leq Ce^{-cT}$, and moreover $\limsup_{T\rightarrow\infty}\E\Big[\big(\ell(\eta^r_T, w^\ast_T, z)\cT_s(\varphi(w^\ast_T,z))\big)^2\Big] < \infty$, so the first term converges to 0 as $T\rightarrow\infty$. The second term also converges to 0 since $(\eta^0,\eta^r,w^\ast,z) \rightarrow \ell(\eta^r,w^\ast,z)\cT_s(\varphi(w^\ast,z))\eta^0$ is continuous and has quadratic growth.  This conclude that $\sP(u_T^{-1},\{u_T^t\}_{t=0}^h) \overset{W_2}{\to} (-\lambda^\ast_\delta u^{\diamond},\{u^t\}_{t=0}^h)$. A similar argument yields $\sP(w_T^\ast,\{w_T^t\}_{t=0}^h) \overset{W_2}{\to} \sP(w^\ast,\{w^t\}_{t=0}^h)$, completing the proof.
\end{proof}

\begin{proof}[Proof of Proposition \ref{prop:art_spec_dmft_coup}-(1)]
For each $T\in \Z_+$, we can construct optimal $W_2$ couplings $\pi_{T,\theta},\pi_{T,\eta},\pi_{T,u}$ and $\pi_{T,w}$ of 
\begin{gather*}
    (\theta^\ast,\{\theta_T^t\}_{t=-T}^h) \text{ and } (\theta^\ast,\{\theta^t\}_{t=0}^h), \quad (z,w^\ast_T,\{\eta_T^t\}_{t=-T}^h) \text{ and } (z,w^\ast,\{\eta^t\}_{t=0}^h)\\
    (u_T^{-1},\{u_T^t\}_{t=-T}^h) \text{ and } (-\lambda^\ast_\delta u^{\diamond},\{u^t\}_{t=0}^h), \quad (w_T^\ast,\{w_T^t\}_{t=-T}^h) \text{ and } (w^\ast,\{w^t\}_{t=0}^h)
\end{gather*}
respectively.

Given a finite collection $\T=(T_1,...,T_k) \in \Z_+^k$, we can form a joint coupling $\pi_{\T,\theta}$ using $\pi_{T_1,\theta},...,\pi_{T_k,\theta}$ by
\begin{enumerate}
    \item Sample $(\theta^\ast,\{\theta^t\}_{t=0}^h) \sim \sP(\theta^\ast,\{\theta^t\}_{t=0}^h)$.
    \item For $j=1,...,k$, sample $(\theta^\ast,\{\theta_{T_j}^t\}_{t=-{T_j}}^h)$ from $\pi_{T_j,\theta}|(\theta^\ast,\{\theta^t\}_{t=0}^h).$
\end{enumerate}
We can use a similar construction to define $\pi_{\T,\eta},\pi_{\T,u},$ and $\pi_{\T,w}$. Verifying that these joint couplings satisfy standard consistency conditions, we can apply Kolmogorov extension theorem to show the existence of couplings $\pi_\theta,\pi_\eta,\pi_u,$ and $\pi_w$ of 
\begin{gather*}
    \{(\theta^\ast,\{\theta_T^t\}_{t=-T}^h)\}_{T\in\Z_+} \text{ and } (\theta^\ast,\{\theta^t\}_{t=0}^h), \quad \{(z,w^\ast_T,\{\eta_T^t\}_{t=-T}^h) \}_{T\in\Z_+}\text{ and } (z,w^\ast,\{\eta^t\}_{t=0}^h)\\
    \{(u_T^{-1},\{u_T^t\}_{t=-T}^h)\}_{T\in\Z_+} \text{ and } (-\lambda^\ast_\delta u^{\diamond},\{u^t\}_{t=0}^h), \quad \{(w_T^\ast,\{w_T^t\}_{t=-T}^h)\}_{T\in\Z_+} \text{ and } (w^\ast,\{w^t\}_{t=0}^h),
\end{gather*}
which preserves the projections $\pi_{\T,\theta},\pi_{\T,\eta},\pi_{\T,u},$ and $\pi_{\T,w}$ for any finite collection $\T=(T_1,...,T_k) \in \Z_+^k$. In particular, this coupling satisfies part (1) of Proposition \ref{prop:art_spec_dmft_coup} using the convergence in law of Lemma \ref{lm:dmft_T_conv}. 

\end{proof}

\paragraph{Convergence of second stage responses}

Next we establish the convergences of the response functions in the second stage ($t> s \geq 0$). As the following lemma shows, a simple induction argument yields the convergences of the artificial response functions $\{r_{T,\theta}(t,s)\}_{t>s\geq 0},\{r_{T,\eta}(t,s)\}_{t>s\geq 0},$ and $\{r_{T,\eta}(t,\ast)\}_{t\geq0}$ in Proposition \ref{prop:art_spec_dmft_coup}-(2). Recall that $\lim_n X_n \overset{L^2}{=} X$ denotes that $X$ is the $L^2$-limit of $X_n$. 
\begin{proof}[Proof of Proposition \ref{prop:art_spec_dmft_coup}-(2)]
    We prove by induction. First note that
    \begin{align*}
        r_{T,\theta}(1,0) = \gamma = r_\theta(1,0)
    \end{align*}
    for all $T$, which shows the base case for $r_{\theta}$. This further implies $R_{T,\theta}(1,0) = R_\theta(1,0)$. Under our chosen coupling, by Lemma \ref{lm:dmft_T_conv},
    \begin{equation*}
    \begin{aligned}
    (\theta^\ast, \{\theta^t_T\}_{t=0}^h) \overset{L^2}{\to} (\theta^\ast, \{\theta^t\}_{t=0}^h), \quad (z,w^\ast_T, \{\eta^t_T\}_{t\geq 0}) \overset{L^2}{\to} (z,w^\ast,\{\eta^t\}_{t\geq 0}),\\
    (u_T^{-1},\{u_T^t\}_{t=0}^h) \overset{L^2}{\to} (-\lambda^\ast_\delta u^{\diamond},\{u^t\}_{t=0}^h), \quad (w^\ast, \{w^t\}_{t=0}^h) \overset{L^2}{\to} (w^\ast, \{w^t\}_{t=0}^h)
    \end{aligned}
    \end{equation*}
    for any fixed $h \geq 0$. Hence
    \begin{align*}
        r_{T,\eta}(1,0) = -\partial_1\ell(\eta^0_T,w^\ast_T,z)\partial_1\ell(\eta^1_T,w^\ast_T,z)R_{T,\theta}(1,0) \overset{L^2}{\rightarrow} -\partial_1\ell(\eta^0,w^\ast,z)\partial_1\ell(\eta^1,w^\ast,z)R_{\theta}(1,0) = r_\eta(1,0),
    \end{align*}
    and 
    \begin{align*}
        r_{T,\eta}(0,\ast) = \partial_2\ell(\eta^0_T,w^\ast_T,z) \overset{L^2}{\to} \partial_2\ell(\eta^0,w^\ast,z) =r_\eta(0,\ast),
    \end{align*}
    concluding the other two base cases. Taking expectation shows the base case for $R_\theta(1,0)$, $R_\eta(1,0)$ and $R_\eta(0,\ast)$.
    
     Now suppose we have shown the claims up to time $t$ and for all $0\leq s<t$. We now show the claims for time $t+1$. Let $t+1>s\geq0$. Taking $T\rightarrow\infty$ limit in (\ref{def:r_T_theta}), we get
    \begin{align*}
        \lim_{T\rightarrow\infty}r_{T,\theta}(t+1,s) &= \lim_{T\rightarrow\infty}\Bigg[\left(1-\gamma\lambda-\gamma\delta\E[\partial_1\ell(\eta_{T}^{t},w_{T}^*,z)]\right) r_{T,\theta}(t,s) - \gamma\sum_{r=s+1}^{t-1}R_{T,\eta}(t,r)r_{T,\theta}(r,s)\Bigg]\\
        &= \left(1-\gamma\lambda-\gamma\delta\E[\partial_1\ell(\eta^{t},w^*,z)]\right) r_{\theta}(t,s) - \gamma\sum_{r=s+1}^{t-1}R_{\eta}(t,r)r_{\theta}(r,s)\\
        &= r_\theta(t+1,s),
    \end{align*}
    which further implies that
    \begin{equation*}
        \lim_{T\rightarrow\infty}R_{T,\theta}(t+1,s) = R_\theta(t+1,s).
    \end{equation*}
    By Assumption \ref{ass: ell lipschitz}, $\partial_1 \ell$ is Lipschitz in its first two arguments, which implies $\partial_1\ell(\eta^t_T, w^\ast_T,z) \overset{L^2}{\to} \partial_1\ell(\eta^t, w^\ast,z)$, so taking $T\rightarrow\infty$ limit in the first display of (\ref{eq:r_T_eta_recursion}), we get
    \begin{align*}
        \lim_{T\rightarrow\infty}r_{T,\eta}(t+1,s) &\overset{L^2}{=} \lim_{T\rightarrow\infty}\Bigg[\partial_1\ell(\eta_{T}^{t+1},w^*_{T},z)\Big( -\sum_{r=s+1}^{t}R_{T,\theta}(t+1,r)r_{T,\eta}(r,s)-\partial_1\ell(\eta_{T}^{s},w^*_{T},z)R_{T,\theta}(t+1,s)\Big)\Bigg]\\
        &\overset{L^2}{=}\partial_1\ell(\eta^{t+1},w^*,z)\Big( -\sum_{r=s+1}^{t}R_{\theta}(t+1,r)r_{\eta}(r,s)-\partial_1\ell(\eta^{s},w^*,z)R_{\theta}(t+1,s)\Big)\\
        &= r_\eta(t+1,s).
    \end{align*}
    Similarly, Assumption \ref{ass: ell lipschitz} guarantees that $\partial_2 \ell(\eta^t_T, w^\ast_T,z) \overset{L^2}{\to} \partial_2\ell(\eta^t, w^\ast,z)$, so taking $T\rightarrow\infty$ limit in the last display of (\ref{eq:r_T_eta_recursion}), we have
    \begin{align*}
        \lim_{T\rightarrow\infty}r_{T,\eta}(t,\ast) &\overset{L^2}{=} \lim_{T\rightarrow\infty}\Big[\partial_1\ell(\eta_{T}^{t+1},w^*_{T},z)\Big(-\sum_{r=0}^{t}r_{T,\eta}(r,*)R_{T,\theta}(t+1,r)\Big) + \partial_2\ell(\eta_{T}^{t+1},w^*_{T},z)\Big]\\
        &\overset{L^2}{=} \partial_1\ell(\eta^{t+1},w^*,z)\Big(-\sum_{r=0}^{t}r_{\eta}(r,*)R_{\theta}(t+1,r)\Big) + \partial_2\ell(\eta^{t+1},w^*,z)\\
        &=r_\eta(t+1,\ast).
    \end{align*}
    Taking expectation shows that $\lim_{T\rightarrow\infty}R_{T,\eta}(t+1,s) = R_\eta(t+1,s)$ and $\lim_{T\rightarrow\infty}R_{T,\eta}(t+1,\ast) = R_\eta(t+1,\ast)$, concluding the induction.
\end{proof}

\paragraph{Convergence of first stage responses}
We have yet to analyze the responses $\{R_\theta(t,\diamond)\}_{t\geq 0},\{R_\eta(t,\diamond)\}_{t\geq 0}$ and $\{R_\eta(t,\diamond\diamond)\}_{t\geq0}$ which arise from the spectral initialization, which requires a deeper analysis of the first stage of the artificial DMFT. We first state two lemmas concerning the convergence of the first stage iterates and exponential decay of the first stage responses. Recall that $\lambda_{1,\infty} = \delta \psi_\delta(\lambda_\delta^\ast),\lambda_{2,\infty} := \delta\psi_\delta(\bar\lambda_\delta)$ denote the almost sure limit of the first and second eigenvalues of $\M_n = \X^\top\diag(\cT_s(\y))\X$ as shown by Proposition \ref{prop: spec est overlap}.

\begin{lemma}\label{lm:first_stage_reg}
    Suppose that assumptions \ref{ass: model assumptions},\ref{ass: spectral initialization} and \ref{ass: ell lipschitz} hold and that $\psi_\delta'(\lambda_\delta^*) > 0$. Then for any $t<0$, there exists constant $C>0$ such that 
    \begin{gather*}
            \pnorm{\theta_T^{t+1}-\theta_T^0}{L^2} \leq C\Big(\frac{\lambda_{2,\infty}}{\lambda_{1,\infty}}\Big)^{T+t+1}, \quad  \pnorm{\eta_T^{t}-\eta_T^0}{L^2} \leq C\Big(\frac{\lambda_{2,\infty}}{\lambda_{1,\infty}}\Big)^{T+t}
        \end{gather*}
    Moreover, The normalizing factor $\beta_T^t$ in (\ref{def:beta_def}) converges exponentially fast to the almost sure limit of the leading eigenvalue of $\M_n$. That is
        \begin{equation}
        \label{eq:beta_exp_conv}
            |\beta_T^t - \lambda_{1,\infty}| \leq C \Big(\frac{\lambda_{2,\infty}}{\lambda_{1,\infty}}\Big)^{T+t}.
        \end{equation}
\end{lemma}
\begin{proof}
Using the DMFT characterization in Proposition \ref{prop:discreteDmftArt}, it holds almost surely that
\begin{align*}
    \pnorm{\theta_T^{t+1}-\theta_T^0}{L^2}^2 &= \lim_{n,d\rightarrow\infty}\frac{1}{d}\pnorm{\btheta_T^{t+1}-\btheta_T^0}{}^2\\
    &\leq 2\limsup_{n,d\rightarrow\infty}\frac{1}{d}\pnorm{\btheta_T^{t+1}-\sqrt{d}\hat{\btheta}^s}{}^2 + 2\limsup_{n,d\rightarrow\infty}\frac{1}{d}\pnorm{\sqrt{d}\hat{\btheta}^s-\btheta_T^0}{}^2\\
    &\stackrel{(i)}{=} 2\limsup_{n,d\rightarrow\infty}\frac{1}{d}\pnorm{\btheta_{T+t+1}^{0}-\sqrt{d}\hat{\btheta}^s}{}^2 + 2\limsup_{n,d\rightarrow\infty}\frac{1}{d}\pnorm{\sqrt{d}\hat{\btheta}^s-\btheta_T^0}{}^2\\
    &\stackrel{(ii)}{\leq} C\Big[\Big(\frac{\lambda_{2,\infty}}{\lambda_{1,\infty}}\Big)^{2(T+t+1)} + \Big(\frac{\lambda_{2,\infty}}{\lambda_{1,\infty}}\Big)^{2T}\Big] \leq C\Big(\frac{\lambda_{2,\infty}}{\lambda_{1,\infty}}\Big)^{2(T+t+1)}.
\end{align*}
Here in $(i)$ we apply the fact that $\btheta^{t+1}_T = \btheta^0_{T+t+1}$, where $\{\btheta^s_{T+t+1}\}_{s\geq T+t+1}$ denotes the same power iteration sequence but starting from $-(T+t+1)$ for such fixed $t$, and $(ii)$ follows from Lemma \ref{lem:hdFirstStage}. A similar argument shows the claim $\pnorm{\eta_T^{t}-\eta_T^0}{L^2} \leq C (\lambda_{2,\infty}/\lambda_{1,\infty})^{T+t}$.

For the convergence of $\beta_T^t$, we have by definition that $\lambda_{1,\infty} = \lim_{n,d\rightarrow\infty} \frac{1}{\sqrt{d}}\pnorm{\M_n\sqrt{d}\hat{\btheta}^s}{}$ and      $\beta_T^{t} = \lim_{n,d\rightarrow\infty} \frac{1}{\sqrt{d}}\pnorm{\M_n \btheta_T^{t-1}}{}$. Hence we have as desired that
 \begin{align*}
|\beta_T^t - \lambda_{1,\infty}| \leq \limsup_{n,d\rightarrow\infty} \pnorm{\M_n}{\op} \frac{1}{\sqrt{d}} \pnorm{\sqrt{\d}\hat{\btheta}^s-\btheta_T^t}{} \leq C\lambda_{1,\infty}\Big(\frac{\lambda_{2,\infty}}{\lambda_{1,\infty}}\Big)^{T+t}.
\end{align*}
\end{proof}

\begin{lemma}\label{lm:first_stage_res_exp_decay}
    Assume that $\psi_\delta'(\lambda_\delta^*) > 0$ and that Assumptions \ref{ass: model assumptions}, \ref{ass: spectral initialization} and \ref{ass: ell lipschitz} hold. For any $t\in \Z_+$ and $M \in (\lambda_{2,\infty},\lambda_{1,\infty})$, there exist $C_1 = C_1(M,t) > 0$ and $c = c(M) > 0$ (but independent of $t$) such that whenever $T > c$, we have
    \begin{align*}
        |R_{T,\theta}(t,s)| \leq C_1\Big(\frac{\lambda_{2,\infty}}{M}\Big)^{-s}, \quad 
        |R_{T,\eta}(t,s)| \leq C_1\Big(\frac{\lambda_{2,\infty}}{M}\Big)^{-s}
    \end{align*}
    for any $s<0$. 
\end{lemma}
\begin{proof}
    Fix any $M \in (\lambda_{2,\infty},\lambda_{1,\infty})$. First note that by (\ref{eq:beta_exp_conv}), there exists a constant $c_0 = c_0(M)>0$ such that for any $k\in[-T,0]$ such that $T + k > c_0$, we have 
    \begin{equation}
        \beta_T^k \geq M >\lambda_{2,\infty}.
    \end{equation}
    We also use that there exists a constant $M'>0$ such that 
    \begin{equation}
       \beta_T^k>M'>0,
    \end{equation}
    for all $T,k$. This is since we know that $\beta_T^k >0$ for any $T,k$ and that $\beta_T^k\rightarrow \lambda_{1,\infty} > 0,$ uniformly with respect to $T,k$, so it must be bounded away from $0$. 
    Let $T>c := 3c_0.$

     When $t >0$, using the fact that $\bUpsilon^s_T = \frac{\M_n}{\beta_T^{s+1}}$ (with $\bUpsilon^s_T$ defined in (\ref{def:Upsilon})) and with $\{\lambda_{j,n}\}_{j=1}^d \in \R_+^{d}$ denoting the $d$ eigenvalues of $\M_n$, we have
     \begin{align*}
     &\Big|\Tr\Big(-\frac{1}{\beta_T^{s+1}}\bOmega_T^{t-1}\cdots\bOmega_T^0\bUpsilon_T^{-1}\cdots\bUpsilon_T^{s+1}\Big)\Big| = \frac{1}{\beta^0_T\ldots \beta^{s+1}_T}\Big|\Tr\Big(\bOmega^{t-1}_T\ldots \bOmega^{0}_T \M_n^{-s-1}\Big)\Big|\\
     &\leq \frac{1}{\beta^0_T\ldots \beta^{s+1}_T} \pnorm{\bOmega^{t-1}_T\ldots \bOmega^{0}_T}{\op}\pnorm{\M_n^{-s-1}}{\ast} = \frac{1}{\beta^0_T\ldots \beta^{s+1}_T} \pnorm{\bOmega^{t-1}_T\ldots \bOmega^{0}_T}{\op} \sum_{j=1}^d \lambda_{j,n}^{-s-1},
     \end{align*}
     where $\pnorm{\cdot}{\ast}$ denotes the matrix nuclear norm. By Lemma \ref{lm: response matrix recursion}, we have
     \begin{align*}
     R_{T,\theta}(t,s) = \lim_{n,d\rightarrow\infty} \frac{1}{d}\Tr\Big(-\frac{1}{\beta_T^{s+1}}\bOmega_T^{t-1}\cdots\bOmega_T^0\bUpsilon_T^{-1}\cdots\bUpsilon_T^{s+1}\Big)
     \end{align*}
     almost surely, so using $\pnorm{\bOmega^t_T}{\op}\leq C$ almost surely for large enough $n,d$ via a similar argument as in Lemma \ref{lm:Omega_exp_decay} below, we have
    \begin{align}
    \begin{split}
    \label{eq:R_T_theta_exp_comp}
            |R_{T,\theta}(t,s)| &= \Big|\lim_{n,d\rightarrow\infty} \frac{1}{d}\Tr\Big(-\frac{1}{\beta_T^{s+1}}\bOmega_T^{t-1}\cdots\bOmega_T^0\bUpsilon_T^{-1}\cdots\bUpsilon_T^{s+1}\Big)\Big|\\
            &\leq\frac{C}{\beta_T^{0}\cdots\beta_T^{s+1}}\limsup_{n,d\rightarrow\infty}\frac{1}{d}\sum_{i=1}^d \lambda_{i,n}^{-s-1}\\
            &\leq \frac{C}{\beta_T^{0}} \Big[\limsup_{n,d\rightarrow\infty}\frac{1}{d}\Big(\frac{\lambda_1^{-s-1}}{\beta_T^{-1}\cdots\beta_T^{s+1}}\Big) + \limsup_{n,d\rightarrow\infty}\Big(\frac{\lambda_2^{-s-1}}{\beta_T^{-1}\cdots\beta_T^{s+1}}\Big)\Big]\\
            &\leq \frac{C}{M'}\Big[\underbrace{\limsup_{n,d\rightarrow\infty}\frac{1}{d}\Big(\frac{\lambda_1^{-s-1}}{\beta_T^{-1}\cdots\beta_T^{s+1}}\Big)}_{(I_s)} + \underbrace{\limsup_{n,d\rightarrow\infty}\Big(\frac{\lambda_{2}^{-s-1}}{\beta_T^{-1}\cdots\beta_T^{s+1}}\Big)}_{(II_s)}\Big].
        \end{split}
        \end{align}
        Note that for any $s<0$, we have $(I_s) = 0.$ To analyze the term $(II_s)$, first consider the case when $s>-T+c.$ In this case since for any $k\in[s+1,-1]$, we have $T+k \geq T+s+1 > c$ which implies $\beta_T^k \geq M$, hence
        \begin{equation*}
            (II_s) \leq \Big(\frac{\lambda_{2,\infty}}{M}\Big)^{-s-1}.
        \end{equation*}
        Now consider the case when $s\in[-T,-T+c].$ For any $k\in[s+c+1,-1]$, we have $T+k \geq T+s+c+1 > c$, so $\beta_T^k \geq M$ whenever $k\in[s+c+1,-1]$. The remaining $c$ terms of $\beta_T^k$ with $k\in[s+1,s+c]$ satisfy $\beta_T^k \geq M'$, so
        \begin{equation*}
            (II_s) \leq \Big(\frac{\lambda_{2,\infty}}{M}\Big)^{-s-c-1}\Big(\frac{\lambda_{2,\infty}}{M'}\Big)^{c}.
        \end{equation*}
        Hence the claim holds by taking $C_1 = \frac{C}{M'} \Big(\frac{\lambda_{2,\infty}}{M}\Big)^{-c-1}\Big(\frac{\lambda_{2,\infty}}{M'}\Big)^{c}$.
The proof for $R_{T,\eta}$ is similar.
\end{proof}

With these lemmas, we are ready to prove part (3) of Proposition \ref{prop:art_spec_dmft_coup}.
\begin{proof}[Proof of Proposition \ref{prop:art_spec_dmft_coup}-(3)]

    \noindent\textbf{Analysis of $R_\theta(t,\diamond)$:} Let $t>0$. By (\ref{def:artificialEta}), we know that
    \begin{equation*}
        \eta_T^t = -\sum_{s=0}^{t-1}\ell(\eta_{T}^{s},w_{T}^{*},z)R_{T,\theta}(t,s) -\cT_s(y_T)\sum_{s=-T}^{-1}\eta_{T}^s R_{T,\theta}(t,s)+ w_{T}^{t}.
    \end{equation*}
    Taking $T\rightarrow\infty$ and using properties of the coupling given in Proposition \ref{prop:art_spec_dmft_coup}-(1)(2), we get
    \begin{equation*}
        \lim_{T\rightarrow\infty}\eta_T^t \overset{L^2}{=} -\sum_{s=0}^{t-1}\ell(\eta^{s},w^{*},z)R_{\theta}(t,s) -\cT_s(y)\lim_{T\rightarrow\infty}\sum_{s=-T}^{-1}\eta_{T}^s R_{T,\theta}(t,s)+ w^{t}.
    \end{equation*}
    By (\ref{def:dmft_eta}), we know that
    \begin{equation*}
        \eta^t = -\sum_{s=0}^{t-1}\ell(\eta^{s},w^*,z)R_\theta(t,s) + \cT(y)w^0R_\theta(t,\diamond) + w^t.
    \end{equation*}
    Since by Proposition \ref{prop:art_spec_dmft_coup}-(1) $ \lim_{T\rightarrow\infty}\eta_T^t \overset{L^2}{=}\eta^t$, we see that 
    \begin{equation}
    \label{eq:eta_equiv_art}
         \cT_s(y)\lim_{T\rightarrow\infty}\sum_{s=-T}^{-1}\eta_{T}^s R_{T,\theta}(t,s) \overset{L^2}{=} -\cT(y)w^0R_\theta(t,\diamond).
    \end{equation}
    We can write the left side as
    \begin{align*}
        \cT_s(y)\lim_{T\rightarrow\infty}\sum_{s=-T}^{-1}\eta_{T}^s R_{T,\theta}(t,s) &= \cT_s(y)\lim_{T\rightarrow\infty}\sum_{s=-T}^{-1}R_{T,\theta}(t,s)(\eta_T^s-\eta_T^0) + \cT_s(y)\lim_{T\rightarrow\infty}\sum_{s=-T}^{-1} R_{T,\theta}(t,s)\eta_T^0\\
        &\overset{L^2}{=}\cT_s(y)\lim_{T\rightarrow\infty}\sum_{s=-T}^{-1}R_{T,\theta}(t,s)(\eta_T^s-\eta_T^0) + \eta^0\cT_s(y)\lim_{T\rightarrow\infty}\sum_{s=-T}^{-1}R_{T,\theta}(t,s).
    \end{align*}
    Next, we claim that $\cT_s(y)\lim_{T\rightarrow\infty}\sum_{s=-T}^{-1}R_{T,\theta}(t,s)(\eta_T^s-\eta_T^0) \overset{L^2}{=} 0$, since 
    \begin{align*}
        &\biggpnorm{\sum_{s=-T}^{-1}R_{T,\theta}(t,s)(\eta_T^s-\eta_T^0)}{L^2} \leq \sum_{s=-T}^{-1}\pnorm{R_{T,\theta}(t,s)(\eta_T^s-\eta_T^0)}{L^2}\\
        &\leq C\sum_{s=-T}^{-1}|R_{T,\theta}(t,s)|\pnorm{\eta_T^s-\eta_T^0}{L^2} \overset{(*)}{\leq} C\sum_{s=-T}^{-1}\Big(\frac{\lambda_{2,\infty}}{M}\Big)^{-s}\Big(\frac{\lambda_{2,\infty}}{M}\Big)^{T+s} = T\Big(\frac{\lambda_{2,\infty}}{M}\Big)^{T} \rightarrow 0
    \end{align*}
    as $T\rightarrow\infty$, where we apply Lemmas \ref{lm:first_stage_reg} and \ref{lm:first_stage_res_exp_decay} in $(\ast)$, and $M \in (\lambda_{2,\infty}, \lambda_{1,\infty})$ is a constant given in Lemma \ref{lm:first_stage_res_exp_decay}. Returning to (\ref{eq:eta_equiv_art}), we now have
    \begin{equation*}
        \eta^0\cT_s(y)\lim_{T\rightarrow\infty}\sum_{s=-T}^{-1}R_{T,\theta}(t,s) \overset{L^2}{=} -\cT(y)w^0R_\theta(t,\diamond).
    \end{equation*}
    Using that $\eta^0 =\cT(y)w^0R_\theta(0,\diamond) + w^0 = (1+\cT(y))w^0$ and some algebra, we conclude that 
    \begin{equation*}
        -\lambda_\delta^\ast\lim_{T\rightarrow\infty}\sum_{s=-T}^{-1}R_{T,\theta}(t,s) = R_\theta(t,\diamond).
    \end{equation*}

    \noindent\textbf{Analysis of $R_\eta(t,\diamond),R_\eta(t,\diamond\diamond)$:} Let $t>0$. By (\ref{def:r_T_theta}) we know that
    \begin{align*}
        \theta_T^{t+1} &= \theta_{T}^t +\gamma\bigg(- \lambda\theta_{T}^t -\delta \E\left[\partial_1\ell(\eta_{T}^{t},w_{T}^{*},z)\right]\theta_{T}^{t}
        \\
        &\quad-\sum_{s=0}^{t-1}R_{T,\eta}(t,s)\theta_{T}^{s} -\sum_{s=-T}^{-1}R_{T,\eta}(t,s)\theta_{T}^{s}-(R_{T,\eta}(t,\ast) + R_{T,\eta}(t,\sharp))\theta^*  +  u_{T}^{t}\bigg).
    \end{align*}
    Taking $T\rightarrow\infty$ and using properties of the coupling given in part (1,2) of Proposition \ref{prop:art_spec_dmft_coup}, we get
    \begin{align*}
           \lim_{T\rightarrow\infty}\theta_T^{t+1} &\overset{L^2}{=} \theta^t +\gamma\bigg(- \lambda\theta^t -\delta \E\left[\partial_1\ell(\eta^{t},w^{*},z)\right]\theta^{t}
        \\
        &\quad-\sum_{s=0}^{t-1}R_{\eta}(t,s)\theta^{s} -\lim_{T\rightarrow\infty}\sum_{s=-T}^{-1}R_{T,\eta}(t,s)\theta_{T}^{s}-(R_{\eta}(t,\ast) + \lim_{T\rightarrow\infty}R_{T,\eta}(t,\sharp))\theta^*  +  u^{t}\bigg).
    \end{align*}
    By (\ref{def:dmft_theta}), we have
    \begin{equation*}
        \theta^{t+1} = \theta^t +\gamma\Big(- \lambda\theta^t -\delta \E\big[\partial_1\ell(\eta^t,w^*,z)\big]\theta^{t}-\sum_{s=0}^{t-1}R_\eta(t,s)\theta^{s} - R_\eta(t,\diamond)\theta^0 - \big(R_\eta(t,\ast) + R_\eta(t,\diamond\diamond)\big)\theta^* + u^{t}\Big).
    \end{equation*} 
    Since $\lim_{T\rightarrow\infty}\theta_T^{t+1}\overset{L^2}{=} \theta^{t+1}$ by Proposition \ref{prop:art_spec_dmft_coup}-(1), we have
    \begin{equation}
    \label{eq:theta_comp}
        \lim_{T\rightarrow\infty}\sum_{s=-T}^{-1}R_{T,\eta}(t,s)\theta_{T}^{s} + \lim_{T\rightarrow\infty}R_{T,\eta}(t,\sharp)\theta^* \overset{L^2}{=} R_\eta(t,\diamond)\theta^0 +  R_\eta(t,\diamond\diamond)\theta^*.
    \end{equation}
    We can write
    \begin{align*}
        \lim_{T\rightarrow\infty}\sum_{s=-T}^{-1}R_{T,\eta}(t,s)\theta_{T}^{s} &= \lim_{T\rightarrow\infty}\sum_{s=-T}^{-1}R_{T,\eta}(t,s)(\theta_{T}^{s}-\theta_T^0) + \lim_{T\rightarrow\infty}\sum_{s=-T}^{-1}R_{T,\eta}(t,s)\theta_{T}^{0}\\
        &=\lim_{T\rightarrow\infty}\sum_{s=-T}^{-1}R_{T,\eta}(t,s)(\theta_{T}^{s}-\theta_T^0) + \theta^0 \lim_{T\rightarrow\infty}\sum_{s=-T}^{-1}R_{T,\eta}(t,s).
    \end{align*}
    Similar to the analysis of $R_\theta(t,\diamond)$ above, we have $\sum_{s=-T}^{-1}R_{T,\eta}(t,s)(\theta_{T}^{s}-\theta_T^0) \overset{L^2}{=} 0 $, since
    \begin{align*}
        &\Bigpnorm{\sum_{s=-T}^{-1}R_{T,\eta}(t,s)(\theta_{T}^{s}-\theta_T^0)}{L^2} \leq \sum_{s=-T}^{-1}|R_{T,\eta}(t,s)|\pnorm{\theta_{T}^{s}-\theta_T^0}{L^2}\\
        &\leq C\sum_{s=-T}^{-1} \Big(\frac{\lambda_{2,\infty}}{M}\Big)^{-s}\Big(\frac{\lambda_{2,\infty}}{M}\Big)^{T+s} = T\Big(\frac{\lambda_{2,\infty}}{M}\Big)^{T} \rightarrow 0
    \end{align*}
    as $T\rightarrow\infty$, where we again use Lemmas \ref{lm:first_stage_reg} and \ref{lm:first_stage_res_exp_decay}. This implies that 
    \begin{align*}
    \lim_{T\rightarrow\infty}\sum_{s=-T}^{-1}R_{T,\eta}(t,s)\theta_{T}^{s} 
        =\theta^0 \lim_{T\rightarrow\infty}\sum_{s=-T}^{-1}R_{T,\eta}(t,s).
    \end{align*}
    Plugging this into (\ref{eq:theta_comp}), we get
    \begin{equation*}
        \theta^0 \lim_{T\rightarrow\infty}\sum_{s=-T}^{-1}R_{T,\eta}(t,s) + \lim_{T\rightarrow\infty}R_{T,\eta}(t,\sharp)\theta^* \overset{L^2}{=} R_\eta(t,\diamond)\theta^0 +  R_\eta(t,\diamond\diamond)\theta^*.
    \end{equation*}
    Using that $\theta^0 = a\theta^\ast+u^\diamond$, we get
    \begin{equation*}
         \lim_{T\rightarrow\infty}\sum_{s=-T}^{-1}R_{T,\eta}(t,s) u^\diamond+\Big( a\lim_{T\rightarrow\infty}\sum_{s=-T}^{-1}R_{T,\eta}(t,s)+ \lim_{T\rightarrow\infty}R_{T,\eta}(t,\sharp)\Big)\theta^* \overset{L^2}{=} R_\eta(t,\diamond)u^\diamond +(aR_\eta(t,\diamond)+  R_\eta(t,\diamond\diamond))\theta^*.
    \end{equation*}
    Since $u^\diamond$ is independent of $\theta^\ast$, we get
    \begin{equation*}
        \lim_{T\rightarrow\infty}\sum_{s=-T}^{-1}R_{T,\eta}(t,s) = R_\eta(t,\diamond).
    \end{equation*} 
    This further implies
    \begin{equation*}
        \lim_{T\rightarrow\infty}R_{T,\eta}(t,\sharp) = R_\eta(t,\diamond\diamond),
    \end{equation*}
    as desired.
\end{proof}

\subsection{Proof of long-time characterization}\label{sec:proof_long_time}

For any integer $t \in \Z_+$, define the matrices
    \begin{equation}\label{def:D_Omega_lim}
        \D^t := \diag\big(\partial_1\ell(\boldeta^t,\boldeta^*,\z)\big) \in \R^{n\times n}, \qquad \bOmega^t := (1-\gamma\lambda)\I_d -\gamma \X^\top\D^t\X \in \R^{d\times d}. 
    \end{equation}
    Furthermore, define the matrix
    \begin{align}\label{def:Upsilon_lim}
        \bUpsilon := \frac{1}{\lambda_{1,\infty}}\X^\top\diag(\cT_s(\y))\X.
    \end{align}
    These could be understood as the $T\rightarrow\infty$ limits of $\D^t_T, \bOmega^t_T, \bUpsilon^t_T$ defined in (\ref{def:D_Omega}) and (\ref{def:Upsilon}).

    We first give a series of lemmas which will be used multiple times in the proof of Theorem \ref{thm:dmft_exp_tti_condition}. The first one states some facts about the high dimensional dynamic when Assumption \ref{ass:exp_tti_cond} is satisfied.

    \begin{lemma}\label{lm:long-time-properties}
        Suppose Assumptions \ref{ass: model assumptions},\ref{ass: spectral initialization}, \ref{ass: ell lipschitz} and \ref{ass:exp_tti_cond} are satisfied. Let $\cE_n$ be the event where Assumption \ref{ass:exp_tti_cond} holds and $\pnorm{\X}{\rm op} < 2(1+\sqrt{\delta})$. Suppose the step size $\gamma < 2/L$ with $L=\lambda + 4(1+\sqrt{\delta})^2\pnorm{\partial_1\ell}{\infty}$. Then on $\cE_n$, there exists some $r\in(0,1)$  independent of $n,d$ such that
            \begin{equation}\label{eq:grad_descent_cauchy}
            \pnorm{\btheta^{t+1}-\btheta^t}{}^2 \leq r^{t-t_0}\pnorm{\btheta^{t_0+1}-\btheta^{t_0}}{}^2
        \end{equation}
        for all $t\geq t_0$. Moreover, $\cE_n$ holds almost surely for large $n,d$.
    \end{lemma}

    \begin{proof}
    On the event $\cE_n$, the contraction (\ref{eq:grad_descent_cauchy}) follows from standard analysis of gradient descent, which we reproduce here for the reader's convenience. Let $F(\x) = \x - \gamma\nabla_{\x}\cL_n(\x)$, so that $\btheta^{t+1} = F(\btheta^t)$. For any $\u,\v\in\R^d$, Taylor expansion yields
    \begin{align*}
    F(\u) - F(\v) = \Big(\I_d - \gamma\int_0^1 \nabla_{\btheta}^2 \cL_n(\v + \tau(\u - \v))\d\tau\Big)(\u - \v).  
    \end{align*}
    When $\u,\v \in \cR_n$, we have $\v + \tau(\u - \v) \in \cR_n$ for any $\tau\in[0,1]$ since $\cR_n$ is convex, hence $\textbf{H} := \int_0^1 \nabla_{\btheta}^2 F(\v + \tau(\u - \v))\d\tau$ satisfies $c_0 \I_d \preceq \textbf{H} \preceq L\I_d$, and
    \begin{align*}
    \pnorm{F(\u) - F(\v)}{} \leq \biggpnorm{\I_d - \gamma\int_0^1 \nabla_{\btheta}^2 F(\v + \tau(\u - \v))\d\tau}{\op}\pnorm{\u - \v}{} \leq \underbrace{\max(|1-\gamma c_0|, |1-\gamma L|)}_{:= \sqrt{r}} \pnorm{\u - \v}{}.
    \end{align*}
    When $\gamma < 2/L$, we have $r\in (0,1)$. Applying this fact to successive iterates of the gradient descent dynamic, we have $\pnorm{\btheta^{t+1} - \btheta^t}{} \leq \sqrt{r}\pnorm{\btheta^{t} - \btheta^{t-1}}{}$ for $t \geq t_0+1$, which yields (\ref{eq:grad_descent_cauchy}) after iteration. Standard tail estimate of $\pnorm{\X}{\op}$ yields that $\cE_n$ holds with probability $1-Cn^{-c}$ for large enough $n$, so the event holds almost surely for large enough $n,d$ by Borel-Cantelli lemma. 
    \end{proof}

    \begin{lemma}\label{lm:Omega_exp_decay}
    Suppose Assumptions \ref{ass: model assumptions},\ref{ass: spectral initialization}, \ref{ass: ell lipschitz} and \ref{ass:exp_tti_cond} hold and the step size $\gamma < 2/L$. Then there exists some $C,c>0$ independent of $n,d$ such that almost surely for large enough $n,d$, $\pnorm{\bOmega^t}{\op} \leq C$ for all $t\in\Z_+$ and $\pnorm{\bOmega^t}{\op} \leq e^{-c}$ for $t\geq t_0$.
    \end{lemma}
    \begin{proof}
    Since $\bOmega^t = \I_d - \gamma(\lambda \I_d + \X^\top \D^t \X)$, we have for all $t\in \Z_+$ that $\pnorm{\bOmega^t}{\op} \leq (1 - \gamma\lambda) + \pnorm{\X}{\op}^2\pnorm{\partial_1\ell}{\infty} \leq C$ almost surely for large enough $n,d$. Next let $\cE_n$ be the event described in Lemma \ref{lm:long-time-properties}. We will show that on $\cE_n$, $0\preceq \bOmega^t \preceq e^{-c\I_d}$ for all $t \geq t_0$, which implies the desired claim since $\cE_n$ holds almost surely for large enough $n,d$. It is easy to see $\bOmega^t \succeq 0$ for all $t\in \Z_+$ since $\pnorm{\lambda \I_d + \X^\top \D^t\X}{\op} \leq \lambda + 4(1+\sqrt{\delta})^2\pnorm{\partial_1\ell}{\infty}$. On the other hand, note that $\lambda \I_d + \X^\top \D^t \X = \nabla_{\btheta}^2 \cL_n(\btheta^t)$, so using Assumption \ref{ass:exp_tti_cond}, for $t\geq t_0$ we have
    \begin{align*}
    \bOmega^t = \I_d - \gamma\nabla_{\btheta}^2 \cL_n(\btheta^t) \preceq \exp(-\gamma \nabla_{\btheta}^2 \cL_n(\btheta^t)) \preceq \exp(-c_0\gamma\I_d),
    \end{align*}
    as desired.
    \end{proof}

\begin{lemma}
     \label{lm: response_mat_relation_2}
     Suppose Assumptions \ref{ass: model assumptions},\ref{ass: spectral initialization}, \ref{ass: ell lipschitz} and \ref{ass:exp_tti_cond} hold, $\psi_\delta'(\lambda_\delta^*) > 0$, and $\gamma < 2/L$. For any $t,s,t',s' \in \Z_+$ such that $t>s, t' > s'$, there exists a constant $C>0$ such that almost surely 
     \begin{equation}
     \begin{aligned}
         &\pnorm{r_\eta(t,s)-r_\eta(t',s')}{L^2}^2 \leq C\limsup_{n,d\rightarrow\infty}\frac{1}{n}\pnorm{\bOmega^{t-1}\cdots\bOmega^{s+1}-\bOmega^{t'-1}\cdots\bOmega^{s'+1}}{F}^2\\
         &\quad+ C^{t-s} \limsup_{n,d\rightarrow\infty}\bigg\{\frac{1}{n}\pnorm{\partial_1\ell(\bbeta^{t'})-\partial_1\ell(\bbeta^{t})}{}^2 + \frac{1}{n}\pnorm{\partial_1\ell(\bbeta^{s'})-\partial_1\ell(\bbeta^{s})}{}^2\bigg\},\label{eq:r_eta_L^2_bound}\\
          &\pnorm{r_\eta(t,\ast)-r_\eta(t',\ast)}{L^2}^2 \leq C\limsup_{n,d\rightarrow\infty}\frac{1}{n}\biggpnorm{\sum_{s=0}^{t-1}\bOmega^{t-1}\cdots\bOmega^{s+1}\X^\top \diag\big(\partial_2\ell(\bbeta^s)\big) - \sum_{s=0}^{t'-1}\bOmega^{t'-1}\cdots\bOmega^{s+1}\X^\top \diag\big(\partial_2\ell(\bbeta^s)\big)}{F}^2 \\
          &\quad + C\limsup_{n,d\rightarrow\infty} \bigg\{\frac{1}{n}\pnorm{\partial_2\ell(\bbeta^t) - \partial_2\ell(\bbeta^{t'})}{}^2  +  \frac{1}{n}\pnorm{\partial_1\ell(\bbeta^{t})-\partial_1\ell(\bbeta^{t'})}{}^2\bigg\},\\
          &\pnorm{r_\eta(t,s)}{L^2}^2 \leq C\limsup_{n,d\rightarrow\infty}\frac{1}{n}\pnorm{\bOmega^{t-1}\cdots\bOmega^{s+1}}{F}^2.
        \end{aligned}
    \end{equation}
    Here $\partial_1 \ell(\bbeta^t)$ (resp. $\partial_2 \ell(\bbeta^t)$) is shorthand for $\partial_1 \ell(\bbeta^t,\bbeta^\ast,\z)$ (resp. $\partial_2 \ell(\bbeta^t,\bbeta^\ast,\z)$).
\end{lemma}
\begin{proof}
    We prove the first claim. By Proposition \ref{prop: high dim response} (with slight extension to two sets of $(t,s)$), almost surely it holds that
    \begin{equation*}
        \lim_{n,d\rightarrow\infty} \frac{1}{n} \sum_{i=1}^n (r_{T,\eta}^{(\bar{\eta}_{T,i})}(t,s)-r_{T,\eta}^{(\bar{\eta}_{T,i})}(t',s'))^2 = \E[(r_{T,\eta}(t,s)-r_{T,\eta}(t',s'))^2] = \pnorm{r_{T,\eta}(t,s)-r_{T,\eta}(t',s')}{L^2}^2.
    \end{equation*}
    Using Proposition \ref{prop:art_spec_dmft_coup}-(2) and taking $T\rightarrow\infty$ limits on both sides, we have almost surely
    \begin{equation*}
        \lim_{T\rightarrow\infty}\lim_{n,d\rightarrow\infty} \frac{1}{n} \sum_{i=1}^n (r_{T,\eta}^{(\bar{\eta}_{T,i})}(t,s)-r_{T,\eta}^{(\bar{\eta}_{T,i})}(t',s'))^2 =  \pnorm{r_{\eta}(t,s)-r_\eta(t',s')}{L^2}^2.
    \end{equation*}
    Then it is sufficient to provide an upper bound to $\lim_{T\rightarrow\infty}\lim_{n,d\rightarrow\infty} \frac{1}{n} \sum_{i=1}^n (r_{T,\eta}^{(\bar{\eta}_{T,i})}(t,s)-r_{T,\eta}^{(\bar{\eta}_{T,i})}(t',s'))^2$.
    Using Lemma \ref{lm: low dim response}, we have 
    \begin{align*}
        &\lim_{T\rightarrow\infty}\lim_{n,d\rightarrow\infty} \frac{1}{n} \sum_{i=1}^n (r_{T,\eta}^{(\bar{\eta}_{T,i})}(t,s)-r_{T,\eta}^{(\bar{\eta}_{T,i})}(t',s'))^2\\
        &\leq \lim_{T\rightarrow\infty}\lim_{n,d\rightarrow\infty}\frac{2}{n}\Bigg[\sum_{i=1}^n (E_{T,\eta}^{t,[s,i]})^2 + \sum_{i=1}^n (E_{T,\eta}^{t',[s',i]})^2\Bigg]\\
        &\quad + 2\lim_{T\rightarrow\infty}\limsup_{n,d\rightarrow\infty} \frac{1}{n} \sum_{i=1}^n \Big(\partial_1\ell(\eta_{T,i}^{t'},\eta^\ast_i,z_i)\deps \eta_{T,i}^{t',[s',i],\eps} - \partial_1\ell(\eta_{T,i}^{t}, \eta^\ast_i, z_i)\deps \eta_{T,i}^{t,[s,i],\eps}\Big)^2\\
        &\leq \underbrace{\lim_{T\rightarrow\infty}\lim_{n,d\rightarrow\infty}\frac{2}{n}\Bigg[\sum_{i=1}^n (E_{T,\eta}^{t,[s,i]})^2 + \sum_{i=1}^n (E_{T,\eta}^{t',[s',i]})^2\Bigg]}_{(I)}\\
        &\quad + \underbrace{4\lim_{T\rightarrow\infty}\limsup_{n,d\rightarrow\infty} \frac{1}{n} \sum_{i=1}^n \big(\partial_1\ell(\eta_{T,i}^{t'},\eta^\ast_i,z_i)\big)^2\Big(\deps \eta_{T,i}^{t',[s',i],\eps} - \deps \eta_{T,i}^{t,[s,i],\eps}\Big)^2}_{(II)}\\
        &\quad + \underbrace{4\lim_{T\rightarrow\infty}\limsup_{n,d\rightarrow\infty} \frac{1}{n} \sum_{i=1}^n \Big(\partial_1\ell(\eta_{T,i}^{t},\eta^\ast_i,z_i)-\partial_1\ell(\eta_{T,i}^{t'},\eta^\ast_i,z_i)\Big)^2\big(\deps \eta_{T,i}^{t,[s,i],\eps}\big)^2}_{(III)}.
    \end{align*}
    Again by Lemma \ref{lm: low dim response}, we know that almost surely
    \begin{equation*}
        (I) = \lim_{T\rightarrow\infty}\lim_{n,d\rightarrow\infty}\frac{2}{n}\Bigg[\sum_{i=1}^n (E_{T,\eta}^{t,[s,i]})^2 + \sum_{i=1}^n (E_{T,\eta}^{t',[s',i]})^2\Bigg] = 0.
    \end{equation*}
    Using Assumption \ref{ass: ell lipschitz} and the matrix computations of the derivatives given in Lemma \ref{lm: response matrix recursion}, we can bound the second term by
    \begin{align*}
        (II)&\leq 4\pnorm{\partial_1 \ell}{\infty}^2\lim_{T\rightarrow\infty}\limsup_{n,d\rightarrow\infty} \frac{1}{n} \sum_{i=1}^n \Big(\deps \eta_{T,i}^{t',[s',i],\eps} - \deps \eta_{T,i}^{t,[s,i],\eps}\Big)^2\\
        &\leq 4\gamma^2C^2\lim_{T\rightarrow\infty}\limsup_{n,d\rightarrow\infty} \frac{1}{n} \sum_{i=1}^n \Big(\X\bOmega^{t'-1}_T\cdots\bOmega^{s'+1}_T\X^\top\D_T^{s'} - \X\bOmega^{t-1}_T\cdots\bOmega^{s+1}_T\X^\top\D_T^s\Big)_{ii}^2\\
         &\leq 4\gamma^2C^2\lim_{T\rightarrow\infty}\limsup_{n,d\rightarrow\infty} \frac{1}{n} \Bigpnorm{\X\bOmega^{t'-1}_T\cdots\bOmega^{s'+1}_T\X^\top\D_T^{s'} - \X\bOmega^{t-1}_T\cdots\bOmega^{s+1}_T\X^\top\D_T^{s}}{F}^2\\
         &\leq 4\gamma^2C^2\lim_{T\rightarrow\infty}\limsup_{n,d\rightarrow\infty} \frac{1}{n} \Bigpnorm{\X\bOmega^{t'-1}_T\cdots\bOmega^{s'+1}_T\X^\top\D_T^{s'} - \X\bOmega^{t-1}_T\cdots\bOmega^{s+1}_T\X^\top\D_T^{s'}}{F}^2\\
         &\quad+ 4\gamma^2C^2\lim_{T\rightarrow\infty}\limsup_{n,d\rightarrow\infty} \frac{1}{n} \Bigpnorm{\X\bOmega^{t-1}_T\cdots\bOmega^{s+1}_T\X^\top\D_T^{s'} - \X\bOmega^{t-1}_T\cdots\bOmega^{s+1}_T\X^\top\D_T^{s}}{F}^2\\
         &\leq 8\gamma^2C^8\lim_{T\rightarrow\infty}\limsup_{n,d\rightarrow\infty} \frac{1}{n} \Bigpnorm{\bOmega^{t'-1}_T\cdots\bOmega^{s'+1}_T - \bOmega^{t-1}_T\cdots\bOmega^{s+1}_T}{F}^2\\
         &\quad + 8\gamma^2 C^{t-s}\lim_{T\rightarrow\infty}\limsup_{n,d\rightarrow\infty}\pnorm{\D_T^{s'}-\D_T^s}{F}^2,
    \end{align*}
    Here $C>0$ is a constant such that $\pnorm{\X}{\op}\vee \pnorm{\partial_1 \ell}{\infty} \vee \pnorm{\bOmega^{k}_T}{\op} \leq C$ for all $k\in\Z_+$ which we know holds almost surely for large $n,d$.
    Next, we claim that 
    \begin{equation*}
        \lim_{T\rightarrow\infty}\limsup_{n,d\rightarrow\infty} \frac{1}{n} \Bigpnorm{\bOmega^{t'-1}_T\cdots\bOmega^{s'+1}_T - \bOmega^{t-1}_T\cdots\bOmega^{s+1}_T}{F}^2 = \limsup_{n,d\rightarrow\infty} \frac{1}{n} \Bigpnorm{\bOmega^{t'-1}\cdots\bOmega^{s'+1} - \bOmega^{t-1}\cdots\bOmega^{s+1}}{F}^2.
    \end{equation*}
    To this end, it suffices to show that
    \begin{equation}\label{eq:inner_prod_claim}
    \begin{gathered}
    \lim_{T\rightarrow\infty}\limsup_{n,d\rightarrow\infty} \frac{1}{n} \pnorm{\bOmega^{t-1}_T\cdots\bOmega^{s+1}_T}{F}^2 = \limsup_{n,d\rightarrow\infty} \frac{1}{n} \pnorm{\bOmega^{t-1}\cdots\bOmega^{s+1} }{F}^2,\\
        \lim_{T\rightarrow\infty}\limsup_{n,d\rightarrow\infty} \frac{1}{n} \pnorm{\bOmega^{t'-1}_T\cdots\bOmega^{s'+1}_T}{F}^2 = \limsup_{n,d\rightarrow\infty} \frac{1}{n} \pnorm{\bOmega^{t'-1}\cdots\bOmega^{s'+1} }{F}^2,\\
        \lim_{T\rightarrow\infty}\limsup_{n,d\rightarrow\infty} \frac{1}{n}\langle \bOmega^{t'-1}_T\cdots\bOmega^{s'+1}_T, \bOmega^{t-1}_T\cdots\bOmega^{s+1}_T\rangle = \limsup_{n,d\rightarrow\infty} \frac{1}{n}\langle \bOmega^{t'-1}\cdots\bOmega^{s'+1}, \bOmega^{t-1}\cdots\bOmega^{s+1}\rangle.
    \end{gathered}
    \end{equation}
    We will show the first claim, and the other two follow from similar arguments. By reverse triangle inequality, we have
    \begin{align*}
        \limsup_{n,d\rightarrow\infty}\frac{1}{\sqrt{n}}\Big|\pnorm{\bOmega^{t-1}_T\cdots\bOmega^{s+1}_T}{F}-\pnorm{\bOmega^{t-1}\cdots\bOmega^{s+1} }{F}\Big| \leq \limsup_{n,d\rightarrow\infty}\frac{1}{\sqrt{n}}\pnorm{\bOmega^{t-1}_T\cdots\bOmega^{s+1}_T-\bOmega^{t-1}\cdots\bOmega^{s+1}}{F}.
    \end{align*}
    For any $k\in \Z_+$, we have
    \begin{align*}
        \pnorm{\bOmega_T^k-\bOmega^k}{F} &= \pnorm{\X^\top(\D^k-\D_T^k)\X}{F}\leq \pnorm{\X}{\op}^2 \pnorm{\D^k-\D_T^k}{F}\leq\pnorm{\X}{\op}^2\pnorm{\partial_1 \ell}{\infty}\pnorm{\btheta^k-\btheta_T^k}{} \leq C^3\pnorm{\btheta^k-\btheta_T^k}{}.
    \end{align*}
    Now using a telescoping decomposition, we have
    \begin{align*}
        \limsup_{n,d\rightarrow\infty}\frac{1}{\sqrt{n}}\pnorm{\bOmega^{t-1}_T\cdots\bOmega^{s+1}_T-\bOmega^{t-1}\cdots\bOmega^{s+1}}{F} &\leq \sum_{k=s+1}^{t-1}\limsup_{n,d\rightarrow\infty}\frac{1}{\sqrt{n}}\pnorm{\bOmega_T^{t-1}\cdots\bOmega_T^{k+1}(\bOmega_T^{k}-\bOmega^k)\bOmega^{k-1}\cdots\bOmega^{s+1}}{F}\\
        &\leq C^{t-s-2} \sum_{k=s+1}^{t-1}\limsup_{n,d\rightarrow\infty}\frac{1}{\sqrt{n}}\pnorm{\bOmega_T^{k}-\bOmega^k}{F}\\
        &\leq C^{t-s+1}\sum_{k=s+1}^{t-1}\lim_{n,d\rightarrow\infty}\frac{1}{\sqrt{n}}\pnorm{\btheta^k-\btheta_T^k}{}\\
        &=C^{t-s+1}\sum_{k=s+1}^{t-2}\pnorm{\theta^k-\theta_T^k}{L^2} \rightarrow 0,
    \end{align*}
    as $T\rightarrow \infty$, with the last step following from Proposition \ref{prop:art_spec_dmft_coup}-(1). This concludes (\ref{eq:inner_prod_claim}). A similar argument yields that
    \begin{align*}
    \lim_{T\rightarrow\infty}\lim_{n,d\rightarrow\infty} \pnorm{\D^{s'}_T - \D^s_T}{F}^2 = \lim_{n,d\rightarrow\infty} \pnorm{\D^{s'} - \D^s}{F}^2,
    \end{align*}
    so in summary, it holds that
    \begin{equation*}
        (II) \leq C\lim_{n,d\rightarrow\infty} \frac{1}{n} \pnorm{\bOmega^{t'-1}\cdots\bOmega^{s'+1} - \bOmega^{t-1}\cdots\bOmega^{s+1}}{F}^2 + C^{t-s}\lim_{n,d\rightarrow\infty}\frac{1}{n}\pnorm{\partial_1\ell(\bbeta^{s'},\bbeta^\ast,\z)-\partial_1\ell(\bbeta^{s},\bbeta^\ast,\z)}{}^2.
    \end{equation*}
    By similar arguments as term $(II)$, we have
    \begin{equation*}
        (III) \leq C^{t-s}\lim_{n,d\rightarrow\infty}\frac{1}{n}\pnorm{\partial_1\ell(\bbeta^{t'},\bbeta^\ast,\z)-\partial_1\ell(\bbeta^{t},\bbeta^\ast,\z)}{}^2.
    \end{equation*}
    This concludes the proof of the first claim. The remaining claims are similar.
\end{proof}

Next we present two lemmas that give the high-dimensional representations of the responses in the DMFT system (\ref{def:dmft_theta0})-(\ref{def:theta_res}), which are analogues of (\ref{eq:art_response_example}) for the responses in the artificial DMFT system (\ref{def:artificialEta})-(\ref{def:artificial_dmft_CR_theta}).

\begin{lemma}
    \label{lm: response_mat_relation_1}
    Suppose Assumptions \ref{ass: model assumptions},\ref{ass: spectral initialization}, \ref{ass: ell lipschitz} and \ref{ass:exp_tti_cond} hold, $\psi_\delta'(\lambda_\delta^*) > 0$, and $\gamma < 2/L$. For $t,s \in \Z_+$ such that $t>s$, it holds almost surely that
    \begin{align}
    \label{eq:R_theta_hd}
        R_\theta(t,s) &= \lim_{n,d\rightarrow\infty} \frac{1}{d} \Tr\Big(\gamma\bOmega^{t-1}\cdots\bOmega^{s+1}\Big),\\
        R_\eta(t,s) &= \lim_{n,d\rightarrow\infty} \frac{1}{d}\Tr\Big(-\gamma \D^t\X\bOmega^{t-1}\cdots\bOmega^{s+1}\X^\top\D^s\Big),\\
        R_\eta(t,\ast) &= \lim_{n,d\rightarrow\infty} \frac{1}{d}\Tr\Big(-\gamma\sum_{s=0}^{t-1}\D^t\X\bOmega^{t-1}\cdots\bOmega^{s+1}\X^\top\diag(\partial^2\ell(\bbeta^s,\bbeta^\ast,\z)) + \diag(\partial_2\ell(\bbeta^t,\bbeta^\ast,\z))\Big),
    \end{align}
    where the limits on the right side exist almost surely. Moreover, it holds that
    \begin{align}\label{eq:R_bivariate_decay}
    |R_\theta(t+s,t)| \leq Ce^{-cs}, \quad \pnorm{r_\eta(t+s,t)}{L^2} \leq Ce^{-cs}.
    \end{align}
\end{lemma}
\begin{proof}
    We prove (\ref{eq:R_theta_hd}). We first show that almost sure the limit $\lim_{n,d\rightarrow}\frac{1}{d}\Tr(\bOmega^{t-1}\ldots \bOmega^{s+1})$ exists. To this end, we first prove the following fact. Suppose a double indexed real sequence $\{a_{n,m}\}_{n,m\geq 0}$ satisfies $a_{n,m}\rightarrow a^\ast_m$ as $n\rightarrow\infty$ for each fixed $m$, and another real sequence $\{b_n\}$ satisfies
    \begin{align}\label{eq:conv_cond}
    \lim_{m\rightarrow\infty}\limsup_{n\rightarrow\infty} |b_n - a_{n,m}| = 0. 
    \end{align}
    Then $\{b_n\}$ must admit a limit. To see this, note that
    \begin{align*}
    \limsup_{n\rightarrow\infty} |b_n - a_m^\ast| \leq \limsup_{n\rightarrow\infty} |b_n - a_{n,m}| + \limsup_{n\rightarrow\infty} |a_{n,m} - a_m^\ast| = \limsup_{n\rightarrow\infty} |b_n - a_{n,m}|.
    \end{align*}
    Taking $m\rightarrow\infty$ on both sides yields
    \begin{align*}
    \lim_{m\rightarrow\infty} \limsup_{n\rightarrow\infty} |b_n - a_m^\ast| = \lim_{m\rightarrow\infty}\limsup_{n\rightarrow\infty} |b_n - a_{n,m}| = 0.
    \end{align*}
    This implies that for any $\eps > 0$, there exists some $m_0 = m_0(\eps)$ such that $\limsup_{n\rightarrow\infty} |b_n - a_{m_0}^\ast| \leq \eps/2$, which further implies the existence of some $N = N(\eps)$ such that $|b_n - a_{m_0}^\ast| \leq \eps/2$ for any $n\geq N$. Hence any $n,n' \geq N$ satisfies $|b_n - b_{n'}| \leq |b_n - a_{m_0}^\ast| + |b_{n'} - a_{m_0}^\ast| \leq \eps$, which implies that $\{b_n\}$ is Cauchy and hence admits a limit.

    Now apply this fact with $b_n = \frac{1}{d}\Tr(\bOmega^{t-1}\ldots \bOmega^{s+1})$ and $a_{n,T} = \frac{1}{d}\Tr(\bOmega_T^{t-1}\ldots \bOmega_T^{s+1})$. Since $a_{n,T} \rightarrow R_{T,\theta}(t,s)/\gamma$ as $n,d\rightarrow\infty$ by Lemma \ref{lm: response matrix recursion} and Proposition \ref{prop: high dim response}, it suffices to show that almost surely
    \begin{align}\label{eq:omega_double_limit}
    \lim_{T\rightarrow\infty}\limsup_{n,d\rightarrow\infty} \Big|\frac{1}{d}\Tr(\bOmega^{t-1}\ldots \bOmega^{s+1}) - \frac{1}{d}\Tr(\bOmega^{t-1}_T\ldots \bOmega_T^{s+1})\Big| = 0.
    \end{align}
    To see this, note that
    \begin{align*}
    &\Big|\Tr\Big(\bOmega^{t-1}_T\ldots \bOmega^{s+1}_T - \bOmega^{t-1}\ldots \bOmega^{s+1}\Big)\Big| \leq \sqrt{d}\pnorm{\bOmega^{t-1}_T\ldots \bOmega^{s+1}_T - \bOmega^{t-1}\ldots \bOmega^{s+1}}{F}\\
    &\leq \sqrt{d}\sum_{k=s+1}^{t-1} \pnorm{\bOmega^{t-1}_T\ldots \bOmega^{k+1}_T (\bOmega^k_T - \bOmega^k)\bOmega^{k-1}\ldots \bOmega^{s+1}}{F} \leq \sqrt{d}C^{t-s}\sum_{k=s+1}^{t-1}\pnorm{\bOmega^k_T - \bOmega^k}{F}, 
    \end{align*}
    where the last step follows since $\pnorm{\bOmega^k}{\op}\vee\pnorm{\bOmega^k_T}{\op} \leq C$ via Lemma \ref{lm:Omega_exp_decay}. On the other hand, for each $s+1\leq k\leq t-1$, Assumption \ref{ass: ell lipschitz} yields that
    \begin{align}\label{eq:omega_bound}
    \pnorm{\bOmega^k_T - \bOmega^k}{F} = \gamma\pnorm{\X^\top (\D^k_T - \D^k)\X}{F} \leq \gamma\pnorm{\X}{\op}^2\pnorm{\partial_1\ell(\bbeta^k_T) - \partial_1\ell(\bbeta^k)}{} \leq C\pnorm{\btheta^k_T - \btheta^k}{}
    \end{align}
    almost surely for large enough $n,d$, which implies
    \begin{align*}
    &\limsup_{n,d\rightarrow\infty} \Big|\frac{1}{d}\Tr(\bOmega^{t-1}\ldots \bOmega^{s+1}) - \frac{1}{d}\Tr(\bOmega^{t-1}_T\ldots \bOmega_T^{s+1})\Big|\\
    &\leq C^{t-s}\sum_{k=s+1}^{t-1}\lim_{n,d\rightarrow\infty} \frac{1}{\sqrt{d}}\pnorm{\btheta^k_T - \btheta^k}{} = C^{t-s}\sum_{k=s+1}^{t-1}\pnorm{\theta^k_T - \theta^k}{L^2} \rightarrow 0,
    \end{align*}
    almost surely as $T\rightarrow\infty$, where the last step follows from Proposition \ref{prop:art_spec_dmft_coup}. This shows (\ref{eq:omega_double_limit}) and hence almost sure existence of limit $\lim_{n,d\rightarrow\infty} \frac{1}{d}\Tr (\bOmega^{t-1}\ldots \bOmega^{s+1})$.

    Next we show the identity of (\ref{eq:R_theta_hd}). Using the artificial dynamic as an intermediary, we have the decomposition
    \begin{align*}
        &\Big| R_\theta(t,s) - \lim_{n,d\rightarrow\infty} \frac{1}{d} \Tr\Big(\gamma\bOmega^{t-1}\cdots\bOmega^{s+1}\Big)\Big|\\
        &\quad \leq \underbrace{|R_\theta(t,s) - R_{T,\theta}(t,s)|}_{(I)} + \underbrace{\Big| R_{T,\theta}(t,s) - \lim_{n,d\rightarrow\infty} \frac{1}{d} \Tr\Big(\gamma\bOmega_T^{t-1}\cdots\bOmega_T^{s+1}\Big)\Big|}_{(II)}\\
        &\quad + \underbrace{\Big| \lim_{n,d\rightarrow\infty} \frac{1}{d} \Tr\Big(\gamma\bOmega_T^{t-1}\cdots\bOmega_T^{s+1}\Big) - \lim_{n,d\rightarrow\infty} \frac{1}{d} \Tr\Big(\gamma\bOmega^{t-1}\cdots\bOmega^{s+1}\Big)\Big|}_{(III)}.
    \end{align*}
    By Proposition \ref{prop:art_spec_dmft_coup}-(2), we have $(I)\rightarrow 0$ as $T\rightarrow\infty$. Moreover, $(II) = 0$ almost surely by Lemma \ref{lm: response matrix recursion} and Proposition \ref{prop: high dim response}, and as $T\rightarrow\infty$
    \begin{align*}
    (III) \leq \limsup_{n,d\rightarrow\infty} \Big|\frac{1}{d}\Tr(\bOmega^{t-1}\ldots \bOmega^{s+1}) - \frac{1}{d}\Tr(\bOmega^{t-1}_T\ldots \bOmega_T^{s+1})\Big| \rightarrow 0
    \end{align*}
    by (\ref{eq:omega_double_limit}). This concludes the proof of (\ref{eq:R_theta_hd}). The other  two claims follow by similar arguments.

    Next we show (\ref{eq:R_bivariate_decay}). Using Lemma \ref{lm: response_mat_relation_1}, it holds almost surely that
    \begin{align*}
        \big|R_\theta(t+s,t)\big| &\leq C\limsup_{n,d\rightarrow\infty}\Big|\frac{1}{d}\Tr(\bOmega^{t+s-1}\cdots\bOmega^{t+1})\Big| \leq C\pnorm{\bOmega^{t+s-1}}{\rm \op}\cdots \pnorm{\bOmega^{t+1}}{\rm \op}.
    \end{align*}
    It suffices to consider the case $s \geq 2t_0$, for which $\pnorm{\bOmega^{t+1}}{\op},\ldots,\pnorm{\bOmega^{t_0}}{\op}\leq C$, and the remaining\\ $\pnorm{\bOmega^{t_0+1}}{\op},\ldots,\pnorm{\bOmega^{t+s-1}}{\op} \leq e^{-c}$ by Lemma \ref{lm:Omega_exp_decay}, so we conclude that $|R_\theta(t+s,t)| \leq C'e^{-cs}$ for some $C' = C'(t_0)$, as desired. To see the claim for $r_\eta$, the last part of Lemma \ref{lm: response_mat_relation_2} yields that
    \begin{align*}
        \pnorm{r_\eta(t+s,t)}{L^2} &\leq C\limsup_{n,d\rightarrow\infty}\frac{1}{\sqrt{n}}\pnorm{\bOmega^{t+s-1}\cdots\bOmega^{t+1}}{F} \leq C\limsup_{n,d\rightarrow\infty}\pnorm{\bOmega^{t+s-1}}{\rm \op}\cdots \pnorm{\bOmega^{t+1}}{\rm \op} \leq Ce^{-cs},
    \end{align*}
    using a similar consideration as above. The proof is complete.
\end{proof}

\begin{lemma}
    \label{lm: response_mat_relation_3}
    Suppose Assumptions \ref{ass: model assumptions},\ref{ass: spectral initialization}, \ref{ass: ell lipschitz} and \ref{ass:exp_tti_cond} hold, $\psi_\delta'(\lambda_\delta^*) > 0$, and $\gamma < 2/L$. For $t\in \Z_+$, with $\D^t, \bOmega^t,\bUpsilon$ defined in (\ref{def:D_Omega_lim}) and (\ref{def:Upsilon_lim}), we have almost surely that
    \begin{align}
    \label{eq:spec_init_res_mat_theta}
        R_\theta(t,\diamond) &=  \frac{\lambda_\delta^\ast}{\lambda_{1,\infty}}\sum_{s=-\infty}^{-1}\lim_{n,d\rightarrow\infty} \frac{1}{d}\Tr\Big(\bOmega^{t-1}\cdots\bOmega^0(\bUpsilon)^{-(s+1)}\Big),\\
        R_\eta(t,\diamond) &= \frac{-1}{\lambda_{1,\infty}}\sum_{s=-\infty}^{-1}\lim_{n,d\rightarrow\infty} \frac{1}{d}\Tr\Big(\D^t\X\bOmega^{t-1}\cdots\bOmega^0(\bUpsilon)^{-(s+1)}\X^\top\diag\big(\cT_s(\y)\big)\Big),\\
        R_\eta(t,\diamond\diamond) &= 
        \label{eq:spec_init_res_mat_2_diamond}\frac{1}{\lambda_{1,\infty}}\sum_{s=-\infty}^{-1}\lim_{n,d\rightarrow\infty}\frac{1}{d}\Tr\Big(\D^t\X\bOmega^{t-1}\cdots\bOmega^{0}(\bUpsilon)^{-(s+1)}\X^\top\diag\big(\cT_s'(\y)\circ\varphi'(\bbeta^*,\z)\circ\bbeta^0\big)\Big),
    \end{align}
    where the limits on the right side exist almost surely, and $(\bUpsilon)^k$ denotes the $k$th power of $\bUpsilon$.
\end{lemma}
\begin{proof} 
By Proposition \ref{prop:art_spec_dmft_coup}-(3), we have almost surely
\begin{align}\label{eq:R_theta_diamond_1}
R_\theta(t,\diamond) = -\lambda_\delta^\ast \lim_{T\rightarrow\infty} \sum_{s=-T}^{-1} R_{T,\theta}(t,s),
\end{align}
where we recall from Proposition \ref{prop: high dim response} and Lemma \ref{lm: response matrix recursion} that almost surely 
\begin{align*}
R_{T,\theta}(t,s) = \lim_{n,d\rightarrow\infty} \frac{1}{d} \Tr\big(-\frac{1}{\beta^{s+1}_T}\bOmega_T^{t-1}\cdots\bOmega_T^0\bUpsilon_T^{-1}\cdots\bUpsilon_T^{s+1}\big).
\end{align*}
We first show that almost surely
\begin{align}\label{eq:omega_upsilon_claim}
\lim_{T\rightarrow\infty} \limsup_{n,d\rightarrow\infty} \frac{1}{d}\Big|\Tr\Big(-\frac{1}{\beta^{s+1}_T}\bOmega_T^{t-1}\cdots\bOmega_T^0\bUpsilon_T^{-1}\cdots\bUpsilon_T^{s+1}\Big) - \Tr\Big(-\frac{1}{\lambda_{1,\infty}}\bOmega^{t-1}\cdots\bOmega^0(\bUpsilon)^{-(s+1)}\Big) \Big| = 0.
\end{align}
By the same reasoning surrounding (\ref{eq:conv_cond}), this implies the almost sure existence of the limit $\lim_{n,d\rightarrow\infty}\Tr\big(-\frac{1}{\lambda_{1,\infty}}\bOmega^{t-1}\cdots\bOmega^0(\bUpsilon)^{-(s+1)}\big)$, as well as the identity
\begin{align}\label{eq:R_infty_theta}
R_{\infty,\theta}(t,s) := \lim_{T\rightarrow\infty}R_{T,\theta} (t,s) = \lim_{n,d\rightarrow\infty}\frac{1}{d}\Tr\Big(-\frac{1}{\lambda_{1,\infty}}\bOmega^{t-1}\cdots\bOmega^0(\bUpsilon)^{-(s+1)}\Big).
\end{align}
To see the claim (\ref{eq:omega_upsilon_claim}), 
\begin{align*}
&\frac{1}{d}\Big|\Tr\Big(-\frac{1}{\beta^{s+1}_T}\bOmega_T^{t-1}\cdots\bOmega_T^0\bUpsilon_T^{-1}\cdots\bUpsilon_T^{s+1}\Big) - \Tr\Big(-\frac{1}{\lambda_{1,\infty}}\bOmega^{t-1}\cdots\bOmega^0(\bUpsilon)^{-(s+1)}\Big) \Big|\\
&\leq \Big|\frac{1}{\beta^{s+1}_T} - \frac{1}{\lambda_{1,\infty}}\Big|\cdot\frac{1}{d}\Big|\Tr\Big(\bOmega^{t-1}_T\ldots\bOmega^0_T\bUpsilon^{-1}_T\ldots \bUpsilon^{s+1}_T)\Big)\Big|\\
&\quad + \frac{1}{\lambda_{1,\infty}}\cdot\frac{1}{d}\Big|\Tr\Big(\bOmega^{t-1}_T\ldots\bOmega^0_T\bUpsilon^{-1}_T\ldots \bUpsilon^{s+1}_T)\Big) - \Tr\Big(\bOmega^{t-1}\ldots\bOmega^0(\bUpsilon)^{-s-1}\Big) \Big| := (I) + (II).
\end{align*}
Using $\beta^{s+1}_T \rightarrow \lambda_{1,\infty}$ for fixed $s$, and $\pnorm{\bOmega^k_T}{\op} \vee \pnorm{\bUpsilon^k_T}{\op} \leq C$ uniformly over $T$, we have $\lim_{T\rightarrow\infty}\limsup_{n,d} (I) = 0$. For $(II)$, we have by (\ref{eq:omega_bound}) that $\pnorm{\bOmega_T^k-\bOmega^k}{F} \leq C\pnorm{\btheta^t-\btheta_T^t}{}$, and moreover
\begin{align*}
    \pnorm{\bUpsilon_T^k-\bUpsilon}{F} &\leq \Big|\frac{1}{\beta^{k+1}_T}-\frac{1}{\lambda_{1,\infty}}\Big|\pnorm{\X^\top\diag(\cT_s(\y))\X}{F} \leq \sqrt{n}\Big|\frac{1}{\beta^{k+1}_T}-\frac{1}{\lambda_{1,\infty}}\Big|\pnorm{\X}{\op}^2 \leq C\sqrt{n}\Big|\frac{1}{\beta^{k+1}_T}-\frac{1}{\lambda_{1,\infty}}\Big|
\end{align*}
almost surely for large enough $n,d$. Hence $\lim_{T\rightarrow\infty}\limsup_{n,d\rightarrow\infty} (II) = 0$ by $\beta^{k+1}_T \rightarrow \lambda_{1,\infty}$ and a similar telescoping decomposition as in (\ref{eq:omega_double_limit}). This establishes the claim (\ref{eq:omega_upsilon_claim}) and thus the identity (\ref{eq:R_infty_theta}). Moreover, by taking $T\rightarrow\infty$ in Lemma \ref{lm:first_stage_res_exp_decay}, it holds that $R_{\infty,\theta}(t,s) \leq Ce^{cs}$.

Next we claim that the right side of (\ref{eq:R_theta_diamond_1}) satisfies
\begin{align}\label{eq:R_theta_diamond_2}
    \lim_{T\rightarrow\infty} \sum_{s=-T}^{-1} R_{T,\theta}(t,s) = \lim_{T\rightarrow\infty}\sum_{s=-T}^{-1} R_{\infty,\theta}(t,s),
\end{align}
where the limit on the right exists for each fixed $t$ since $\sum_{s=-\infty}^{-1}|R_{\infty,\theta}(t,s)| \leq C\sum_{s=-\infty}^{-1}e^{cs} < \infty$.
To see the claim, we have
\begin{align}
\begin{split}
\label{eq:move_T_lim}
    &\Big|\sum_{s=-T}^{-1} R_{T,\theta}(t,s) - \sum_{s=-T}^{-1} R_{\infty,\theta}(t,s)\Big| \leq \sum_{s=-\Delta}^{-1} \Big|R_{T,\theta}(t,s) - R_{\infty,\theta}(t,s)\Big| + \sum_{s=-T}^{-\Delta} \big(|R_{T,\theta}(t,s)| + |R_{\infty,\theta}(t,s)|\big)\\
    &\leq \sum_{s=-\Delta}^{-1} \Big|R_{T,\theta}(t,s) - R_{\infty,\theta}(t,s)\Big|  + C\sum_{s=-T}^{-\Delta} e^{cs} \leq \sum_{s=-\Delta}^{-1} \Big|R_{T,\theta}(t,s) - R_{\infty,\theta}(t,s)\Big| + Ce^{-c\Delta},
\end{split}
\end{align}
where we use the exponential decay of $|R_{T,\theta}(t,s)|, |R_{\infty,\theta}(t,s)|$ in Lemma \ref{lm:first_stage_res_exp_decay}. The result follows by first taking $T\rightarrow\infty$ then $\Delta\rightarrow\infty$.
Combining (\ref{eq:R_theta_diamond_1}) and (\ref{eq:R_theta_diamond_2}) and further applying (\ref{eq:R_infty_theta}), we have almost surely
\begin{align*}
    R_\theta(t,\diamond) = -\lambda_\delta^\ast \sum_{s=-\infty}^{-1} R_{\infty,\theta}(t,s) = -\lambda_\delta^\ast \sum_{s=-\infty}^{-1} \lim_{n,d\rightarrow\infty}\frac{1}{d}\Tr\Big(-\frac{1}{\lambda_{1,\infty}}\bOmega^{t-1}\cdots\bOmega^0(\bUpsilon)^{-(s+1)}\Big),
\end{align*}
as desired. The proof for $R_\eta(t,\diamond)$ is similar.

Now we prove the claim for $R_\eta(t,\diamond\diamond)$. We first show that the limit
\begin{align}\label{eq:r_eta_diamond_limit}
    \sum_{s=-\infty}^{-1}\lim_{n,d\rightarrow\infty}\frac{1}{d}\Tr\Big(\D^t\X\bOmega^{t-1}\cdots\bOmega^{0}(\bUpsilon)^{-(s+1)}\X^\top\diag\big(\cT_s'(\y)\circ\varphi'(\bbeta^*,\z)\circ\bbeta^0\big)\Big)
\end{align}
exists. To simplify some notation, let $\tilde{\D}_T^k = \diag(\cT_s'(\y)\circ\varphi'(\bbeta^*,\z)\circ\bbeta_T^k)$ and $\tilde{\D}^0 = \diag(\cT_s'(\y)\circ\varphi'(\bbeta^*,\z)\circ\bbeta^0)$. We proceed in two steps. First, we show that the inner $n,d$ limit exists for fixed $t,s$.
Appealing to condition (\ref{eq:conv_cond}), we need to verify that
\begin{align}\label{eq:r_eta_diamond_claim}
\lim_{T\rightarrow\infty}\limsup_{n,d\rightarrow\infty}\frac{1}{d}\Big|\Tr\Big(\D_T^t\X\bOmega_T^{t-1}\cdots\bOmega_T^0\bUpsilon_T^{-1}\cdots\bUpsilon_T^{s+1}\tilde{\D}_T^s-\D^t\X\bOmega^{t-1}\cdots\bOmega^0(\bUpsilon)^{-(s+1)}\tilde{\D}^0\Big)\Big| = 0.
\end{align}
We can bound 
\begin{align*}
    &\Big|\Tr\Big(\D_T^t\X\bOmega_T^{t-1}\cdots\bOmega_T^0\bUpsilon_T^{-1}\cdots\bUpsilon_T^{s+1}\tilde{\D}_T^s-\D^t\X\bOmega^{t-1}\cdots\bOmega^0(\bUpsilon)^{-(s+1)}\tilde{\D}^0\Big)\Big|\\
    &\leq \sqrt{n}\biggpnorm{\D_T^t\X\bOmega_T^{t-1}\cdots\bOmega_T^0\bUpsilon_T^{-1}\cdots\bUpsilon_T^{s+1}\tilde{\D}_T^s-\D^t\X\bOmega^{t-1}\cdots\bOmega^0(\bUpsilon)^{-(s+1)}\tilde{\D}^0}{F}\\
    &\leq C\sqrt{n}\Big(\pnorm{\D_T^t-\D^t}{F} + \sum_{k=0}^{t-1}\pnorm{\bOmega_T^k-\bOmega^k}{F} + \sum_{k=s+1}^{-1}\pnorm{\bUpsilon_T^k-\bUpsilon}{F}+\pnorm{\tilde{\D}_T^s-\tilde{\D}^0}{F}\Big),
\end{align*}
where the last line follows from a telescoping decomposition and the fact that $\pnorm{\D_T^t}{\op} \vee \pnorm{\X}{\op}\vee\pnorm{\bOmega_T^k}{\op} \leq C$ almost surely for large enough $n,d$. Using further the bounds
\begin{equation*}
\begin{gathered}
\pnorm{\D^t_T - \D^t}{F} \leq C\pnorm{\btheta^t_T - \btheta^t}{},\quad \pnorm{\bOmega^k_T - \bOmega^t}{F} \leq C\pnorm{\btheta^k_T - \btheta^k}{},\\
\pnorm{\bUpsilon^k_T - \bUpsilon^k}{F} \leq C\sqrt{n}\Big|\frac{1}{\beta^{k+1}_T} - \frac{1}{\lambda_{1,\infty}}\Big| \leq C\sqrt{n}e^{-c(T+k)}, \quad \pnorm{\tilde\D^s_T - \tilde\D^0}{F} \leq C\pnorm{\btheta^s_T - \btheta^0}{},
\end{gathered}
\end{equation*}
we have
\begin{align*}
    &\limsup_{n,d\rightarrow\infty}\frac{1}{d}\Big|\Tr\Big(\D_T^t\X\bOmega_T^{t-1}\cdots\bOmega_T^0\bUpsilon_T^{-1}\cdots\bUpsilon_T^{s+1}\tilde\D^s_T - \D^t\X\bOmega^{t-1}\cdots\bOmega^0(\bUpsilon)^{-(s+1)}\tilde\D^0\Big)\Big|\\
    &\leq C\sum_{k=0}^t\pnorm{\theta_T^k-\theta^k}{L^2} + C\sum_{k=s}^{-1}e^{-c(T+k)} + Ce^{-c(T+s)} \rightarrow 0
\end{align*}
as $T\rightarrow\infty$. Similar considerations as above further shows that
\begin{align}\label{eq:r_eta_diamond_claim_2}
\lim_{T\rightarrow\infty}\limsup_{n,d\rightarrow\infty}\frac{1}{d}\Big|\Tr\Big(\frac{1}{\beta^{s+1}_T}\D_T^t\X\bOmega_T^{t-1}\cdots\bOmega_T^0\bUpsilon_T^{-1}\cdots\bUpsilon_T^{s+1}\tilde{\D}_T^s - \frac{1}{\lambda_{1,\infty}}\D^t\X\bOmega^{t-1}\cdots\bOmega^0(\bUpsilon)^{-(s+1)}\tilde{\D}^0\Big)\Big| = 0.
\end{align}
Denote the inner limit by 
\begin{align*}
A(t,s) := \lim_{n,d\rightarrow\infty}\frac{1}{d}\Tr\Big(\D^t\X\bOmega^{t-1}\cdots\bOmega^{0}(\bUpsilon)^{-(s+1)}\X^\top\diag\big(\cT_s'(\y)\circ\varphi'(\bbeta^*,\z)\circ\bbeta^0\big)\Big).
\end{align*}
By similar considerations as in Lemma \ref{lm:first_stage_res_exp_decay}, we have $|A(t,s)| \leq C(t)e^{cs}$, hence $\sum_{s=-\infty}^{-1} |A(t,s)| \leq C(t)\sum_{s=-\infty}^{-1}e^{cs} < \infty$, which shows the existence of the outer limit in (\ref{eq:r_eta_diamond_limit}).

Lastly, to see (\ref{eq:spec_init_res_mat_2_diamond}), by Proposition \ref{prop:art_spec_dmft_coup}-(3), we have $R_{\eta}(t,\diamond\diamond)=\lim_{T\rightarrow\infty}R_{T,\eta}(t,\sharp)$, so further applying the expression of $R_{T,\eta}(t,\sharp)$ in Proposition \ref{prop: high dim response} and Lemma \ref{lm: response matrix recursion}, we have
\begin{align*}
    R_{\eta}(t,\diamond\diamond)&=\lim_{T\rightarrow\infty}\sum_{s=-T}^{-1}\lim_{n,d\rightarrow\infty}\frac{1}{d}\frac{1}{\beta_T^{s+1}}\Tr\Big(\D_T^t\X\bOmega_T^{t-1}\cdots\bOmega_T^{0}\bUpsilon_T^{-1}\cdots\bUpsilon_T^{-s+1}\X^\top\tilde{\D}_T^s\Big)\\
    &\overset{(i)}{=}\sum_{s=-\infty}^{-1}\lim_{T\rightarrow\infty}\lim_{n,d\rightarrow\infty}\frac{1}{d}\frac{1}{\beta_T^{s+1}}\Tr\Big(\D_T^t\X\bOmega_T^{t-1}\cdots\bOmega_T^{0}\bUpsilon_T^{-1}\cdots\bUpsilon_T^{-s+1}\X^\top\tilde{\D}_T^s\Big)\\
    &\overset{(ii)}{=}\frac{1}{\lambda_{1,\infty}}\sum_{s=-\infty}^{-1}\lim_{n,d\rightarrow\infty}\frac{1}{d}\Tr\Big(\D^t\X\bOmega^{t-1}\cdots\bOmega^{0}(\bUpsilon)^{-(s+1)}\X^\top\tilde{\D}^0\Big).
\end{align*}
Here $(i)$ follows from similar reasoning as in (\ref{eq:R_theta_diamond_2}), and $(ii)$ follows from (\ref{eq:r_eta_diamond_claim_2}). The proof is complete.
\end{proof}

We are now ready to complete the proof of Theorem \ref{thm:dmft_exp_tti_condition}.
\begin{proof}[Proof of Theorem \ref{thm:dmft_exp_tti_condition}]
Let us check individually the five conditions of Definition \ref{def:exp_conv_spec_init}.
    \begin{enumerate}
        \item Let $\cE_n$ be the event given in 
        Lemma \ref{lm:long-time-properties}. By Lemma \ref{lm:long-time-properties}, we have
        \begin{equation*}
            \pnorm{\btheta^{t+1}-\btheta^t}{}^2 \leq r^{t-t_0}\pnorm{\btheta^{t_0+1}-\btheta^{t_0}}{}^2,
        \end{equation*}
        on the event $\cE_n$. Let $D_0 := \lim_{n,d\rightarrow\infty}\frac{1}{d}\pnorm{\btheta^{t_0+1} - \btheta^{t_0}}{}^2$ be the almost sure limit. By Theorem \ref{thm: discrete dmft asymp}, we have $D_0<\infty$. We show the existence of $\theta^\infty$ by showing that $\theta^t$ is a Cauchy sequence. For any $t_0 \leq s<t$, we have
        \begin{align*}
        \pnorm{\btheta^t - \btheta^s}{} \leq \sum_{k=s}^{t-1} \pnorm{\btheta^{k+1} - \btheta^k}{} \leq \sum_{k=s}^{t-1} r^{\frac{k-t_0}{2}} \pnorm{\btheta^{t_0+1} - \btheta^{t_0}}{} \leq \frac{r^{\frac{s-t_0}{2}}}{1-\sqrt{r}}\pnorm{\btheta^{t_0+1} - \btheta^{t_0}}{}.
        \end{align*}
        Then by Theorem \ref{thm: discrete dmft asymp}, we have
        \begin{align*}
            \pnorm{\theta^t-\theta^s}{L^2}^2 &= \lim_{n,d\rightarrow\infty} \frac{1}{d}\pnorm{\btheta^t-\btheta^s}{}^2 \leq \frac{r^{s-t_0}}{(1-\sqrt{r})^2}D_0 \rightarrow 0,
        \end{align*}
        as $s\rightarrow\infty$. This shows that $\{\theta^t\}$ is an $L^2$ Cauchy sequence, and hence it admits an $L^2$ limit $\theta^\infty$. Moreover, the convergence is exponential with rate
        \begin{align*}
            \pnorm{\theta^t-\theta^\infty}{L^2}^2 \leq  \frac{r^{t-t_0}}{(1-\sqrt{r})^2}D_0.
        \end{align*} To show the corresponding result for $\{\eta^t\}$, note that 
        \begin{align*}
            \pnorm{\eta^t-\eta^s}{L^2}^2 &= \lim_{n,d\rightarrow\infty} \frac{1}{n}\pnorm{\bbeta^t-\bbeta^s}{}^2 \leq (C')^2\lim_{n,d\rightarrow\infty}\frac{1}{n}\pnorm{\btheta^t-\btheta^s}{}^2,
        \end{align*}
        from which we can apply the same arguments as for $\{\theta^t\}$.
        \item We show $\{u^t\}$ is an $L^2$ Cauchy sequence. For any $t_0 \leq s<t$, we have
        \begin{align*}
            \frac{1}{\delta}\pnorm{u^t-u^s}{L^2}^2 &\overset{(i)}{=}\E[\ell(\eta^t,\eta^\ast,z)^2] +\E[\ell(\eta^s,\eta^\ast,z)^2] - 2\E[\ell(\eta^t,\eta^\ast,z)\ell(\eta^s,\eta^\ast,z)] 
            \\
            &\overset{(ii)}{=}\lim_{n,d\rightarrow\infty}\Bigg[\frac{1}{n}\sum_{i=1}^n\ell(\eta^t_i,\eta_i^\ast,z_i)^2 + \frac{1}{n}\sum_{i=1}^n\ell(\eta^s_i,\eta_i^\ast,z_i)^2 - \frac{2}{n}\sum_{i=1}^n\ell(\eta^t_i,\eta^\ast_i,z_i)\ell(\eta^s_i,\eta_i^\ast,z_i)\Bigg]\\
            &= \lim_{n,d\rightarrow\infty}\frac{1}{n}\sum_{i=1}^n (\ell(\eta_i^t,\eta_i^\ast,z_i)-\ell(\eta_i^s,\eta_i^\ast,z_i))^2\\
            &\leq C'\lim_{n,d\rightarrow\infty}\frac{1}{n}\sum_{i=1}^n (\eta_i^t-\eta_i^s)^2 = \frac{(C')^3}{\delta}\lim_{n,d\rightarrow\infty}\frac{1}{d}\pnorm{\btheta^t-\btheta^s}{2}^2\overset{(iii)}{\leq} \frac{(C')^3D_0}{\delta(1-\sqrt{r})^2}r^{s-t_0}.
        \end{align*}
    Here $(i)$ follows from the covariance structure of $u^t$, i.e., $\E[u^tu^s] = C_\eta(t,s) = \delta\E[\ell(\eta^t,w^\ast,z)\ell(\eta^s,w^\ast,z)]$, $(ii)$  follows from Theorem \ref{thm: discrete dmft asymp}, and $(iii)$ follows from similar computations as in part 1. This shows the existence of $u^\infty$, and the same convergence rate of $\pnorm{u^t - u^\infty}{L^2}$. Similarly, for any $t_0 \leq s<t$, we have
    \begin{align*}
            \pnorm{w^t-w^s}{L^2}^2 &=\E[(\theta^t)^2] +\E[(\theta^s)^2] - 2\E[\theta^t\theta^s] 
            \\
            &=\lim_{n,d\rightarrow\infty}\Bigg[\frac{1}{d}\sum_{j=1}^d(\theta^t_j)^2 + \frac{1}{d}\sum_{j=1}^d(\theta^s_j)^2 - \frac{2}{d}\sum_{j=1}^d\theta^t_j\theta^s_j\Bigg]\\
            &= \lim_{n,d\rightarrow\infty}\frac{1}{d}\sum_{j=1}^d (\theta_j^t-\theta_j^s)^2 = \lim_{n,d\rightarrow\infty}\frac{1}{d}\pnorm{\btheta^t-\btheta^s}{}^2 \leq \frac{D_0}{(1-\sqrt{r})^2}r^{s-t_0},
        \end{align*}
        concluding the claim for $\{w^t\}$.

    \item \textbf{Claim for $r_\eta(s)$} To show the existence of $r_\eta(s)$, we check that for any fixed $s\in\Z_+$, the sequence $r_\eta(t+s,t)$ in $t$ is Cauchy. A sufficient condition to check is that $\sum_{t=t_0}^\infty \pnorm{r_\eta(t+1+s,t+1)-r_\eta(t+s,t)}{L^2} <\infty.$ Using Lemma \ref{lm: response_mat_relation_2}-(1), we have for some $C > 0$ that may depend on $s$ that
    \begin{align*}
        &\pnorm{r_\eta(t+1+s,t+1)-r_\eta(t+s,t)}{L^2} \leq C\limsup_{n,d\rightarrow\infty}\frac{1}{\sqrt{n}}\pnorm{\bOmega^{t+s}\cdots\bOmega^{t+2}-\bOmega^{t+s-1}\cdots\bOmega^{t+1}}{F}\\
        &\quad+ C\limsup_{n,d\rightarrow\infty}\frac{1}{\sqrt{n}}\pnorm{\partial_1\ell(\bbeta^{t+1+s})-\partial_1\ell(\bbeta^{t+s})}{} +  C\limsup_{n,d\rightarrow\infty}\frac{1}{\sqrt{n}}\pnorm{\partial_1\ell(\bbeta^{t+1})-\partial_1\ell(\bbeta^{t})}{}.
    \end{align*}
    Using a telescoping decomposition and $\pnorm{\bOmega^k}{\op}\leq e^{-c} \leq 1$ for $k\geq t\geq t_0$ via Lemma \ref{lm:Omega_exp_decay}, we have
    \begin{equation*}
        \limsup_{n,d\rightarrow\infty}\frac{1}{\sqrt{n}}\pnorm{\bOmega^{t+s}\cdots\bOmega^{t+2}-\bOmega^{t+s-1}\cdots\bOmega^{t+1}}{F} \leq  \sum_{k=t+1}^{t+s-1} \limsup_{n,d\rightarrow\infty}\frac{1}{\sqrt{n}}\pnorm{\bOmega^{k+1}-\bOmega^k}{F},
    \end{equation*}
    where for any $k\geq t+1$, almost surely it holds that
    \begin{align*}
        &\limsup_{n,d\rightarrow\infty}\frac{1}{n}\pnorm{\bOmega^{k+1}-\bOmega^k}{F}^2 \leq \gamma^2\limsup_{n,d\rightarrow\infty}\frac{1}{n}\pnorm{\X}{\rm op}^4\pnorm{\partial_1\ell(\bbeta^{k+1},\bbeta^\ast,\z)-\partial_1\ell(\bbeta^{k},\bbeta^\ast,\z)}{}^2\\
        &\leq C\limsup_{n,d\rightarrow\infty}\frac{1}{n}\pnorm{\bbeta^{k+1}-\bbeta^k}{}^2 = C\pnorm{\eta^{k+1}-\eta^k}{L^2}^2 \leq Cr^k
    \end{align*}
    for some $r\in (0,1)$, where the last step follows from part (1) proved above. This entails that
    \begin{align}\label{eq:Omega_cauchy_bound}
    \limsup_{n,d\rightarrow\infty}\frac{1}{\sqrt{n}}\pnorm{\bOmega^{t+s}\cdots\bOmega^{t+2}-\bOmega^{t+s-1}\cdots\bOmega^{t+1}}{F} \leq Cr^{t/2}.
    \end{align}
    Similar consideration yields
    \begin{align*}
    \limsup_{n,d\rightarrow\infty} \frac{1}{\sqrt{n}}\pnorm{\partial_1\ell(\bbeta^{t+1+s})-\partial_1\ell(\bbeta^{t+s})}{} \vee \limsup_{n,d\rightarrow\infty}\frac{1}{\sqrt{n}}\pnorm{\partial_1\ell(\bbeta^{t+1})-\partial_1\ell(\bbeta^{t})}{} \leq Cr^{t/2}.
    \end{align*}
    Substituting these estimates back gives $\pnorm{r_\eta(t+1+s,t+1)-r_\eta(t+s,t)}{L^2}\leq Cr^{t/2}$ for any $t\geq t_0$, which implies $\sum_{t=t_0}^\infty \pnorm{r_\eta(t+1+s,t+1)-r_\eta(t+s,t)}{L^2} < \infty$, as desired. This establishes the existence of the limit of $r_\eta(t+s,t)$ as $t\rightarrow\infty$ for fixed $s$, which we denote by $r_\eta(s)$.

    \noindent\textbf{Claim for $r_\eta^\ast$} Next we show that $r_\eta(t,\ast)$ is a Cauchy sequence in $L^2$. For any $t,t' \in \Z_+$, Lemma \ref{lm: response_mat_relation_2} yields that
    \begin{align*}
        &\pnorm{r_\eta(t,\ast)-r_\eta(t',\ast)}{L^2}\\
        &\leq C \underbrace{\limsup_{n,d\rightarrow\infty}\frac{1}{\sqrt{n}}\biggpnorm{\sum_{s=0}^{t-1}\bOmega^{t-1}\cdots\bOmega^{s+1}\X^\top\diag\big(\partial_2\ell(\bbeta^s)\big)-\sum_{s=0}^{t'-1}\bOmega^{t'-1}\cdots\bOmega^{s+1}\X^\top\diag\big(\partial_2\ell(\bbeta^s)\big)}{F}}_{(I)}\\
        &+C\underbrace{\limsup_{n,d\rightarrow\infty}\frac{1}{\sqrt{n}}\pnorm{\partial_2\ell(\bbeta^{t})-\partial_2\ell(\bbeta^{t'})}{}}_{(II)} + C\underbrace{\limsup_{n,d\rightarrow\infty}\frac{1}{\sqrt{n}}\pnorm{\partial_1\ell(\bbeta^{t})-\partial_1\ell(\bbeta^{t'})}{}}_{(III)}.
    \end{align*}
    Using the assumption that $\partial_2\ell$ is Lipschitz, we can bound the second term by
    \begin{align*}
        (II) \leq C\limsup_{n,d\rightarrow\infty}\frac{1}{\sqrt{n}}\pnorm{\bbeta^t-\bbeta^{t'}}{2} = \pnorm{\eta^t-\eta^{t'}}{L^2} \rightarrow 0
    \end{align*}
    as $t\wedge t'\rightarrow\infty$, where the last step follows from part (1) established above.
    Similarly, it holds that $(III) \rightarrow 0$ as $t\wedge t'\rightarrow\infty$.
    For the first term, we can reindex to get
    \begin{equation*}
        (I) = \limsup_{n,d\rightarrow\infty}\frac{1}{\sqrt{n}}\biggpnorm{\sum_{s=0}^{t-1}\bOmega^{t-1}\cdots\bOmega^{t-s}\X^\top \diag\big(\partial_2 \ell(\bbeta^{t-s-1})\big)-\sum_{s=0}^{t'-1}\bOmega^{t'-1}\cdots\bOmega^{t'-s}\X^\top \diag\big(\partial_2 \ell(\bbeta^{t'-s-1})\big)}{F}.
    \end{equation*}
    On the event $\cE_n$, the exponential decay proved in Lemma \ref{lm:long-time-properties} shows that there exists a limit $\btheta^\infty := \lim_{t\rightarrow\infty}\btheta^t$ and we call $\bbeta^\infty := \X\btheta^\infty.$ Define the matrices
    \begin{equation}\label{def:D_Omega_lim_infty}
        \D^\infty := \diag\big(\partial_1\ell(\boldeta^\infty,\boldeta^*,\z)\big) \in \R^{n\times n}, \qquad \bOmega^\infty := (1-\gamma\lambda)\I_d -\gamma \X^\top\D^\infty\X \in \R^{d\times d}. 
    \end{equation}
    We emphasize that these quantities are almost surely well-defined for large enough $n,d$ (since $\cE_n$ holds so), hence all the $\limsup_{n,d\rightarrow\infty}$ in the following are well-defined.

    By adding and subtracting $(\bOmega^\infty)^{s}$ ($s$th power of $\bOmega^\infty$) and triangle inequality, we have $(I)\leq (I_1) + (I_2) + (I_3)$, where
    \begin{align*}
        (I_1) &= \sum_{s=0}^{t-1}\limsup_{n,d\rightarrow\infty} \frac{1}{\sqrt{n}}\Bigpnorm{\bOmega^{t-1}\cdots\bOmega^{t-s}\X^\top \diag\big(\partial_2 \ell(\bbeta^{t-s-1})\big)-(\bOmega^\infty)^{s}\X^\top \diag\big(\partial_2 \ell(\bbeta^\infty)\big)}{F},\\
        (I_2) &= \sum_{s=0}^{t'-1}\limsup_{n,d\rightarrow\infty} \frac{1}{\sqrt{n}}\Bigpnorm{\bOmega^{t'-1}\cdots\bOmega^{t'-s}\X^\top \diag\big(\partial_2 \ell(\bbeta^{t'-s-1})\big)-(\bOmega^\infty)^s\X^\top \diag\big(\partial_2 \ell(\bbeta^\infty)\big)}{F},\\
        (I_3) &= \limsup_{n,d\rightarrow\infty}\frac{1}{\sqrt{n}}\biggpnorm{\sum_{s=0}^{t-1}(\bOmega^\infty)^s\X^\top \diag\big(\partial_2 \ell(\bbeta^\infty)\big) - \sum_{s=0}^{t'-1}(\bOmega^\infty)^s\X^\top \diag\big(\partial_2 \ell(\bbeta^\infty)\big)}{F}.
    \end{align*}
    We claim that $(I) \rightarrow 0$ as $t\wedge t'\rightarrow\infty$, since such claim holds for $(I_1)$-$(I_3)$. To see that the claim holds for $(I_3)$, note that $\pnorm{\bOmega^\infty}{\op} \leq e^{-c}$ almost surely for large enough $n,d$ for some $c > 0$ using a similar argument of Lemma \ref{lm:Omega_exp_decay}, so assuming $t' \geq t$, we have
    \begin{align*}
    (I_3) \leq C\pnorm{\partial_2\ell}{\infty}\lim_{n,d\rightarrow\infty} \pnorm{\X}{\op}\sum_{s=t}^{t'-1} \pnorm{(\bOmega^\infty)^s}{\op} \leq C\sum_{s=t}^{\infty} e^{-cs} \rightarrow 0
    \end{align*}
    as $t\rightarrow\infty$. Next we show $(I_1) \rightarrow 0$ as $t\rightarrow\infty$, and a similar argument applies to $(I_2)$. Let
    \begin{align*}
        A_{t}(s) &:=\limsup_{n,d\rightarrow\infty} \frac{1}{\sqrt{n}}\Bigpnorm{\bOmega^{t-1}\cdots\bOmega^{t-s}\X^\top \diag\big(\partial_2\ell(\bbeta^{t-s-1})\big)-(\bOmega^\infty)^{s}\X^\top \diag\big(\partial_2\ell(\bbeta^\infty)\big)}{F},
    \end{align*}
    which can be further bounded by $A_t(s) \leq A_t'(s) + A_t''(s)$, where
    \begin{align*}
    A_t'(s) &= \limsup_{n,d\rightarrow\infty} \frac{1}{\sqrt{n}}\Bigpnorm{\bOmega^{t-1}\cdots\bOmega^{t-s}\X^\top \diag\big(\partial_2\ell(\bbeta^{t-s-1})\big)-(\bOmega^\infty)^{s}\X^\top \diag\big(\partial_2\ell(\bbeta^{t-s-1})\big)}{F},\\
    A_t''(s) &= \limsup_{n,d\rightarrow\infty} \frac{1}{\sqrt{n}}\Bigpnorm{(\bOmega^\infty)^s\X^\top \diag\big(\partial_2\ell(\bbeta^{t-s-1})\big)-(\bOmega^\infty)^{s}\X^\top \diag\big(\partial_2\ell(\bbeta^\infty)\big)}{F}.
    \end{align*}
    We first bound $A_t'(s)$. Note that
    \begin{align*}
    A_t'(s) &\leq \pnorm{\partial_2\ell}{\infty}\limsup_{n,d\rightarrow\infty} \pnorm{\X}{\op}\frac{1}{\sqrt{n}}\pnorm{\bOmega^{t-1}\cdots\bOmega^{t-s}-(\bOmega^\infty)^{s}}{F} \leq C\limsup_{n,d\rightarrow\infty} \frac{1}{\sqrt{n}}\pnorm{\bOmega^{t-1}\cdots\bOmega^{t-s}-(\bOmega^\infty)^{s}}{F}.
    \end{align*}
    To bound $\limsup_{n,d\rightarrow\infty} \frac{1}{\sqrt{n}}\pnorm{\bOmega^{t-1}\cdots\bOmega^{t-s}-(\bOmega^\infty)^{s}}{F}$, we consider two cases. For the case $s \leq c_0 t$ for some small enough $c_0$. Note that almost surely for large $n,d$,
    we have
    \begin{equation*}
        \limsup_{n,d\rightarrow\infty}\frac{1}{\sqrt{n}}\pnorm{\bOmega^k-\bOmega^\infty}{F} \leq C\lim_{n,d\rightarrow\infty} \frac{1}{\sqrt{n}}\pnorm{\btheta^k - \btheta^\infty}{}.
    \end{equation*}
    We can further bound almost surely
    \begin{align}
    \begin{split}
    \label{eq:theta_inf_hd_bound}
        \lim_{n,d\rightarrow\infty} \frac{1}{\sqrt{n}}\pnorm{\btheta^k - \btheta^\infty}{}
        &\leq C \lim_{n,d\rightarrow\infty}\limsup_{s\rightarrow\infty} \frac{1}{\sqrt{n}}\pnorm{\btheta^k - \btheta^s}{} + C \lim_{n,d\rightarrow\infty}\limsup_{s\rightarrow\infty} \frac{1}{\sqrt{n}}\pnorm{\btheta^s - \btheta^\infty}{}\\
        &\leq C \lim_{n,d\rightarrow\infty}\limsup_{s\rightarrow\infty} \frac{1}{\sqrt{n}}\sum_{v=k}^{s-1} \pnorm{\btheta^{v+1}-\btheta^v}{}\\
        &\overset{(i)}{\leq} C \lim_{n,d\rightarrow\infty}\frac{1}{\sqrt{n}}\pnorm{\btheta^{t_0+1}-\btheta^{t_0}}{}\limsup_{s\rightarrow\infty} \sum_{v=k}^{s-1} r^{\frac{v-t_0}{2}}\\
        &\overset{(ii)}{\leq} C\sum_{v=k}^{\infty} r^{\frac{v-t_0}{2}} \leq Ce^{-ck}.
    \end{split}
    \end{align}
    Here for (i), we apply Lemma \ref{lm:long-time-properties} and for (ii), we apply Theorem \ref{thm: discrete dmft asymp}.
     So using a telescoping decomposition and $\pnorm{\bOmega^t}{\op}\leq C$  for all $t\geq 0$ via Lemma \ref{lm:Omega_exp_decay}, we have
    \begin{align*}
        &\limsup_{n,d\rightarrow\infty} \frac{1}{\sqrt{n}}\pnorm{\bOmega^{t-1}\cdots\bOmega^{t-s-1}-(\bOmega^\infty)^s}{F} \\
        &\leq C^s\sum_{k=t-s}^{t-1}\limsup_{n,d\rightarrow\infty}\frac{1}{\sqrt{n}}\pnorm{\bOmega^k-\bOmega^\infty}{F}\leq Cse^{Cs}e^{-c(t-s)} \leq Ce^{-c_0't}
    \end{align*}
    for some $c_0' > 0$. For the other case $s > c_0t$, we apply
    \begin{align*}
        \limsup_{n,d\rightarrow\infty} \frac{1}{\sqrt{n}}\pnorm{\bOmega^{t-1}\cdots\bOmega^{t-s-1}-(\bOmega^\infty)^s}{F} &\leq \limsup_{n,d\rightarrow\infty} \frac{1}{\sqrt{n}}\pnorm{\bOmega^{t-1}\cdots\bOmega^{t-s}}{F} + \limsup_{n,d\rightarrow\infty} \frac{1}{\sqrt{n}}\pnorm{(\bOmega^\infty)^s}{F}\\
        &\leq C\limsup_{n,d\rightarrow\infty}\pnorm{\bOmega^{t-1}}{\rm op}\cdots\pnorm{\bOmega^{t-s}}{\rm op} + C\limsup_{n,d\rightarrow\infty} \pnorm{\bOmega^\infty}{\op}^s.
    \end{align*}
    When $s > c_0t$ and $t$ is large enough, it can be checked via Lemma \ref{lm:Omega_exp_decay} that all the operator norms above are bounded by $e^{-c}$ except for at most $t_0$ terms in the first product, which implies that $\limsup_{n,d\rightarrow\infty} \frac{1}{\sqrt{n}}\pnorm{\bOmega^{t-1}\cdots\bOmega^{t-s-1}-(\bOmega^\infty)^s}{F} \leq Ce^{-c_0't}$ again holds in this case. So we conclude that $A_t'(s) \leq Ce^{-c_0't}$ for all $s\in [0,t-1]$. Next, to bound $A_t''(s)$, we have
    \begin{align*}
    A_t''(s) \leq \limsup_{n,d\rightarrow\infty} \pnorm{\bOmega^\infty}{\op}^s\pnorm{\X}{\op}\frac{1}{\sqrt{n}}\pnorm{\bbeta^{t-s-1} - \bbeta^\infty}{}\leq Ce^{-cs}\limsup_{n,d\rightarrow\infty}\frac{1}{\sqrt{n}}\pnorm{\bbeta^{t-s-1} - \bbeta^\infty}{}. 
    \end{align*}
    When $s \geq c_0t$ for some $c_0 > 0$, using $\limsup_{n,d\rightarrow\infty}\frac{1}{\sqrt{n}}\pnorm{\bbeta^{t-s-1} - \bbeta^\infty}{} \leq C$, we have $A_t''(s) \leq e^{-c_0't}$, otherwise we have 
    \begin{align*}
    \limsup_{n,d\rightarrow\infty} \frac{1}{\sqrt{n}}\pnorm{\bbeta^{t-s-1} - \bbeta^{\infty}}{} \leq C\limsup_{n,d\rightarrow\infty} \frac{1}{\sqrt{n}}\pnorm{\btheta^{t-s-1} - \btheta^{\infty}}{} \leq Ce^{-c(t-s)} \leq Ce^{-c_0't}.
    \end{align*}
    This concludes that $A_t''(s) \leq Ce^{-c_0't}$ for all $s\in[0,t-1]$, so the same bound holds for $A_t(s)$, and hence $(I_1) = \sum_{s=0}^{t-1} A_t(s) \rightarrow 0$ as $t\rightarrow\infty$. This concludes that $r_\eta(t,\ast)$ is an $L^2$-Cauchy sequence in $t$ which shows the existence of $r_\eta^\ast$.

    \noindent\textbf{Claim for $R_\theta(s)$}
    Lastly, to show the convergence of $R_\theta(t+s,t)$, we check that $R_\theta(t+s,t)$ is a Cauchy sequence in $t$. A sufficient condition to verify is that $\sum_{t=t_0}^\infty|R_\theta(t+s+1,t+1)-R_\theta(t+s,t)| < \infty.$
    Using the high-dimensional representation in Lemma \ref{lm: response_mat_relation_1}, we have
    \begin{align*}
        |R_\theta(t+s+1,t+1)-R_\theta(t+s,t)| &\leq \limsup_{n,d\rightarrow\infty}\frac{1}{d}\Big|\Tr\Big(\gamma\bOmega^{t+s}\cdots\bOmega^{t+2}\Big)-\Tr\Big(\gamma\bOmega^{t+s-1}\cdots\bOmega^{t+1}\Big)\Big|\\
        &= \gamma\limsup_{n,d\rightarrow\infty}\frac{1}{d}\Big|\Tr\Big(\bOmega^{t+s}\cdots\bOmega^{t+2}-\bOmega^{t+s-1}\cdots\bOmega^{t+1}\Big)\Big|\\
        &\leq\gamma\limsup_{n,d\rightarrow\infty}\frac{1}{\sqrt{d}}\pnorm{\bOmega^{t+s}\cdots\bOmega^{t+2}-\bOmega^{t+s-1}\cdots\bOmega^{t+1}}{F}.
    \end{align*}
    Using the same computations as in (\ref{eq:Omega_cauchy_bound}), there exists some constant $C'>0$ such that
    \begin{equation*}
        \gamma\limsup_{n,d\rightarrow\infty}\frac{1}{\sqrt{d}}\pnorm{\bOmega^{t+s}\cdots\bOmega^{t+2}-\bOmega^{t+s-1}\cdots\bOmega^{t+1}}{F} \leq C'r^{t/2}.
    \end{equation*}
    As $r\in(0,1)$, we see that
    \begin{equation*}
        \sum_{t=t_0}^\infty|R_\theta(t+s+1,t+1)-R_\theta(t+s,t)| < \sum_{t=t_0}^\infty C'r^{t/2} < \infty,
    \end{equation*}
    concluding that $R_\theta(t+s,t)$ is Cauchy as a sequence in $t$ which shows the existence of $R_\theta(s) = \lim_{t\rightarrow\infty} R_\theta(t+s,t).$
    
    \item We first check that $R_\theta(t,\diamond) \rightarrow 0$. Using the expression of $R_\theta(t,\diamond)$ in Lemma \ref{lm: response_mat_relation_3}, we have
    \begin{align*}
        R_\theta(t,\diamond) &\leq C\sum_{s=-\infty}^{-1}\underbrace{\limsup_{n,d\rightarrow\infty}\frac{1}{\sqrt{d}}\pnorm{\bOmega^{t-1}\cdots\bOmega^0}{F}}_{(I)}\cdot\underbrace{\limsup_{n,d\rightarrow\infty}\frac{1}{\sqrt{d}}\pnorm{(\bUpsilon)^{-(s+1)}}{F}}_{(II)}.
    \end{align*}
    Using Lemma \ref{lm:Omega_exp_decay}, we have $(I) \leq C\limsup_{n,d\rightarrow\infty} \pnorm{\bOmega^{t-1}}{\op}\ldots \pnorm{\bOmega^0}{\op} \leq Ce^{-c(t-t_0)}$ for $t\geq t_0$. Next, note that with $\lambda_{i,n}$ denoting the $i$th largest eigenvalue of $\X^\top \diag\big(\cT_s(\y)\big)\X$, 
    \begin{align*}
        \frac{1}{d}\pnorm{(\bUpsilon)^{-(s+1)}}{F}^2 = \frac{1}{d}\sum_{i=1}^n\Big(\frac{\lambda_{i,n}}{\lambda_{1,\infty}}\Big)^{-2(s+1)}\leq \frac{1}{d}\Big(\frac{\lambda_{1,n}}{\lambda_{1,\infty}}\Big)^{-2(s+1)} +\frac{n-1}{d}\Big(\frac{\lambda_{2,n}}{\lambda_{1,\infty}}\Big)^{-2(s+1)},
    \end{align*}
    which further implies that
    \begin{equation*}
        (II) = \limsup_{n,d\rightarrow\infty} \frac{1}{d}\pnorm{(\bUpsilon)^{-(s+1)}}{F}^2 \leq C \Big(\frac{\lambda_{2,\infty}}{\lambda_{1,\infty}}\Big)^{-2(s+1)} = Ce^{cs}.
    \end{equation*}
    Combining the bounds of $(I)$ and $(II)$, we have
    \begin{equation*}
        R_\theta(t,\diamond) \leq Ce^{-c(t-t_0)}\sum_{s=-\infty}^{-1}e^{cs} \leq Ce^{-c(t-t_0)} \rightarrow 0
    \end{equation*}
    as $t\rightarrow\infty$. Similar argument holds for $R_\eta(t,\diamond)$.

    We now check the claim for $R_\eta(t,\diamond\diamond)$. By Lemma \ref{lm: response_mat_relation_3}, we have
    \begin{align*}
        R_\eta(t,\diamond\diamond) &= \sum_{s=-\infty}^{-1}\lim_{n,d\rightarrow\infty}\frac{1}{d}\Tr\Big(\frac{1}{\lambda_{1,\infty}}\D^t\X\bOmega^{t-1}\cdots\bOmega^{0}\bUpsilon^{-(s+1)}\X^\top\tilde{\D}^0\Big)\\
        &\leq \frac{1}{\lambda_{1,\infty}}\sum_{s=-\infty}^{-1}\limsup_{n,d\rightarrow\infty}\frac{1}{\sqrt{d}}\pnorm{\tilde{\D}^0\D^t\X\bOmega^{t-1}\cdots\bOmega^{0}}{F}\limsup_{n,d\rightarrow\infty}\frac{1}{\sqrt{d}}\pnorm{\bUpsilon^{-(s+1)}\X^\top}{F}\\
        &\leq C\sum_{s=-\infty}^{-1}\underbrace{\limsup_{n,d\rightarrow\infty}\pnorm{\bOmega^{t-1}}{\op}\ldots\pnorm{\bOmega^{0}}{\op}\frac{1}{\sqrt{d}}\pnorm{\tilde{\D}^0}{F}}_{(I)} \cdot \underbrace{\limsup_{n,d\rightarrow\infty}\frac{1}{\sqrt{d}}\pnorm{\bUpsilon^{-(s+1)}}{F}}_{(II)},
    \end{align*}
    where we recall that $\tilde{\D}^0=\diag(\cT_s'(\y)\circ\varphi'(\bbeta^*,\z)\circ\bbeta^0)$, and in the last line we use the fact that $\pnorm{\X}{\op} \vee \pnorm{\D^t}{\op}\leq C$ almost surely for large enough $n,d$.
    By similar arguments as in the above $R_\theta(t,\diamond)$ case and further $\limsup_{n,d}\pnorm{\tilde \D^0}{F}/\sqrt{d} \leq C\limsup_{n,d\rightarrow\infty} \pnorm{\bbeta^0}{}/\sqrt{d} \leq C$, we can conclude that for $t>t_0$, we have $(I) \leq Ce^{-c(t-t_0)}$ and $(II) \leq Ce^{cs}$, so once again $R_\eta(t,\diamond\diamond)\rightarrow 0$ as $t\rightarrow\infty$, as desired.
    
    \item This has been shown in Lemma \ref{lm: response_mat_relation_1}.
    \end{enumerate}

    Lemma \ref{lm:long-time-properties} shows the existence of $\btheta^\infty$ on the event $\cE_n$, which holds almost surely for large enough $n,d$, so with probability $1$, $\btheta^\infty$ is well-defined for large enough $n,d$. It remains to check that $\frac{1}{d}\sum_{j=1}^d \delta_{\theta^\infty_j} \overset{W_2}{\to} \sP(\theta^\infty)$. For any $t\geq 0$, we have the decomposition
    \begin{align*}
    W_2\Big(\frac{1}{d}\sum_{j=1}^d \delta_{\theta^\infty_j},\sP(\theta^\infty)\Big)\leq \underbrace{W_2\Big(\frac{1}{d}\sum_{j=1}^d \delta_{\theta^\infty_j},\frac{1}{d}\sum_{j=1}^d \delta_{\theta^t_j}\Big)}_{(I)} + \underbrace{W_2\Big(\frac{1}{d}\sum_{j=1}^d \delta_{\theta^t_j},\sP(\theta^t)\Big)}_{(II)} + \underbrace{W_2\Big(\sP(\theta^t),\sP(\theta^\infty)\Big)}_{(III)}.
\end{align*}
By Theorem \ref{thm: discrete dmft asymp}, $(II) \rightarrow 0$ as $n,d\rightarrow \infty$ for any $t$, and part (1) of Definition \ref{def:exp_conv_spec_init} yields that $(III) \rightarrow 0$ as $t\rightarrow \infty$. It remains to note that $\limsup_{n,d\rightarrow\infty}(I) \rightarrow 0$ as $t\rightarrow \infty$, which holds because
\begin{align*}
   \limsup_{n,d\rightarrow\infty}(I) \leq \limsup_{n,d\rightarrow\infty} \frac{1}{\sqrt{d}} \pnorm{\btheta^t-\btheta^\infty}{} \overset{(\ast)}{\leq} Ce^{-ct} \xrightarrow{t} 0,
\end{align*}
where the step $(\ast)$ follows from the computations in (\ref{eq:theta_inf_hd_bound}). The proof is now complete.
\end{proof}

\appendix

\section{State evolution of spectral initialized AMP (Theorem \ref{thm: amp spec})}
\label{sec: spec init amp}
In this section, we prove the state evolution of the spectral initialized AMP given in Theorem \ref{thm: amp spec}. Throughout the remainder of this section, we implicitly assume the same assumptions as in \ref{thm: amp spec}, namely we assume Assumptions \ref{ass: model assumptions}, \ref{ass: spectral initialization}, \ref{ass: ell lipschitz} hold, and that $\psi_\delta'(\lambda_\delta^*) > 0$. The proof strategy is to approximate the spectral initialized AMP in (\ref{eq:spec_init_AMP}), by an AMP algorithm with random initialization. For ease of reference, we first state this latter algorithm. 

Let $\s = (s_i)_{i=1}^n\in \R^n$ and $\v = (v_j)_{j=1}^d\in \R^{d\times 2}$ be auxiliary randomness independent of $\X$, such that for some variables $S\in\R,V\in\R^2$, almost surely their empirical distributions satisfy
\begin{align}\label{eq:auxiliary_converge}
\frac{1}{d}\sum_{j=1}^d \delta_{v_j} \overset{W_2}{\to} \sP(V), \quad \frac{1}{n}\sum_{i=1}^n \delta_{s_i} \overset{W_2}{\to} \sP(S). 
\end{align}

For any integer $T > 0$, we consider the following AMP algorithm indexed from $-T$. For any $i\geq -T$, let  $f_i^T = (f^T_{i,1}, f^T_{i,2}):\R^{2(i+T+1)}\times \R \rightarrow \R^2$ and $g_i^T = (g^T_{i,1}, g^T_{i,2}): \R^{2(i+T)}\times \R^2 \rightarrow\R^2$ be Lipschitz functions, and
\begin{align}\label{eq:random_init_amp}
\begin{split}
    \b^i_T &= \X g_i^T(\a^{-T+1}_T,...,\a^i_T;\v) + \frac{1}{\delta} \sum_{j=-T}^{i-1} f_j^T(\b^{-T}_T,...,\b^j_T;\s)\zeta^T_{i,j} \in \R^{n\times 2},\\
    \a^{i+1}_T &= -\frac{1}{\delta}\X^\top f^T_i(\b^{-T}_T,...,\b^i_T;\s) + \sum_{j=-T}^{i} g^T_j(\a^{-T+1}_T,...,\a^j_T;\v)\xi^T_{i,j} \in \R^{d\times 2},
\end{split}
\end{align}
where the initialization is $\b^{-T}_T = \X g_{-T}^T(\v)$.

Let $\{u^i_T\}_{i\geq -T+1} = \{(u^i_{T,1}, u^i_{T,2})\}_{i\geq -T+1}, \{w^i_T\}_{i\geq -T} = \{(w^i_{T,1}, w^i_{T,2})\}_{i\geq -T}$ be centered Gaussian vectors with covariance
\begin{align}\label{eq:vector_amp_se}
\begin{split}
\E\big[u^{i+1}_T(u^{j+1}_T)^\top\big] &= \frac{1}{\delta}\E\Big[f^T_{i}(w^{-T}_T,...,w^{i}_T;S)f^T_{j}(w^{-T}_T,...,w^{j}_T;S)^\top\Big] \in \R^{2\times 2}, \quad -T\leq j \leq i < \infty,\\
\E\big[w^i_T(w^{j}_T)^\top\big] &=\E\Big[g^T_{i}(u_T^{-T+1},...,u^{i}_T;V)g^T_{j}(u^{-T+1}_T,...,u^{j}_T;V)^\top\Big]\in \R^{2\times 2}, \quad -T\leq j \leq i < \infty,
\end{split}
\end{align}
initialized from $\E[w^{-T}_T (w^{-T}_T)^\top] = \E[g_{-T}^T(V)(g_{-T}^T(V))^\top]$, with $V \sim \sP(V)$. Then the Onsager correction terms in (\ref{eq:random_init_amp}) are defined by
\begin{align}\label{eq:vector_amp_onsager}
\begin{split}
    \xi^T_{ij} &= \Big(\E\Big[\frac{\partial}{\partial w^j_T}f_i^T(w^{-T}_T,...,w^i_T;S)\Big]\Big)^\top \in \R^{2\times 2}, \quad -T\leq j \leq i,\\
    \zeta^T_{ij} &= \Big(\E\Big[\frac{\partial}{\partial u_T^{j+1}}g^T_i(u^{-T+1}_T,...,u^i_T;V)\Big]\Big)^\top\in\R^{2\times 2}, \quad -T\leq j \leq i-1.
\end{split}
\end{align}

The state evolution characterizes the asymptotic law of the AMP iterates as in the following theorem.

\begin{theorem}\label{thm:standard_amp_se} 
 Assume that the nonlinearities $\{f_i^T\},\{g_i^T\}$ are Lipschitz, $\lim_{n,d\rightarrow\infty} n/d = \delta$, and the distribution of $\X\in\R^{n\times d}$ satisfies Assumption \ref{ass: model assumptions}-(3). Suppose $\z$ and $\s$ are independent of $\X$ and satisfy (\ref{eq:auxiliary_converge}) almost surely for some variables $S,V$ such that $\E[S^2], \E[\pnorm{V}{}^2] < \infty$. Then for any fixed integer $m > -T$, almost surely as $n,d \rightarrow \infty$, we have
    \begin{align}
        \begin{split}
            \frac{1}{d}\sum_{j=1}^d \delta_{(\a^{-T+1}_T)_j,...,(\a^{m}_T)_j,(\v)_j} &\overset{W_2}{\to} \mathsf{P}(u^{-T+1}_T,...,u^m_{T},V),\\
            \frac{1}{n}\sum_{i=1}^n \delta_{(\b^{-T}_T)_i,...,(\b^{m}_T)_i,(\s)_i}  &\overset{W_2}{\to}\mathsf{P}(w^{-T}_T,...,w^m_{T},S).
        \end{split}
    \end{align}
\end{theorem}
Theorem \ref{thm:standard_amp_se} follows from standard analysis of AMP state evolution \cite{javanmard2013state}, where we have re-indexed the algorithm to start at $-T$.

\subsection{Construction of artificial AMP}
To connect to the spectral initialized AMP (\ref{eq:spec_init_AMP}), we split the above AMP algorithm into two stages, with negative and positive indices respectively. In the first stage, we choose the nonlinearities so that the iterates approximate power iteration of the matrix $\M_n = \sum_{i=1}^n \mathcal{T}_s(y_i)\x_i\x_i^\top$. The nonlinearities in the second stage are chosen to match the spectral initialized AMP in (\ref{eq:spec_init_AMP}). By Theorem \ref{thm:standard_amp_se} and taking the limit $T\rightarrow\infty$, this allows us to derive the state evolution of spectral initialized AMP (\ref{eq:spec_init_AMP}) in Theorem \ref{thm: amp spec}. 

\paragraph{First Stage}
The first stage consists of AMP iterates $\{\b^{-T}_T, \a_T^{-T+1},\b_T^{-T+1},\ldots,\b_T^{-1},\a^0_T\}$ in (\ref{eq:random_init_amp}). The external randomness $\v\in\R^{d\times 2}$ and $\s \in \R^n$ are specified as
\begin{equation}\label{eq:randomness_v}
    \v = (\v_1,\v_2) = (\alpha \btheta^*+\sqrt{1-\alpha^2}\n,\btheta^*), \quad \s = \z,
\end{equation}
where $0<\alpha<1$ and $\n \sim \N(0,\I_d)$ independent of $\btheta^*$. The nonlinearities $\{g^T_i\}_{-T\leq i\leq -1}$ and $\{f^T_i\}_{-T \leq i \leq -1}$ in this stage are specified as
\begin{align}\label{eq:nonlinearity_first_stage}
\begin{split}
g_{-T}^T(v) &= \Big(\frac{1}{\sqrt{\delta}}v_1,v_2\Big),\\
g_i^T(a_T^{-T+1},...,a_T^i;v) &= \left(\frac{1}{\beta_i^T\sqrt{\delta}}((a_T^i)_1-v_2(\xi_{i-1,i-1}^T)_{21}),v_2\right), \quad -T+1\leq i\leq -1,\\
f_{i}^T(b_T^{-T},...,b_T^{i};s) &= \left(-\delta\cT(\varphi((b_T^i)_2,s))(b_T^i)_1,0\right), \quad -T\leq i\leq -1,
\end{split}
\end{align}
where $\cT(y) = \frac{\cT_s(y)}{\lambda_\delta^*-\cT_s(y)}$ is given in (\ref{def:T_func}), $\beta_0^T = 1$, and $\beta_{i+1}^T = \sqrt{(\xi_{ii}^T)_{21}^2+\E[(u_T^{i+1})_1^2]}$ for $-T\leq i\leq -2$. With the above choice, $g_i^T(a_T^{-T+1},\ldots,a_T^i;v)$ only depends on $a_T^i$ and $f_i^T(b_T^{-T},\ldots,b^i_T;s)$ only depends on $b^i_T$, leading to 
\begin{align}\label{eq:onsager_simplify_1}
\zeta^T_{ij} = 0 \text{ for } -T\leq i\leq -1, j < i-1, \quad \xi_{ij}^T = 0, -T \leq i\leq -1, j < i.
\end{align}


The key result of the first stage is that the spectral estimator $\hat{\btheta}^{s}$ can be approximated based on the final AMP iterate $\a_T^0$, as shown by the following lemma. We present its proof and the subsequent Lemmas \ref{lem: conv gaussian amp} and \ref{lem: conv hd amp} in Appendix \ref{subsec:proof_auxiliary_amp}.
\begin{lemma}\label{lem: AMP stage 1} 
    The final AMP iterate of the first stage satisfies
    \begin{equation}
        \lim_{T\rightarrow\infty}\lim_{n,d\rightarrow\infty} \frac{1}{d} \pnorm{\sqrt{d}\hat{\btheta}^s - \sqrt{\delta}\big((\a_T^0)_1 - (\xi_{-1,-1}^T)_{21}\btheta^*\big)}{}^2 \overset{a.s.}{\to} 0.
    \end{equation}
    Moreover, $(\xi^T_{-1,1})_{21} \rightarrow -a/\sqrt{\delta}$ as $T \rightarrow \infty$, and almost surely, 
    \begin{align*}
        \lim_{n,d\rightarrow\infty} W_2\Big(\frac{1}{d}\sum_{j=1}^d \delta_{\sqrt{d}(\hat{\btheta}^s)_j,(\btheta^*)_j}, \sP(\theta^0,\theta^\ast)\Big) &= 0,\\
        \lim_{T\rightarrow\infty} \lim_{n,d\rightarrow\infty} W_2\Big(\frac{1}{d}\sum_{j=1}^d \delta_{\sqrt{\delta}\big(\a^0_T-(\xi_{-1,-1}^T)_{21}\btheta^*\big)_j,(\btheta^*)_j}, \mathsf{P}(\theta^0,\theta^*)\Big) &= 0,
    \end{align*}
    where $\sP(\theta^0,\theta^\ast)$ is given by the joint distribution of $(\theta^0,\theta^\ast)$ in (\ref{def:dmft_theta0}).
\end{lemma}
\paragraph{Second Stage}
This stage consists of the iterates $\{\b^{0}_T, \a^1_T,\b^{1}_T,\ldots\}$ in (\ref{eq:random_init_amp}). 
The nonlinearities $\{g^T_i\}_{i\geq 0}$ and $\{f^T_i\}_{i\geq 0}$ in this stage are specified as
\begin{align*}
g_0^T\big(a_T^{-T+1},...,a_T^{0};v\big) = g_{0}\Big(\sqrt{\delta}\big((a_T^0)_1 - (\xi_{-1,-1}^T)_{21}v_2\big),v_2\Big) = \Big(\sqrt{\delta}\big((a_T^0)_1 - (\xi_{-1,-1}^T)_{21}v_2\big),v_2\Big),
\end{align*}
and for $i \geq 0$,
\begin{align}
\begin{split}
    g_{i+1}^T(a_T^{-T+1},...,a_T^{i+1};v) &= g_{i+1}\Big(a_T^{1},...,a_T^{i+1};\sqrt{\delta}\big((a_T^0)_1 - (\xi_{-1,-1}^T)_{21}v_2\big),v_2\Big),\\
    f_{i}^T(b_T^{-T},...,b_T^i;s) &= f_{i}(b_T^{0},...,b_T^i;s),
\end{split}
\end{align}
where $\{g_i\}_{i\geq 0}, \{f_i\}_{i\geq 0}$ are the nonlinearities from the spectral initialized AMP in (\ref{eq:spec_init_AMP}). In other words, $\{g^T_i\}_{i\geq 0}, \{f^T_i\}_{i\geq 0}$ are chosen to be the same as $\{g_i\}_{i\geq 0}, \{f_i\}_{i\geq 0}$, except that the argument $\btheta^0 = \sqrt{d}\hat\btheta^s$ in $g_i^T$ is replaced by its approximation $\sqrt{\delta}\big((\a_T^0)_1 - (\xi_{-1,-1}^T)_{21}\btheta^*\big)$ given in Lemma \ref{lem: AMP stage 1}. 

With this choice of nonlinearity and the auxiliary randomness in (\ref{eq:randomness_v}), the AMP iterates in the second stage satisfy: for $i\geq 0$, 
\begin{align}\label{eq:vector_amp_second_stage}
\begin{split}
    \b^i_T &= \X g_i\Big(\a^{1}_T,...,\a^i_T;\sqrt{\delta}\big((\a_T^0)_1 - (\xi_{-1,-1}^T)_{21}\btheta^*\big),\btheta^*\Big) + \frac{1}{\delta} \sum_{j=-T}^{i-1} f_j^T(\b^{-T}_T,...,\b^j_T;\z)\zeta^T_{i,j}, \\
    \a^{i+1}_T &= -\frac{1}{\delta}\X^\top f_i(\b^{0}_T,...,\b^i_T;\z) + \sum_{j=-T}^{i} g^T_j\Big(\a^{-T+1}_T,...,\a^j_T;\sqrt{\delta}\big((\a_T^0)_1 - (\xi_{-1,-1}^T)_{21}\btheta^*\big),\btheta^\ast\Big)\xi^T_{i,j}.
\end{split}
\end{align}
We can make a few simplifications for the Onsager correction terms. For $i\geq0$, $f_i^T(w^{-T}_T,...,w^i_T;z)$ in (\ref{eq:vector_amp_onsager}) does not depend on $w_T^{-T},...,w_T^{-1}$, and similarly, $g_i^T(u^{-T+1}_T,...,u^i_T;v)$ in (\ref{eq:vector_amp_onsager}) does not depend on $u_T^{-T+1},...,u_T^{-1}$, which implies that
\begin{align}\label{eq:onsager_simplify_2}
    \zeta_{ij}^T = 0 \text{ for } i\geq 0, j< -1, \quad\quad
    \xi_{ij}^T = 0 \text{ for }  i\geq 0, j< 0.
\end{align}
Then (\ref{eq:vector_amp_second_stage}) can be simplified as 
\begin{align}
\begin{split}
\label{eq: artifical AMP iterates}
    \b^i_T &= \X g_i\Big(\a^{1}_T,...,\a^i_T;\sqrt{\delta}\big((\a_T^0)_1 - (\xi_{-1,-1}^T)_{21}\btheta^*\big),\btheta^*\Big) + \frac{1}{\delta} \sum_{j=0}^{i-1} f_j(\b^{0}_T,...,\b^j_T;\z)\zeta^T_{i,j} + \frac{1}{\delta}f_{-1}^T(\b^{-T}_T,...,\b^{-1}_T;\z)\zeta^T_{i,-1},\\
    \a^{i+1}_T &= -\frac{1}{\delta}\X^\top f_i(\b^{0}_T,...,\b^i_T;\z) + \sum_{j=0}^{i} g_j\Big(\a^{1}_T,...,\a^j_T;\sqrt{\delta}\big((\a_T^0)_1 - (\xi_{-1,-1}^T)_{21}\btheta^*\big),\btheta^*\Big)\xi^T_{i,j}.
\end{split}
\end{align}
\paragraph{Taking the limit}
The behavior of the spectral initialized AMP can be approximated by that of the second stage in the above two-stage AMP algorithm, upon taking $T \rightarrow \infty$ in the first stage. To quantify this, we first show convergence of the Gaussian variables $\{u_T^i\}_{i\geq0}$ and $\{w_T^i\}_{i\geq 0}$.

\begin{lemma}\label{lem: conv gaussian amp}
    For any fixed $k>0$, it holds as $T\rightarrow \infty$ that
    \begin{align*}
    W_2\Big(\sP(\sqrt{\delta}u_T^0,u_T^1,...,u_T^k),\sP(u^0,u^1,...,u^k)\Big) &\rightarrow 0,\\
        W_2\Big(\sP(w_T^0,w_T^1,...,w_T^k),\sP(w^0,w^1,...,w^k)\Big) &\rightarrow 0.
    \end{align*}
\end{lemma}

As a corollary, using Lemma 6 in \citet{bayati_message_passing_2011} extended to general $\R^k$-valued functions, we also have convergence of the Onsager correction terms
\begin{align}\label{eq:converge_onsager}
    \lim_{T\rightarrow \infty} \xi_{ij}^T \rightarrow \xi_{ij},\quad \lim_{T\rightarrow \infty} \zeta_{ij}^T \rightarrow \zeta_{ij} \text{ for } -1 < j \leq i-1, \quad \lim_{T\rightarrow \infty} \zeta_{i,-1}^T \rightarrow \sqrt{\delta}\zeta_{i,-1} \text{ for } i\geq 0.
\end{align}

We also have convergence of high-dimensional AMP iterates in the high-dimensional limit.

\begin{lemma}\label{lem: conv hd amp}
    For $i\geq 0$, it holds almost surely that
    \begin{align*}
        \lim_{T\rightarrow\infty}\lim_{n,d\rightarrow\infty}\frac{1}{d}\pnorm{(\a_T^{i+1})_1-\a^{i+1}_1}{}^2 \rightarrow 0, &\qquad \lim_{T\rightarrow\infty}\lim_{n,d\rightarrow\infty}\frac{1}{d}\pnorm{(\a_T^{i+1})_2-\a^{i+1}_2}{}^2 \rightarrow 0, \\
        \lim_{T\rightarrow\infty}\lim_{n,d\rightarrow\infty}\frac{1}{n}\pnorm{(\b_T^i)_1-\b^i_1}{}^2 \rightarrow 0, &\qquad \lim_{T\rightarrow\infty}\lim_{n,d\rightarrow\infty}\frac{1}{n}\pnorm{(\b_T^i)_2-\b^i_2}{}^2 \rightarrow 0.
    \end{align*}
\end{lemma}
We pause to give some intuition of the above result. Comparing the AMP iterates in (\ref{eq: artifical AMP iterates}) and (\ref{eq:spec_init_AMP}), the result of Lemma \ref{lem: conv hd amp} is rather intuitive except for why the additional Onsager correction term $\frac{1}{\delta}f_{-1}^T(\b^{-T}_T,...,\b^{-1}_T;\s)\zeta^T_{i,-1}$ in (\ref{eq: artifical AMP iterates}) converges to $\left(-\frac{1}{\lambda_\delta^*}\bZ_s \X\btheta^0, \mathbf{0}\right)\zeta_{i,-1}$ in (\ref{eq:spec_init_AMP}) in the high dimensional limit followed by $T\rightarrow\infty$. Since $f_{-1,2}^T = 0$ as in (\ref{eq:nonlinearity_first_stage}) and $\zeta_{i,-1}^T\ \rightarrow \sqrt{\delta}\zeta_{i,-1}$ as in (\ref{eq:converge_onsager}), it suffices to check that 
\begin{align}\label{eq: conv first onsager}
\lim_{T\rightarrow\infty}\lim_{n,d\rightarrow\infty}\frac{1}{n} \biggpnorm{\frac{1}{\delta}f_{-1,1}^T(\b^{-T}_T,...,\b^{-1}_T;\s) + \frac{1}{\sqrt{\delta}\lambda_\delta^*}\bZ_s \X\btheta^0}{}^2 = 0.
\end{align}
By the definition of $f_{-1,1}^T$ in (\ref{eq:nonlinearity_first_stage}), we know that $\frac{1}{\delta}f_{-1,1}^T(\b^{-T}_T,...,\b^{-1}_T;\s) = -\bZ(\b_T^{-1})_1$, where $\bZ = \diag(\cT(\y))$, and using (\ref{eq:onsager_simplify_1}) and $(\zeta^T_{-1,-2})_{11} = \frac{1}{\beta_{-1}^T\sqrt{\delta}}$, 
\begin{align*}
 (\b_T^{-1})_1 &= \X g_{-1,1}^T(\a_T^{-T+1},...,\a_T^{-1};\v) + \frac{1}{\delta}f_{-2,1}^T(\b_T^{-T},...,\b_{T}^{-2};\s)(\zeta_{-1,-2}^T)_{11}\\
 &=\frac{1}{\beta_{-1}^T\sqrt{\delta}}\X\big((\a_T^{-1})_1-\btheta^*(\xi_{-2,-2}^T)_{21}\big) - \frac{1}{\beta_{-1}^T\sqrt{\delta}}\bZ(\b_T^{-2})_1.
\end{align*}
Using $\beta_{-1}^T \rightarrow \frac{1}{\sqrt{\delta}}$ as $T\rightarrow\infty$ (\cite[Lemma B.2]{mondelli_approximate_2022}), the above display implies that
\begin{align*}
\lim_{T\rightarrow\infty}\lim_{n,d\rightarrow\infty}\frac{1}{n}\Bigpnorm{(\b_T^{-1})_1 + \bZ(\b_T^{-2})_1 - \X\big((\a_T^{-1})_1-\btheta^*(\xi^T_{-2,-2})_{21}\big)}{}^2 = 0.
\end{align*}
Using that the AMP iterates of the first stage converge in the high-dimensional limit (\cite[Lemma 5.3]{mondelli_optimal_2022}), we have 
\begin{align*}
\lim_{T\rightarrow\infty}\lim_{n,d\rightarrow\infty}\Big\{\frac{1}{n}\pnorm{(\b_T^{-1})_1-(\b_T^{-2})_1}{}^2, \frac{1}{n}\pnorm{(\a_T^{0})_1-(\a_T^{-1})_1}{}^2\Big\} = 0, \quad \lim_{T\rightarrow\infty} (\xi^T_{-2,-2})_{21} = \lim_{T\rightarrow\infty}(\xi_{-1,-1}^T)_{21},
\end{align*}
which further implies that 
\begin{equation*}
\lim_{T\rightarrow\infty}\lim_{n,d\rightarrow\infty}\frac{1}{n}\Bigpnorm{(\I+\bZ)(\b_T^{-1})_1 - \X\big((\a_T^{0})_1-\btheta^*(\xi^T_{-1,-1})_{21}\big)}{}^2 = 0.
\end{equation*}
Now applying Lemma \ref{lem: AMP stage 1}, we can conclude
\begin{equation*}\lim_{T\rightarrow\infty}\lim_{n,d\rightarrow\infty}\frac{1}{n}\Bigpnorm{(\b_T^{-1})_1 - (\I+\bZ)^{-1}\frac{1}{\sqrt{\delta}}\X\btheta^0}{}^2 = 0.
\end{equation*}
Now using that $\frac{1}{\delta}f_{-1,1}^T(\b^{-T}_T,...,\b^{-1}_T;\z) = -\bZ(\b_T^{-1})_1$ and $(\I+\bZ)^{-1}\bZ = \frac{1}{\lambda_\delta^*}\bZ_s$, we get the desired (\ref{eq: conv first onsager}).

With these results, we are now ready to prove Theorem \ref{thm: amp spec}. 
\begin{proof}[Proof of Theorem \ref{thm: amp spec}]
We show for any 2-Pseudo Lipschitz function $\psi: \R^{2i+2} \rightarrow \R$, we have
    \begin{equation}
        \frac{1}{d}\sum_{j=1}^d \psi\Big((\a^1)_j,...,(\a^i)_j;(\btheta^0)_j,(\btheta^*)_j\Big) \rightarrow \E\left[\psi(u^1,...,u^i;a\theta^*+u^0_1,\theta^*)\right].
    \end{equation}
    The analogous statement for the $\b^j$ iterates is similar. Let $(u_T^0,u_T^1,...,u_T^i),(u^0,u^1,...,u^i)$ be an optimal $W_2$ coupling. For any integer $T > 0$, we have the decomposition
    \begin{align*}
        \begin{split}
            &\limsup_{n,d\rightarrow\infty}\Big|\frac{1}{d}\sum_{j=1}^d \psi\Big((\a^1)_j,\ldots,(\a^i)_j;(\btheta^0)_j,(\btheta^*)_j\Big) - \E\left[\psi(u^1,\ldots,u^i;a\theta^*+u^0_1,\theta^*)\right] \Big|\\
            &\leq \underbrace{\limsup_{n,d\rightarrow\infty}\Big|\frac{1}{d}\sum_{j=1}^d \psi\Big((\a^1)_j,\ldots,(\a^i)_j;(\btheta^0)_j,(\btheta^*)_j\Big) - \frac{1}{d}\sum_{j=1}^d \psi\Big((\a^1_T)_j,\ldots,(\a^i_T)_j;\sqrt{\delta}((\a_T^0)_1 - (\xi_{-1,-1}^T)_{21}\btheta^*)_j,(\btheta^*)_j\Big)\Big|}_{(I)}\\
            &+\underbrace{\limsup_{n,d\rightarrow\infty}\Big|\frac{1}{d}\sum_{j=1}^d \psi\Big((\a^1_T)_j,...,(\a^i_T)_j;\sqrt{\delta}((\a_T^0)_1 - (\xi_{-1,-1}^T)_{21}\btheta^*)_j,(\btheta^*)_j\Big) - \E\left[\psi(u^1_T,...,u^i_T;\sqrt{\delta}((u^0_T)_1) - (\xi^T_{-1,-1})_{21}\theta^*),\theta^*)\right]\Big|}_{(II)}\\
            &+\underbrace{\Big|\E\left[\psi(u^1_T,...,u^i_T;\sqrt{\delta}((u^0_T)_1)-(\xi^T_{-1,-1})_{21}\theta^*,\theta^*)\right] - \E\left[\psi(u^1,...,u^i;a\theta^*+u^0_1,\theta^*)\right]\Big|}_{(III)}.
        \end{split}
    \end{align*}
    Here $\{\a^i\}_{i\geq 0}$ are given by (\ref{eq:spec_init_AMP}) with respect to the randomness of $(\btheta^\ast, \X,\z)$, and $\{\a^i_T\}_{i\geq 0}$
    are given by (\ref{eq:random_init_amp}) with respect to the randomness of $(\btheta^\ast, \n, \X,\z)$, coupled with the same $(\btheta^\ast, \X,\z)$ independent of $\n$.
    
    We will show that almost surely as $T\rightarrow\infty$, terms $(I)$-$(III)$ converge to $0$.
    First note that for any finite $T > 0$, $(II) = 0$ by Theorem \ref{thm:standard_amp_se}. For $(I)$, we have by triangle inequality
    \begin{align}
        \begin{split}
            &\Big|\frac{1}{d}\sum_{j=1}^d \psi\Big((\a^1)_j,...,(\a^i)_j;(\btheta^0)_j,(\btheta^*)_j\Big) - \frac{1}{d}\sum_{j=1}^d \psi\Big((\a^1_T)_j,...,(\a^i_T)_j;\sqrt{\delta}((\a_T^0)_1 - (\xi_{-1,-1}^T)_{21}\btheta^*)_j,(\btheta^*)_j\Big)\Big| \\
            &\leq \frac{1}{d}\sum_{j=1}^d\Big|\psi\Big((\a^1)_j,...,(\a^i)_j;(\btheta^0)_j,(\btheta^*)_j\Big) - \psi\Big((\a^1_T)_j,...,(\a^i_T)_j;\sqrt{\delta}((\a_T^0)_1 - (\xi_{-1,-1}^T)_{21}\btheta^*)_j,(\btheta^*)_j\Big)\Big|.
        \end{split}
    \end{align}
Since $\psi$ is 2-PL, we have for each $j\in[d]$ that
\begin{align}
    \begin{split}
        &\Big|\psi\Big((\a^1)_j,...,(\a^i)_j;(\btheta^0)_j,(\btheta^*)_j\Big) -  \psi\Big((\a^1_T)_j,...,(\a^i_T)_j;\sqrt{\delta}((\a_T^0)_1 - (\xi_{-1,-1}^T)_{21}\btheta^*)_j,(\btheta^*)_j\Big)\Big| \\
        &\leq C\left(1+\pnorm{\alpha_j}{}+\pnorm{\alpha_j^T}{}\right)\pnorm{\alpha_j-\alpha_j^T}{},
    \end{split}
\end{align}
where $\alpha_j,\alpha_j^T$ are length $2(2i+1)$ vectors defined by
\begin{align*}
    \alpha_j &= \Big((\a^1)_{1j},(\a^1)_{2j},...,(\a^i)_{1j},(\a^i)_{2j},(\btheta^0)_j,(\btheta^*)_j\Big)\\
    \alpha_j^T &= \Big((\a^1_T)_{1j},(\a^1_T)_{2j},...,(\a^i_T)_{1j},(\a^i_T)_{2j},\sqrt{\delta}((\a_T^0)_1 - (\xi_{-1,-1}^T)_{21}\btheta^*)_j,(\btheta^*)_j\Big).
\end{align*}
Define $\Xi_T, \Theta_T \in \R_+^d$ as 
\begin{align*}
\Xi_T = \{\pnorm{\alpha_j-\alpha_j^T}{}\}_{j=1}^d, \quad \Theta_T = \{1+\pnorm{\alpha_j}{}+\pnorm{\alpha_j^T}{}\}_{j=1}^d,
\end{align*}
then
\begin{align*}
(I) \leq \limsup_{n,d\rightarrow\infty} \frac{C}{d}\langle\Xi_T,\Theta_T\rangle \leq \limsup_{n,d\rightarrow\infty}\frac{C}{d}\pnorm{\Xi_T}{}\pnorm{\Theta_T}{}.
\end{align*}
It remains to show that, almost surely, $\lim_{T\rightarrow\infty}\lim_{n,d\rightarrow\infty}\frac{1}{\sqrt{d}}\pnorm{\Xi_T}{}=0$ and $\limsup_{T\rightarrow\infty}\limsup_{n,d\rightarrow\infty}\frac{1}{\sqrt{d}}\pnorm{\Theta_T}{} < \infty$. For the first claim, up to a constant $C > 0$ depending on $i$, we have almost surely that
\begin{align*}
    \frac{1}{d}\pnorm{\Xi_T}{}^2 \leq C\Big[\frac{1}{d}\sum_{k=1}^i\pnorm{(\a_T^k)_1-\a^k_1}{}^2 + \frac{1}{d}\sum_{k=1}^i\pnorm{(\a_T^k)_2-\a^k_2}{}^2 + \frac{1}{d}\pnorm{\btheta^0 - \sqrt{\delta}((\a_T^0)_1 - (\xi_{-1,-1}^T)_{21}\btheta^*)}{}^2\Big] \rightarrow 0,
\end{align*}
where we apply Lemmas \ref{lem: AMP stage 1} and \ref{lem: conv hd amp}. For the second claim, note that
\begin{align*}
\frac{1}{\sqrt{d}} \pnorm{\Theta_T}{} &= \frac{1}{\sqrt{d}} \sqrt{\sum_{j=1}^d (1+\pnorm{\alpha_j}{} +\pnorm{\alpha_j^T}{})^2}\leq C \sqrt{1 + \frac{1}{d}\sum_{j=1}^d\pnorm{\alpha_j}{}^2 + \frac{1}{d}\sum_{j=1}^d\pnorm{\alpha_j^T}{}^2}\\
&\leq C \sqrt{1 + \frac{1}{d}\sum_{j=1}^d\pnorm{\alpha_j-\alpha_j^T}{}^2 + \frac{1}{d}\sum_{j=1}^d\pnorm{\alpha_j^T}{}^2}.
\end{align*}
By Lemma \ref{lem: conv hd amp}, we have $\lim_{T\rightarrow\infty}\lim_{n,d\rightarrow\infty}\frac{1}{d}\sum_{j=1}^d\pnorm{\alpha_j-\alpha_j^T}{}^2 \rightarrow 0$ almost surely. By Theorem \ref{thm:standard_amp_se}, we have $\lim_{n,d\rightarrow \infty}\frac{1}{d}\sum_{j=1}^d\pnorm{\alpha_j^T}{}^2 = \E[\pnorm{\alpha^T}{}^2]$ almost surely for any $T > 0$, where 
\begin{align*}
\alpha^T := \Big((u^1_T)_{1},(u^1_T)_{2},\ldots,(u^i_T)_{1},(u^i_T)_{2}, \sqrt{\delta}\big((u^0_T)_1 - (\xi_{-1,-1}^T)_{21}\theta^*\big) ,\theta^*\Big).
\end{align*}
Moreover, we have $\E[\pnorm{\alpha^T - \alpha}{}^2] \rightarrow 0$ as $T\rightarrow\infty$ by Lemma \ref{lem: conv gaussian amp}, hence
\begin{equation}
\lim_{T\rightarrow\infty}\lim_{n,d\rightarrow \infty}\frac{1}{d}\sum_{j=1}^d\pnorm{\alpha_j^T}{2}^2 = \E[\pnorm{\alpha}{}^2],
\end{equation}
where $\alpha := ((u^1)_{1},(u^1)_{2},...,(u^i)_{1},(u^i)_{2}, \theta^0 ,\theta^*)$.
This concludes the second claim, thereby showing that $(I) \rightarrow 0$ almost surely as $T\rightarrow\infty$. 

For $(III)$, we have
\begin{align*}
&\Big|\E\Big[\psi\Big(u^1_T,...,u^i_T;\sqrt{\delta}\big((u^0_T)_1  - (\xi^T_{-1,-1})_{21}\theta^*\big),\theta^*\Big)\Big] - \E\Big[\psi\Big(u^1,...,u^i;a\theta^*+u^0_1,\theta^*\Big)\Big]\Big|\\
    &\leq \E \Big[\Big|\psi\Big(u^1_T,...,u^i_T;\sqrt{\delta}\big((u^0_T)_1 - (\xi^T_{-1,-1})_{21}\theta^*\big),\theta^*\Big) - \psi\Big(u^1,...,u^i;a\theta^*+u^0_1,\theta^*\Big)\Big|\Big]\\
    &\leq C\E\Big[(1+\pnorm{\alpha^T}{}+\pnorm{\alpha}{})\pnorm{\alpha^T-\alpha}{}\Big]\\
    &\leq C\sqrt{\E\Big[(1+\pnorm{\alpha^T}{}^2+\pnorm{\alpha}{}^2)\Big]} \cdot \sqrt{\E[\pnorm{\alpha^T-\alpha}{}^2]},
\end{align*}
It is easy to see that $\limsup_{T\rightarrow\infty} \sqrt{\E\left[(1+\pnorm{\alpha^T}{}^2+\pnorm{\alpha}{}^2)\right]}$ is bounded, and $\lim_{T\rightarrow\infty}\sqrt{\E[\pnorm{\alpha^T-\alpha}{}^2]} = 0$ by Lemma \ref{lem: conv gaussian amp}, thereby showing $(III)\rightarrow 0$ as $T\rightarrow\infty$. The proof is complete.
\end{proof}

\subsection{Proof of auxiliary lemmas}\label{subsec:proof_auxiliary_amp}
\subsubsection{Proof of Lemma \ref{lem: AMP stage 1}}
We check that the first stage of our AMP is equivalent to the first stage of the artificial GAMP given in Section 5 and B.1 in \citet{mondelli_approximate_2022}. More precisely, the GAMP algorithm is: for $t\geq 0$, 
\begin{align}\label{eq:gamp}
\begin{split}
\tilde{\x}^{t+1} &= \frac{1}{\sqrt{\delta}} \X^\top\tilde{h}_t(\tilde{\u}^t,\y) - \tilde{c}_t \tilde{f}_t(\tilde{\x}^t),\\
\tilde{\u}^{t+1} &= \frac{1}{\sqrt{\delta}}\X\tilde{f}_{t+1}(\tilde{\x}^{t+1}) - \tilde{b}_{t+1}\tilde{h}_t(\tilde{\u}^t,\y),
\end{split}
\end{align}
with initialization $\tilde{\x}^0 = \alpha \btheta^\ast + \sqrt{1-\alpha^2}\n$ for some $\alpha\in(0,1)$ and $\tilde{\u}^0 = \frac{1}{\delta} \X \tilde{f}_0(\tilde{\x}^0) = \frac{1}{\delta}\X \tilde{\x}^0$ with $\tilde{f}_0(x) = x$. Here the Onsager coefficients are
\begin{align*}
\tilde{c}_t = \E[\tilde{h}_t'(\tilde{U}_t, Y)], \quad \tilde{b}_{t+1} = \frac{1}{\delta}\E[\tilde{f}_{t+1}'(\tilde{X}_{t+1})],
\end{align*}
where, with $\omega^\ast,W_{\tilde{U},t}, W_{\tilde{X},t}$ denoting independent standard Gaussian variables and $Y = \varphi(\omega^\ast,z)$, the state evolution variables are defined by 
\begin{equation}\label{eq:power_iteration_SE}
\begin{gathered}
\tilde{U}_t = \mu_{\tilde{U},t} \omega^\ast + \sigma_{\tilde{U},t}W_{\tilde{U},t}, \quad \tilde{X}_{t} = \mu_{\tilde{X},t} \theta^\ast + \sigma_{\tilde{X},t}W_{\tilde{X},t},
\end{gathered}
\end{equation}
and the state evolution coefficients $\{\mu_{\tilde U, t}\}, \{\mu_{\tilde X,t}\}, \{\sigma_{\tilde U,t}\}, \{\sigma_{\tilde X,t}\}$ are given by
\begin{equation}\label{eq:power_iteration_coef}
\begin{gathered}
\mu_{\tilde U,t} = \frac{1}{\sqrt{\delta}}\E[\theta^\ast \tilde f_t(\tilde X_t)],\quad \sigma_{\tilde U,t}^2 = \frac{1}{\delta}\E[\tilde f_t(\tilde X_t)^2] - \mu_{\tilde U,t}^2, \\
\mu_{\tilde X,t+1} = \sqrt{\delta}\E[\omega^\ast \tilde h_t(\tilde U_t, Y)] - \E[\tilde h_t'(\tilde U_t,Y)]\E[\theta^\ast \tilde f_t(\tilde X_t)],\quad \sigma_{\tilde X,t+1}^2 = \E[\tilde h_t^2(\tilde U_t,Y)], 
\end{gathered}
\end{equation}
initialized from $\mu_{\tilde X,0} = \alpha$ and $\sigma^2_{\tilde X,0} = 1-\alpha^2$. The choices for the nonlinearities are
\begin{align}\label{eq:power_iter_nonlinearity}
\tilde{h}_t(u,y) = \sqrt{\delta}u\cT(y), \quad \tilde{f}_t(x) = \frac{x}{\beta_t}, 
\end{align}
with $\cT(y) = \frac{\cT_s(y)}{\lambda^\ast_\delta - \cT_s(y)}$ given in (\ref{def:T_func}) and $\beta_t = \sqrt{\mu_{\tilde{X},t}^2 + \sigma_{\tilde{X},t}^2}$.

Specifically, we show the following correspondence for $i=-T,-T+1,\dots,-1,$
\begin{equation}\label{eq:amp_gamp_correspondence}
\begin{gathered}
    (\b_T^i)_1 = \Tilde{\u}^{i+T}, \quad (w_T^i)_1 \overset{d}{=} \Tilde{U}_{i+T}\\
    (\a_T^{i+1})_1 - \btheta^* (\xi^T_{i,i})_{21} = \Tilde{\x}^{i+T+1}, \quad (u_T^{i+1})_1-(\xi^T_{i,i})_{21}\theta^* \overset{d}{=} \Tilde{X}_{i+T+1}\\
    \beta_i^T = \beta_{i+T},
\end{gathered}
\end{equation}
where the right side is the GAMP notation from \citet{mondelli_approximate_2022}. With these equivalences, the desired conclusion follows from Theorem 1 in \citet{mondelli_approximate_2022}.

We show (\ref{eq:amp_gamp_correspondence}) by induction.

\textbf{Base Case ($i = -T$)} 
Let $\v = (\v_1,\v_2) = (\alpha\btheta^*+\sqrt{1-\alpha^2}\n,\btheta^*)$ be as in (\ref{eq:randomness_v}), $\bZ = \diag(\cT(y_1),\ldots, \cT(y_n))$ with $\cT(y) = \frac{\cT_s(y)}{\lambda_\delta^*-\cT_s(y)}$ as given by (\ref{def:T_func}). The first iterates of the artificial GAMP (\ref{eq:gamp}) are
\begin{align}
    \Tilde{\u}^{0} = \frac{1}{\sqrt{\delta}}\X\v_1, \quad \Tilde{\x}^{1} = \frac{1}{\sqrt{\delta}}\X^\top\bZ\X\v_1 - \sqrt{\delta}\E[\cT(Y)]\v_1,
\end{align}
where the expectation is over $Y = \varphi(G,z)$ with $G\sim \N(0,1)$ independent of $z \sim \sP(z)$.
The SE iterates of the GAMP are
\begin{align*}
    \Tilde{U}_0 = \mu_{\Tilde{U},0} \omega^* +\sigma_{\Tilde{U},0} W_{\Tilde{U},0}, \quad \Tilde{X}_1 = \mu_{\Tilde{X},1} \theta^* +\sigma_{\Tilde{X},1} W_{\Tilde{X},1},
\end{align*}
where $\theta^* \sim \sP(\theta^*)$ and $\omega^*,W_{\Tilde{U},0},W_{\Tilde{X},1}$ are independent standard normals and
\begin{equation}\label{eq:gamp_se_init}
\begin{aligned}
    \mu_{\tilde{U},0} = \frac{\alpha}{\sqrt{\delta}}, &\qquad \sigma_{\tilde{U},0}^2 = \frac{1-\alpha^2}{\delta},\\
    \mu_{\tilde{X},1} = \frac{\alpha}{\sqrt{\delta}}, &\qquad \sigma_{\tilde{X},1}^2 = \alpha^2\E[\cT(Y)^2(\omega^*)^2]+(1-\alpha^2)\E[\cT(Y)^2],
\end{aligned}
\end{equation}
using Equation B.2 in \citet{mondelli_approximate_2022}.
In the above notation, $Y=\varphi(\omega^*,z)$ where $z\sim \sP(z)$. 

To see the correspondence involving $\b^{-T}_T$ and $w^{-T}_T$, we have by the AMP iterate (\ref{eq:random_init_amp}) that
\begin{equation}
    \b_T^{-T} = g_{-T}^T(\v) = \Big(\frac{1}{\sqrt{\delta}}\X\v_1,\X\btheta^*\Big),
\end{equation}
which verifies $(\b_T^{-T})_1 = \Tilde{\u}^0.$ We have the centered Gaussian $w_T^{-T}$ with covariance
\begin{equation}
    \E[w_T^{-T}(w_T^{-T})^\top] = \begin{pmatrix}
        \frac{1}{\delta} & \frac{\alpha}{\sqrt{\delta}}\\
        \frac{\alpha}{\sqrt{\delta}} & 1
    \end{pmatrix}.
\end{equation}
This implies $\big((w_T^{-T})_1,(w_T^{-T})_2\big) \overset{d}{=} \big(\frac{\alpha}{\sqrt{\delta}} \omega^* + \sqrt{\frac{1-\alpha^2}{\delta}}W_{w,-T},\omega^\ast\big)$, where $\omega^*$ and $W_{w,-T}$ are independent standard Gaussian variables, verifying that $(w_T^{-T})_1 \overset{d}{=} \Tilde{U}_{0}$. 

To see the correspondence involving $\a^{-T+1}_T$ and $u^{-T+1}_T$, we have by the AMP iterate (\ref{eq:random_init_amp}) that
\begin{align*}
    \a_{T}^{-T+1} = \left(\frac{1}{\sqrt{\delta}}\X^\top\bZ\X\v_1,\mathbf{0}\right) + \left(\frac{1}{\sqrt{\delta}}\v_1,\btheta^*\right)\xi_{-T,-T}^T,
\end{align*}
where the Onsager correction $\xi_{-T,-T}^T$ has the form
\begin{equation*}
    \begin{pmatrix}
        \E\left[\frac{\partial}{\partial(w_T^{-T})_1}f_{-T,1}(w_T^{-T};z)\right] & 0\\
        \E\left[\frac{\partial}{\partial(w_T^{-T})_2}f_{-T,1}(w_T^{-T};z)\right] & 0
    \end{pmatrix}.
\end{equation*}
Using the definition in (\ref{eq:nonlinearity_first_stage}) and $(w^{-T}_T)_2 \overset{d}{=} \omega^\ast$, we can compute the top left term as
\begin{align*}
     \E\left[\frac{\partial}{\partial(w_T^{-T})_1}f_{-T,1}(w_T^{-T};z)\right] = -\delta\E[\cT(Y)],
\end{align*}
which implies that
\begin{align*}
    (\a_T^{-T+1})_1 - \btheta^* (\xi^T_{-T,-T})_{21} = \frac{1}{\sqrt{\delta}}\X^\top\bZ\X\v_1 - \sqrt{\delta}\E[\cT(Y)]\v_1 = \Tilde{\x}^{1}.
\end{align*}
We also have the centered Gaussian variable $(u_T^{-T+1})_1$ with variance
\begin{align*}
\E[(u_T^{-T+1})_1^2]  = \delta\E[(\cT(Y)(w_T^{-T})_1)^2] = \alpha^2\E[\cT(Y)^2(\omega^*)^2]+(1-\alpha^2)\E[\cT(Y)^2].
\end{align*}
This matches $\sigma_{\tilde{X},1}^2$ in (\ref{eq:gamp_se_init}), hence to show $(u_T^{-T+1})_1-(\xi^T_{-T,-T})_{21}\theta^* \overset{d}{=} \tilde{X}_1 \overset{d}{=} \mu_{\tilde{X},1}\theta^\ast + \sigma_{\tilde{X},1}W_{\tilde{X},1}$, it suffcies to verify $(\xi^T_{-T,-T})_{21}=-\frac{\alpha}{\sqrt{\delta}}$. Using Gaussian integration by parts and that $(w_T^{-T})_2|(w_T^{-T})_1 \sim \N(\alpha\sqrt{\delta}(w_T^{-T})_1,1-\alpha^2)$, we can compute
\begin{align*}
    (\xi^T_{-T,-T})_{21} = \E\left[\frac{\partial}{\partial(w_T^{-T})_2}f_{-T,1}(w_T^{-T};z)\right] = \frac{1}{1-\alpha^2}\E\left[\left((w_T^{-T})_2-\alpha\sqrt{\delta}(w_T^{-T})_1\right)f_{-T,1}(w_T^{-T};z)\right].
\end{align*}
Now using $(w_T^{-T})_2 = \omega^*$ and $(w_T^{-T})_1 = \frac{\alpha}{\sqrt{\delta}} \omega^* + \sqrt{\frac{1-\alpha^2}{\delta}}W_{w,-T}$, we can compute that
\begin{align*}
    (\xi_{-T,-T}^T)_{21} = -\sqrt{\delta}\alpha(\E[\cT(Y)(\omega^*)^2] - \E[\cT(Y)]) = -\frac{\alpha}{\sqrt{\delta}},
\end{align*}
where the last line uses that $\E[\cT(Y)(\omega^*)^2] - \E[\cT(Y)] = \frac{1}{\delta}$, which was computed in Equation 5.38 in \citet{mondelli_optimal_2022}. Lastly, $\beta_{-T}^T = \beta_0 = 1$ by definition, establishing the base case.

\noindent\textbf{Inductive Step}
Suppose that the equivalences (\ref{eq:amp_gamp_correspondence}) hold at the $(i-1)$-th iterate for some $i\geq -T+1$. We show that they also hold for the $i$-th iterate. The GAMP iterates (\ref{eq:gamp}) at this stage satisfy 
\begin{align}\label{eq:gamp_artificial}
\begin{split}
    \Tilde{\u}^{i+T} &= \frac{1}{\beta_{i+T}\sqrt{\delta}}\X\Tilde{\x}^{i+T} - \frac{1}{\sqrt{\delta}\beta_{i+T}}\bZ \Tilde{\u}^{i-1+T},\\
    \tx^{i+T+1} &= \X^\top\bZ\tu^{i+T} - \frac{\sqrt{\delta}}{\beta_{i+T}}\E[\cT(Y)]\tx^{i+T}.
\end{split}
\end{align}
Under the nonlinearity choice (\ref{eq:power_iter_nonlinearity}), the corresponding state evolution variables are the same as in (\ref{eq:power_iteration_SE}): 
\begin{align}\label{eq:gamp_artificial_dist}
\begin{split}
    \Tilde{U}_{i+T} &= \mu_{\Tilde{U},i+T} \omega^* +\sigma_{\Tilde{U},i+T} W_{\Tilde{U},i+T}\\
    \Tilde{X}_{i+T+1} &= \mu_{\Tilde{X},i+T+1} \theta^* +\sigma_{\Tilde{X},i+T+1} W_{\Tilde{X},i+T+1},
\end{split}
\end{align}
where $\theta^* \sim \sP(\theta^*)$ and $\omega^*,W_{\Tilde{U},i+T},W_{\Tilde{X},i+T+1}$ are independent standard normals, and the coefficients in (\ref{eq:power_iteration_SE}) are given by
\begin{align}\label{eq:gamp_artificial_se}
\begin{split}
    \mu_{\Tilde{U},i+T} = \frac{\mu_{\Tilde{X},i+T}}{\sqrt{\delta}\beta_{i+T}}, &\qquad \sigma_{\Tilde{U},i+T}^2 = \frac{\sigma_{\Tilde{X},i+T}^2}{\delta\beta_{i+T}^2},\\
    \mu_{\Tilde{X},i+T+1} = \frac{\mu_{\Tilde{X},i+T}}{\sqrt{\delta}\beta_{i+T}}, &\qquad \sigma_{\Tilde{X},i+T+1}^2 = \mu_{\Tilde{X},i+T}^2\E[\cT(Y)^2\omega^*]+\sigma_{\Tilde{X},i+T}^2\E[\cT(Y)^2].
\end{split}
\end{align}
We first establish the correspondence involving $(\b_T^i)_1$ and $(w_T^i)_1$. By the AMP iterate in (\ref{eq:random_init_amp}), we have
\begin{align*}
    \b_T^i = \left(\frac{1}{\beta_i^T\sqrt{\delta}}\X\big((\a_T^i)_1-\btheta^*(\xi_{i-1,i-1}^T)_{21}\big) - \frac{1}{\sqrt{\delta}\beta_{i}^T}\bZ \b_T^{i-1},\X\btheta^*\right).
\end{align*}
Then using the inductive assumption that $(\a^i_T)_1 - \btheta^\ast(\xi^T_{i-1,i-1})_{21} = \tilde\x^{i+T}$, we have verified that $(\b^i_T)_1 = \tu^{i+T}$ with the latter given in (\ref{eq:gamp_artificial}). The corresponding Gaussian variable $w_i^T$ has covariance
\begin{align*}
    \E[w_T^i(w_T^i)^\top] &= \begin{pmatrix}
        \E\big[\frac{1}{\beta_i^T\sqrt{\delta}}\big(u_T^i-(\xi_{i-1,i-1}^T)_{21}\theta^*)\big)\big]^2 & \E\big[\frac{1}{\beta_i^T\sqrt{\delta}}\big(u_T^i-(\xi_{i-1,i-1}^T)_{21}\theta^*\big)\theta^*\big]\\
        \E\big[\frac{1}{\beta_i^T\sqrt{\delta}}\big(u_T^i-(\xi_{i-1,i-1}^T)_{21}\theta^*\big)\theta^*\big] & 1
    \end{pmatrix}\\
    &= \begin{pmatrix}
        \frac{1}{\beta_{i+T}^2\delta} (\mu_{\Tilde{X},i+T}^2 + \sigma_{\Tilde{X},i+T}^2)& \frac{1}{\beta_{i+T}\sqrt{\delta}} \mu_{\Tilde{X},i+T}\\
        \frac{1}{\beta_{i+T}\sqrt{\delta}} \mu_{\Tilde{X},i+T} & 1
    \end{pmatrix},
\end{align*}
where in the second equality we use the inductive hypothesis. This implies 
\begin{equation}\label{eq: wT decomposition}
    \big((w_T^i)_1,(w_T^i)_2\big) \overset{d}{=}\Big( \frac{\mu_{\Tilde{X},i+T}}{\sqrt{\delta}\beta_{i+T}} \omega^* + \frac{\sigma_{\Tilde{X},i+T}}{\sqrt{\delta}\beta_{i+T}}W_{w,i}, \omega^\ast\Big),
\end{equation}
with $\omega^\ast, W_{w,i}$ being independent standard Gaussian variables, which verifies that $(w_T^i)_1 \overset{d}{=} \Tilde{U}_{i+T}.$

Next we establish the correspondence involving $(\a_T^{i+1})_1$ and $(u_T^{i+1})_1$. By the AMP iterate in (\ref{eq:random_init_amp}), we have
\begin{align}
    \notag(\a_T^{i+1})_1 &= -\frac{1}{\delta}\X^\top f_{i,1}^T(\b_{T}^{-T},...,\b_T^i;\z)+g_{i,1}^T(\a^{-T+1}_T,...,\a^i_T;\v)(\xi_{ii}^T)_{11} + g_{i,2}^T(\a^{-T+1}_T,...,\a^i_T;\v)(\xi_{ii}^T)_{21}\\
    \notag &=\X^\top \bZ(\b_T^i)_1 + \frac{1}{\beta_i^T\sqrt{\delta}}((\a_T^i)_1-\btheta^*(\xi_{i-1,i-1}^T)_{21})(\xi_{ii}^T)_{11}+\btheta^*(\xi_{ii}^T)_{21}\\
    &=\X^\top\bZ\Tilde{\u}^{i+T} + \frac{1}{\beta_{i+T}\sqrt{\delta}}\tx^{i+T}(\xi_{ii}^T)_{11}+\btheta^*(\xi_{ii}^T)_{21},
\end{align}
where in the last equality, we use the inductive assumptions.
The Onsager correction has the form
\begin{equation}
 \xi_{ii}^T =    \begin{pmatrix}
        \E\left[\frac{\partial}{\partial(w_T^{i})_1}f_{i,1}(w_T^{-T},...,w_T^i;z)\right] & 0\\
        \E\left[\frac{\partial}{\partial(w_T^{i})_2}f_{i,1}(w_T^{-T},...,w_T^i;z)\right] & 0
    \end{pmatrix},
\end{equation}
whose top left term can be computed as
\begin{equation}
    \E\left[\frac{\partial}{\partial(w_T^{i})_1}f_{i,1}(w_T^{-T},...,w_T^i;Z)\right] = -\delta\E[\cT(Y)].
\end{equation}
Plugging this back in shows that $(\a_T^{i+1})_1 - \btheta^*(\xi_{ii}^T)_{21} = \tx^{i+T+1}$ with the latter given in (\ref{eq:gamp_artificial}). We now have the Gaussian variable $(u_T^{i+1})_1$ which has variance 
\begin{equation}\label{eq:var_u}
\frac{1}{\delta}\E[f_{i,1}^T(w_T^{-T},...,w_T^i;Z)^2] = \delta\E[\cT(Y)^2(w_T^i)^2] = \mu_{\Tilde{X},i+T}^2\E[\cT(Y)^2\omega^*]+\sigma_{\Tilde{X},i+T}^2\E[\cT(Y)^2],
\end{equation}
where to get the last equality, we use $(w_T^i)_1 \overset{d}{=} \frac{\mu_{\Tilde{X},i+T}}{\sqrt{\delta}\beta_{i+T}} \omega^* + \frac{\sigma_{\Tilde{X},i+T}}{\sqrt{\delta}\beta_{i+T}}W_{w,i}$ established in (\ref{eq: wT decomposition}) and $Y = \varphi(\omega^\ast,z)$. Let us verify $(u_T^{i+1})_1-(\xi_{ii}^T)_{21}\theta^* \overset{d}{=} \Tilde{X}_{i+T+1} = \mu_{\tilde{X},i+T+1}\theta^\ast + \sigma_{\tilde{X},i+T+1} W_{\tilde{X},i+T+1}$, with $\mu_{\tilde{X},i+T+1}, \sigma_{\tilde{X},i+T+1}$ given in (\ref{eq:gamp_artificial_se}). Note that the variance in (\ref{eq:var_u}) is identical to $\sigma_{\tilde{X},i+T+1}^2$, so it suffices to check that 
 \begin{align}\label{eq:xi_ii_calc}
 (\xi_{ii}^T)_{21} = -\mu_{\tilde X,i+T+1} = -\frac{\mu_{\Tilde{X},i+T}}{\sqrt{\delta}\beta_{i+T}}.
 \end{align}
 Note by the covariance structure of $w^i_T$ in (\ref{eq: wT decomposition}), we know the conditional distribution $(w^i_T)_2|(w^i_T)_1 \sim \N\left(\frac{\mu_{\Tilde{X},i+T}\beta_{i+T}\sqrt{\delta}}{\mu_{\Tilde{X},i+T}^2+\sigma_{\Tilde{X},i+T}^2}(w^i_T)_1,\frac{\sigma_{\Tilde{X},i+T}^2}{\mu_{\Tilde{X},i+T}^2+\sigma_{\Tilde{X},i+T}^2}\right)$.
Then by using Gaussian integration by parts with respect to $(w^i_T)_2|(w^i_T)_1$, we can compute
\begin{align*}
   \notag(\xi_{ii}^T)_{21} &= \E\left[\frac{\partial}{\partial(w_T^i)_2}f_{i,1}(w_T^{-T},...,w_T^{i};Z)\right]\\
   \notag&= \frac{\mu_{\Tilde{X},i+T}^2+\sigma_{\Tilde{X},i+T}^2}{\sigma_{\Tilde{X},i+T}^2}\E\left[\left((w^i_T)_2-\frac{\mu_{\Tilde{X},i+T}\beta_{i+T}\sqrt{\delta}}{\mu_{\Tilde{X},i+T}^2+\sigma_{\Tilde{X},i+T}^2}(w^i_T)_1\right)f_{i,1}(w_T^{-T},...,w_T^{i};Z)\right]\\
   \notag&=-\frac{\mu_{\Tilde{X},i+T}^2+\sigma_{\Tilde{X},i+T}^2}{\sigma_{\Tilde{X},i+T}^2}\E\left[\left((w^i_T)_2-\frac{\mu_{\Tilde{X},i+T}\beta_{i+T}\sqrt{\delta}}{\mu_{\Tilde{X},i+T}^2+\sigma_{\Tilde{X},i+T}^2}(w^i_T)_1\right)\delta\cT(Y)(w_T^i)_1\right]\\
   \notag&=-\sqrt{\delta}\frac{\mu_{\Tilde{X},i+T}}{\beta_{i+T}}\left(\E[\cT(Y)(\omega^*)^2]-\E[\cT(Y)]\right)\\
   &=-\frac{\mu_{\Tilde{X},i+T}}{\sqrt{\delta}\beta_{i+T}},
\end{align*}
which verifies (\ref{eq:xi_ii_calc}). Here the second to last line follows from (\ref{eq: wT decomposition}), and the last equality again follows from the fact $\E[\cT(Y)(\omega^*)^2] - \E[\cT(Y)] = \frac{1}{\delta}$, established in Equation 5.38 in \citet{mondelli_optimal_2022}. 

Lastly, we have 
\begin{align*}
(\beta_i^T)^2 = (\xi^T_{i-1,i-1})^2 + \E[(u^i_T)_1^2] = \E\big[\big((u^i_T)_1 - (\xi^T_{i-1,i-1})_{21}\theta^\ast\big)^2\big] = \E[\tilde{X}_{i+T}^2] = \mu_{\tilde{X},i+T}^2 + \sigma_{\tilde{X},i+T}^2 = \beta_{i+T}^2,
\end{align*}
completing the induction.
\qed

\subsubsection{Proof of Lemma \ref{lem: conv gaussian amp}}
    We will use induction to show that (i) the base case $\sqrt\delta u^0_T \overset{W_2}{\to} u^0$; (ii) For any $k\geq 0$, 
    \begin{align*}
    W_2\left((\sqrt{\delta}u_T^0,u_T^1,...,u_T^k),(u^0,u^1,...,u^k)\right) \rightarrow 0 \Longrightarrow W_2\left((w_T^0,w_T^1,...,w_T^k),(w^0,w^1,...,w^k)\right) \rightarrow 0,
    \end{align*}
    and (iii) For any $k\geq 0$,
    \begin{align*}
    W_2\left((w_T^0,w_T^1,...,w_T^k),(w^0,w^1,...,w^k)\right) \rightarrow 0 \Longrightarrow W_2\left((\sqrt{\delta}u_T^0,u_T^1,...,u_T^{k+1}),(u^0,u^1,...,u^{k+1})\right) \rightarrow 0.  
    \end{align*}
    Note that as we have finite dimensional centered Gaussian processes, we only need to check the covariances converge.

 In the base case, it suffices to check that 
 \begin{align*}
 \delta\E[u_T^0(u_T^0)^\top] \rightarrow \E[u^0(u^0)^\top] = 
 \begin{pmatrix}
 1-a^2 & 0\\
 0 & 0
 \end{pmatrix}.
 \end{align*}
 To see this, first note that by the nonlinearity choice in (\ref{eq:nonlinearity_first_stage}), $(u^0_T)_2 = 0$. Then the claim follows by using $(u^0_T)_1 \overset{d}{=} (\xi_{-1,-1}^T)_{21}\theta^\ast + \tilde{X}_{T} = ((\xi_{-1,-1}^T)_{21} + \tilde\mu_{X,T})\theta^\ast + \tilde\sigma_{X,t}W_{\tilde X,T}$ via (\ref{eq:amp_gamp_correspondence}), $(\xi^T_{-1,-1})_{21} = -\mu_{\tilde X, T}$ via (\ref{eq:xi_ii_calc}), and $\sigma^2_{\tilde X,T} \rightarrow (1-a^2)/\delta$ as $T\rightarrow\infty$ by 
 the fixed point analysis (Lemma B.2) of GAMP in \citet{mondelli_approximate_2022}.

Let us now show $(ii)$, and $(iii)$ will follow from similar arguments. Let $\Sigma_T$ be the covariance matrix of $(w_T^0,w_T^1,...,w_T^k)$ and $\Sigma$ be that of of $(w^0,w^1,...,w^k)$. 
To check $\lim_{T\rightarrow\infty }\Sigma_T = \Sigma$, pick an optimal $W_2$-coupling of $(\sqrt{\delta}u_T^0,u_T^1,...,u_T^k,\theta^\ast), (u^0,u^1,...,u^k,\theta^\ast)$.  For any $i,j \in \{0,...,k\}$, we have
\begin{align*}
    \E[(w^i)_1(w^j)_1] &= \E\Big[g_{i,1}(u^1,...,u^i;a\theta^*+u_1^0,\theta^*)g_{j,1}(u^1,...,u^j;a\theta^*+u_1^0,\theta^*)\Big]\\
    \E[(w^i_T)_1(w^j_T)_1] &=\E\Big[g_{i,1}(u^1_T,...,u^i_T;\sqrt\delta((u^0_T)_1 - (\xi^T_{-1,-1})_{21}\theta^\ast),\theta^*)g_{j,1}(u^1_T,...,u^j_T;\sqrt\delta((u^0_T)_1 - (\xi^T_{-1,-1})_{21}\theta^\ast),\theta^*)\Big].
\end{align*}
Using that $g_{i,1},g_{j,1}$ both are Lipschitz and thus have linear growth, we know that, assuming $i \leq j$, the function $(u^1_T,...,u^j_T,v_1,v_2) \mapsto g_{i,1}(u^1_T,...,u^i_T;v_1,v_2)g_{j,1}(u^1_T,...,u^j_T;v_1,v_2)$ is at most quadratic growth, so $\E[(w^i_T)_1(w^j_T)_1] \rightarrow_{T\rightarrow\infty} \E[(w^i)_1(w^j)_1]$ follows from the $W_2$ convergence of the induction hypothesis. The same reasoning also yields
\begin{align*}
    \lim_{T\rightarrow\infty} \E[(w^i_T)_2(w^j_T)_2] = \E[(w^i)_2(w^j)_2], \quad
    \lim_{T\rightarrow\infty} \E[(w^i_T)_1(w^j_T)_2] = \E[(w^i)_1(w^j)_2],
\end{align*}
which concludes $\Sigma_T \rightarrow \Sigma$, as desired.
\qed

\subsubsection{Proof of Lemma \ref{lem: conv hd amp}}
    We can check this statement inductively. In the proof we will use $C>0$ to denote a constant independent of $T,n,d$ but may change line to line. We note that by model Assumption \ref{ass: model assumptions}, for large enough $C$, we have $\pnorm{\X}{\rm op} < C$ almost surely for large $n,d$.
    \paragraph{Base Case}
    The first AMP iterate of the second stage is
    \begin{align}
    \label{eq:b_0}
        \b^0_T &= \X\left(\sqrt{\delta}((\a_T^0)_1 - (\xi_{-1,-1}^T)_{21}\btheta^*),\btheta^*\right) + \frac{1}{\delta}f_{-1}^T(\b^{-T}_T,...,\b^{-1}_T;\z)\zeta^T_{0,-1}.
    \end{align}
    By the nonlinearity choice (\ref{eq:nonlinearity_first_stage}), we have $(\b_T^0)_2 = \X\btheta^* = (\b^0)_2$, verifying the claim for the second component. For the first component, by definition of $\b^0_1$ we need to verify 
    \begin{equation*}
        \lim_{T\rightarrow\infty}\lim_{n,d\rightarrow\infty}\frac{1}{\sqrt{n}}\pnorm{(\b^0_T)_1 - \b_1^0}{} = \lim_{n,d\rightarrow\infty} \frac{1}{\sqrt{n}}\biggpnorm{(\b^0_T)_1 - \Big(\X\btheta^0-\frac{1}{\lambda_\delta^*}\bZ_s \X\btheta^0\Big)}{} \rightarrow 0.
    \end{equation*}
    Using (\ref{eq:b_0}) and that $\xi_{0,-1}^T = [\sqrt{\delta},0;0,0]$, we can write the first component as
    \begin{equation*}
        (\b_T^0)_1 = \X\sqrt{\delta}\big((\a_T^0)_1 - (\xi_{-1,-1}^T)_{21}\btheta^*\big) + \frac{1}{\delta}f_{-1,1}^T(\b^{-T}_T,...,\b^{-1}_T;\z)\sqrt{\delta}. 
    \end{equation*}
    
    Then we can bound
    \begin{align*}
        \lim_{T\rightarrow\infty}\lim_{n,d\rightarrow\infty} \frac{1}{\sqrt{n}}\biggpnorm{(\b^0_T)_1 - \Big(\X\btheta^0-\frac{1}{\lambda_\delta^*}\bZ_s \X\btheta^0\Big)}{}&\leq\underbrace{\lim_{T\rightarrow\infty}\lim_{n,d\rightarrow\infty}\frac{1}{\sqrt{n}}\bigpnorm{\X\sqrt{\delta}\big((\a_T^0)_1 - (\xi_{-1,-1}^T)_{21}\btheta^*\big)-\X\btheta^0}{}}_{(I)}\\
        &\quad+\underbrace{\lim_{T\rightarrow\infty}\lim_{n,d\rightarrow\infty}\frac{1}{\sqrt{n}}\pnorm{\frac{1}{\delta}f_{-1,1}^T(\b^{-T}_T,...,\b^{-1}_T;\z)\sqrt{\delta} + \frac{1}{\lambda_\delta^\ast}\bZ_s\X\btheta^0}{}}_{(II)}.
    \end{align*}
    Now $(I) = 0$ by Lemma \ref{lem: AMP stage 1} and $(II) = 0$ by (\ref{eq: conv first onsager}), concluding the base case.
    \paragraph{Induction Step} Suppose that the lemma statement holds for the iterates $\b_T^0,...,\b_T^i$ and $\a_T^0,...,\a_T^i$. We now check that $\lim_{T\rightarrow\infty}\lim_{d\rightarrow\infty}\frac{1}{d}\pnorm{(\a_T^{i+1})_1-\a^{i+1}_1}{}^2 \rightarrow 0$ and $\lim_{T\rightarrow\infty}\lim_{d\rightarrow\infty}\frac{1}{d}\pnorm{(\a_T^{i+1})_2-\a^{i+1}_2}{}^2 \rightarrow 0$. We can bound $\pnorm{(\a_T^{i+1})_1-\a^{i+1}_1}{}$ by
    \begin{align}
    \begin{split}
        &\bigpnorm{(\a_T^{i+1})_1-\a^{i+1}_1}{}^2 \leq C \Bigpnorm{\X^\top \big(f_{i,1}(\b_T^0,...,\b_T^i;\z) - f_{i,1}(\b^0,...,\b^i;\z)\big)}{}^2\\
        &+ 
        C\sum_{j=0}^i \Bigpnorm{g_{j,1}\Big(\a_T^1,...,\a_T^j;\sqrt{\delta}\big((\a_T^0)_1-(\xi_{-1,-1}^T)_{21}\btheta^*\big),\btheta^*\Big)(\xi_{i,j}^T)_{11}-g_{j,1}\Big(\a^1,...,\a^j;\btheta^0,\btheta^*\Big)(\xi_{i,j})_{11}}{}^2\\
        &+
        C\sum_{j=0}^i \Bigpnorm{g_{j,2}\Big(\a_T^1,...,\a_T^j;\sqrt{\delta}\big((\a_T^0)_1-(\xi_{-1,-1}^T)_{21}\btheta^*\big),\btheta^*\Big)(\xi_{i,j}^T)_{21}-g_{j,2}\Big(\a^1,...,\a^j;\btheta^0,\btheta^*\Big)(\xi_{i,j})_{21}}{}^2.
    \end{split}
    \end{align}
    Using the inductive assumption and the fact that $f_{i,1}$ is Lipschitz, we can bound the first term by
    \begin{align*}
    \lim_{n,d\rightarrow\infty}\Bigpnorm{\X^\top (f_{i,1}(\b_T^0,...,\b_T^i;\z) - f_{i,1}(\b^0,...,\b^i;\z))}{}^2 &\leq C \lim_{n,d\rightarrow\infty} \pnorm{\X}{\op}^2 \Bigpnorm{(\b_T^0,...,\b_T^i)-(\b^0,...,\b^i)}{F}^2 \xrightarrow{T} 0,
    \end{align*}
    since $\pnorm{\X}{\op}^2$ is almost surely bounded for large enough $n,d$.

    Next, we can bound
    \begin{align}
        \begin{split}
            &\Bigpnorm{g_{j,1}\Big(\a_T^1,...,\a_T^j;\sqrt{\delta}((\a_T^0)_1-(\xi_{-1,-1}^T)_{21}\btheta^*),\btheta^*\Big)(\xi_{i,j}^T)_{11}-g_{j,1}\Big(\a^1,...,\a^j;\btheta^0,\btheta^*\Big)(\xi_{i,j})_{11}}{}^2\\
            &\leq C \underbrace{\Bigpnorm{g_{j,1}\Big(\a_T^1,...,\a_T^j;\sqrt{\delta}((\a_T^0)_1-(\xi_{-1,-1}^T)_{21}\btheta^*),\btheta^*\Big)}{}^2 \big|(\xi_{i,j}^T)_{11} - (\xi_{i,j})_{11}\big|^2}_{(I)}\\
            &\quad + C\underbrace{\Bigpnorm{g_{j,1}\Big(\a_T^1,...,\a_T^j;\sqrt{\delta}((\a_T^0)_1-(\xi_{-1,-1}^T)_{21}\btheta^*),\btheta^*\Big) - g_{j,1}\Big(\a^1,...,\a^j;\btheta^0,\btheta^*\Big)}{}^2|(\xi_{i,j})_{11}|^2}_{(II)}.
        \end{split}
    \end{align}
    Using $\xi_{i,j}^T \rightarrow \xi_{i,j}$ via (\ref{eq:converge_onsager}), we have $\lim_{T\rightarrow\infty}\limsup_{n,d\rightarrow\infty} (I) = 0$.
    Next using that $g_{j,1}$ is Lipschitz and the inductive assumption, we also have $\lim_{T\rightarrow\infty}\lim_{n,d\rightarrow\infty} (II) = 0$, which concludes that
    \begin{equation*}
    \lim_{T\rightarrow\infty}\lim_{n,d\rightarrow\infty}\sum_{j=0}^i \Bigpnorm{g_{j,1}\Big(\a_T^1,...,\a_T^j;\sqrt{\delta}((\a_T^0)_1-(\xi_{-1,-1}^T)_{21}\btheta^*),\btheta^*\Big)(\xi_{i,j}^T)_{11}-g_{j,1}\Big(\a^1,...,\a^j;\btheta^0,\btheta^*\Big)(\xi_{i,j})_{11}}{}^2 = 0.
    \end{equation*}
    A similar argument shows that for any $j=0,...,i$,
    \begin{equation*}
        \lim_{T\rightarrow\infty} \lim_{n,d\rightarrow\infty}\Bigpnorm{g_{j,1}\Big(\a_T^1,...,\a_T^j;\sqrt{\delta}((\a_T^0)_1-(\xi_{-1,-1}^T)_{21}\btheta^*),\btheta^*\Big) - g_{j,1}\Big(\a^1,...,\a^j;\btheta^0,\btheta^*\Big)}{}^2|(\xi_{i,j})_{11}|^2 = 0,
    \end{equation*}
    so we can conclude that almost surely $\lim_{T\rightarrow\infty}\lim_{n,d\rightarrow\infty}\frac{1}{d}\pnorm{(\a_T^{i+1})_1-\a^{i+1}_1}{}^2 \rightarrow 0$. The same argument shows that $\lim_{T\rightarrow\infty}\lim_{n,d\rightarrow\infty}\frac{1}{d}\pnorm{(\a_T^{i+1})_2-\a^{i+1}_2}{}^2 \rightarrow 0$.
    A similar line of reasoning shows that if the lemma statement holds for the iterates $\b_T^0,...,\b_T^i$ and $\a_T^0,...,\a_T^{i+1}$, then the lemma statement also holds for $\b_T^{i+1}$, completing the induction.
\qed

\section{Proof of artificial DMFT characterization (Proposition \ref{prop:discreteDmftArt})}\label{appendix:proof_power_dmft}

The proof is similar to that of Theorem \ref{thm: discrete dmft asymp}, but based on reduction to the AMP algorithm (\ref{eq:random_init_amp}) stated in Section \ref{sec: spec init amp} with independent initialization. We set the external randomness $\v \in \R^{d\times 2},\s\in \R^n$ therein to be
\begin{equation*}
    \v = (\btheta^\ast,\btheta^\ast) \qquad \s = \z.
\end{equation*}
The following lemma sets the nonlinearites $g_j,f_j$ used in the AMP algorithm.

\begin{lemma}
\label{lem:amp_nonlin_art}
    For $j\geq -T$, there exist continuous non-linearities $g_j^T=(g^T_{j,1},g^T_{j,2}):\R^{2(j+T)}\times \R^2\rightarrow\R^2$ and $f_j^T=(f_{j,1}^T,f_{j,2}^T):\R^{2(j+T+1)}\times \R\rightarrow\R^2$ such that $g_j^T(a^{-T+1}_T,\ldots,a^j_T;\theta^\ast,\theta^\ast)$ is Lipschitz in $\big((a^{-T+1}_T)_1,\ldots, (a^j_T)_1\big)$, $f_j^T(b^{-T}_T, \ldots, b^j_T;z)$ is Lipschitz in $\big((b^{-T}_T)_1,\ldots, (b^{j}_T)_1\big)$, and
    \begin{align}
        \label{eq:postprocessing_g_art}
        g^T_j(\a^{-T+1}_T,...,\a^j_T;\btheta^\ast,\btheta^*) &= (\btheta^j_T,\btheta^*) &j\geq -T\\
        \label{eq:postprocessing_f_art}
        f^T_j(\b^{-T}_T,...,\b^{j}_T;\z) &= \Big(\diag(\cT_s(\y))\X\btheta^j_T,\mathbf{0}_n\Big) &j<0\\
    \label{eq:postprocessing_f_art_2}
         f^T_j(\b^{-T}_T,...,\b^{j}_T;\z) &= \Big(\ell(\X\btheta^j_T,\X\btheta^*,\z),\mathbf{0}_n\Big)  &j\geq 0.
    \end{align}
\end{lemma}
\begin{proof} 
    \noindent\textbf{Base Case}
    Define $g_{-T}^T(\theta^\ast,\theta^*) = (\theta^\ast,\theta^*)$. Then $g_{-T}^T(\btheta^\ast,\btheta^*) = (\btheta^\ast,\btheta^*) = (\btheta_T^{-T},\btheta^\ast)$ verifying (\ref{eq:postprocessing_g_art}). Having defined $g_{-T}^T$, the next AMP iterate is $\b_T^{-T} = \big(\X\btheta_T^{-T},\X\btheta^\ast\big)$. Thus we can define
    \begin{equation*}
        f_{-T}^T(b_{T}^{-T},z) = \Big(\cT_s(\varphi(b_{T,2}^{-T},z))b_{T,1}^{-T},0\Big),
    \end{equation*}
    satisfying (\ref{eq:postprocessing_f_art}).
    
    \noindent\textbf{Inductive Step (First Stage $j < 0$)} Assume we have defined up to $f_{j-1}^T$ and $g_{j}^T$ where $j< 0$. In this step we will first define $f_{j}^T$ and then $g_{j+1}^T$. By the inductive assumption, the first component of the next AMP iterate is
    \begin{align*}
        \b_{T,1}^{j} &= \X g_{j,1}^T(\a^{-T+1}_T,...,\a^j_T;\btheta^\ast,\btheta^\ast) + \frac{1}{\delta} \sum_{i=-T}^{j-1} f_{i,1}^T(\b^{-T}_T,...,\b^i_T;\z)(\zeta^T_{j,i})_{11}\\
        &=\X \btheta_T^{j} + \frac{1}{\delta} \sum_{i=-T}^{j-1} f_{i,1}^T(\b^{-T}_T,...,\b^i_T;\z)(\zeta^T_{j,i})_{11}.
    \end{align*}
    Hence to satisfies (\ref{eq:postprocessing_f_art}), we can define
    \begin{equation}
    \label{def:f_postprocess_art_1}
        f_j^T(b^{-T}_T,...,b^{j}_T;z) = \bigg(\cT_s\big(\varphi(b_{T,2}^{-T},z)\big)\bigg[b_{T,1}^{j} - \frac{1}{\delta}\sum_{i=-T}^{j-1}f^T_{i,1}(b^{-T}_T,...,b^{i}_T;z)(\zeta^T_{ji})_{11}\bigg],0\bigg).
    \end{equation}
    
    The first component of the next AMP iterate is
    \begin{align*}
         \a^{j+1}_{T,1} &= -\frac{1}{\delta}\X^\top f^T_{j,1}(\b^{-T}_T,...,\b^j_T;\z) + \sum_{i=-T}^{j} \Big(g^T_i(\a^{-T+1}_T,...,\a^i_T;\btheta^\ast,\btheta^\ast)\xi^T_{j,i}\Big)_{\cdot 1}\\
         &=-\frac{1}{\delta}\X^\top \diag(\cT_s(\y))\X\btheta_T^j + \sum_{i=-T}^{j} g^T_{i,1}(\a^{-T+1}_T,...,\a^i_T;\btheta^\ast,\btheta^\ast)(\xi^T_{j,i})_{11}+ \sum_{i=-T}^{j} \btheta^\ast(\xi^T_{j,i})_{21}.
    \end{align*}
    In particular,
    \begin{equation}
    \label{eq:preactivation_amp_art}
    \X^\top \diag\big(\cT_s(\y)\big)\X\btheta_T^j = \delta\sum_{i=-T}^{j} g^T_{i,1}(\a^{-T+1}_T,...,\a^i_T;\btheta^\ast,\btheta^\ast)(\xi^T_{j,i})_{11}+ \delta\sum_{i=-T}^{j} \btheta^\ast(\xi^T_{j,i})_{21}-\delta\a_{T,1}^{j+1}.
    \end{equation}
    Using (\ref{eq:preactivation_amp_art}) and the state evolution in Theorem \ref{thm:standard_amp_se}, we see that 
    \begin{equation}\label{eq:hat_beta}
        \hat{\beta}_T^{j+1} = \sqrt{\E\bigg[\delta\sum_{i=-T}^{j}g^T_{i,1}(u^{-T+1}_T,...,u^i_T;\theta^\ast,\theta^\ast)(\xi_{j,i}^T)_{11} + \delta\sum_{i=-T}^{j}\theta^\ast(\xi^T_{j,i})_{21}-\delta (u_T^{j+1})_1\bigg]^2}.
    \end{equation}
    This quantity is well-defined as it only uses previously constructed state evolution iterates so we can use it to define $g_{j+1}^T$. Since $\btheta^{j+1}_T = \X^\top \diag(\cT_s(\y))\X\btheta^j_T/\hat\beta^{j+1}_T$ for all $j\leq -1$, we can define
    \begin{equation}
    \label{def:g_postprocess_art_1}
        g_{j+1}^T(a^{-T+1}_T,...,a^{j+1}_T;\theta_T^\ast,\theta^*) = \bigg(\frac{1}{\hat\beta_T^{j+1}}\bigg[\delta\sum_{i=-T}^{j}g_{i,1}(a^{-T+1}_T,...,a^i_T;\theta^\ast,\theta^*)(\xi_{j,i}^T)_{11} + \delta\sum_{i=-T}^{j}\theta^*(\xi^T_{j,i})_{21}-\delta (a_T^{j+1})_1\bigg],\theta^*\bigg)
    \end{equation}
    satisfying (\ref{eq:postprocessing_g_art}). This specifies the nonlinearities up to $g^T_0$ and $f^T_{-1}$.

    \noindent\textbf{Inductive Step (Second Stage $j \geq 0$)}
    Now suppose that we have defined up to $f_{j-1}^T$ and $g_{j}^T$ for some $j \geq 0$. We now define $f_j^T$ and $g_{j+1}^T$. The first component of the next AMP iterate is 
    \begin{align}
    \label{eq:def_preactivation}
        \notag\b_{T,1}^{j} &= \X g_{j,1}^T(\a^{-T+1}_T,...,\a^j_T;\btheta^\ast,\btheta^\ast) + \frac{1}{\delta} \sum_{i=-T}^{j-1} f_{i,1}^T(\b^{-T}_T,...,\b^i_T;\z)(\zeta^T_{j,i})_{11}\\
        &=\X \btheta_T^{j} + \frac{1}{\delta} \sum_{i=-T}^{j-1} f_{i,1}^T(\b^{-T}_T,...,\b^i_T;\z)(\zeta^T_{j,i})_{11}.
    \end{align}
    Then to satisfy (\ref{eq:postprocessing_f_art_2}), using the fact that $(\b^0_T)_2 = \X\btheta^\ast$, we can define
            \begin{align}
    \begin{split}
    \label{def:f_postprocess_art_2}
        f_{j,1}^T(b^{-T}_T,...,b^{j}_T;z) &= \ell\Big((b_T^{j})_1 -\frac{1}{\delta} \sum_{i=-T}^{j-1} f^T_{i,1}(b^{-T}_T,...,b^{i}_T;z)(\zeta^T_{j,i})_{11},(b_T^{0})_2,z\Big),\\
        f^T_{j,2}(b^{-T}_T,...,b^{j}_T;z) &= 0.
    \end{split}
    \end{align}
    The first component of the next AMP iterate is
    \begin{align*}
         \a^{j+1}_{T,1} &= -\frac{1}{\delta}\X^\top f^T_{j,1}(\b^{-T}_T,...,\b^j_T;\z) + \sum_{i=-T}^{j} \Big(g^T_i(\a^{-T+1}_T,...,\a^i_T;\btheta^\ast,\btheta^\ast)\xi^T_{j,i}\Big)_{\cdot 1}\\
         &=-\frac{1}{\delta}\X^\top \ell\Big(\X\btheta_T^j,\X\btheta^\ast,\z\Big) + \sum_{i=-T}^{j} g^T_{i,1}(\a^{-T+1}_T,...,\a^i_T;\btheta^\ast,\btheta^\ast)(\xi^T_{j,i})_{11}+ \sum_{i=-T}^{j} \btheta^\ast(\xi^T_{j,i})_{21},
    \end{align*}
    so
    \begin{equation*}
        \X^\top \ell\Big(\X\btheta_T^j,\X\btheta^\ast,\z\Big) = \delta\sum_{i=-T}^{j} g^T_{i,1}(\a^{-T+1}_T,...,\a^i_T;\btheta^\ast,\btheta^\ast)(\xi^T_{j,i})_{11}+ \delta\sum_{i=-T}^{j} \btheta^\ast(\xi^T_{j,i})_{21}-\delta\a^{j+1}_{T,1}.
    \end{equation*}
    On the other hand, we have
    \begin{align*}
        \btheta^{j+1}_T&= \btheta^{j}_T - \gamma\X^\top \ell(\X\btheta^{j}_T,\X\btheta^*,\z) - \gamma\lambda \btheta^{j}_T\\
        &= (1-\gamma\lambda)g^T_{j,1}(\a_T^{-T+1},...,\a_T^j;\btheta^\ast,\btheta^*) + \gamma\delta\Big[\a_{T,1}^{j+1} - \sum_{i=-T}^j g^T_{i,1}(\a_T^{-T+1},...,\a_T^i;\btheta^\ast,\btheta^*)(\xi^T_{ji})_{11} - \sum_{i=-T}^j \btheta^* (\xi^T_{ji})_{21}\Big],
    \end{align*}
    Thus with
        \begin{align}
        \notag g_{j+1,1}^T(a^{-T+1}_T,...,a^{j+1}_T;\theta^\ast,\theta^*) &= (1-\gamma\lambda)g_{j,1}^T(a^{-T+1}_T,...,a^{j}_T;\theta^\ast,\theta^*)\\
        \notag&\quad + \gamma\delta\Big[(a^{j+1}_T)_1 - \sum_{i=-T}^{j} g_{i,1}^T(a^{-T+1}_T,...,a^{i}_T;\theta^\ast,\theta^*)(\xi^T_{j,i})_{11} - \sum_{i=-T}^{j} \theta^* (\xi^T_{j,i})_{21}\Big],\\
        \label{def:g_postprocess_art_2}
        g_{j+1,2}^T(a^{-T+1}_T,...,a^{j+1}_T;\theta^\ast,\theta^*) &= \theta^*,
    \end{align}
    we can verify that this definition satisfies (\ref{eq:postprocessing_g_art}), completing the induction. For the regularity requirements, $f_j$ and $g_j$ in the first stage satisfy them under boundedness of $\cT_s(\cdot)$, and those in the second stage satisfy them using the same reasoning as in Lemma \ref{lm: AMP nonlinearity choice}.
\end{proof}

With this AMP algorithm, we can now complete the proof.
\begin{proof}[Proof of Proposition \ref{prop:discreteDmftArt}]
    Let $\{f_j^T\},\{g_j^T\}$ be defined as in Lemma \ref{lem:amp_nonlin_art}. We first construct the low-dimensional counterparts $\theta_T^i$ and $\eta_T^i$ of $\btheta_T^i$ and $\bbeta_T^i$. By construction of $g_i^T$, we know that
    \begin{equation}
        g_{i,1}^T(\a_T^{-T+1},...,\a_T^i;\btheta^\ast,\btheta^\ast) = \btheta_T^i,
    \end{equation}
    so we define
    \begin{equation}\label{eq:art_dmft_theta}
        \theta_{T,\gamma}^{i} := g^T_{i,1}(u_T^{-T+1},...,u_T^i;\theta^\ast,\theta^\ast).
    \end{equation}
    Next, recall from (\ref{eq:def_preactivation}) that
    \begin{equation}
        \X\btheta_T^i =\b_{T,1}^i-\frac{1}{\delta}\sum_{j=-T}^{i-1}f_{j,1}^T(\b_T^{-T},...,\b_T^j;\z)(\zeta_{i,j}^T)_{11}.
    \end{equation}
    Then we define
    \begin{equation}\label{eq:art_dmft_eta}
        \eta_{T,\gamma}^i := w_{T,1}^i -\frac{1}{\delta}\sum_{j=-T}^{i-1}f^T_{j,1}(w^{-T}_T,...,w^{j}_T;z)(\zeta^T_{ij})_{11}.
    \end{equation}
    This definition satisfies
    \begin{equation}
        f_{j,1}^T(w_T^{-T},...,w_T^j;z) =\begin{cases}
            \cT_s(\varphi(w_{T,2}^{-T},z))\eta_{T,\gamma}^j &j<0\\
            \ell(\eta_{T,\gamma}^j,w_{T,2}^0,z) &j\geq 0.
        \end{cases}
    \end{equation}
    Applying the AMP state evolution (\ref{thm:standard_amp_se}), for any fixed integer $m\geq -T$, almost surely as $n,d \rightarrow\infty$, we have that
    \begin{align}
        \frac{1}{d}\sum_{j=1}^d \delta_{\left((\theta^{-T+1}_T)_j,...,(\theta^{m}_T)_j,(\theta^\ast)_j,(\theta^{*})_j\right)} &\xrightarrow{W_2} \mathsf{P}\left(\theta_T^{-T+1},...,\theta_T^{m},\theta^\ast,\theta^*\right)\\
        \frac{1}{n}\sum_{i=1}^n\delta_{\left((\eta^{-T}_T)_i,...,(\eta_T^{m})_i,(\eta^{*})_i,z_i\right)} &\xrightarrow{W_2} \sP\left(\eta_T^{-T},...,\eta_T^{m},w_T^*,z\right),
    \end{align}
    where $w_T^\ast := u_{T,2}^0 = u_{T,2}^{-T}.$
    We define the following quantities (on the left) using the state evolution variables (on the right):
\begin{equation}
\begin{gathered}
    \{u_{T,\gamma}^{i}\}_{i\geq -T}:= \delta \{u_{T,1}^{i+1}\}_{i\geq -T},  \qquad(w_{T,\gamma}^\sharp,w^*_{T,\gamma}, \{w_\gamma^{i}\}_{i\geq -T}) := (w_{T,2}^{-T},w^0_{T,2}, \{w^i_{T,1}\}_{i\geq -T}), \\
    \Big\{r^\gamma_{T,\theta}(i,j)\Big\}_{i > j \geq -T} := \Big\{\frac{1}{\delta}\frac{\partial g^T_{i,1}}{\partial u^{j+1}_{T,1}}\Big\}_{i>j\geq -T},\\
    \Big\{r^\gamma_{T,\eta}(i,j)\Big\}_{i>j\geq -T} := \Big\{\frac{\partial f^T_{i,1}}{\partial w_{T,1}^j}\Big\}_{i>j\geq -T}, \quad \Big\{r^\gamma_{T,\eta}(i,*)\Big\}_{i\geq 0} := \Big\{\frac{\partial f^T_{i,1}}{\partial w_{T,2}^{0}}\Big\}_{i \geq 0}, \quad \Big\{r^\gamma_{T,\eta}(i,\sharp)\Big\}_{i\geq -T} := \Big\{\frac{\partial f^T_{i,1}}{\partial w_{T,2}^{-T}}\Big\}_{i \geq -T} \label{eq:discrete_equiva_4}.
\end{gathered}
\end{equation}
We also define 
\begin{align}
\begin{split}
    R^\gamma_{T,\theta}(i,j) &:= \E\Big[r^\gamma_{T,\theta}(i,j)\Big] = \frac{1}{\delta}\E\Big[\frac{\partial g^T_{i,1}}{\partial u^{j+1}_{T,1}}\Big] = \frac{1}{\delta}(\zeta_{ij}^T)_{11},\\
    R^\gamma_{T,\eta}(i,j) &:= \delta\E\Big[r^\gamma_{T,\eta}(i,j)\Big] = \delta\E\Big[\frac{\partial f^T_{i,1}}{\partial w_{T,1}^j}\Big] = \delta(\xi_{ij}^T)_{11},\\
    R^\gamma_{T,\eta}(i,*) &:= \delta\E\Big[r^\gamma_{T,\eta}(i,*)\Big] = \delta\E\Big[\frac{\partial f^T_{i,1}}{\partial w_{T,2}^0}\Big] = \delta(\xi_{i0}^T)_{21},\\
    R^\gamma_{T,\eta}(i,\sharp) &:= \delta\E\Big[r^\gamma_{T,\eta}(i,\sharp)\Big] = \delta\E\Big[\frac{\partial f^T_{i,1}}{\partial w_{T,2}^{-T}}\Big] = \delta(\xi_{i(-T)}^T)_{21}.
\end{split}
\end{align}
Following a similar calculation as in (\ref{eq:check_covariance}), we have
\begin{align*}
\E[u^i_{T,\gamma} u^j_{T,\gamma}] = \E[u^i_Tu^i_T], \quad \E[w^i_{T,\gamma}w^j_{T,\gamma}] = \E[w^i_{T}w^j_{T}], \quad \E[w^i_{T,\gamma}w^\ast_{T,\gamma}] = \E[w^i_{T}w^\ast_{T}], \quad w^\ast_{T,\gamma} = w^\sharp_{T,\gamma}, 
\end{align*}
with the variables on the right given by (\ref{eq:w_art}) and (\ref{eq:u_art}), indicating that $\{u^i_{T,\gamma}\}_{i\geq -T} \overset{d}{=} \{u^i_{T}\}_{i\geq -T}$ and $(w^\ast_{T,\gamma}, w^\sharp_{T,\gamma}, \{w^i_{T,\gamma}\}_{i\geq -T}) \overset{d}{=} (w^\ast_{T}, w^\sharp_{T}, \{w^i_{T}\}_{i\geq -T})$. 

Next we specify the recursions of $r_{T,\theta}^{\gamma}(t,s), r_{T,\eta}^\gamma(t,s), r^\gamma_{T,\eta}(t,\ast), r^\gamma_{T,\eta}(t,\sharp)$.

\noindent\textbf{Recursion of $r_{T,\eta}^\gamma(t,s)$} For the case $s < t < 0$, using $\frac{\partial f^T_{s,1}}{\partial w^s_{T,1}} = \cT_s(y^\ast_{T,\gamma})$, we have by (\ref{def:f_postprocess_art_1}) that
\begin{align*}
r_{T,\eta}^\gamma(t,s) &= \cT_s(y^\ast_{T,\gamma})\Big[-\frac{1}{\delta}\sum_{\ell=s}^{t-1} \frac{\partial f^T_{\ell,1}}{\partial w^s_{T,1}} (\zeta^T_{t\ell})_{11}\Big] = \cT_s(y^\ast_{T,\gamma})\Big[-\sum_{\ell=s+1}^{t-1} r_{T,\eta}^{\gamma}(\ell,s)R_\theta^\gamma(t,\ell) - \cT_s(y^\ast_{T,\gamma})R_\theta^\gamma(t,s)\Big].
\end{align*}
For the case $t\geq 0 > s$, we have by (\ref{def:f_postprocess_art_2}) that
\begin{align*}
r_{T,\eta}^\gamma(t,s) &= \partial_1\ell(\eta^t_{T,\gamma},w^\ast_{T,\gamma},z)\Big[-\frac{1}{\delta}\sum_{\ell=s}^{t-1} \frac{\partial f^T_{\ell,1}}{\partial w^s_{T,1}} (\zeta^T_{t\ell})_{11}\Big]\\
&= \partial_1\ell(\eta^t_{T,\gamma},w^\ast_{T,\gamma},z)\Big[-\sum_{\ell=s+1}^{t-1} r_{T,\eta}^{\gamma}(\ell,s)R_\theta^\gamma(t,\ell) - \cT_s(y^\ast_{T,\gamma})R_\theta^\gamma(t,s)\Big].
\end{align*}
For the case $t > s\geq 0$, we have by (\ref{def:f_postprocess_art_2}) and $\frac{\partial f^T_{s,1}}{\partial w^s_{T,1}} = \partial_1 \ell(\eta^t_\gamma, w^\ast_{T,\gamma}, z)$ that
\begin{align*}
r_{T,\eta}^\gamma(t,s) = \partial_1\ell(\eta^t_{T,\gamma},w^\ast_{T,\gamma},z)\Big[-\sum_{\ell=s+1}^{t-1} r_{T,\eta}^{\gamma}(\ell,s)R_\theta^\gamma(t,\ell) - \partial_1 \ell(\eta^t_\gamma, w^\ast_{T,\gamma}, z) R_\theta^\gamma(t,s)\Big].
\end{align*}
\noindent\textbf{Recursion of $r_{T,\eta}^\gamma(t,\sharp)$} If $t < 0$, we have by (\ref{def:f_postprocess_art_1}) that
\begin{align*}
r_{T,\eta}^\gamma(t,\sharp) &= \cT_s'(y^\ast_{T,\gamma})\varphi'(w^\ast_{T,\gamma},z)\eta^t_{T,\gamma} - \cT_s(y^\ast_{T,\gamma}) \frac{1}{\delta} \sum_{\ell=-T}^{t-1} \frac{\partial f_{\ell,1}^T}{\partial w^{-T}_{T,2}}(\zeta_{t\ell}^T)_{11}\\
&=\cT_s'(y^\ast_{T,\gamma})\varphi'(w^\ast_{T,\gamma},z)\eta^t_{T,\gamma} -  \cT_s(y^\ast_{T,\gamma}) \sum_{\ell=-T}^{t-1} r^\gamma_{T,\eta}(\ell,\sharp)R_{T,\theta}^\gamma(t,\ell).
\end{align*}
If $t \geq 0$, we have by (\ref{def:f_postprocess_art_2}) that
\begin{align*}
r_{T,\eta}^\gamma(t,\sharp) = \partial_1 \ell(\eta^t_{T,\gamma}, w^\ast_{T,\gamma},z) \Big[-\frac{1}{\delta}\sum_{\ell=-T}^{t-1} \frac{\partial f_{\ell,1}^T}{\partial w^{-T}_{T,2}}(\zeta_{t\ell}^T)_{11}\Big] = -\partial_1 \ell(\eta^t_{T,\gamma}, w^\ast_{T,\gamma},z)\sum_{\ell=-T}^{t-1} r_{T,\eta}^\gamma(\ell,\sharp)R_{T,\theta}^\gamma(t,\ell).
\end{align*}
\noindent\textbf{Recursion of $r_{T,\eta}^\gamma(t,\ast)$} By (\ref{def:f_postprocess_art_2}) we have
\begin{align*}
r_{T,\eta}^\gamma(t,\ast) &= \partial_1\ell(\eta^t_{T,\gamma}, w^\ast_{T,\gamma},z)\Big[-\frac{1}{\delta}\sum_{\ell=0}^{t-1} \frac{\partial f_{\ell,1}^T}{\partial w^{0}_{T,2}}(\zeta_{t\ell}^T)_{11}\Big] + \partial_2 \ell(\eta^t_{T,\gamma}, w^\ast_{T,\gamma},z)\\
&= -\partial_1\ell(\eta^t_{T,\gamma}, w^\ast_{T,\gamma},z)\sum_{\ell=0}^{t-1} r_{T,\eta}^\gamma(\ell,\ast)R_{T,\theta}^\gamma(t,\ell) + \partial_2 \ell(\eta^t_{T,\gamma}, w^\ast_{T,\gamma},z).
\end{align*}
\noindent \textbf{Recursion of $r_{T,\theta}^\gamma(t,s)$} We first consider the case $t \geq 0$. If $t = s$ (in which case $s\geq 0$) it holds by (\ref{def:g_postprocess_art_2}) that $r_{T,\theta}^\gamma(s+1,s) = \gamma$, and for $t\geq s+1$ we have
\begin{align*}
\frac{\partial g_{t+1,1}^T}{\partial u^{s+1}_{T,1}} &= (1-\gamma\lambda)\frac{\partial g_{t,1}^T}{\partial u^{s+1}_{T,1}} +  \gamma\delta \Big[-\sum_{r=s+1}^t \frac{\partial g_{r,1}^T}{\partial u^{s+1}_{T,1}} (\xi^T_{tr})_{11} \Big]\\
&= \Big(1-\gamma\lambda - \gamma \delta (\xi^T_{tt})_{11}\Big)\frac{\partial g_{t,1}^T}{\partial u^{s+1}_{T,1}} - \gamma\sum_{r=s+1}^{t-1} \frac{\partial g_{r,1}^T}{\partial u^{s+1}_{T,1}}R_{T,\eta}^\gamma(t,r)\\
&= \Big(1-\gamma\lambda - \gamma \delta \E[\cT_s(y^\ast_{T,\gamma})]\Big)\frac{\partial g_{t,1}^T}{\partial u^{s+1}_{T,1}} - \gamma\sum_{r=s+1}^{t-1} \frac{\partial g_{r,1}^T}{\partial u^{s+1}_{T,1}}R_{T,\eta}^\gamma(t,r).
\end{align*}
This implies
\begin{align*}
r_{T,\theta}^{\gamma}(t+1,s) = \Big(1-\gamma\lambda - \gamma \delta \E[\cT_s(y^\ast_{T,\gamma})]\Big)r_{T,\theta}^\gamma(t,s) - \gamma\sum_{r=s+1}^{t-1} R_{T,\eta}^\gamma(t,r) r_{T,\theta}^\gamma(r,s).
\end{align*}
For the case $t < 0$, it holds by (\ref{def:g_postprocess_art_1}) that $r_{T,\eta}^\gamma(s+1,s) = -1/\beta^{s+1}_T$, and for $t\geq s+1$, 
\begin{align*}
\frac{\partial g_{t+1,1}^T}{\partial u^{s+1}_{T,1}} = \frac{1}{\beta^{t+1}_T} \Big[\delta\sum_{r=s+1}^t \frac{\partial g_{r,1}^T}{\partial u^{s+1}_{T,1}} (\xi_{tr}^T)_{11}\Big] = \frac{1}{\beta^{t+1}_T}\Big[\sum_{r=s+1}^{t-1} \frac{\partial g_{r,1}^T}{\partial u^{s+1}_{T,1}} R_{T,\eta}^\gamma(t,r) + \delta\frac{\partial g_{t,1}^T}{\partial u^{s+1}_{T,1}}\E[\cT_s(y^\ast_{T,\gamma})] \Big].
\end{align*}
This implies
\begin{align*}
r_{T,\theta}^\gamma(t+1,s) = \frac{1}{\beta^{t+1}_T}\Big[\sum_{r=s+1}^{t-1} r_{T,\theta}^\gamma(r,s) R_{T,\eta}^\gamma(t,r) + \delta r_{T,\theta}^\gamma(t,s)\E[\cT_s(y^\ast_{T,\gamma})] \Big].
\end{align*}
Lastly we derive the recursion for $\{\theta^t_{T,\gamma}\}$ in (\ref{eq:art_dmft_theta}) and $\{\eta^t_{T,\gamma}\}$ in (\ref{eq:art_dmft_eta}). For $\{\theta^t_{T,\gamma}\}$, in the case $t < 0$, we have by (\ref{def:g_postprocess_art_1}) that
\begin{align*}
\theta^{t+1}_{T,\gamma} &= \frac{1}{\beta_T^{j+1}}\bigg[\delta\sum_{i=-T}^t g_{i,1}(u^{-T+1}_T,...,u^i_T;\theta^\ast,\theta^*)(\xi_{j,i}^T)_{11} + \delta\sum_{i=-T}^t\theta^*(\xi^T_{j,i})_{21}-\delta (u_T^{t+1})_1\bigg]\\
&= \frac{1}{\beta_T^{j+1}}\bigg[\sum_{i=-T}^t \theta^i_{T,\gamma}R_{T,\eta}^\gamma(t,i) + R_{T,\eta}^\gamma(t,\sharp) \theta^\ast - u^t_{T,\gamma}\bigg]\\
&= \frac{1}{\beta_T^{j+1}}\bigg[\delta\E[\cT_s(y^\ast_{T,\gamma})]\theta^t_{T,\gamma} + \sum_{i=-T}^{t-1} \theta^i_{T,\gamma}R_{T,\eta}^\gamma(t,i) + R_{T,\eta}^\gamma(t,\sharp) \theta^\ast - u^t_{T,\gamma}\bigg]
\end{align*}
where we use the fact that $\sum_{i=-T}^t(\xi^T_{t,i})_{21} = (\xi_{t,-T}^T)_{21} = R_{T,\eta}^\gamma(t,\sharp)/\delta$. For the case $t\geq 0$, we have by (\ref{def:g_postprocess_art_2})
\begin{align*}
\theta^{t+1}_{T,\gamma} &= (1-\gamma\lambda)g_{t,1}^T(u^{-T+1}_T,...,u^t_T;\theta^\circ,\theta^*)\\
        &\quad + \gamma\delta\Big[(u^{t+1}_T)_1 - \sum_{i=-T}^t g_{i,1}^T(u^{-T+1}_T,...,u^{i}_T;\theta^\circ,\theta^*)(\xi^T_{t,i})_{11} - \theta^\ast\sum_{i=-T}^t (\xi^T_{t,i})_{21}\Big]\\
        &= (1-\gamma\lambda)\theta^t_{T,\gamma} + \gamma\Big[u^t_{T,\gamma} - \sum_{i=-T}^t \theta^i_{T,\gamma}R_{T,\eta}^\gamma(t,i) -  \delta((\xi_{t,-T}^T)_{21} + (\xi_{t,0}^T)_{21})\theta^\ast\Big]\\
        &=(1-\gamma\lambda)\theta^t_{T,\gamma} + \gamma\Big[-\delta\E[\partial_1\ell(\eta^t_{T,\gamma},w^\ast_{T,\gamma},z)]\theta^t_{T,\gamma} - \sum_{i=-T}^{t-1}\theta^i_{T,\gamma}R_{T,\eta}^\gamma(t,i) - (R_{T,\eta}^\gamma(t,\sharp) + R_{T,\eta}^\gamma(t,\ast))\theta^\ast + u^t_{T,\gamma}\Big],
\end{align*}
where we apply the fact that $(\xi_{ti}^T)_{21} = 0$ except for $i = -T$ and $i = 0$.

Lastly, for $\{\eta^t_{T,\gamma}\}$, when $t \leq 0$ we have by (\ref{def:f_postprocess_art_1}) that
\begin{align*}
\eta^t_{T,\gamma} = w_{T,1}^t - \frac{1}{\delta}\sum_{i=-T}^{t-1}f^T_{i,1}(w^{-T}_T,...,w^{i}_T;z)(\zeta^T_{ti})_{11} = w^t - \sum_{i=-T}^{t-1} \cT_s(y^\ast_{T,\gamma})\eta^i_{T,\gamma}R_{T,\theta}^\gamma(t,i).
\end{align*}
When $t > 0$, we have by (\ref{def:f_postprocess_art_2}) that
\begin{align*}
\eta^t_{T,\gamma} &= w_{T,1}^t - \frac{1}{\delta}\sum_{i=-T}^{t-1}f^T_{i,1}(w^{-T}_T,...,w^{i}_T;z)(\zeta^T_{ti})_{11}\\
&= w^t_{T,\gamma} - \sum_{i=-T}^{-1} \cT_s(y^\ast_{T,\gamma})\eta^i_{T,\gamma}R_{T,\theta}^\gamma(t,i) - \sum_{i=0}^{t-1} \ell(\eta^i_{T,\gamma},w^\ast_{T,\gamma},z)R_{T,\theta}^\gamma(t,i).
\end{align*}
In view of the above recursions and the expression of $\hat\beta^t_T$ in (\ref{eq:hat_beta}), a simple induction argument yields that $y^\ast_{T,\gamma} \overset{d}{=} y_T$, and
\begin{equation*}
\begin{gathered}
R_{T,\theta}^\gamma(i,j) = R_{T,\theta}(i,j), \quad R_{T,\eta}^\gamma(i,j) = R_{T,\eta}(i,j), \quad R_{T,\eta}^\gamma(i,\ast) = R_{T,\eta}(i,\ast), \quad R_{T,\eta}^\gamma(i,\sharp) = R_{T,\eta}(i,\sharp), \\
\theta^{t}_{T,\gamma} \overset{d}{=} \theta^t_{T}, \quad \eta^{t}_{T,\gamma} \overset{d}{=} \eta^t_{T}, \quad \hat\beta^t_T = \beta^t_T,
\end{gathered}
\end{equation*}
which concludes the proof.
\end{proof}

\section{Local convexity of phase retrieval (Proposition \ref{lm:pr_exp_tti_ass})}
\label{sec:phase_conv}

Note that \cite{li2020toward} uses a slightly different (but equivalent) scaling than ours. To better apply their technical lemmas, define
    \begin{align*}
        \tilde{\x}_i = \sqrt{d}\x_i, \quad \tilde{\btheta}^\ast = \frac{1}{\sqrt{d}} \btheta^\ast,\quad \tilde{\btheta} = \frac{1}{\sqrt{d}} \btheta,
    \end{align*}
    where $\tilde\btheta$ is the new optimizing variable, $\tilde\x_i \sim \N(0,\I_d)$, and $\pnorm{\tilde\btheta^\ast}{} = 1$. Then we have the correspondence $ \tilde{\x}_i \leftrightarrow \a_i, \tilde{\btheta}^\ast \leftrightarrow \x, \tilde{\btheta} \leftrightarrow \z$, where the notation on the right is in \cite{li2020toward}. With these definitions, the loss function in (\ref{eq:smoothed_WF_loss}) is
    \begin{align*}
    \cL_n(\tilde\btheta) = \frac{1}{2}\sum_{i=1}^n \Big((\tilde\x_i^\top \tilde\btheta)^2-(|\tilde\x_i^\top\tilde\btheta^\ast|)^2\Big)^2 h\big((\tilde\x_i^\top\tilde\btheta)^2\big)h\big((|\tilde\x_i^\top\tilde\btheta^\ast|)^2\big).
    \end{align*}
    In contrast, \cite{li2020toward} uses
    \begin{align*}
    \tilde\cL_n(\tilde\btheta) = \frac{1}{2}\sum_{i=1}^n \Big((\tilde\x_i^\top \tilde\btheta)^2-(|\tilde\x_i^\top\tilde\btheta^\ast|)^2\Big)^2 h\Big(\frac{(\tilde\x_i^\top\tilde\btheta)^2}{\pnorm{\tilde\btheta}{}^2}\Big)h\Big(\frac{n(|\tilde\x_i^\top\tilde\btheta^\ast|)^2}{\pnorm{\y}{1}}\Big).
    \end{align*}
    The approximation $\pnorm{\y}{1}/n \approx \E[\pnorm{\y}{1}/n] = 1$ is justified by \cite[Lemma 3.3]{li2020toward}, and the approximation eror brought by the $1/\pnorm{\tilde\btheta}{}^2$ scaling is handled by the following lemma.

    \begin{lemma}
    \label{lm:phase_loss_comparison}
        For any $\tilde{\btheta}$ such that $\pnorm{\tilde{\btheta}-\tilde{\btheta}^
        \ast}{} < \frac{1}{5}$, fixed non-negative integers $s,t$, and given $\eps>0$, we have for $g\in \{h,h',h''\}$ that
        \begin{equation*}
            \biggpnorm{\frac{1}{n}\sum_{i=1}^n(\tilde{\x}_i^\top\tilde{\btheta})^s(\tilde{\x}_i^\top\tilde{\btheta}^\ast)^t\bigg[g\Big((\tilde{\x}_i^\top\tilde{\btheta})^2\Big)-g\Big(\frac{(\tilde{\x}_i^\top\tilde{\btheta})^2}{\pnorm{\tilde{\btheta}}{}^2}\Big)\bigg]h\big((\tilde{\x}_i^\top\tilde{\btheta}^\ast)^2\big)\tilde{\x}_i\tilde{\x}_i^\top}{\rm op} \leq CU^{(t+s)/2}(\sqrt{L}e^{-cL} + \epsilon),
        \end{equation*}
        with probability at least $1-e^{-c(\eps)n}$ when $n\geq C(\eps)d$, where $c(\eps)$ (resp. $C(\eps)$) is some small (resp. large) enough constant depending on $\eps$.
    \end{lemma}
    \begin{proof}
        Since $\frac{4}{5} \leq \pnorm{\tilde{\btheta}}{} \leq \frac{6}{5}$ and $\pnorm{h}{\infty}\vee\pnorm{h'}{\infty}\vee\pnorm{h''}{\infty} \leq 1$, we have
        \begin{align*}
            \bigg|g\Big((\tilde{\x}_i^\top\tilde{\btheta})^2\Big)-g\Big(\frac{(\tilde{\x}_i^\top\tilde{\btheta})^2}{\pnorm{\tilde{\btheta}}{}^2}\Big)\bigg| &\leq \mathbf{1}\bigg(\frac{4}{5}\sqrt{L}\leq |\tilde{\x}_i^\top\tilde{\btheta}|\leq \frac{6}{5}\sqrt{U}\bigg),
        \end{align*}
        which further implies
        \begin{align*}
        |\tilde{\x}_i^\top\tilde{\btheta}|^s\bigg|g\Big((\tilde{\x}_i^\top\tilde{\btheta})^2\Big)-g\Big(\frac{(\tilde{\x}_i^\top\tilde{\btheta})^2}{\pnorm{\tilde{\btheta}}{}^2}\Big)\bigg|&\leq |\tilde{\x}_i^\top\tilde{\btheta}|^s\mathbf{1}\bigg(\frac{4}{5}\sqrt{L}\leq (\tilde{\x}_i^\top\tilde{\btheta})^2\leq \frac{6}{5}\sqrt{U}\bigg) \leq C U^{s/2}\mathbf{1}\bigg(\frac{4}{5}\sqrt{L}\leq |\tilde{\x}_i^\top\tilde{\btheta}|\bigg).
        \end{align*}
        Using also $|\tilde{\x}_i^\top\tilde{\btheta}^\ast|^th((\tilde{\x}_i^\top\tilde{\btheta}^\ast)^2) \leq CU^{t/2}$, we have
    \begin{align*}
        &\biggpnorm{\frac{1}{n}\sum_{i=1}^n(\tilde{\x}_i^\top\tilde{\btheta})^s(\tilde{\x}_i^\top\tilde{\btheta}^\ast)^t\bigg[g\Big((\tilde{\x}_i^\top\tilde{\btheta})^2\Big)-g\Big(\frac{(\tilde{\x}_i^\top\tilde{\btheta})^2}{\pnorm{\tilde{\btheta}}{}^2}\Big)\bigg]h((\tilde{\x}_i^\top\tilde{\btheta}^\ast)^2)\tilde{\x}_i\tilde{\x}_i^\top}{\rm op}\\
        &\leq CU^{(t+s)/2} \biggpnorm{\frac{1}{n}\sum_{i=1}^n\mathbf{1}\bigg(\frac{4}{5}\sqrt{L}\leq |\tilde{\x}_i^\top\tilde{\btheta}|\bigg)\tilde{\x}_i\tilde{\x}_i^\top}{\rm op} \leq CU^{(t+s)/2}(\sqrt{L}e^{-cL} + \epsilon),
    \end{align*}
    where the last step applies \cite[Lemma A.2]{li2020toward}.
    \end{proof}

\begin{proof}[Proof of Proposition \ref{lm:pr_exp_tti_ass}]
    Let us first check that $\pnorm{\hat\btheta^s - \tilde\btheta^\ast}{}\leq 1/5$. Recall that $\hat\btheta^s$ is the unit-norm leading eigenvector of the matrix $\M_n = \sum_{i=1}^n \x_i\x_i^\top (\x_i^\top \btheta^\ast)^2\bm{1}\{|\x_i^\top \btheta^\ast| \leq U\}$, which satisfies 
    \begin{align*}
    \E[\M_n] = \delta(\xi_1 \tilde\btheta^\ast\tilde\btheta^{\ast\top} + \xi_2 \I_d),
    \end{align*}
    with $\xi_1 = \E_{Z\sim \N(0,1)}[Z^4\bm{1}\{|Z|\leq U\}] - \E_{Z\sim \N(0,1)}[Z^2\bm{1}\{|Z|\leq U\}] > 0$ and $\xi_2 = \E_{Z\sim \N(0,1)}[Z^2\bm{1}\{|Z|\leq U\}]$. By the Davis-Kahan theorem, we have
    \begin{align}\label{eq:davis_khan}
    \pnorm{\hat\btheta^s - \tilde\btheta^\ast}{} \leq C\frac{\pnorm{\M_n - \E[\M_n]}{\op}}{\Delta(\E[\M_n])}, 
    \end{align}
    where $\Delta(\E[\M_n]) = \delta\xi_1$ is the difference between the first and second eigenvalue of $\E[\M_n]$. To estimate the numerator, first note that $\pnorm{\M_n - \E[\M_n]}{\op} \leq 2\sup_{u\in \N_d} \big|u^\top (\M_n - \E[\M_n])u\big|$ where $\N_d$ is a $1/4$-net of the unit ball in $\R^d$ with cardinality at most $9^d$. Now $u^\top (\M_n - \E[\M_n])u = \sum_{i=1}^n (Z_i - \E[Z_i])$ with $Z_i = (\x_i^\top u)^2 (\x_i^\top\btheta^\ast)\bm{1}\{|\x_i^\top \btheta^\ast|\leq U\}$ being sub-exponential with norm bounded by $C/d$, hence Bernstein inequality yields
    \begin{align*}
    \P(|u^\top (\M_n - \E[\M_n])u|\geq t) \leq 2\exp\Bigg(-c\Big(\frac{d^2t^2}{n} \wedge dt\Big)\Bigg)
    \end{align*}
    for any $t\geq 0$. Taking the union bound over $u\in\N$ and choosing $t = \delta^{2/3}$ yields that $\pnorm{\M_n - \E[\M_n]}{\op} \leq \delta^{2/3}$ with probability $1-e^{-cn}$ for large enough $\delta$. Coming back to (\ref{eq:davis_khan}), we conclude that $\pnorm{\hat\btheta^s - \tilde\btheta^\ast}{} \leq C/\delta^{1/3} \leq 1/5$ for large enough $\delta$.

    Next, we check that $\psi'_\delta(\lambda^\ast_\delta) > 0$. Using Proposition \ref{prop: spec est overlap}, we see that it is sufficient to check that the limiting overlap is nonzero, i.e.
    \begin{equation*}
        \lim_{n,d\rightarrow \infty}\frac{1}{\sqrt{d}}\langle\hat{\btheta}^s,\btheta^\ast\rangle > 0.
    \end{equation*} 
    For sufficiently large $\delta$, we can bound the overlap by
    \begin{align*}
        \frac{1}{\sqrt{d}}\langle\hat{\btheta}^s,\frac{\btheta^\ast}{\sqrt{d}}\rangle &= \frac{1}{2} \Big(\pnorm{\hat{\btheta}^s}{}^2 + \Bigpnorm{\frac{\btheta^\ast}{\sqrt{d}}}{}^2 - \Bigpnorm{\hat{\btheta}^s-\frac{\btheta^\ast}{\sqrt{d}}}{}^2\Big)\geq \frac{1}{2} \Big(\pnorm{\hat{\btheta}^s}{}^2 + \biggpnorm{\frac{\btheta^\ast}{\sqrt{d}}}{}^2 -\frac{1}{25}\Big) = 1-1/50 > 0,
    \end{align*}
    which holds for large enough $n,d$, as desired.
    
    Now we lower bound the smallest eigenvalue of the Hessian. Using
    \begin{align}\label{eq:part_1_ell_phase}
        \partial_1\ell(a,b) = 2(3a^2-b^2)h(a^2)h(b^2) + (9a^2 - b^2)(a^2-b^2) h'(a^2)h(b^2) +2a^2(a^2-b^2)^2h''(a^2)h(b^2),
\end{align}    
    in terms of $(\tilde\x_i, \tilde\btheta^\ast,\tilde\btheta)$ the Hessian is 
    \begin{align*}
        \nabla_{\btheta}^2 L(\btheta) &= \underbrace{\frac{1}{n}\sum_{i=1}^n2\Big(3(\tilde{\x}_i^\top\tilde{\btheta})^2-(\tilde{\x}_i^\top\tilde\btheta^\ast)^2\Big)h((\tilde{\x}_i^\top\tilde{\btheta})^2)h((\tilde{\x}_i^\top\btheta^\ast)^2)\tilde{\x}_i\tilde{\x}_i^\top}_{(I)}\\
        &\quad+ \underbrace{\frac{1}{n}\sum_{i=1}^n\Big(9(\tilde{\x}_i^\top\tilde{\btheta})^2 - (\tilde\x_i^\top \tilde\btheta^\ast)^2\Big)\Big((\tilde{\x}_i^\top\tilde{\btheta})^2-(\tilde{\x}_i^\top\btheta^\ast)^2\Big)h'((\tilde{\x}_i^\top\tilde{\btheta})^2)h((\tilde{\x}_i^\top\btheta^\ast)^2)\tilde\x_i\tilde\x_i^\top}_{(II)}\\
        &\quad+\underbrace{\frac{1}{n}\sum_{i=1}^n2(\tilde{\x}_i^\top\tilde{\btheta})^2((\tilde{\x}_i^\top\tilde{\btheta})^2-(\tilde{\x}_i^\top\btheta^\ast)^2)^2h''((\tilde{\x}_i^\top\tilde{\btheta})^2)h((\tilde{\x}_i^\top\btheta^\ast)^2)\tilde\x_i\tilde\x_i^\top}_{(III)}.
    \end{align*}
    
    We will now lower bound $\lambda_{\min}((I))$ and upper bound $\pnorm{(II)}{\op},\pnorm{(III)}{\op}$. Using two applications of Lemma \ref{lm:phase_loss_comparison}, we have the bound
    \begin{align*}
    \biggpnorm{(I) - \frac{1}{n}\sum_{i=1}^n 2\Big(3(\tilde{\x}_i^\top\tilde{\btheta})^2-(\tilde{\x}_i^\top\tilde{\btheta}^\ast)^2\Big)h\Big(\frac{(\tilde{\x}_i^\top\tilde{\btheta})^2}{\pnorm{\tilde{\btheta}}{}^2}\Big)h\Big((\tilde{\x}_i^\top\tilde{\btheta}^\ast)^2\Big)\tilde{\x}_i\tilde{\x}_i^\top}{\op} \leq CU(\sqrt{L}e^{-\kappa_2 L} + \eps).
    \end{align*} 
    By Eq. (9) in the proof of Theorem 2.1 in \cite{li2020toward}, it holds with probability $1-e^{-cn}$ that
    \begin{align*}
     \lambda_{\min}\bigg(\frac{1}{n}\sum_{i=1}^n 2(3(\tilde{\x}_i^\top\tilde{\btheta})^2-(\tilde{\x}_i^\top\tilde{\btheta}^\ast)^2)h\Big(\frac{(\tilde{\x}_i^\top\tilde{\btheta})^2}{\pnorm{\tilde{\btheta}}{}^2}\Big)h\Big((\tilde{\x}_i^\top\tilde{\btheta}^\ast)^2\Big)\tilde{\x}_i\tilde{\x}_i^\top\Big)\bigg) \geq \frac{2}{25}- CU^{\kappa_3}(\eps+\eps^{-1}e^{-\kappa_1\eps^{-2}}+e^{-\kappa_2L}).
    \end{align*}
    Then for $U = 2L$ and both large enough, it holds with probability $1-e^{-cn}$ that $\lambda_{\min}((I)) \geq 1/25$. On the other hand, terms $(II)$ and $(III)$ are both linear combinations of terms of the form
    \begin{align*}
    B_{s,t}(\tilde{\btheta}) := \frac{1}{n}\sum_{i=1}^n (\tilde{\x}_i^\top \tilde{\btheta})^s(\tilde{\x_i}^\top \tilde\btheta^\ast)^t g\big((\tilde{\x}_i^\top \tilde{\btheta})^2\big)h\big((\tilde\x_i^\top \tilde\btheta^\ast)\big),
    \end{align*}
    where $s,t$ are nonnegative integers, and $g = h',h''$ satisfies $g(x) = 0$ on $[0,L]\cup[U,\infty)$ and $|g(x)|\leq 1$ for $x\in [L,U]$. Further define
    \begin{align*}
    B_{s,t}'(\tilde\btheta) &:= \frac{1}{n}\sum_{i=1}^n (\tilde{\x}_i^\top \tilde{\btheta})^s(\tilde{\x_i}^\top \tilde\btheta^\ast)^t g\Big(\frac{(\tilde{\x}_i^\top \tilde{\btheta})^2}{\pnorm{\tilde\btheta}{}^2}\Big)h\big((\tilde\x_i^\top \tilde\btheta^\ast)^2\big),\\
    B_{s,t}''(\tilde\btheta) &:= \frac{1}{n}\sum_{i=1}^n (\tilde{\x}_i^\top \tilde{\btheta})^s(\tilde{\x_i}^\top \tilde\btheta^\ast)^t g\Big(\frac{(\tilde{\x}_i^\top \tilde{\btheta})^2}{\pnorm{\tilde\btheta}{}^2}\Big)h\Big(\frac{n(\tilde\x_i^\top \tilde\btheta^\ast)^2}{\pnorm{\y}{1}}\Big).
    \end{align*}
    By Lemma \ref{lm:phase_loss_comparison} and \cite[Lemma 3.3]{li2020toward}, for any preset $\eps > 0$, by choosing $U = 2L$ and both large enough, we have
    \begin{align*}
    \pnorm{B_{s,t}(\tilde\btheta) - B'_{s,t}(\tilde\btheta)}{\op} \vee \pnorm{B'_{s,t}(\tilde\btheta) - B''_{s,t}(\tilde\btheta)}{\op} \leq \eps
    \end{align*}
    for any $\tilde\btheta$ such that $\pnorm{\tilde\btheta - \tilde\btheta^\ast}{}\leq 1/5$, so using further $\pnorm{B''_{s,t}(\tilde\btheta)}{\op} \leq \eps$ for the same set of $\tilde\btheta$ via \cite[Lemma 3.4]{li2020toward}, we conclude that with probability $1-e^{-cn}$ it holds that $\pnorm{(II)}{\op} + \pnorm{(III)}{\op} \leq 1/50$. By choosing $\eps$ small enough, this concludes that $\lambda_{\min}(\nabla_{\btheta}^2 \cL_n(\btheta)) \geq 1/50$ for $\pnorm{\tilde\btheta - \tilde\btheta^\ast}{}\leq 1/5$ with probability $1-e^{-cn}$, as desired.

    Finally we derive the system of fixed point equations. Since $\btheta^\infty = \btheta^\ast$ by \cite{li2020toward}, we have $\theta^\infty = \theta^\ast$, which further gives
    \begin{align*}
        \bbeta^\infty = \X\btheta^\infty = \X\btheta^\ast = \bbeta^\ast, \quad
        \eta^\infty = w^\infty = w^\ast.
    \end{align*}
    By (\ref{eq:fix3}), we can conclude that $u^\infty = 0.$ By substituing (\ref{eq:fix6}) into (\ref{eq:fix5}) and some further algebra, we can derive the equation
    \begin{align}
    \label{eq:pr_fix_1}
    \frac{1}{\delta} = \E\Big[\frac{\partial_1\ell(w^\ast,w^\ast,0)R_\theta^\infty}{1+\partial_1\ell(w^\ast,w^\ast,0)R_\theta^\infty}\Big]
    \end{align}
    for $R_\theta^\infty$. We claim it admits an unique solution. By some further algebra, the equation is equivalent to 
    \begin{equation}
        \label{eq:R_theta}F(R_\theta^\infty) = 1-\frac{1}{\delta}, \text{ where } F(x) = \E\Big[\frac{1}{1+\partial_1\ell(w^\ast,w^\ast,0)x}\Big].
    \end{equation}
    As $\partial_1\ell(w^\ast,w^\ast,0) > 0$, $F$ is strictly monotone so the equation (\ref{eq:R_theta}) admits an unique solution. The equation 
    \begin{align}
    \label{eq:pr_fix_2}
    R_\eta^\infty &= \frac{1}{R_\theta^\infty} - \delta\E\Big[\frac{\partial_1\ell(w^\ast,w^\ast,0)}{1+\partial_1\ell(w^\ast,w^\ast,0)}\Big] 
    \end{align} follows directly from (\ref{eq:fix6}). To get the final equation, note that $\partial_2\ell(w^\ast,w^\ast,0) = -\partial_1\ell(w^\ast,w^\ast,0)$, so using further (\ref{eq:fix7}), we get
    \begin{align*}
        R_\eta^\ast &= -\delta\E\Big[\frac{\partial_1 \ell(\eta^\infty,w^*,0)}{1+\partial_1 \ell(\eta^\infty,w^*,0)R_\theta^\infty}\Big] = -\delta\E[\partial_1\ell(w^\ast,w^\ast,0)] - R_\eta^\infty,
    \end{align*}
    where the last step follows from algebra, (\ref{eq:pr_fix_1}), and (\ref{eq:pr_fix_2}).
\end{proof}

\bibliographystyle{abbrvnat}
\bibliography{dmft}
\end{document}